\documentclass[10pt,a4paper,reqno]{amsart}
\usepackage{amsmath,amsthm,amsfonts,amssymb,euscript,bbm,color,slashed}

\usepackage{marvosym,accents}

\newtheorem{theorem}{Theorem}[section]
\newtheorem{proposition}[theorem]{Proposition}
\newtheorem{lemma}[theorem]{Lemma}
\newtheorem{remark}[theorem]{Remark}
\newtheorem{definition}[theorem]{Definition}
\newtheorem{corollary}[theorem]{Corollary}
\newtheorem{conjecture}[theorem]{Conjecture}
\newtheorem{example}[theorem]{Example}
\newtheorem{problem}[theorem]{Problem}

\usepackage[margin=1in,dvips]{geometry}

\numberwithin{equation}{section}
\numberwithin{theorem}{section}

\usepackage{bbm}
\usepackage{graphics}
\usepackage{enumerate}
\def\ub {\underline{u}}
\def\th {\theta}

\def\Hb {\underline{H}}
\def\chib {\underline{\chi}}
\def\chih {\hat{\chi}}
\def\chibh {\hat{\underline{\chi}}}
\def\omegab {\underline{\omega}}
\def\etab {\underline{\eta}}
\def\betab {\underline{\beta}}

\def\hot{\widehat{\otimes}}

\def\sigmac{\check{\sigma}}

\def\mub{\underline{\mu}}

\def\alp {\alpha}

\def\bt {\beta}

\def\nab {\slashed{\nabla}}

\def\ep {\epsilon}
\def\om {\omega}
\def\omb {\underline{\omega}}

\def\f {\frac}

\def\rd {\partial}
\def\ls {\lesssim}
\def\de {\delta}
\def\i {\infty}
\def\Om {\Omega}
\def\Omg{\Omega}
\def\nub{\underline{\nu}}
\newcommand{\ud}{\mathrm{d}}

\def\Ub{\underline{U}}

\newcommand{\pfstep}[1]{\vspace{.5em} {\it \noindent #1.}}

\newcommand{\bea}{\begin{eqnarray}}
\newcommand{\eea}{\end{eqnarray}}
\def\beaa{\begin{eqnarray*}}
\def\eeaa{\end{eqnarray*}}

\renewcommand{\div}{\slashed{\mathrm{div}}}
\newcommand{\curl}{\slashed{\mathrm{curl}}}
\newcommand{\trchb}{\slashed{\mathrm{tr}}\chib}
\def\trch{\slashed{\mathrm{tr}}\chi}
\newcommand{\tr}{\slashed{\mathrm{tr}}}

\begin{document}

\title[High-frequency limits and null dust shells]{High-frequency limits and null dust shell solutions\\ in general relativity}

\begin{abstract}
Consider the characteristic initial value problem for the Einstein vacuum equations \emph{without any symmetry assumptions}. Impose a sequence of data on two intersecting null hypersurfaces, each of which is foliated by spacelike $2$-spheres. Assume that the sequence of data is such that the derivatives of the metrics along null directions are only uniformly bounded in $L^2$ but the derivatives of the metrics along the directions tangential to the $2$-spheres obey higher regularity bounds uniformly. By the results in [J.~Luk and I.~Rodnianski, \emph{Nonlinear interaction of impulsive gravitational waves for the vacuum Einstein equations}, Camb.~J.~Math.~5(4), 2017], it follows that the sequence of characteristic initial value problems gives rise to a sequence of vacuum spacetimes $(\mathcal M, g_n)$ in a fixed double-null domain $\mathcal M$. Since the existence theorem requires only very low regularity, the sequence of solutions may exhibit both oscillations and concentrations, and the limit need not be vacuum. We prove nonetheless that, after passing to a subsequence, the metrics converge in $C^0$ and weakly in $W^{1,2}$ to a solution of the Einstein--null dust system with two families of (potentially measure-valued) null dust.

We show moreover that all sufficiently regular solutions to the Einstein--null dust system (with potentially measure-valued null dust) adapted to a double null coordinate system arise locally as weak limits of solutions to the Einstein vacuum system in the manner described above. As a consequence, we also give the first general local existence and uniqueness result for solutions to the Einstein--null dust system for which the null dusts are only measures. This in particular includes as a special case solutions featuring propagating and interacting shells of null dust.

\end{abstract}

\author{Jonathan Luk}
\address{Department of Mathematics, Stanford University, Stanford, CA 94305, USA}
\email{jluk@stanford.edu}
\author{Igor Rodnianski}
\address{Department of Mathematics, Princeton University, Princeton, NJ 08544, USA}
\email{irod@math.princeton.edu}

  \maketitle

\tableofcontents

\section{Introduction}

This paper studies two circles of problems in general relativity, namely the problem of high-frequency limits and the problem of null dust shell solutions. We moreover show that there is a close relationship between the two problems. Here is a summary of what we achieve in this paper:
\begin{enumerate}
\item The first circle of problems (see Section~\ref{sec:problem.1}) concerns \textbf{high-freqeuncy limits of vacuum solutions}, i.e.~we seek to understand ``effective matter field'' that arises in appropriately defined weak limits of solutions to the Einstein vacuum equations. For ``angularly regular'' spacetimes adapted to a double null foliation (but \emph{without any symmetry assumptions}), we give a complete characterization of possible high-frequency limits. Namely we show that all high-frequency limits are isometric to solutions to the Einstein--null dust system; and conversely, all solutions to the Einstein--null dust system also arise locally as high-frequency limits of vacuum spacetimes.
\item The second circle of problems (see Section~\ref{sec:problem.2}) concerns \textbf{null dust shell solutions}, i.e.~solutions to the Einstein--null dust system with a ``shell of null dust''  for which the stress-energy-momentum tensor is a delta measure on an embedded null hypersurface. We prove an existence and uniqueness result (again \emph{with no symmetry assumptions}) for the Einstein--null dust system which describes solutions featuring propagation and interaction of null dust shells (and also more general solutions where the null dust is measure-valued).
\item We show that the problem of high-frequency limits and the problem of null dust shells are closely related. In fact, they can both be studied and understood from the point of view of \textbf{low-regularity problems of the Einstein equations}. In particular, in this paper we study these problems using the low-regularity local existence and uniqueness result in our previous papers \cite{LR, LR2} (which was originally developed to understand the propagation and interaction of impulsive gravitational waves, i.e.~solutions to the Einstein vacuum equation such that some curvature components admit a delta function singularity on an embedded null hypersurface).
\item Moreover, our construction of null dust shell solutions is based on studying vacuum solutions and then taking appropriate high-frequency limits. Put differently, using the characterization of high-frequency limit, existence and uniqueness of \emph{solutions to the Einstein--null dust} can be established by studying \emph{vacuum solutions}. See Section~\ref{sec:results.dust}.
\item Conversely, we also illustrate through the example of \textbf{formation of trapped surfaces} how the study of the Einstein--null dust system illuminates our understanding of the Einstein vacuum equations. See Section~\ref{sec:intro.addendum}.
\end{enumerate}

We will explain this further in the remainder of the introduction and give a first descriptions of the main results. In \textbf{Section~\ref{sec:problems}}, we will first introduce the two circles of problems regarding high-frequency limits and null dust shell solutions. In \textbf{Section~\ref{sec:intro.main.results}}, we then give a first description of the main results in this paper, and explain how they relate to \cite{LR, LR2}. We then discuss some related works in \textbf{Section~\ref{sec:related.works}} and the ideas of the proof in \textbf{Section~\ref{sec:proof.intro}}.

\subsection{The problems}\label{sec:problems}
\subsubsection{High-frequency limits}\label{sec:problem.1}

Consider a sequence of $(3+1)$-dimensional spacetimes $\{(\mathcal M, g_n)\}_{n=1}^{+\infty}$ which solve the Einstein vacuum equations:
\begin{equation}\label{EVE}
Ric_{\mu \nu}(g_n)=0.
\end{equation} 
Suppose that $g_n \to g_\infty$ in $C^0_{loc}$ and the derivatives of $g_n$ converge \underline{weakly}\footnote{Remark that if the convergence of the derivatives of $g_n$ is \underline{strong} in $L^2_{loc}$, then the limit is necessarily vacuum.} in $L^2_{loc}$. Explicit examples are known (see for instance \cite{Burnett, GW2}) in symmetry classes such that the limit $(\mathcal M, g_\infty)$ may satisfy the Einstein equations
\begin{equation}\label{EE}
Ric_{\mu \nu}(g_\i)-\frac 12  (g_\i)_{\mu\nu} R(g_\i) = T_{\mu\nu},
\end{equation}
with a \emph{non-vanishing stress-energy-momentum tensor} $T_{\mu\nu}$, where $R(g_\i)$ is the scalar curvature of the limit metric. Physically, $T_{\mu\nu}$ can be interpreted as an effective stress-energy-momentum tensor arising from limits of high-frequency gravitational waves. Mathematically, $T_{\mu\nu}$ can be thought of as a defect measure that arises because taking weak limits do not commute with taking products.

This phenomenon raises the following question:
\begin{problem}\label{prob:Burnett.general}
Give a description of the non-vanishing stress-energy-momentum tensors that arise in the limiting process as described above.
\end{problem}

There are two guiding conjectures concerning Problem~\ref{prob:Burnett.general}, both of which were introduced by Burnett. In \cite{Burnett}, Burnett considered more restrictive assumptions on the convergence of $g_n$. Namely, he required that for some $C>0$, $\lambda_n\to 0$,
\begin{equation}\label{eq:Burnett.assumptions}
|g_n - g_\infty|\leq \lambda_n,\quad  |\rd g_n|\leq C,\quad |\rd^2 g_n|\leq C\lambda_n^{-1}
\end{equation}
in some local coordinate system.

Under these assumptions, Burnett made the following conjectures \cite{Burnett}:
\begin{conjecture}[Burnett's conjecture]\label{conj:Burnett}
Any such limit $(\mathcal M, g_\infty)$ is isometric to a solution to the Einstein--massless Vlasov system for some appropriate choice of Vlasov field.
\end{conjecture}

\begin{conjecture}[Reverse Burnett's conjecture]\label{conj:reverse.Burnett}
Any solution to the Einstein--massless Vlasov system arises as a limit of solutions to the Einstein vacuum equations in the sense described above.
\end{conjecture}
Here, ``Einstein--massless Vlasov system'' is to be interpreted in an appropriate generalized sense which allows the Vlasov field to be a measure on the cotangent bundle. In particular, it includes the known examples for which the limit is described by the Einstein--null dust system.

Conjectures~\ref{conj:Burnett} and \ref{conj:reverse.Burnett} remain open in full generality, although there is some recent progress when $(\mathcal M, g_n)$ is assumed to be $\mathbb U(1)$ symmetric \cite{HL.HF,HL.Burnett}; see Section~\ref{sec:U(1)}. In between the full problem and the $\mathbb U(1)$ symmetric problem, it is of interest to study a setting which is not completely general, but nonetheless does not require any exact symmetries.
\begin{problem}\label{prob:Burnett}
Find a setting in $(3+1)$-dimensions \underline{without any exact symmetry} such that Conjectures~\ref{conj:Burnett} and \ref{conj:reverse.Burnett} can be studied.
\end{problem}

Beyond the original conjectures of Burnett, one can try to understand weak limits of vacuum solutions without imposing \eqref{eq:Burnett.assumptions}, but requiring only the weaker convergence ($g_n\to g_0$ uniformly and $\rd g_n \to \rd g_\infty$ weakly in $L^2$) introduced in the beginning of Section~\ref{sec:problem.1}. Notice in particular that while \eqref{eq:Burnett.assumptions} allows for oscillations in $g_n$, it prohibits concentrations. On the other hand, concentrations can in principle occur if we only require $g_n\to g_\infty$ in $C^0$ and $\rd g_n \to \rd g_\infty$ weakly in $L^2$. This motivates
\begin{problem}\label{prob:concentration}
Study (appropriate analogues of) Conjectures~\ref{conj:Burnett} and \ref{conj:reverse.Burnett} when \underline{concentrations} are present (in addition to oscillations).
\end{problem}

\subsubsection{Null shell solutions}\label{sec:problem.2}

In 1957, Synge \cite{Synge} discovered a solution to the Einstein equations describing a propagating null shell of dust. More precisely, he constructed a spacetime with a distinguished null hypersurface so that the metric is isometric to Schwarzschild to the one side of that null hypersurface, and is isometric to Minkowski to the other side of that null hypersurface. Along the separating null hypersurface, the spacetime is not vacuum. Instead, a component of the Ricci curvature is a delta function supported on this null hypersurface, and the spacetime can be thought of as containing a null shell of dust. Since then, many other explicit solutions have been discovered; see Section~\ref{sec:null.dust.in.physics} for further discussions.

In view of the explicit solutions, it is desirable to develop a local theory for null dust shells which does not impose any symmetry assumptions. Note that the difficulty in such an endeavor is that a null shell has much lower regularity than that allowed by standard local well-posedness results.
\begin{problem}\label{prob:null.shell}
Prove a local existence and uniqueness theorem for the Einstein--null dust system which incorporates the \underline{propagation} of null shells of dust.
\end{problem}

Once Problem~\ref{prob:null.shell} is understood, it is natural to further extend the class of initial data for which one can develop a local theory. Motivated by explicit solutions featuring the interaction of two null dust shells (see for instance \cite{tDgtH85, tDgtH86, iR85}), it is desirable to understand more generally the \emph{interaction} of two null shells, described by the transversal intersection of two null hypersurfaces which support the null shells of dust.
\begin{problem}\label{prob:interaction}
Prove a local existence and uniqueness theorem for the Einstein--null dust system which incorporates the \underline{interaction} of null shells of dust.
\end{problem}

From a PDE point of view, a spacetime containing a null dust shell is a solution to the Einstein--null dust system for which the stress-energy-momentum tensor of the null dust is merely a measure (which is not absolutely continuous with respect to the Lebesgue measure). From this perspective, it is of interest to study more general solutions to the Einstein--null dust system for which the stress-energy-momentum tensor is a measure with singular parts, but not necessarily a measure supported on a single null hypersurface.

\begin{problem}\label{prob:general.shell}
Construct more general solutions to the Einstein--null dust system where the null-dust is \underline{measure-valued}.
\end{problem}

Notice that while null dust shells are particular measure-valued solutions to the Einstein--null dust system, they are very special. Indeed they are so special that often in the physics literature they are constructed (under symmetry assumptions) by considering a ``junction condition'' across the null hypersurface on which the dust is supported. This is no longer the case for more general measures.

\subsection{Main results}\label{sec:intro.main.results}

In this subsection, we give the informal statements of the main results concerning the problems discussed in Sections~\ref{sec:problem.1} and \ref{sec:problem.2}; see Sections~\ref{sec:results.HF} and \ref{sec:results.dust}. Before that, however, we first discuss our previous low-regularity well-posedness result in Section~\ref{sec:intro.LR2}, which as we will show is closely related to the problems at hand. 

\subsubsection{A local well-posedness result with $L^2$ Christoffel symbols}\label{sec:intro.LR2}
We recall our earlier local well-posedness result in \cite{LR, LR2}. The setup of \cite{LR, LR2} is as follows. We seek a solution $(\mathcal M = [0,u_*]\times [0,\ub_*]\times \mathbb S^2, g)$ to the Einstein vacuum equations  in double null coordinates:
\begin{equation}\label{eq:metricform.intro}
g = -2\Omega^2(\ud u\otimes \ud\ub+\ud\ub\otimes \ud u)+\gamma_{AB}(\ud\th^A-b^A\ud u)\otimes (\ud\th^B-b^B\ud u),
\end{equation}
where $\vartheta = (\th^1, \th^2)$ is a local coordinate system on $\mathbb S^2$, $\Omega$ is a strictly positive function, $b$ is a vector field tangent to $\mathbb S^2$, and for every $(u,\ub)$, $\gamma$ a Riemannian metric on $\mathbb S^2$. 
A characteristic initial value problem (for $(\Omega, b, \gamma)$) is considered in \cite{LR, LR2}, i.e.~characteristic data are prescribed on $\Hb_0:=[0, I ] \times \{0\}\times \mathbb S^2$ and $H_0:=\{0\}\times [0,\underline{I} ] \times \mathbb S^2$. The function $\Omega$ can be arbitrarily prescribed on $H_0$ and $\Hb_0$. The vector field 
$b^A$ can be prescribed arbitrarily on $\Hb_0$ (but not $H_0$), and additionally $\f{\rd b^A}{\rd \ub}$ can be prescribed arbitrarily on the sphere $S_{0,0} := \{0\}\times\{0\}\times \mathbb S^2$. Finally, the metric $\gamma$ can be prescribed on $H_0$ and $\Hb_0$ subject to some \emph{constraints} equations (see \eqref{eq:constraints.first.time} in Section~\ref{sec:reduced.data}).


In \cite{LR2}, we consider data obeying the following estimates\footnote{In fact we only needed slightly weaker estimates, but \eqref{eq:intro.metric.bds} is slightly more concise to state. We remark also that the constraints equations together with \eqref{eq:intro.metric.bds} imply bounds for the Ricci coefficients} (where $\f{\rd }{\rd\vartheta}$ denotes $\f{\rd }{\rd\th^1}$ or $\f{\rd }{\rd\th^2}$ derivatives):
\begin{equation}\label{eq:intro.metric.bds}
\begin{split}
&\: \sum_{\mathfrak g \in \{\gamma, \log\det\gamma, \log\Om, b\}}  \sum_{i\leq 5} \|(\f{\rd }{\rd\vartheta})^i \mathfrak g \restriction_{S_{0,0}}\|_{L^2(S)} + \sum_{i\leq 5} \|(\f{\rd }{\rd\vartheta})^i \f{\rd b}{\rd\ub} \restriction_{S_{0,0}}\|_{L^2(S)}    \\
&\: + \sum_{\mathfrak g \in \{\gamma, \log\det\gamma, \log\Om, b\}}\sum_{i\leq 5}(\|(\f{\rd}{\rd\vartheta})^i\f{\rd \mathfrak g}{\rd\ub}\restriction_{H_0}\|_{L^2_{\ub} L^2(S)} + \|(\f{\rd}{\rd\vartheta})^i\f{\rd \mathfrak g}{\rd u}\restriction_{\Hb_0}\|_{L^2_{u} L^2(S)}) \leq C.
\end{split}
\end{equation}


\begin{theorem}[L.--R.~\cite{LR2}]\label{thm:existence.intro}
Given characteristic initial data satisfying the bounds \eqref{eq:intro.metric.bds}, there exists $\ep>0$ sufficiently small \textbf{depending only on $C$} such that for any $u_* \in (0,I]$ and $\ub_* \in (0,\ep]$, there exists a unique solution to the Einstein vacuum equations in double null coordinates in $[0,u_*]\times [0,\ub_*]\times \mathbb S^2$ which achieves the given data. The solution is $C^0\cap W^{1,2}$ with additional regularity in $\f{\rd}{\rd\vartheta}$ directions, with estimates \textbf{depending only on $C$} in \eqref{eq:intro.metric.bds}. 
\end{theorem}

A more precise version is given in Theorems~\ref{ext.thm} and \ref{thm:ext.est}.

The key point here is that when measured in the worst directions (i.e.~the $\f{\rd}{\rd u}$ and $\f{\rd}{\rd \ub}$ directions), the components of the metric are merely in $W^{1,2}$.

\subsubsection{Results on high-frequency limits}\label{sec:results.HF}

To study high-frequency limits, we consider exactly the setting of the results in Section~\ref{sec:intro.LR2}, where in spite of the very low regularity of the data we still have a well-posedness theory (see~Problem~\ref{prob:Burnett}). Our result on high-frequency limits is most easily formulated in the language of Theorem~\ref{thm:existence.intro}. As emphasized above, the assumptions of Theorem~\ref{thm:existence.intro} allow the components of the metric to be merely bounded in $W^{1,2}$. Therefore, given a sequence of characteristic initial data obeying uniformly the estimates in Theorem~\ref{thm:existence.intro},  the first derivatives of the metric components do not necessarily have strong limits. In particular, the limits, if they exist, can in principle have non-trivial stress-energy-momentum tensors as discussed in Section~\ref{sec:problem.1}. Our first main result shows that non-trivial stress-energy-momentum tensors must correspond to that of null dust. It can be viewed as a resolution of Conjecture~\ref{conj:Burnett} in our particular setting where the metric is adapted to a double null foliation gauge\footnote{Note that in our setting, due to the angular regularity given by Theorem~\ref{thm:existence.intro}, we only obtain null dust in the limit, as opposed to more general Vlasov field as in the case of Conjecture~\ref{conj:Burnett} in general.}.
\begin{theorem}\label{thm:limit.intro}
Take a sequence of characteristic initial data which obey the bounds in Theorem~\ref{thm:existence.intro} uniformly. Then the following holds:
\begin{enumerate}
\item There exists a sequence of metric 
$$g_n = -2\Omega_n^2(\ud u\otimes \ud\ub+\ud\ub\otimes \ud u)+(\gamma_n)_{AB}(\ud\th^A-b_n^A\ud u)\otimes (\ud\th^B-b_n^B\ud u)$$
in a uniform domain of existence $[0,u_*]\times [0,\ub_*]\times \mathbb S^2$. 
\item After passing to a subsequence $g_{n_k}$, there exists a metric
$$g_\infty = -2\Omega_\infty^2(\ud u\otimes \ud\ub+\ud\ub\otimes \ud u)+(\gamma_\infty)_{AB}(\ud\th^A-b_\infty^A\ud u)\otimes (\ud\th^B-b_\infty^B\ud u)$$
so that $g_n \to g_\infty$ in $C^0$ and weakly in $W^{1,2}$ in $[0,u_*]\times [0,\ub_*]\times \mathbb S^2$.
\item Moreover $g_\infty$ satisfies (weakly) the Einstein--null dust system with two families of null dusts which are potentially measure-valued.
\end{enumerate}
\end{theorem}
In establishing Theorem~\ref{thm:limit.intro}, the existence theorem (Theorem~\ref{thm:existence.intro}) plays two important roles:
\begin{itemize}
\item Theorem~\ref{thm:existence.intro} gives a \underline{uniform} region of existence of solutions, which allows us to study the limit in this setting. 
\item The regularity properties of the solutions proven in Theorem~\ref{thm:existence.intro} allows us to treat the limits of the nonlinear terms. In particular, the regularity properties dictate that only for specific nonlinear terms could the product of the weak limits be different from the weak limit of the product. This implies that nothing ``worse'' than two families of null dusts can arise in the limit.
\end{itemize}
Understanding the second point above in particular requires studying the effects of compensated compactness.

Moreover, the setup in Theorem~\ref{thm:limit.intro} allows \emph{concentrations} (in addition to oscillations) to occur in the limiting process. These seem to be first known examples of limiting effective stress-energy-momentum tensors being created by concentrations; see Problem~\ref{prob:concentration}.

One consequence of Theorem~\ref{thm:limit.intro} is that we know when the limiting spacetime metric is vacuum: since the null dust satisfies a transport equation, it is vanishing if and only if it has vanishing data. 
\begin{corollary}\label{cor:vac.cond.intro}
Let the sequence $g_n$, the subsequence $g_{n_k}$ and the limit $g_\infty$ be as in Theorem~\ref{thm:limit.intro}. Then the limiting spacetime metric $g_\infty$ is a (weak) solution to the Einstein vacuum equations if and only if the initial data sets for $g_{n_k}$ converge to a limiting initial data set to the Einstein vacuum equations.
\end{corollary}

In fact, slightly more can be said, and that the solution is determined only by the limit of the initial data. In other words, we have the following uniqueness theorem:
\begin{theorem}\label{thm:uniqueness.intro}
Given two sequences of characteristic initial data satisfying the assumptions of Theorem~\ref{thm:limit.intro} which moreover have the same limit on the initial characteristic hypersurfaces, then in fact the limiting spacetime metrics given by Theorem~\ref{thm:limit.intro} also coincide.
\end{theorem}

\subsubsection{Results on null shells}\label{sec:results.dust}

Once we have obtained the existence and uniqueness of the limits as solutions to the Einstein--null dust system (see Theorems~\ref{thm:limit.intro} and \ref{thm:uniqueness.intro}), we can use this to obtain an existence and uniqueness theory for the Einstein--null dust system where the null dust is merely a measure (with sufficient angular regularity). More precisely, given an initial data set to the Einstein--null dust system with a potentially measure-valued null dust, we approximate it by a sequence of initial data sets to the Einstein vacuum equations, and then use Theorem~\ref{thm:limit.intro} (!) to construct a solution to the Einstein--null dust system. Uniqueness of solutions constructed in this manner is then given by Theorem~\ref{thm:uniqueness.intro}. 

Our main result on null shells is the following existence and uniqueness theorem for the characteristic initial value problem for the Einstein--null dust system with measure-valued null dust:
\begin{theorem}\label{thm:null.shells.intro}
Consider a characteristic initial value problem with the Einstein--null dust system with strongly angularly regular\footnote{We refer the reader to Definition~\ref{def:SARCID} for the precise regularity assumptions. For now we just emphasize that the angular regularity that we require for these characteristic initial data is stronger than the angular regularity for the solutions. We thus distinguish then with the terms ``angularly regular spacetimes'' and ``strongly angularly regular data''. This is related to a well-known loss of derivative associate to the characteristic initial value problem for second order hyperbolic system \cite{Hagen}.} initial data with a measure-valued null dust. 

Then, in an appropriate local double null domain, there exists a unique angularly regular weak solution to the Einstein--null dust system.
\end{theorem}

Theorem~\ref{thm:null.shells.intro} provides an existence and uniqueness result for a large class of data with measure-valued null dust. These include in particular data for which the solutions feature the propagation and interaction of null dust shells. In other words, it simultaneously addresses Problems~\ref{prob:null.shell}, \ref{prob:interaction} and \ref{prob:general.shell}. We emphasize that Theorem~\ref{thm:null.shells.intro} imposes \underline{no symmetry assumptions}.

We remark that in Theorem~\ref{thm:null.shells.intro} one also gets a stability statement, which follows from the proof of Theorem~\ref{thm:uniqueness.intro}. We will however not formulate this precisely for the sake of brevity.


Finally, as we explained above, the proof of Theorem~\ref{thm:null.shells.intro} not only gives existence and uniqueness of measure-valued solutions to the Einstein--null dust system, but it also shows that any such solution is an appropriate limit of vacuum solutions. As a result, we also resolve Conjecture~\ref{conj:reverse.Burnett} in our setting. 
\begin{corollary}\label{cor:reverse.Burnett.intro}
Let $(\mathcal M =[0,u_*]\times [0,\ub_*] \times \mathbb S^2 ,\, g_\i,\, \{\ud \nu_u\}_{u\in [0,u_*]}, \, \{\ud \nub_{\ub}\}_{\ub\in [0,\ub_*]})$ be  an angularly regular solution\footnote{Here, $\ud\nu$ and $\ud\nub$ here denote the measure-valued null dust; see Definition~\ref{def:ang.reg.null.dust}.} to the Einstein--null dust system (potentially with measure-valued dusts) with strongly angularly regular data. Then for any $p\in \mathcal M$, there exist $p\in \mathcal M'\subseteq \mathcal M$ and a sequence of smooth angularly regular vacuum solutions $(\mathcal M',g_n)$ such that $g_n \to g_\i$ in $C^0$ and weakly in $W^{1,2}$ in $\mathcal M'$.
\end{corollary}

Combining Theorem~\ref{thm:limit.intro} and Corollary~\ref{cor:reverse.Burnett.intro}, we have thus answered Problem~\ref{prob:Burnett} in the class of angularly regular solutions with strongly angularly regular data.

\subsection{Related works}\label{sec:related.works}

\subsubsection{High-frequency limits in general relativity}
The study of limits of high-frequency spacetimes has a long tradition in the physics literature, beginning with the pioneering works of Isaacson \cite{I1, I2} and Choquet-Bruhat \cite{CB.HF}, who already observed using some form of ``averaging'' or ``expansion'' that high-frequency limits of gravitational waves could lead to an effective stress-energy-momentum tensor mimicking that of null dust. (See also \cite{Penrose.massless}.) This was further discussed and explored by MacCallum--Taub \cite{MacCallumTaub}. 

More relevant to our paper is the work of Burnett \cite{Burnett}, in which he formulated high-frequency limits of gravitational waves in the language of weak limits. In the same paper, he introduced Conjectures~\ref{conj:Burnett} and \ref{conj:reverse.Burnett}. Within Burnett's framework of weak limits, various examples have been constructed, see for instance \cite{pHtF93, GW2, SGWK, SW, HL.HF, Lott2}.
 
Finally, we note interesting connections between high-frequency limits and inhomogeneities in cosmology \cite{GW1, GW2, GW.FLRW, GW.simple}, as well as late-time asymptotics in cosmological spacetimes \cite{Lott1, Lott3, Lott2}. See also \cite{Penrose2018, SC.standing} for other applications.

\subsubsection{Burnett's conjecture in $\mathbb U(1)$ symemtry}\label{sec:U(1)}
As mentioned earlier, Burnett's conjecture (Conjecture~\ref{conj:Burnett}) in its full generality remains open. Nevertheless, imposing an $\mathbb U(1)$ symmetry and an elliptic gauge condition,  Burnett's conjecture\footnote{We remark that the exact conditions in \cite{HL.Burnett} are slightly stronger that in Conjecture~\ref{conj:Burnett}; see \cite{HL.Burnett} for details.} has been solved recently in \cite{HL.Burnett}.

In fact, under a slightly more restrictive symmetry and gauge assumption, there is a partial result for the reverse Burnett conjecture (Conjecture~\ref{conj:reverse.Burnett}) \cite{HL.HF}. It was shown that all generic, smooth, small data solutions to the Einstein--null dust system arise as suitable weak limit of solutions to the Einstein vacuum equations.

\subsubsection{Null dust shell solutions in the physics literature}\label{sec:null.dust.in.physics}

As described earlier, to the best of our knowledge the first null dust shell solution was constructed in \cite{Synge}. This was later generalized by \cite{tDgtH85}. The interaction of two null dust shells under symmetry assumptions has been studied in \cite{tDgtH85, tDgtH86, iR85}.

Due to its simplicity, null dust shell solutions are also used as a simplified model to study gravitational collapse; see for instance \cite{Penrose.shell, Hawking.shell, Tod.shell, Barrabes.shell, cBwIeP90, cBwIpL91}.

For further references, see the book \cite{BH.book}.

\subsubsection{Low-regularity solutions to the Einstein equations} Our result in this paper heavily relies on the low-regularity existence and uniqueness result in Theorem~\ref{thm:existence.intro}. Low-regularity results in general relativity are themselves of independent interest. Perhaps the most celebrated such result is the bounded $L^2$ curvature theorem:
\begin{theorem}[Klainerman--R.--Szeftel \cite{L21}]\label{thm:L2curv}
The time of existence (with respect to a maximal foliation) of a classical solution to the Einstein vacuum equations depends only on the $L^2$-norm of the curvature and a lower bound of the volume radius of the corresponding initial data set.
\end{theorem}

Theorem~\ref{thm:L2curv} handles a very general class of data. This is in contrast to Theorem~\ref{thm:existence.intro}, which although allows for lower regularity when measured in the worst directions, the theorem also requires the data to be of a more specific form.

Very recently, in an ongoing work, L.--Van de Moortel \cite{LVdM1, LVdM2} study the problem of transversal interaction of \emph{three} impulsive gravitational waves under polarized $\mathbb U(1)$ symmetry. While \cite{LVdM1, LVdM2} relies heavily on the symmetry assumptions, in view of the presence of three impulsive gravitational waves, the problem requires geometric construction beyond the double null foliation used in \cite{LR, LR2}. 




\subsection{Brief discussion of the proof}\label{sec:proof.intro}

\subsubsection{Examples in symmetries}

Before we discuss the proof, it is illuminating to look at a few very simple examples in symmetry. The first example is given already in \cite{Burnett}, which shows in the plane wave setting how oscillations give rise to a null dust. This is the basic example of the phenomenon that we explore in our paper (where we consider the much more general case with no symmetry assumptions).

\begin{example}[The Burnett example \cite{Burnett}]\label{ex.Burnett}

Consider the following metric on $\mathbb R^4$:
\begin{equation}\label{Burnett.form}
g=-2 du d\ub+H(\ub)^2(e^{G(\ub)} dX^2+ e^{-G(\ub)} dY^2),
\end{equation}
where $H(\ub)$ and $G(\ub)$ are real-valued functions of $\ub$. This defines a Lorentzian metric as long as $H>0$. The Ricci curvature tensor is given by 
\begin{equation}\label{eq:Burnett.Ric}
\mathrm{Ric}(g) = \{-\f 12 \left(G'(\ub)\right)^2-\f{2 H''(\ub)}{H(\ub)} \}\, \ud\ub \otimes \ud\ub.
\end{equation}
Burnett considered a one-parameter family of solutions to the vacuum Einstein equations which take the form \eqref{Burnett.form}. The family of solutions is parametrized by $\lambda$. For $\lambda>0$, define $G_\lambda$ by
$$G_\lambda(\ub)=\lambda k(\ub)\sin (\f{\ub}{\lambda}),$$
where $k(\ub)$ is some fixed smooth function. Also, define $H_\lambda$ by the following ordinary differential equation
\begin{equation}\label{H.ODE}
\begin{cases}
-\f 12 \left(G_{\lambda}'(\ub)\right)^2-\f{2 H_{\lambda}''(\ub)}{H_{\lambda}(\ub)}=0,\\
H_{\lambda}(0)=1,\quad H_{\lambda}'(0)=0,
\end{cases}
\end{equation}
so that by \eqref{eq:Burnett.Ric} the metric $g_\lambda = -2 du d\ub+H_\lambda(\ub)^2(e^{G_\lambda(\ub)} dX^2+ e^{-G_\lambda(\ub)} dY^2)$ is vacuum. 

We now consider the limit $\lambda \to 0$. Clearly, 
\begin{equation}\label{Burnett.G.def}
G_0(\ub):=\lim_{\lambda\to 0} G_{\lambda}(\ub)=0.
\end{equation}
By standard theory of ordinary differential equations, there exists $\ep>0$ such that \eqref{H.ODE} can be solved for $\ub \in [0,\ep]$. It is easy moreover to show that $H_\lambda$ has a limit in $C^1([0,\ep])$ after taking $\ep$ to be smaller if necessary. We define $H_0(\ub):=\lim_{\lambda\to 0} H_\lambda(\ub)$.

For $\ub \in [0,\ep]$, the spacetime metric given by
$$g_0=-2 du d\ub+H_0(\ub)^2(e^{G_0(\ub)} dX^2+ e^{-G_0(\ub)} dY^2)$$
satisfies (by \eqref{eq:Burnett.Ric})
$$\mathrm{Ric}(g_0)=\f 14 \left(k(\ub)\right)^2\,\ud\ub \otimes \ud \ub = \f 12 w\mbox{-}\lim_{\lambda\to 0} (G_\lambda'(\ub))^2 \,\ud\ub \otimes \ud \ub.$$
In particular, if $k \not\equiv 0$, then $g_0$ is not a solution to the Einstein vacuum equation, but instead solves the Einstein null dust system.

\end{example}

Still within the category of explicit examples in symmetry class, one can also go beyond one family of null dust and get a limit with two families of null dust.  We refer the reader to \cite{GW2} for details.

\begin{example}[Green--Wald example in Gowdy symmetry \cite{GW2}]
Green and Wald gave an example of a sequence of vacuum \underline{polarized Gowdy} spacetimes whose limit is non-vacuum and in fact can be thought of as having two families of null dust. The spacetimes they constructed have topology\footnote{We remark that as we are only interested in the local behavior of high-frequency, the topology plays no role here.} $\mathbb R\times \mathbb T^3$ so that in a coordinate system $(\tau,\th,\sigma,\de)\in \mathbb R\times \mathbb T^3$, the sequence of metrics are given by
$$g_n=e^{\f{(\tau-\alpha_n)}{2}}(-e^{-2\tau}d\tau^2+d\th^2)+e^{-\tau}\left(e^{P_n} d\sigma^2+e^{-P_n} d\de^2\right),$$
where $P_n$ and $\alpha_n$ take the following form
\begin{equation}\label{GW.P.alpha.def}
\begin{split}
P_n:=&\f{A}{\sqrt{n}}J_0(n e^{-\tau})\sin(n\th),\\
\alpha_n:=&-\f{A^2e^{-\tau}}{2}J_1(n e^{-\tau})J_0(n e^{-\tau})\cos(2n\th)\\
&-\f{A^2 n e^{-2\tau}}{4}\left(\left(J_0(n e^{-\tau})\right)^2+2\left(J_1(n e^{-\tau})\right)^2-J_0(n e^{-\tau})J_2(n e^{-\tau})\right).
\end{split}
\end{equation}
Here, $J_k$ and $Y_k$ denote the standard Bessel functions of first and second kind respectively; and $A$ is some fixed real valued constant. 

As is shown in \cite{GW2}, these metrics are all vacuum and the sequence has a uniform limit on compact subsets of $\mathbb R\times\mathbb T^3$ to the limiting spacetime
$$g_{\infty}=e^{\f{(\tau+\f{A^2 e^{-\tau}}{\pi})}{2}}(-e^{-2\tau}d\tau^2+d\th^2)+e^{-\tau}\left( d\sigma^2+ d\de^2\right),$$
which has a non-trivial Einstein tensor with the following non-vanishing components 
$$G_{\tau\tau}=\f{A^2e^{-\tau}}{4\pi},\quad G_{\th\th}=\f{A^2e^\tau}{4\pi}.$$
This corresponds to a solution to the Einstein equations with two families of null dust\footnote{It can be easily checked after introducing the null variables $\ub:=-e^{-\tau}+\th$ and $u:=-e^{-\tau}-\th$, the non-vanishing components of the Einstein tensor in the $(u,\ub,\sigma,\de)$ coordinates are
$$G_{uu}=G_{\ub\ub}=\f{A^2}{8\pi H^2}.$$}.
\end{example}

While the explicit examples feature only oscillations and have a limit which is smooth, it is not difficult to modify \eqref{ex.Burnett} so that the limiting stress-energy-momentum tensor still corresponds to null dust, but is only a \emph{measure-valued} null dust shell. Our main result will in particular generalize this simple example to general measure-valued null dust without symmetry assumptions.

\begin{example}[Null shell in plane symmetry]\label{example:intro.shell}
Let $k(\ub)\in C^\infty_c$, $\mathrm{supp}(k) \subseteq [-\f 12, \f 12]$, $f\geq 0$. Moreover, assume
\begin{equation}\label{shell.k.bd}
\int_{-\infty}^{\infty} \left(k'(\ub)\right)^2 d\ub = 1.
\end{equation}
Now, we construct a one-parameter family of solutions of the form \eqref{Burnett.form} to the Einstein vacuum equation by setting
$$G_\lambda(\ub)=\lambda^{\f12} k(\f{\ub}{\lambda})$$
for $\lambda>0$.
$H_\lambda$ is then defined to be solutions to \eqref{H.ODE} so that we obtain a $1$-parameter family of vacuum solutions. 
Notice that $G_{\lambda}$ is much more singular than that in \eqref{ex.Burnett} as $\lambda\to 0$. In particular, $G'_{\lambda}$ is not uniformly bounded in $\lambda$. It is easy to see that $G_\lambda(\ub)$ converges uniformly to $0$. We thus define
\begin{equation}\label{shell.G.def}
G_0(\ub)= \lim_{\lambda\to 0} G_\lambda(\ub) = 0.
\end{equation}
An ODE argument shows that there is an interval $\ub \in [-\ep ,\ep]$ such that $H_\lambda(\ub)$ admits a $C^0 \cap W^{1,p}$ limit (for\footnote{Note however that $H'_{\lambda}$ does not have a uniform limit.} $p \in [1,+\infty)$).

For $\ub \in [-\ep,\ep]$, the spacetime metric given by
$$g_0=-2 du d\ub+H_0(\ub)^2(e^{G_0(\ub)} dX^2+ e^{-G_0(\ub)} dY^2)$$
satisfies (by \eqref{eq:Burnett.Ric})
$$\mathrm{Ric}(g_0)= \f 12 \de(\ub)\,\ud \ub\otimes \ud \ub,$$
which corresponds exactly to a null dust shell.

\end{example}

\subsubsection{Characterization of the high-frequency limit}

The above examples, while extremely specific, already illustrate the basic phenomenon. In the more general case in Theorem~\ref{thm:limit.intro}, we will prove that there are (at most) two non-trivial components of the Einstein tensor that can be generated in the weak limit. In particular, as asserted in Theorem~\ref{thm:limit.intro}, the limiting spacetime is isometric to a solution to the Einstein--null dust system. Unlike in the above examples, however, no explicit computations will be available and we will rely on compactness, particularly compensated compactness, arguments.

Our starting point is Theorem~\ref{thm:existence.intro}, which states that given a sequence of initial data obeying the estimates in Theorem~\ref{thm:existence.intro} uniformly, there is a uniform region of existence. Moreover, it follows from the estimates established in the proof of Theorem~\ref{thm:existence.intro} that the sequence of spacetime metrics are uniformly bounded in $C^\alp \cap W^{1,2}$. Using standard compactness results, this is sufficient to extract a subsequence $(\mathcal M, g_{n_k})$ and a limiting spacetime so that the metrics converge strongly in $C^0$ and weakly in $W^{1,2}$.

To obtain more information about the limiting spacetime, and to show that it indeed satisfies the Einstein--null dust system, we need a more precise understanding of the convergence. Introducing the Ricci coefficients $\eta$, $\etab$, $\trch$, $\trchb$, $\chih$, $\chibh$, $\om$ and $\omb$ with respect to a null frame adapted to a double null coordinate system (see Section~\ref{sec:Ricci.coeff}), to understand whether the limiting Ricci coefficients amounts to checking whether quadratic products of these Ricci coefficients converge weakly to the products of the weak limits. The following are the main observations:
\begin{enumerate}
\item The Ricci coefficients $\eta$, $\etab$, $\trch$ and $\trchb$ converge (up to a subsequence) \emph{strongly} in the (say, spacetime\footnote{In fact, stronger convergence holds (and the sense of convergence is different for different Ricci coefficients), but strong spacetime $L^2$ convergence is sufficient to ensure that they do not contribution to convergence defect for the Ricci curvature.}) $L^2$ norm. In particular, in all the quadratic terms where $\eta$, $\etab$, $\trch$ and $\trchb$ is one of the Ricci coefficients, the weak limit of the product coincides with the product of the weak limits.
\item Notice that even though $\eta$, $\etab$, $\trch$ and $\trchb$ all have strong $L^2$ limits, the precise sense of limit (and the proof) is different. The components $\eta$ and $\etab$ admit (subsequential) uniform pointwise limit, and can be proven by an Arzela--Ascoli type argument. However, $\trch$ and $\trchb$ only converge in $L^p$ (for $p\neq +\infty$) (up to a subsequence) and to prove this we rely on the compactness of BV and the Aubin--Lions lemma.
\item The Ricci coefficients which only have weak $L^2$ limits, i.e.~$\chih$, $\chibh$, $\om$ and $\omb$, exhibit some \emph{compensated compactness}. For instance, since $\chih$ is more regular along constant-$\ub$ hypersurfaces and $\chibh$ is more regular along constant-$u$ hypersurfaces, we have $w\mbox{-}\lim_{k\to +\infty} (\chih \otimes \chibh) = (w\mbox{-}\lim_{k\to +\infty} \chih) \otimes (w\mbox{-}\lim_{k\to +\infty} \chibh)$, \emph{even though both $\chih_{n_k}$ and $\chibh_{n_k}$ only admit weak limits}.
\item Finally, the \emph{only} quadratic terms of the Ricci coefficients in the definition of the Ricci curvature such that the weak limit of the product differs from the product of the weak limits are $|\chih|^2_\gamma$ and $|\chibh|^2_\gamma$. In particular, $\mathrm{weak}\mbox{-*} \lim_{k\to +\infty} |\chih_{n_k}|^2_{\gamma_{n_k}} - |\mathrm{weak}\mbox{-*} \lim_{k\to +\infty} \chih_{n_k}|^2_{\gamma_{\infty}}$ and $\mathrm{weak}\mbox{-*} \lim_{k\to +\infty} |\chibh_{n_k}|^2_{\gamma_{n_k}} - |\mathrm{weak}\mbox{-*} \lim_{k\to +\infty} \chibh_{n_k}|^2_{\gamma_{\infty}}$ are in general non-trivial non-negative measures corresponding to the two families of null dust.
\end{enumerate}
Using the above three observations, it already follows that with respect to the null frame $\{e_1,\,e_2,\,e_3,\,e_4\}$, the only potentially non-vanishing components of the Ricci curvature of the limiting spacetime are $\mathrm{Ric}(e_3,e_3)$ and $\mathrm{Ric}(e_4, e_4)$.

To prove that the limit solves the Einstein--null dust system, we also need to show the propagation equation of the null dust\footnote{In fact we will also show some higher order equations such as transport equations for some first angular derivatives of the Ricci coefficients and a hyperbolic system for the renormalized curvature components. They are strictly speaking not necessary for Theorem~\ref{thm:limit.intro}, but are used to prove the uniqueness of the limit.}. For this we need to understand convergence properties of some higher derivative terms and also cubic terms. These turn out to follow from strong convergence statements and compensated compactness statements which are similar to those above but for higher order derivatives.

\subsubsection{Proof of the uniqueness theorem}

For the uniqueness theorem (Theorem~\ref{thm:uniqueness.intro}), first note that since the limiting spacetime metrics are obtained as limits of the metrics from Theorem~\ref{thm:existence.intro}, the metric components are mostly in similar function spaces as in \cite{LR, LR2} (with the notable exception that $\trch$, $\trchb$ are only BV instead of $W^{1,1}$, and that there are in addition two families of null dust, which are in general only measure-valued; more on this later). This suggests uniqueness to be proven in function spaces that are used in \cite{LR2}. 

In order to obtain the uniqueness result, the renormalization introduced in \cite{LR, LR2} plays an important role. (In fact, in this paper, the presence of the measure-valued null dust makes the renormalization in \cite{LukWeakNull} more convenient to use than that in our original \cite{LR, LR2}). The main point of the renormalization in \cite{LR, LR2} is to identify some quantities we called the renormalized curvature components, which are more regular than the spacetime curvature components themselves. Introducing the renormalization moreover allowed us to obtain a closed system of equations which does not involve the spacetime curvature components $R(e_4, e_A, e_4, e_B)$ and $R(e_3, e_A, e_3, e_B)$, which are in general only distribution-valued. 

After introducing the renormalization, we can recast the Einstein--null dust system in double null coordinates as a coupled quasilinear system of hyperbolic, elliptic and transport equations for the metric components, the Ricci coefficients, the renormalized curvature components and the null dust. We then use this system to prove estimates and deduce the uniqueness statement.

One crucial difference of our uniqueness argument with the proof of a priori estimates in \cite{LR2} is that a priori we only know that the solutions obey the Einstein--null dust system in an appropriate weak sense. Another (perhaps more important) difference is the presence of the null dust, which is moreover only measure-valued. To prove our uniqueness result will involve estimating differences of the null dust and the null mean curvatures $\trch$ and $\trchb$ using the transport equations they satisfy, and this needs to be carried out in appropriate function spaces which avoid a potential loss of derivative; see  Sections~\ref{sec:diff.trch} and \ref{sec:diff.null.dust}.

\subsubsection{Approximating the initial data}

The final result that we prove is Theorem~\ref{thm:null.shells.intro} (from which Corollary~\ref{cor:reverse.Burnett.intro} also follows), which asserts the local existence and uniqueness of solutions to the Einstein--null dust system with measure-valued null dust. Given that all limits satisfy the Einstein--null dust system (Theorem~\ref{thm:limit.intro}) and that the limiting spacetime depends only on the limiting data (Theorem~\ref{thm:uniqueness.intro}), the final necessary ingredient is a statement that for every given data set to the Einstein--null dust system (with potentially a measure-valued null dust), we can find a sequence of smooth vacuum data which 
\begin{enumerate}
\item obey uniformly the estimates required in Theorem~\ref{thm:limit.intro}, and 
\item moreover limits in an appropriate weak sense to the given data. 
\end{enumerate}
Once this approximation result is achieved, we prove the existence part of Theorem~\ref{thm:null.shells.intro} using Theorem~\ref{thm:limit.intro} to extract a limit which is a weak solution to the Einstein--null dust system. We then prove the uniqueness part of Theorem~\ref{thm:null.shells.intro} by Theorem~\ref{thm:uniqueness.intro}.

To obtain the approximation result requires solving on the initial hypersurfaces the null constraint equations. For this we rely on the fact, elucidated in \cite{Chr}, that the null constraint equations can be solved by first prescribing appropriate ``free data'', related to the conformal class of $\gamma$, and then solving transport equations. We carry out the proof of the approximation result in two steps\footnote{In particular, even to generate the null dust shell, we first regularize the initial data for the null dust, and then approximate it by high-frequency oscillation in the metric, in contrast to Example~\ref{example:intro.shell}.}, which we now describe. 

In the first step, we show that (up to technical assumptions\footnote{We need a technical assumption that require the null dust to vanish near some angular direction; see Section~\ref{sec:approx.smooth.dust.by.vac}.}) any \emph{smooth} null dust data can be approximated by highly oscillatory but smooth vacuum data. We prescribe highly oscillatory data for $\f{\rd \hat{\gamma}_n}{\rd \ub}$, where $\hat{\gamma}_n$ is an appropriate representation of the conformal class of $\gamma_n$. We choose a sequence $\f{\rd \hat{\gamma}_n}{\rd \ub}$ carefully so that $|\f{\rd \hat{\gamma}_n}{\rd \ub}|^2_{\hat{\gamma}_n} - |\f{\rd \hat{\gamma}^{(dust)}}{\rd \ub}|^2_{\hat{\gamma}^{(dust)}}$ converges weakly to the prescribed function corresponding to the null dust. Moreover, using the high-frequency parameter in $\gamma_n$ as a smallness parameter, we solve the vacuum null constraints equations, prove necessary estimates, and show that the limit indeed solves the null constraint equations for the Einstein--null dust system.

In the second step, we show that any \emph{measure-valued} null dust data (with additional angular regularity) can be approximated by \emph{smooth} null dust data (and obtain the desired result after combining with Step~1). 
To obtain this approximation result, we smooth out the given data for the null dust and the metric and prove a stability-type result for the null constraint equations in a low-regularity class consistent with the null dust only being a measure.

We emphasize that such an approximation result is possible precisely because we only need uniform bounds for $\chih$ and $\chibh$ in $L^2$ (with only additional angular regularity but not regularity in the null directions). 

\subsection{Formation of trapped surfaces}\label{sec:intro.addendum} A main theme of this paper and the discussion thus far is the connection between the Einstein vacuum equations and the Einstein--null dust system via weak limits. In particular, our construction of solutions to the Einstein--null dust system with measure-valued null dust explicitly exploits this connection. 

The connection between the Einstein vacuum equations and the Einstein--null dust system can also shed light on the problem of the formation of trapped surfaces. Because of the complexity of the formation of trapped surfaces in vacuum, the problem was first studied in simplified settings, such as the collapse of a null dust shell \cite{Gibbons.shell, Penrose.shell, Synge}. We will show that the solutions in \cite{Gibbons.shell, Penrose.shell, Synge} in fact arise as suitable weak limits of \emph{vacuum} solutions. These vacuum solutions are moreover exactly those constructed by Christodoulou in his groundbreaking work \cite{Chr}. In other words, at least for a specific solution regime, the collapse of a null dust shell captures the dynamics of gravitational collapse in vacuum. See further discussions in Section~\ref{sec:trapped.surfaces}.

\subsection{Outline of the paper}

We end the introduction with an outline of the remainder of the paper. In \textbf{Section~\ref{sec:geo.prelim}}, we introduce the geometric setup in this paper. In particular, we will describe the double null foliation gauge. We will also define the precise notions of weak solutions that will be relevant in this paper. After these preliminary discussions, we recall the existence theorem in \cite{LR2} in \textbf{Section~\ref{sec.existence}}. We then state the precise version of the main theorems (Theorems~\ref{thm:limit.intro}, \ref{thm:uniqueness.intro}, \ref{thm:null.shells.intro} and Corollary~\ref{cor:reverse.Burnett.intro}) in \textbf{Section~\ref{sec.main.thm}} (see Theorems~\ref{main.thm}, \ref{thm:uniqueness}, \ref{thm:main.local.dust} and \ref{thm:reverse.Burnett}). 

The remainder of the paper is then devoted to the proofs of these four theorems. 
\begin{itemize}
\item Sections~\ref{sec:gen.compactness}--\ref{sec:eqns.for.limit} are devoted to proving that a limit exists and solves the Einstein--null dust system (Theorem \ref{main.thm}). In \textbf{Section~\ref{sec:gen.compactness}}, we begin with some general compactness results. In \textbf{Section~\ref{sec:existence}}, the compactness results will be applied to extract a limiting spacetime with regularity properties. In \textbf{Section~\ref{sec:eqns.for.limit}}, we then show that the limit satisfies the Einstein--null dust system.
\item \textbf{Section~\ref{sec:proof.uniqueness}} will be devoted to proving the uniqueness theorem (Theorem~\ref{thm:uniqueness}).
\item \textbf{Section~\ref{sec.approx.thm}} will be devoted to proving an approximation theorem for the characteristic initial data, from which the local existence and uniqueness result for the Einstein--null dust system (Theorem~\ref{thm:main.local.dust}) and the result on approximation by vacuum spacetimes (Theorem~\ref{thm:reverse.Burnett}) follow.
\end{itemize}

In \textbf{Section~\ref{sec:trapped.surfaces}}, we end the paper with a discussion regarding the relation between null dust shell solutions and the formation of trapped surfaces result of Christodoulou \cite{Chr}.

Finally, we include two appendices. In \textbf{Appendix~\ref{app:est}}, we derive from \cite{LR2} some additional estimates that are used in this paper. In \textbf{Appendix~\ref{app:CC}}, we prove our main compensated compactness lemma. 
\\ \\
{\bf Acknowledgements:} J. Luk thanks the Cambridge mathematical general relativity group for stimulating discussions. In particular, he thanks Jan Sbierski for raising some interesting questions. 

J. Luk is supported by NSF grants DMS-1709458, DMS-2005435 and a Terman Fellowship. I. Rodnianski is supported by NSF grant DMS-1709270 and a Simons Investigator Award.

\section{Geometric preliminaries}\label{sec:geo.prelim}

\subsection{Weak solutions to the Einstein equations}\label{sec.weak.sol}

In this subsection, we give some definitions of weak solutions to the Einstein vacuum equations and the Einstein--null dust system which make sense under very weak regularity assumptions. We will very soon further restrict the class of solutions that we consider, but it is useful to keep these more general definitions in mind.

\begin{definition}\label{def:vacuum}
Let $(\mathcal M, g)$ be a $C^0\cap W^{1,2}_{\mathrm{loc}}$ time-oriented Lorentzian manifold. We say that $(\mathcal M, g)$ is a \textbf{weak solution to the Einstein vacuum equations} if for all smooth and compactly supported vector fields $X$ and $Y$,
$$\int_{\mathcal M} \left((D_\mu X^\mu)(D_\nu Y^\nu)-D_\mu X^\nu D_\nu Y^\mu\right) \,\mathrm{dVol}_g = 0,$$
where $D$ and $\mathrm{dVol}_g$ respectively denote the Levi--Civita connection and the volume form associated to $g$.
\end{definition}

\begin{definition}\label{def:null.dust}
Let $(\mathcal M, g)$ be a $C^0\cap W^{1,2}_{\mathrm{loc}}$ time-oriented Lorentzian manifold and $\ud\nu$, $\ud\nub$ be non-negative Radon measures on $\mathcal M$. Let $u,\,\ub:\mathcal M \to \mathbb R$ be two $C^1$ functions satisfying
$$g^{-1}(\ud u, \ud u) = g^{-1}(\ud \ub, \ud \ub) = 0.$$

We say that $(\mathcal M, g, \ud \nu, \ud \nub)$ is a \textbf{weak solution to the Einstein--null dust system} if the following holds:
\begin{enumerate}
\item For all smooth and compactly supported vector fields $X$ and $Y$,
$$\int_{\mathcal M} \left((D_\mu X^\mu)(D_\nu Y^\nu)-D_\mu X^\nu D_\nu Y^\mu\right) \,\mathrm{dVol}_g = \int_{\mathcal M} (Xu)(Yu) \,\ud\nu + \int_{\mathcal M} (X\ub)(Y\ub) \,\ud\nub.$$
\item For every smooth and compactly supported real-valued function $\varphi$,
$$\int_{\mathcal M} g^{-1}(\ud u, \ud \varphi) \, \ud \nu = 0 = \int_{\mathcal M} g^{-1}(\ud \ub, \ud \varphi) \, \ud \nub.$$
\end{enumerate}
\end{definition}

\begin{remark}\label{rmk:null.dust.1}
Consider the particular case where $g$ is $C^2$ and there exist $C^1$ functions $f_\mu$ and $f_\nu$ such that $\ud\mu = f^2_\mu \,\mathrm{dVol}_g$ and $\ud\nu = f^2_\nu \,\mathrm{dVol}_g$. Then the two conditions in Definition~\ref{def:null.dust} are equivalent to
\begin{enumerate}
\item $$\mathrm{Ric}(g) = f^2_\mu\, \ud u\otimes \ud u + f^2_{\nu} \, \ud \ub \otimes \ud \ub.$$
\item $$2 g^{-1}(\ud u, \ud f_\mu) + (\Box_g u) f_\mu =0 = 2 g^{-1}(\ud \ub, \ud f_\nu) + (\Box_g \ub) f_\nu,$$
where $\Box_g$ denotes the Laplace--Beltrami operator associated to $g$.
\end{enumerate}
\end{remark}

\begin{remark}\label{rmk:null.dust.2}
Definition~\ref{def:null.dust} does \underline{not} give the most general form of the Einstein--null dust system. Even when the density of the null dust is given by an $L^1$ function, one in general allows, for null $1$-forms $\xi$ and $\underline{\xi}$,
\begin{enumerate}
\item $$\mathrm{Ric} = f^2_\xi \, \xi \otimes \xi + f^2_{\underline{\xi}} \, \underline{\xi} \otimes \underline{\xi},$$
\item $$2 g^{-1}(\xi, \ud f_\xi) + (D^\alp \xi_\alp) f_\xi =0 = 2 g^{-1}(\underline{\xi}, \ud f_{\underline{\xi}}) + (D^\alp \underline{\xi}_\alp) f_{\underline{\xi}},$$
\item $$D_\xi \xi = 0 = D_{\underline{\xi}} \underline{\xi},$$
\end{enumerate}
where $\xi$ and $\underline{\xi}$ are not necessarily exact. We restrict however our attention to Definition~\ref{def:null.dust} since this is precisely what arises in the limit in our setting. 
\end{remark}

\subsection{Double null foliation and double null coordinates}\label{sec.dnf}

In this subsection, we define spacetimes with double null foliation and double null coordinates. From now on we consider $\mathcal M$ being a manifold with corners with\footnote{In fact, all of our results apply if we replace $\mathbb S^2$ by any compact $2$-surface $S$. The choice of $S=\mathbb S^2$ is so that we have consistent notation with \cite{LR, LR2}.}
\begin{equation}\label{eq:M.topology}
\mathcal M = [0,u_*] \times [0,\ub_*] \times \mathbb S^2,
\end{equation}
where $u_*, \ub_*>0$. 

\begin{definition}\label{def:sets}
We introduce the following notations for subsets of $\mathcal M$ satisfying \eqref{eq:M.topology}:
\begin{enumerate}
\item $$H_u:= \{ (u',\ub', \vartheta) \in [0,u_*]\times [0,\ub_*] \times \mathbb S^2: u' = u\},$$
\item $$\Hb_u:= \{ (u',\ub', \vartheta) \in [0,u_*]\times [0,\ub_*] \times \mathbb S^2: \ub' = \ub\},$$
\item $$S_{u,\ub} := H_u \cap \Hb_{\ub} = \{ (u',\ub', \vartheta) \in [0,u_*]\times [0,\ub_*] \times \mathbb S^2: u' = u,\,\ub' = \ub\}.$$
\end{enumerate}
\end{definition}

\begin{definition}
For a type $T^q_p$ tensor field $\xi$ on $\mathcal M$, we say that it is $S$-tangent if for every vector fields $X_1, ..., X_p\in T_x\mathcal M$, $\xi(X_1,...,X_p)$ is in $\otimes^q T_x S_{u,\ub}$ and $\xi(X_1,...,X_p)=0$ if any one of $X_1, ..., X_p$ equals to $e_3$ or $e_4$.
\end{definition}

We introduce the following convention for indices for the remainder of the paper. {\bf We will use the convention that the lower case Greek indices run through the spacetime indices $\mu,\nu=1,2,3,4$ while the upper case Latin indices run through the indices on the $2$-surfaces $A,B=1,2$.} We will use Einstein's summation convention unless otherwise stated.

\begin{definition}[$C^0\cap W^{1,2}_{loc}$ metrics in double null coordinates]\label{double.null.def}
A Lorentzian metric $g$ on $\mathcal M$ satisfying \eqref{eq:M.topology} is said to be a \textbf{$C^0\cap W^{1,2}_{loc}$ metric in double null coordinates} if the following hold:
\begin{itemize}
\item There exists an atlas $\{U_i\}_{i=1}^N \subset \mathbb S^2$ such that given coordinates $(\th^1, \th^2)$ in each coordinate chart $U_i$, the metric takes the form\footnote{Recall our notation that capital Latin letters are summed from $1$ to $2$.} 
\begin{equation}\label{double.null.coordinates}
g=-2\Omega^2(du\otimes d\ub+d\ub\otimes du)+\gamma_{AB}(d\th^A-b^Adu)\otimes (d\th^B-b^Bdu),
\end{equation}
where 
\begin{itemize}
\item $\Omega$ is a real-valued function in $C^0\cap W^{1,2}_{loc}$, 
\item $b = b^A\f{\rd}{\rd\th^A}$ is a $C^0\cap W^{1,2}_{loc}$ $S$-tangent vector field 
tangential to $S$, 
\item $\gamma = \gamma_{AB}\,\ud\th^A\,\ud \th^B$ is a $C^0\cap W^{1,2}_{loc}$ $S$-tangent symmetric covariant $2$-tensor, which restricts to a positive definitie metric on $T S$.
\end{itemize}
\end{itemize}
\end{definition}

\begin{remark}[Regularity assumptions]
In Definition~\ref{double.null.def}, the metric components are required only to be in $C^0\cap W^{1,2}_{loc}$. This is the minimal assumption that we need for many of the definitions below. Nevertheless, the spacetime metrics that we actually construct will have higher regularity, at least when viewed in some directions (see already Definition~\ref{double.null.def.2}).
\end{remark}

\begin{remark}[Geometric significance of the double null coordinate system] At least when the metric is sufficiently regular, \eqref{double.null.coordinates} has the following geometric interpretation:
\begin{enumerate}
\item $u$ and $\ub$ are null variables. In particular, $H_u$ and $\Hb_{\ub}$ (recall \eqref{def:sets}) are null hypersurfaces.
\item $(\ud u)^\sharp$ and $(\ud\ub)^\sharp$ are null geodesic vector fields (where $\sharp$ denotes the metric dual).
\item The coordinate functions $\th^1$ and $\th^2$ are constant along the integral curves of $(du)^\sharp$.
\end{enumerate}
\end{remark}

\subsection{Ricci coefficients and the Gauss curvature of the $2$-spheres}\label{sec:Ricci.coeff}

We will define the Ricci coefficients for a metric satisfying Definition~\ref{double.null.def}. For the rest of this subsection, we assume that $(\mathcal M, g)$ satisfying Definition~\ref{double.null.def} is given and fixed. 

\begin{definition}
The normalized null pair are defined as follows:
$$e_3=\Omega^{-1}(\f{\rd}{\rd u}+b^A\f{\rd}{\rd\th^A}),\quad e_4=\Omega^{-1}\f{\rd}{\rd\ub}.$$
We will also write $\{e_A\}_{A=1,2}$ to denote an arbitrary local frame tangent to $S_{u,\ub}$.
\end{definition}

We now define the Ricci coefficients as the following $S$-tangent tensors, where $D$ is the Levi--Civita connection with respect to the spacetime metric $g$:
\begin{definition}[Ricci coefficients]\label{def.RC}
\begin{enumerate}
\item Define the following Ricci coefficients\footnote{Note that this is well-defined with $C^0\cap W^{1,2}_{loc}$ regularity of the metric coefficients.} such that for vector fields $\slashed X$, $\slashed Y$ tangential to $S$:
\begin{equation*}
\begin{split}
&\chi(\slashed X, \slashed Y)=g(D_{\slashed X} e_4, \slashed Y),\, \,\, \quad \chib(\slashed X,\slashed Y)=g(D_{\slashed X} e_3,\slashed Y),\\
&\eta(\slashed X)=-\frac 12 g(D_3 \slashed X,e_4),\quad \etab(\slashed X)=-\frac 12 g(D_4 \slashed X,e_3),\\
&\omega=-\frac 14 g(D_4 e_3,e_4),\quad\,\,\, \omegab=-\frac 14 g(D_3 e_4,e_3).
\end{split}
\end{equation*}
\item Define also 
$$\chih=\chi-\f 12 \trch \gamma, \quad \chibh=\chib-\f 12 \trchb \gamma,$$
where $\chih$ (resp.~$\chibh$) is the traceless part of $\chi$ (resp.~$\chib$) (with the trace taken with respect to the metric $\gamma$ on $S_{u,\ub}$) and $\trch$ (resp.~$\trchb$) is the trace of $\chi$ (resp.~$\chib$). 
\item We will use $\chi_{AB}$, $\chib_{AB}$, $\eta_A$, etc.~to denote the components of the Ricci coefficients with respect to a local coordinate system on $S$ in a double null coordinate system.
\end{enumerate}
\end{definition}

We define the following quantities which depend only on the intrinsic geometry of a Riemannian $2$-manifold. Note that in application $S = S_{u,\ub}$ for some $u,\,\ub$.
\begin{definition}\label{def:isoperimetric} Let $(S,\gamma)$ be a closed Riemannian $2$-manifold.
\begin{enumerate}
\item Define $\mathrm{Area}(S,\gamma)$ to be the total area of $S$ with respect to the metric $\gamma$.
\item Define the isoperimetric constant by 
\begin{equation}\label{eq:def.IPC}
{\mathbf I}(S,\gamma) = \sup_{\substack{ U \\ \rd U \in C^1}} \f{ \min\{\mathrm{Area}(U),\,\mathrm{Area}(U^c)\} }{(\mathrm{Perimeter}(\rd U))^2 } .
\end{equation}
\item Define the Gauss curvature $K$ by
\begin{equation}\label{Gauss.def}
\gamma_{BC} K=\f{\rd}{\rd\th^A}\slashed{\Gamma}^{A}_{BC}-\f{\rd}{\rd\th^C}\slashed{\Gamma}^A_{BA}+\slashed{\Gamma}^A_{AD}\slashed{\Gamma}^D_{BC}-\slashed{\Gamma}^A_{CD}\slashed{\Gamma}^D_{BA},
\end{equation}
where 
\begin{equation}\label{Gamma.def}
\slashed{\Gamma}_{BA}^C=\f 12(\gamma^{-1})^{CD}\left(\f{\rd}{\rd\th^B}\gamma_{AD}+\f{\rd}{\rd\th^A}\gamma_{BD}-\f{\rd}{\rd\th^D}\gamma_{AB}\right).
\end{equation}
\end{enumerate}
When there is no danger of confusion, we write $\mathrm{Area}(S) = \mathrm{Area}(S,\gamma)$ and ${\mathbf I}(S) = {\mathbf I}(S,\gamma)$.
\end{definition}

Note that our regularity assumptions are sufficient to make sense of $\mathrm{Area}(S_{u,\ub},\gamma)$ and ${\mathbf I}(S_{u,\ub},\gamma)$. However, for $(\mathcal M, g)$ satisfying Definition~\ref{double.null.def}, the Gauss curvature $K$ of the surfaces $S_{u,\ub}$ is only to be understood as a distribution.

\subsection{Differential operators on $S$-tangent tensor fields}

\begin{definition}[Covariant derivatives for $S$-tangent tensor fields]\label{def:nabla}
Define $\nab_3$ and $\nab_4$ to be the projections to $S_{u,\ub}$ of the covariant derivatives $D_3= D_{e_3}$ and $D_4 = D_{e_4}$ respectively. Define $\nab_A$ to be the Levi--Civita connection with respect to the metric $\gamma$.
\end{definition}

The following identities hold for the connections $\nab_3$, $\nab_4$ and $\nab$. The proofs are straightforward and omitted.
\begin{proposition}\label{diff.formula}
In the double null coordinate system, for every covariant tensor field $\phi$ of rank $r$ tangential to the spheres $S_{u,\ub}$, we have
\begin{equation}\label{nab3.def}
\begin{split}
(\nab_3 \phi)_{A_1 A_2 ... A_r}
=&\:\Om^{-1}(\frac{\partial}{\partial u}+b^C\f{\rd}{\rd\th^C}) \phi_{A_1 A_2 ... A_r}-\sum_{i=1}^r(\chib^B{ }_{A_i}-\Omega^{-1}\frac{\partial b^B}{\partial\th^{A_i}})\phi_{A_1\dots\hat{A_i}B\dots A_r},
\end{split}
\end{equation}
where $\hat{A_i}$ denotes that the $A_i$ which was originally present is removed. Similarly, we have
\begin{equation}\label{nab4.def}
\begin{split}
(\nab_4 \phi)_{A_1 A_2 ... A_r}
=&\:\Omega^{-1}\frac{\partial}{\partial \ub} \phi_{A_1 A_2 ... A_r}-\sum_{i=1}^r \chi^B{ }_{A_i}\phi_{A_1\dots\hat{A_i}B\dots A_r}.
\end{split}
\end{equation}
Finally, $\nab$ is given by the Levi--Civita connection associated to the meetric $\gamma$, i.e.,
\begin{equation}\label{nab.def}
\nab_B \phi_{A_1 A_2 \dots A_r}=\f{\rd}{\rd \th^B}\phi_{A_1 A_2 \dots A_r}-\sum_{i=1}^r \slashed{\Gamma}_{BA_i}^C \phi_{A_1 A_2\dots \hat{A_i} C\dots A_r},\end{equation}
where $\slashed{\Gamma}$ is as in \eqref{Gamma.def}.

\end{proposition}

We introduce the following differential operators:
\begin{definition}\label{def:slashedL}
Define $\slashed{\mathcal L}$ to be the projection of the Lie derivative to the tangent space of $S_{u,\ub}$
\end{definition}

\begin{definition}\label{def:div.curl}
For a totally symmetric tensor field $\phi$ of rank $r+1$, define
$$(\div \phi)_{A_1\cdots A_r} := (\gamma^{-1})^{BC}\nab_B \phi_{CA_1\cdots A_r},\quad (\curl \phi)_{A_1\cdots A_r} := \in^{BC}\nab_B \phi_{CA_1\cdots A_r},$$
where $\in$ denotes the volume form with respect to $\gamma$.

Finally, define the following operator on $S$-tangent $1$-forms:
$$(\nab\otimes \phi)_{AB} := \nab_A \phi_B + \nab_B \phi_A - \gamma_{AB} \div \phi.$$
\end{definition}

\subsection{Some basic identities}

\begin{proposition}\label{prop:metric.der}
The Ricci coefficients can be expressed as derivatives of the metric components in the double null coordinate system (recall \eqref{double.null.coordinates}). More precisely, we have\footnote{In coordinates, the relations in \eqref{metric.derivative.invar} above read as follows:
\begin{equation}\label{metric.derivative}
\begin{split}
\f{\rd}{\rd \ub} \gamma_{AB} = &2\Omega \chi_{AB},\quad (\f{\rd}{\rd u}+b^C\f{\rd}{\rd\th^C}) \gamma_{AB}+\f{\rd b^C}{\rd \th^A}\gamma_{BC}+\f{\rd b^C}{\rd \th^B}\gamma_{AC} = 2\Omega \chib_{AB},\\
\f{\rd}{\rd \ub} b^A= &-2\Omega^2(\eta^A-\etab^A).
\end{split}
\end{equation}}
\begin{equation}\label{metric.derivative.invar}
\begin{split}
\slashed {\mathcal L}_{\f{\rd}{\rd \ub}} \gamma = 2\Omega \chi,\quad \slashed{\mathcal L}_{(\f{\rd}{\rd u}+b^A\f{\rd}{\rd\th^A})} \gamma = 2\Omega \chib,\quad \slashed {\mathcal L}_{\f{\rd}{\rd \ub}} b=-2\Omega^2(\eta^\sharp-\etab^\sharp),
\end{split}
\end{equation}
where $\slashed{\mathcal L}$ is as in Definition~\ref{def:slashedL} and ${ }^\sharp$ denotes the metric dual with respect to $\gamma$.
\end{proposition}

\begin{proposition}
The following relation holds:
\begin{equation}\label{Ricci.relation}
\begin{split}
\omega=-\f12 (e_4\log\Omega),\quad \omegab=-\f12(e_3\log\Omega), \quad \f 12 (\eta_A+ \etab_A)= \nab_A\log\Omega.
\end{split}
\end{equation}
\end{proposition}

\subsection{Function spaces and angularly regular metrics}

In this paper, we will mostly consider a slightly more restricted class of spacetimes in double null coordinates; see already Definition~\ref{double.null.def.2}. The main feature of this class is that the metric is more regular along the directions tangential to $S_{u,\ub}$ (despite that it is not in a better \emph{isotropic} Sobolev space than $W^{1,2}$). Moreover, the precise regularity is different for different Ricci coefficients, which can be seen as capturing the null structure of the Einstein equations in a double null coordinate system. The regularity that we impose is consistent with the regularity of spacetimes obtained in \cite{LR2}.

\begin{definition}
Denote by $\mathrm{dA}_\gamma$ the standard volume form induced by the (Riemannian) metric $\gamma$ on $S_{u,\ub}$, i.e.~in local coordinates, $\mathrm{dA}_\gamma:= \sqrt{\det\gamma}\,\ud\th^1\,\ud\th^2$.
\end{definition}

\begin{definition}[Definition of $L^p$ spaces]\label{def:Lp}
In this definition, let $\phi$ be an $S$-tangent rank $r$ tensor field on $[0,u_*]\times [0,\ub_*]\times \mathbb S^2$.

\begin{enumerate}
\item For every $(u,\ub)$, define, for $p\in [1,+\infty)$
\begin{equation*}
\begin{split}
\|\phi\|_{L^p(S_{u,\ub},\gamma)} := &\: (\int_{S_{u,\ub}} |\phi|^p_{\gamma} \, \mathrm{dA}_\gamma)^{\f 1p} \\
= &\: (\int_{S_{u,\ub}} ( (\gamma^{-1})^{A_1 A_1'} \cdots (\gamma^{-1})^{A_r A_r'} \phi_{A_1\dots A_r} \phi_{A_1'\dots A_r'})^{\f p2} \, \mathrm{dA}_\gamma)^{\f 1p};
\end{split}
\end{equation*}
and for $p=+\infty$, define
$$\|\phi\|_{L^\infty(S_{u,\ub},\gamma)} := \mathrm{ess\,sup}_{S_{u,\ub}} |\phi|_{\gamma} = \mathrm{ess\,sup}_{S_{u,\ub}} ((\gamma^{-1})^{A_1 A_1'} \cdots (\gamma^{-1})^{A_r A_r'} \phi_{A_1\dots A_r} \phi_{A_1'\dots A_r'})^{\f 12}.$$
We will often view $\|\phi\|_{L^p(S_{u,\ub},\gamma)}$ as a function of $u$ and $\ub$.
\item For $q\in [1,+\infty)$, $p\in [1,+\infty]$, define
$$\|\phi\|_{L^q_u L^p(S_{u,\ub},\gamma)} := (\int_0^{u_*} \|\phi\|_{L^p(S_{u,\ub},\gamma)}^q \, \ud u)^{\f 1q},$$
$$ \|\phi\|_{L^q_{\ub} L^p(S_{u,\ub},\gamma)} := (\int_0^{\ub_*} \|\phi\|_{L^p(S_{u,\ub},\gamma)}^q \, \ud \ub)^{\f 1q}.$$
These two terms will be viewed as functions of $\ub$ and $u$ respectively. Define also $L^\i_u L^p(S_{u,\ub},\gamma)$ and $L^\i_{\ub} L^p(S_{u,\ub},\gamma)$ after the obvious modifications.
\item For $r\in [1, +\infty)$, $p,\,q\in [1,+\infty]$, define
$$\|\phi\|_{L^r_u L^q_{\ub} L^p(S_{u,\ub},\gamma)} := (\int_0^{u_*} \|\phi\|_{L^q_{\ub} L^p(S_{u,\ub},\gamma)}^r \,\ud u)^{\f 1r},$$
$$ \|\phi\|_{L^r_{\ub} L^q_u L^p(S_{u,\ub},\gamma)} := (\int_0^{\ub_*} \|\phi\|_{L^q_{\ub} L^p(S_{u,\ub},\gamma)}^r \,\ud \ub)^{\f 1r}.$$
In a similar manner as (2), we also allow $r=+\infty$ after the obvious modifications.
\end{enumerate}
\end{definition}

\begin{definition}[Definition of Sobolev spaces]
Let $\phi$ be an $S$-tangent tensor field. 
\begin{enumerate}
\item For every $m\in \mathbb N\cup \{0\}$ and $p\in [1,+\infty]$, define
$$\|\phi\|_{W^{m,p}(S_{u,\ub},\gamma)}:= \|\nab^m \phi\|_{L^p(S_{u,\ub},\gamma)},$$
where $\nab$ is the Levi--Civita connection associated to $\gamma$. 
\item Define also $L^q_{\ub}W^{m,p}(S_{u,\ub},\gamma)$, $L^r_u L^q_{\ub} W^{m,p}(S_{u,\ub},\gamma)$, etc.~in a similar manner as Definition~\ref{def:Lp}.2 and \ref{def:Lp}.3 after replacing $L^p(S_{u,\ub},\gamma)$ by $W^{m,p}(S_{u,\ub},\gamma)$.
\end{enumerate}
\end{definition}

\begin{definition}[Definition of the BV spaces]\label{def:BV}
Let $\phi$ be an $S$-tangent rank $r$ tensor field on $[0,u_*]\times [0,\ub_*]\times \mathbb S^2$. Define
\begin{equation*}
\begin{split}
\|\phi\|_{BV(H_u, \gamma)} := &\: \int_0^{\ub_*} \|\phi\|_{L^1(S_{u,\ub},\gamma)}\,\ud \ub + \sup \left\{| \int_0^{\ub_*} \int_{S_{u,\ub}} (\f{\rd}{\rd \ub}\varphi) \phi \, \mathrm{dA}_{\gamma} \,\ud \ub| : \varphi \in C^1_c,\,|\varphi|\leq 1\right\} \\
&\: + \sup \left\{| \int_0^{\ub_*} \int_{S_{u,\ub}} (\div\slashed{X}) \phi \, \mathrm{dA}_{\gamma} \,\ud \ub| : \slashed{X} \in C^1_c,\,\sup_{u,\,\ub} \|\slashed{X} \|_{L^1(S_{u,\ub},\gamma)}\leq 1\right\}.
\end{split}
\end{equation*}
Similarly, we define
\begin{equation*}
\begin{split}
\|\phi\|_{BV(\Hb_{\ub}, \gamma)} := &\: \int_0^{u_*} \|\phi\|_{L^1(S_{u,\ub},\gamma)}\,\ud u + \sup \left\{| \int_0^{u_*} \int_{S_{u,\ub}} (\f{\rd}{\rd u}\varphi) \phi \, \mathrm{dA}_{\gamma} \,\ud u| : \varphi \in C^1_c,\,|\varphi|\leq 1\right\} \\
&\: + \sup \left\{| \int_0^{u_*} \int_{S_{u,\ub}} (\div\slashed{X}) \phi \, \mathrm{dA}_{\gamma} \,\ud u| : \slashed{X} \in C^1_c,\,\sup_{u,\,\ub} \|\slashed{X} \|_{L^1(S_{u,\ub},\gamma)}\leq 1\right\}.
\end{split}
\end{equation*}
\end{definition}

\begin{definition}[Continuity in $u$ and/or $\ub$]
Define the space $C^0_u C^0_{\ub} W^{m,p}(S_{u,\ub},\gamma)$ as the completion of smooth tensor fields under the $L^\i_u L^\i_{\ub} W^{m,p}(S_{u,\ub},\gamma)$. Define $C^0_u L^q_{\ub} W^{m,p}(S_{u,\ub},\gamma)$, $C^0_{\ub} L^q_{u} W^{m,p}(S_{u,\ub},\gamma)$ in a similar manner.
\end{definition}

At this point, let us recall that BV functions, even though they are defined only a.e., have well-defined traces on Lipschitz hypersurfaces. We will in particular need the following statement (whose proof can be found for instance in \cite[Theorem~5.6]{Evans}):
\begin{lemma}\label{lem:trace}
Let $f:[0,u_*]\times [0,\ub_*]\times \mathbb S^2 \to \mathbb R$ be such that $f \in C^0_u BV(H_u,\gamma)$. Then the following holds.
\begin{enumerate}
\item For every $(u,\ub)\in [0,u_*]\times (0,\ub_*)$, there is an $L^1(S_{u,\ub},\gamma)$ function $f^-(u,\ub,\theta)$ such that
$$\lim_{\ep\to 0^+} \f 1\ep\int_{\ub-\ep}^{\ub}\int_{S_{u,\ub'}} |f^-(u,\ub,\vartheta) - f(u,\ub,\vartheta)|\,\mathrm{dA}_\gamma\,\ud \ub'= 0.$$
\item For every $(u,\ub)\in [0,u_*]\times [0,\ub_*]$, there is an $L^1(S_{u,\ub},\gamma)$ function $f^+(u,\ub,\theta)$ such that
$$\lim_{\ep\to 0^+} \f 1\ep\int_{\ub}^{\ub+\ep}\int_{S_{u,\ub'}} |f^+(u,\ub,\vartheta) - f(u,\ub,\vartheta)|\,\mathrm{dA}_\gamma\,\ud \ub'= 0.$$
\end{enumerate}
Similar statements hold for $f\in C^0_{\ub} BV(\Hb_{\ub},\gamma)$ after swapping $u$ and $\ub$; we omit the details.
\end{lemma}

We need one more definition before we introduce the class of spacetime we consider. We introduce an auxiliary metric in all of $\mathcal M = [0,u_*]\times [0,\ub_*]\times \mathbb S^2$ to measure the regularity of $\gamma$. We define this to be the Lie-transported $\gamma \restriction_{S_{0,0}}$.
\begin{definition}\label{def:gamma00}
Define $\gamma_{0,0}$ on $\mathcal M = [0,u_*]\times [0,\ub_*]\times \mathbb S^2$ so that on the initial $2$-sphere $S_{0,0}$, $\gamma_{0,0} = \gamma\restriction_{S_{0,0}}$, and that\footnote{Note that this is possible since $[\f{\rd}{\rd u}, \f{\rd}{\rd\ub}] = 0$.} 
$$\slashed{\mathcal L}_{\f{\rd}{\rd u}} \gamma_{0,0} = 0,\quad \slashed{\mathcal L}_{\f{\rd}{\rd \ub}} \gamma_{0,0} = 0$$ 
everywhere else in $\mathcal M$.
\end{definition}

We are now ready to define the class of spacetimes that we study for the remainder of the paper.
\begin{definition}[Angularly regular double null metrics]\label{double.null.def.2}
Let $(\mathcal M = [0,u_*]\times [0,\ub_*]\times S, g)$ be a spacetime in double null coordinates (see~Definition~\ref{double.null.coordinates}). We say that $(\mathcal M, g)$ is \textbf{angularly regular} if 
\begin{enumerate}
\item $\gamma - \gamma_{0,0},\,b,\,\log\Omega \in C^0_u C^0_{\ub} W^{2,4}(S_{u,\ub}) \cap L^\i_u L^\i_{\ub} W^{3,2}(S_{u,\ub})$,
\item $\sup_{u,\ub}({\bf I}(S_{u,\ub},\gamma) + \mathrm{Area}(S_{u,\ub},\gamma) + (\mathrm{Area}(S_{u,\ub},\gamma))^{-1}) <+\infty$, $\log \f{\det\gamma}{\det\gamma_{0,0}} \in C^0_u C^0_{\ub} C^0(S_{u,\ub})$,
\item $K \in  C^0_u C^0_{\ub} L^4(S_{u,\ub}) \cap L^\i_u L^2_{\ub} W^{2,2}(S_{u,\ub}) \cap L^\i_{\ub} L^2_{u} W^{2,2}(S_{u,\ub})$,
\item $\chi,\,\om \in L^2_{\ub} L^\i_u W^{2,2}(S_{u,\ub}) \cap C^0_u L^2_{\ub} W^{2,2}(S_{u,\ub})\cap L^\i_u L^2_{\ub} W^{3,2}(S_{u,\ub})$, $\chib,\,\omb \in L^2_{u} L^\i_{\ub} W^{2,2}(S_{u,\ub}) \cap C^0_{\ub} L^2_{u} W^{2,2}(S_{u,\ub}) \cap L^\i_{\ub} L^2_{u} W^{3,2}(S_{u,\ub})$,
\item $\eta,\,\etab \in C^0_u C^0_{\ub} W^{1,4}(S_{u,\ub}) \cap L^\i_u L^\i_{\ub} W^{2,2}(S_{u,\ub})\cap L^\i_u L^2_{\ub} W^{3,2}(S_{u,\ub}) \cap L^\i_{\ub} L^2_u W^{3,2}(S_{u,\ub})$,
\item $\trch \in C^0_u L^1_{\ub} W^{2,1}(S_{u,\ub}) \cap L^\i_u BV(H_u) \cap L^\i_u L^\i_{\ub} W^{3,2}(S_{u,\ub})$, $\trchb \in C^0_{\ub} L^1_u W^{2,1}(S_{u,\ub}) \cap L^\i_{\ub} BV(\Hb_{\ub}) \cap L^\i_u L^\i_{\ub} W^{3,2}(S_{u,\ub})$.
\end{enumerate}
In the above, we have written $S_{u,\ub}$, $H_u$ and $\Hb_{\ub}$ instead of $(S_{u,\ub},\gamma)$, $(H_u,\gamma)$ and $(\Hb_{\ub},\gamma)$ to simplify the notations.
\end{definition}

\begin{remark}
Note that because of the presence of $C^0_u$ and $C^0_{\ub}$ in the above spaces, we can talk about $\gamma$, $b$, $\log\Omg$, $\eta$ and $\etab$ for every $u$ and $\ub$ (and not just for almost every $u$ and $\ub$). We can also talk about $\chi$ and $\om$ as an $L^2_{\ub}W^{2,2}(S_{u,\ub})$ function and $\trch$ as a $L^1_{\ub} W^{2,1}(S_{u,\ub})$ function for every $u$ (as opposed to just for almost every $u$). Similarly, we can talk about $\chib$ and $\omb$ as an $L^2_{u}W^{2,2}(S_{u,\ub})$ function and $\trchb$ as a $L^1_u W^{2,1}(S_{u,\ub})$ function for every $\ub$ (as opposed to just for almost every $\ub$).
\end{remark}

\subsection{Null structure equations}\label{sec:null.structure.eqn}

We introduce the \emph{null structure equations}, which are equations for the Ricci coefficients which hold on solutions to the Einstein vacuum equations. To state these equations, we need the following definitions (in addition to Definition~\ref{def:div.curl}):
\begin{definition}\label{def:contractions}
Define the following contractions
$$\phi^{(1)}\cdot\phi^{(2)} := (\gamma^{-1})^{AC}(\gamma^{-1})^{BD}\phi^{(1)}_{AB}\phi^{(2)}_{CD} \quad\mbox{for symmetric $2$-tensors $\phi^{(1)}_{AB}$, $\phi^{(2)}_{AB}$,}$$
$$\phi^{(1)}\cdot\phi^{(2)} := (\gamma^{-1})^{AB}\phi^{(1)}_{A}\phi^{(2)}_{B} \quad\mbox{for $1$-forms $\phi^{(1)}_{A}$, $\phi^{(2)}_{A}$,}$$
$$(\phi^{(1)}\cdot\phi^{(2)})_A := (\gamma^{-1})^{BC}\phi^{(1)}_{AB}\phi^{(2)}_{C} \quad\mbox{for a symmetric $2$-tensor $\phi^{(1)}_{AB}$ and a $1$-form $\phi^{(2)}_{A}$,}$$
$$(\phi^{(1)}\hot\phi^{(2)})_{AB} := \phi^{(1)}_A\phi^{(2)}_B+\phi^{(1)}_B\phi^{(2)}_A-\gamma_{AB}((\gamma^{-1})^{CD}\phi^{(1)}_C\phi^{(2)}_D) \quad\mbox{for one forms $\phi^{(1)}_A$, $\phi^{(2)}_A$,}$$
$$\phi^{(1)}\wedge\phi^{(2)} := \in^{AB}(\gamma^{-1})^{CD}\phi^{(1)}_{AC}\phi^{(2)}_{BD}\quad\mbox{for symmetric two tensors $\phi^{(1)}_{AB}$, $\phi^{(2)}_{AB}$}.$$
Define $^*$ of $1$-forms and symmetric $2$-tensors respectively as follows (note that on $1$-forms this is the Hodge dual on $S_{u,\ub}$):
\begin{align*}
^*\phi_A :=  & \gamma_{AC} \in^{CB} \phi_B, \quad ^*\phi_{AB} \doteq  \gamma_{BD} \in^{DC} \phi_{AC}.
\end{align*}
Define also the trace for totally symmetric tensors of rank $r$ to be
\begin{equation}\label{tr.def}
(\tr\phi)_{A_1...A_{r-1}}:= (\gamma^{-1})^{BC}\phi_{BCA_1...A_{r-1}}.
\end{equation}
\end{definition}

We now give a list of the null structure equations.
\begin{proposition}\label{prop:null.structure}
Let $(\mathcal M = [0,u_*]\times [0,\ub_*]\times S,g)$ be a $C^2$ spacetime in double null coordinates. If $(\mathcal M,g)$ solves the Einstein vacuum equations, then the following \textbf{null structure equations} hold:
\begin{align}
\nab_4 \trch+\frac 12 (\trch)^2=&\: -|\chih|_\gamma^2-2\omega \trch, \label{Ric44}\\
\nab_3 \trchb+\frac 12 (\trchb)^2=&\: -|\chibh|_\gamma^2-2\omegab \trchb, \label{Ric33} \\
\nab_4\eta + \f 34 \trch (\eta-\etab) =&\:  \div\chih -\frac 12 \nab \trch - \f 12(\eta - \etab)\cdot  \chih, \label{Ric4A} \\
\nab_3\etab +\f 34 \trchb (\etab-\eta) = &\: \div\chibh - \frac 12 \nab \trchb - \f 12(\etab-\eta) \cdot \chibh, \label{Ric3A} \\ 
\nab_4 \trchb+\trch \trchb =&\: 2\omega \trchb -2 K +2\div \etab +2|\etab|_\gamma^2,\label{trRicAB}\\
\nab_3 \trch+ \trchb \trch =&\: 2\omegab \trch -2K +2\div \eta+2|\eta|_\gamma^2, \label{trRicAB.1}\\
\nab_4\chibh +\f 12 \trch \chibh =&\: \nab\widehat{\otimes} \etab+2\omega \chibh -\f 12\trchb \chih + \etab\widehat{\otimes} \etab, \label{RicAB} \\
\nab_3\chih +\f 12 \trchb \chih =&\: \nab\widehat{\otimes} \eta+2\omegab \chih -\f 12\trch \chibh + \eta\widehat{\otimes} \eta, \label{RicAB.1} \\
\nab_4\omegab-2\omega\omegab+ \eta\cdot\etab-\f 12|\eta|_\gamma^2 =&\: -\f12(K-\f12 \chih\cdot\chibh+\f14 \trch\trchb), \label{Ric34} \\
\nab_3\omega-2\omega\omegab+\eta\cdot\etab-\f 12|\etab|_\gamma^2 =&\:-\f12(K-\f12 \chih\cdot\chibh+\f14 \trch\trchb). \label{Ric34.1}
\end{align}
\end{proposition}
\begin{proof}
See the derivation for instance in \cite{Chr}. \qedhere
\end{proof}

\subsection{Weak formulation of transport equations}\label{sec:weak.transport}

\begin{definition}\label{def:weak.transport}
Let $(\mathcal M,g)$ be an angularly regular metric in double null coordinates. Consider the transport equations
\begin{equation}\label{eq:transport.3}
\nab_3 \phi = F,
\end{equation}
\begin{equation}\label{eq:transport.4}
 \nab_4 \psi = G,
\end{equation}
where $\phi$, $F$, $\psi$, $G$ are $S$-tangent covariant tensor fields of rank $r$ in $C^0_u C^0_{\ub} L^p(S)$ for some $p\in [1,+\infty]$. 

We say that \eqref{eq:transport.3} (resp.~\eqref{eq:transport.4}) is satisfied in the \textbf{integrated sense} if for every $C^1$ contravariant $S$-tangent tensor $\varphi$ of rank $r$, the following holds $\forall 0\leq u_1< u_2\leq u_*,\,\forall \ub \in [0,\ub_*]$
\begin{equation*}
\begin{split}
\int_{S_{u_2,\ub}} \langle\varphi, \phi\rangle \Omg \,\ud A_{\gamma} &\: - \int_{S_{u_1,\ub}} \langle\varphi, \phi\rangle \Omg \,\ud A_{\gamma} - \int_{u_1}^{u_2} \int_{S_{u',\ub}} (\langle \varphi, F + (\trchb - 2\omb)\phi \rangle + \langle \nab_3\varphi, \phi\rangle) \Omg^2 \,\ud A_{\gamma}\, \ud u' =0,
\end{split}
\end{equation*}
(resp.~$\forall 0\leq \ub_1< \ub_2\leq \ub_*,\,\forall \ub \in [0,u_*]$
\begin{equation*}
\begin{split}
\int_{S_{u,\ub_1}} \langle\varphi, \psi\rangle \Omg \,\ud A_{\gamma} &\: - \int_{S_{u,\ub_2}} \langle\varphi, \psi\rangle \Omg \,\ud A_{\gamma} - \int_{\ub_1}^{\ub_2} \int_{S_{u,\ub'}} (\langle \varphi, G + (\trch - 2\om)\psi\rangle + \langle \nab_4\varphi,\psi\rangle) \Omg^2 \,\ud A_{\gamma}\, \ud \ub' = 0. )
\end{split}
\end{equation*}
\end{definition}

\begin{definition}\label{def:weaker.transport}
Let $(\mathcal M,g)$ be an angularly regular metric in double null coordinates. \eqref{eq:transport.3} and \eqref{eq:transport.4} where $\phi$, $F$ are $S$-tangent covariant tensor fields of rank $r$ in $C^0_u L^2_{\ub} L^p(S)$ for some $p\in [1,+\infty]$; and $\psi$, $G$ are $S$-tangent covariant tensor fields of rank $r$ in $C^0_{\ub} L^2_u L^p(S)$ for some $p \in [1,+\infty]$. 

We say that \eqref{eq:transport.3} (resp.~\eqref{eq:transport.4}) is satisfied in the \textbf{weak integrated sense} if for every $C^1$ contravariant $S$-tangent tensor $\varphi$ of rank $r$, the following holds $\forall 0\leq u_1< u_2\leq u_*$:
\begin{equation*}
\begin{split}
\int_0^{\ub_*} \int_{S_{u_2,\ub}} \langle\varphi, \phi\rangle \Omg \,\ud A_{\gamma}\, \ud \ub &\: - \int_0^{\ub_*} \int_{S_{u_1,\ub}} \langle\varphi, \phi\rangle \Omg \,\ud A_{\gamma}\,\ud \ub \\
&\: - \int_0^{\ub_*}\int_{u_1}^{u_2} \int_{S_{u',\ub}} (\langle \varphi, F + (\trchb - 2\omb)\phi \rangle + \langle \nab_3\varphi, \phi\rangle) \Omg^2 \,\ud A_{\gamma}\, \ud u'\,\, \ud \ub =0,
\end{split}
\end{equation*}
(resp.~$\forall 0\leq \ub_1< \ub_2\leq \ub_*$:
\begin{equation*}
\begin{split}
\int_0^{u_*} \int_{S_{u,\ub_1}} \langle\varphi, \psi\rangle \Omg \,\ud A_{\gamma} \,\ud u &\: - \int_0^{u_*} \int_{S_{u,\ub_2}} \langle\varphi, \psi\rangle \Omg \,\ud A_{\gamma} \,\ud u \\
&\: - \int_0^{u_*} \int_{\ub_1}^{\ub_2} \int_{S_{u,\ub'}} (\langle \varphi, G + (\trch - 2\om)\psi\rangle + \langle \nab_4\varphi,\psi\rangle) \Omg^2 \,\ud A_{\gamma}\, \ud \ub'\,\ud u = 0. )
\end{split}
\end{equation*}
\end{definition}

\begin{remark}
It is easy to see by integration by parts, \eqref{metric.derivative.invar} and \eqref{Ricci.relation} that 
$$\boxed{\mbox{classical sense}} \implies \boxed{\mbox{integrated sense}} \implies \boxed{\mbox{weak integrated sense}}.$$
\end{remark}

\subsection{Weak formulation of Einstein vacuum equations in the double null gauge}\label{sec.weak.double.null}

In this subsection, we give a weak formulation of the Einstein vacuum equations in the double null gauge, which is slightly stronger than that in Definition~\ref{def:vacuum} and takes advantage of angular regularity. Our formulation relies on notions introduced in Sections~\ref{sec:null.structure.eqn} and \ref{sec:weak.transport}.

\begin{definition}\label{def:weak.sol.vac.ang.reg}
Let $(\mathcal M= [0,u_*]\times [0,\ub_*]\times \mathbb S^2,g)$ be an angularly regular spacetime in double null coordinates. 

We say that $(\mathcal M,g)$ is an \textbf{angularly regular weak solution to the Einstein vacuum equations} if the following holds:
\begin{enumerate}
\item \eqref{Ric44}--\eqref{Ric3A} are satisfied in the integrated sense (Definition~\ref{def:weak.transport}).
\item \eqref{trRicAB}--\eqref{Ric34.1} are satisfied in the weak integrated sense (Definition~\ref{def:weaker.transport}).
\end{enumerate}
\end{definition}

\begin{remark}
We remark that in order to make sense of Definition~\ref{def:weak.sol.vac.ang.reg}, we do not need to full regularity assumptions in Definition~\ref{double.null.def.2}. We make the stronger assumptions in Definition~\ref{double.null.def.2} because that will be the relevant class of spacetimes in the later parts of the paper.
\end{remark}

The following proposition clarifies the relation between the notions of solutions Definitions~\ref{def:vacuum} and \ref{double.null.def.2}, as well as their relation to classical solutions. Part (1) is an immediate consequence of Proposition~\ref{prop:null.structure}; part (2) is a direct computation. We omit the details.
\begin{proposition}\label{prop:weak.sol.vac.ang.reg}
\begin{enumerate}
\item Suppose $(\mathcal M = [0,u_*]\times [0,\ub_*]\times \mathbb S^2,g)$ is a $C^2$ (classical) solutions to the Einstein vacuum equations in double null coordinates (see Definition~\ref{double.null.def}), then $(\mathcal M, g)$ angularly regular weak solution to the Einstein vacuum equations in the sense of Definition~\ref{def:weak.sol.vac.ang.reg}.
\item Suppose $(\mathcal M= [0,u_*]\times [0,\ub_*]\times \mathbb S^2,g)$ be an angularly regular weak solution to the Einstein vacuum equations in the sense of Definition~\ref{def:weak.sol.vac.ang.reg}, then $(\mathcal M, g)$ is a weak solution to the Einstein vacuum equations in the sense of Definition~\ref{def:vacuum}.
\end{enumerate}
\end{proposition}

\subsection{Weak formulation of the Einstein--null dust system in the double null gauge}\label{sec.weak.double.null.dust}

In analogy to Definition~\ref{def:weak.sol.vac.ang.reg}, we introduce in this subsection a weak formulation of the Einstein--null dust system in the double null gauge that uses angular regularity; see already Definition~\ref{def:weak.sol.ang.reg}.

We begin with defining the class of measures (which will represent the null dusts) which we will consider.

\begin{definition}\label{def:ang.reg.null.dust}
Let $(\mathcal M= [0,u_*]\times [0,\ub_*]\times \mathbb S^2,g)$ be an angularly regular spacetime in double null coordinates. Let $\{\ud \nu_u\}_{u\in [0,u_*]}$ be a $1$-parameter family of measures such that for every $u\in [0,u_*]$, $\ud\nu_u$ is a non-negative Radon measure on $\{u\}\times (0,\ub_*) \times \mathbb S^2$. Similarly, let $\{\ud \nub_{\ub}\}_{\ub \in [0,\ub_*]}$ be a $1$-parameter family of measures such that for every $\ub \in [0,\ub_*]$, $\ud\nub_{\ub}$ is a non-negative Radon measure on $(0,u_*)\times \{\ub\}\times \mathbb S^2$.

We say that $\{\ud\nu_u\}_{u \in [0,u_*]}$ and $\{\ud\nub_{\ub}\}_{\ub \in [0,\ub_*] }$ are \textbf{angularly regular} if the following holds:
\begin{enumerate}
\item $\ud\nu_u$ is continuous in $u$ and $\ud\nub_{\ub}$ is continuous in $\ub$ with respect to the weak-* topology, i.e.
$$\lim_{u'\to u} \int_{\{u'\}\times (0,\ub_*)\times \mathbb S^2} \varphi \,\ud \nu_{u'} = \int_{\{u\}\times (0,\ub_*)\times \mathbb S^2} \varphi \,\ud \nu_{u},\quad \forall \varphi\in C^0_c([0,u_*]\times (0,\ub_*)\times \mathbb S^2),$$
$$\lim_{\ub'\to \ub} \int_{(0,u_*)\times \{\ub'\}\times \mathbb S^2} \varphi \,\ud \nu_{\ub'} = \int_{(0,u_*)\times \{\ub\}\times \mathbb S^2} \varphi \,\ud \nu_{u},\quad \forall \varphi\in C^0_c((0,u_*)\times [0,\ub_*]\times \mathbb S^2).$$
\item Angular regularity holds in the following sense. For every $u$, let
\begin{equation*}
\begin{split}
\accentset{\scalebox{.7}{\mbox{\tiny (0)}}}{{\mathfrak X}}_u:= \{ \varphi \mbox{ real valued function} : &\: \varphi \in C^0_{\ub}L^1(S_{u,\ub}),\,\|\varphi \|_{L^\i_{\ub}L^1(S_{u,\ub})} \leq 1\}, 
\end{split}
\end{equation*}
\begin{equation*}
\begin{split}
\accentset{\scalebox{.7}{\mbox{\tiny (1)}}}{{\mathfrak X}}_u:= \{ \slashed X \mbox{ $S$-tangent vector field} : &\: \slashed X,\,\div\slashed{X} \in C^0_{\ub}L^1(S_{u,\ub}),\,\|\slashed X\|_{L^\i_{\ub}L^1(S_{u,\ub})} \leq 1\}, 
\end{split}
\end{equation*}
\begin{equation*}
\begin{split}
\accentset{\scalebox{.7}{\mbox{\tiny (2)}}}{{\mathfrak X}}_u:= \{ (\slashed X, \slashed Y) : &\: \slashed X, \,\slashed Y \mbox{ $S$-tangent vector fields}, \slashed X,\,\slashed Y,\, \div (\slashed{X}\otimes \slashed Y),\\
&\: \div \div (\slashed{X}\otimes\slashed{Y}) \in C^0_{\ub}L^1(S_{u,\ub}),\, \|\slashed X\otimes \slashed Y\|_{L^\i_{\ub}L^{\f 43}(S_{u,\ub})} \leq 1 \}.
\end{split}
\end{equation*}
Similarly, for every $\ub$, define
\begin{equation*}
\begin{split}
\accentset{\scalebox{.7}{\mbox{\tiny (0)}}}{{\mathcal X}}_{\ub}:= \{ \varphi \mbox{ real valued function} : &\: \varphi \in C^0_u L^1(S_{u,\ub}),\,\|\varphi \|_{L^\i_u L^1(S_{u,\ub})} \leq 1\}, 
\end{split}
\end{equation*}
$$\accentset{\scalebox{.7}{\mbox{\tiny (1)}}}{{\mathcal X}}_{\ub}:= \{ \slashed X \mbox{ $S$-tangent vector field} : \slashed X,\,\div \slashed{X} \in C^0_{u}L^1(S_{u,\ub}),\,\|\slashed X\|_{L^\i_{u}L^1(S_{u,\ub})} \leq 1\},$$
\begin{equation*}
\begin{split}
\accentset{\scalebox{.7}{\mbox{\tiny (2)}}}{{\mathcal X}}_{\ub}:= \{ (\slashed X, \slashed Y) : &\: \slashed X, \,\slashed Y \mbox{ $S$-tangent vector fields}, \slashed X,\,\slashed{Y},\, \div (\slashed{X}\otimes \slashed Y),\\
&\: \div \div (\slashed{X}\otimes\slashed{Y})  \in C^0_{u}L^1(S_{u,\ub}),\, \|\slashed X\otimes \slashed Y\|_{L^\i_{u}L^{\f 43}(S_{u,\ub})} \leq 1 \}.
\end{split}
\end{equation*}
Then there exists $C>0$ such that
\begin{equation}\label{eq:nu.add.reg}
\begin{split}
&\: \sup_{u \in [0,u_*]} \left(\sup_{\varphi \in \accentset{\scalebox{.7}{\mbox{\tiny (0)}}}{{\mathfrak X}}_{u} } \left|\int_{\{u\}\times [0,\ub_*]\times \mathbb S^2} \varphi \,\ud\nu_u\right|  + \sup_{\slashed X\in \accentset{\scalebox{.7}{\mbox{\tiny (1)}}}{{\mathfrak X}}_{u} } \left|\int_{\{u\}\times [0,\ub_*]\times \mathbb S^2} \div \slashed{X}\,\ud\nu_u\right| \right) \\
&\: \qquad + \sup_{u \in [0,u_*]} \sup_{(\slashed X,\,\slashed Y) \in \accentset{\scalebox{.7}{\mbox{\tiny (2)}}}{{\mathfrak X}}_{u} } \left|\int_{\{u\}\times [0,\ub_*]\times \mathbb S^2} \div\, \div (\slashed{X}\otimes\slashed{Y}) \,\ud\nu_u\right| \leq C,
\end{split}
\end{equation}
and
\begin{equation}\label{eq:nub.add.reg}
\begin{split}
&\: \sup_{\ub \in [0,\ub_*]} \left(\sup_{\varphi \in \accentset{\scalebox{.7}{\mbox{\tiny (0)}}}{{\mathcal X}}_{\ub}} \left|\int_{[0,u_*]\times \{\ub\} \times \mathbb S^2} \varphi \,\ud\nu_{\ub} \right|  +  \sup_{\slashed X\in \accentset{\scalebox{.7}{\mbox{\tiny (1)}}}{{\mathcal X}}_{\ub} } \left|\int_{[0,u_*]\times \{\ub\} \times \mathbb S^2} \div \slashed{X}\,\ud\underline{\nu}_{\ub} \right|\right) \\
&\: \qquad + \sup_{\ub \in [0,\ub_*]} \sup_{(\slashed X,\,\slashed Y) \in \accentset{\scalebox{.7}{\mbox{\tiny (2)}}}{{\mathcal X}}_{\ub} } \left|\int_{[0,u_*]\times \{\ub\}\times \mathbb S^2} \div\, \div (\slashed{X}\otimes\slashed{Y}) \, \ud\underline{\nu}_{\ub}\right| \leq C,
\end{split}
\end{equation}
where we used the convention $\div\, \div (\slashed{X}\otimes\slashed{Y}):= \nab_A \nab_B (X^A Y^B)$.
\end{enumerate}
\end{definition}

\begin{definition}\label{def:weak.sol.ang.reg}
Let $(\mathcal M= [0,u_*]\times [0,\ub_*]\times \mathbb S^2,g)$ be an angularly regular spacetime in double null coordinates, and let $(\{\ud\nu_u\}_{u\in [0,u_*]}, \{\ud\nu_{\ub}\}_{\ub \in [0,\ub_*]})$ be angularly regular non-negative Radon measures (see Definition~\ref{def:ang.reg.null.dust}).

We say that $(\mathcal M,g,\{\ud\nu_u\}_{u\in [0,u_*]}, \{\ud\nu_{\ub}\}_{\ub \in [0,\ub_*]})$ is an \textbf{angularly regular weak solution to the Einstein--null dust system} if the following holds:
\begin{enumerate}
\item \eqref{Ric4A} and \eqref{Ric3A} are satisfied in the integrated sense (Definition~\ref{def:weak.transport}).
\item \eqref{trRicAB}--\eqref{Ric34.1} are satisfied in the weak integrated sense (Definition~\ref{def:weaker.transport}).
\item The following equations hold for $\trch$ and $\trchb$ (instead of \eqref{Ric44} and \eqref{Ric33}). For all $0\leq u_1<u_2 \leq u_*$, $\ub\in [0,\ub_*]$ and $C^1$ function $\varphi: [0,u_*]\times [0,\ub_*] \times \mathbb S^2\to \mathbb R$,
\begin{equation}\label{eq:trchb}
\begin{split}
&\: \int_{S_{u_2,\ub}}  \varphi \Omg \trchb^- \,\mathrm{dA}_{\gamma} - \int_{S_{u_1,\ub}} \varphi \Omg \trchb^+ \,\mathrm{dA}_{\gamma} \\
= &\: \int_{u_1}^{u_2} \int_{S_{u,\ub}} ((e_3\varphi)\trchb -4 \varphi \omb \trchb+ \f12 \varphi (\trchb)^2 - \varphi|\chibh|_{\gamma}^2)\Omg^2\,\mathrm{dA}_{\gamma} \,\ud u -\int_{(u_1, u_2)\times \{\ub\}\times \mathbb S^2}  \varphi\,\ud \underline{\nu}_{\ub},
\end{split}
\end{equation}
For all $0\leq \ub_1 < \ub_2 \leq \ub_*$, $u\in [0,u_*]$ and $C^1$ function $\varphi: [0,u_*]\times [0,\ub_*] \times \mathbb S^2\to \mathbb R$,
\begin{equation}\label{eq:trch}
\begin{split}
&\: \int_{S_{u,\ub_2}}  \varphi \Omg \trch^- \,\mathrm{dA}_{\gamma} - \int_{S_{u,\ub_1}} \varphi \Omg \trch^+ \,\mathrm{dA}_{\gamma} \\
= &\: \int_{\ub_1}^{\ub_2} \int_{S_{u,\ub}} ((e_4 \varphi)\trch -4 \varphi \om \trch+ \f12 \varphi (\trch)^2 - \varphi|\chih|_{\gamma}^2)\Omg^2\,\mathrm{dA}_{\gamma} \,\ud \ub -\int_{\{u\}\times (\ub_1, \ub_2)\times \mathbb S^2}  \varphi\,\ud \nu_u.
\end{split}
\end{equation}

\item The following equations hold for $\ud \nu_u$ and $\ud\underline{\nu}_{\ub}$. For every $0\leq u_1 <u_2 \leq u_*$, and every $C^1_c$ function $\varphi: [0,u_*]\times (0,\ub_*) \times \mathbb S^2\to \mathbb R$,
\begin{equation}\label{eq:nu}
\begin{split}
&\: \int_{\{u_2\}\times (0,\ub_*) \times \mathbb S^2} \varphi \,\ud \nu_{u_2} \\
= &\: \int_{\{u_1\}\times  (0,\ub_*) \times \mathbb S^2} \varphi\,\ud\nu_{u_1} + \int_{u_1}^{u_2} \int_{\{u\}\times  (0,\ub_*) \times \mathbb S^2} (\f{\rd\varphi}{\rd u} + \nab_{b} \varphi ) \,\ud \nu_{u}\,\ud u.
\end{split}
\end{equation}
For every $0\leq \ub_1 <\ub_2 \leq \ub_*$, and every $C^1_c$ function $\varphi: (0,u_*)\times [0,\ub_*] \times \mathbb S^2\to \mathbb R$,
\begin{equation}\label{eq:nub}
\int_{(0,u_*) \times \{\ub_2\} \times \mathbb S^2} \varphi \,\ud \nub_{\ub_2} = \int_{(0,u_*) \times \{\ub_1\} \times \mathbb S^2} \varphi\,\ud\nub_{\ub_1} + \int_{\ub_1}^{\ub_2} \int_{(0,u_*) \times \{\ub\} \times \mathbb S^2} \f{\rd\varphi}{\rd \ub}  \,\ud \nub_{\ub}\,\ud \ub.
\end{equation}
\end{enumerate}
\end{definition}

The following is easy to check (cf.~part (2) of Proposition~\ref{prop:weak.sol.vac.ang.reg} in the vacuum case); we omit the proof.
\begin{lemma}
Suppose $(\mathcal M,g,\{\ud\nu_u\}_{u\in [0,u_*]}, \{\ud\nu_{\ub}\}_{\ub \in [0,\ub_*]})$ is an angularly regular weak solution to the Einstein--null dust system in the sense of Definition~\ref{def:weak.sol.ang.reg}. Then, for $\ud \nu:= \Omg^2\, \ud\nu_u \,\ud u$ and $\ud\nub:= \Omg^2 \,\ud\nub_{\ub}\,\ud\ub$, $(\mathcal M, g, \ud\nu, \ud\nub)$ is a weak solution to the Einstein--null dust system in the sense of Definition~\ref{def:null.dust}.
\end{lemma}

\subsection{Renormalized curvature components and the renormalized Bianchi equations}

Given an angularly regular spacetime $(\mathcal M,g)$ (see~Definition~\ref{double.null.def.2}), define the following $S$-tangent tensor fields:
\begin{definition}\label{def:curv}
Define
$$\beta  :=-\div\chih + \frac 12 \slashed{\nabla} \trch - \f 12(\eta-\etab) \cdot (\chi - \trch\gamma),\quad 
\betab := \div\chibh - \frac 12 \slashed{\nabla} \trchb - \f 12(\eta-\etab)\cdot (\chib -\trchb\gamma),\quad \sigmac := \curl\eta.$$
\end{definition}

We will call $(\beta,\,\betab,\,\sigmac)$ and $K$ the \textbf{renormalized curvature components}. The relation of $(\beta,\,\betab,\,\sigmac)$ to the spacetime curvature components is given by the following:
\begin{lemma}
If $(\mathcal M,g)$ is a $C^2$ (classical) solution to the Einstein vacuum equations, then
$$\bt_A = \frac 1 2 R(e_A,  e_4, e_3, e_4),\quad\sigmac = \frac 1 4  \,^*R(e_4,e_3, e_4,  e_3)+ \f12 \in^{AB}(\gamma^{-1})^{CD}\chibh_{AC}\chih_{BD},$$ 
$$\betab_A = \frac 1 2 R(e_A,  e_3,  e_3, e_4),$$
where $R$ is the Riemann curvature tensor of $g$ and $\, ^*R$ denotes the Hodge dual of $R$.
\end{lemma}

For sufficiently regular spacetimes in double null coordinates, the renormalized curvature components obey the following \textbf{renormalized Bianich equations} (recall definitions from Definitions~\ref{def:div.curl} and \ref{def:contractions}):
\begin{proposition}\label{prop:Bianchi}
If $(\mathcal M,g)$ is a $C^3$ (classical) solution to the Einstein vacuum equations, then the following system of equations holds:
\begin{align}
\nab_3\beta+\trchb\beta=&\: -\slashed{\nabla} K  +^*\slashed{\nabla}\sigmac + 2\omegab \beta+2\chih\cdot\betab-3(\eta K-^*\eta\sigmac)+\frac 1 2(\slashed{\nabla}(\chih\cdot\chibh)+^*\slashed{\nabla}(\chih\wedge\chibh)) \nonumber\\
&+\f 32(\eta\chih\cdot\chibh+^*\eta\chih\wedge\chibh)-\frac 14 (\nab\trch \trchb+\trch\nab\trchb)-\frac 34 \eta\trch\trchb,\label{eq:null.Bianchi.1}\\
\nab_4\sigmac+\frac 32\trch\sigmac=&\: -\div^*\beta-\f 12(\eta - \etab)\wedge\beta-2\etab\wedge
\beta-\frac 12 \chih\wedge(\nab\widehat{\otimes}\etab)-\frac 12 \chih\wedge(\etab\widehat{\otimes}\etab),\label{eq:null.Bianchi.2}\\
\nab_4 K+\trch K=&\: -\div\beta-\f 12(\eta - \etab)\cdot\beta-2\etab\cdot\beta+\frac 12 \chih\cdot\nab\widehat{\otimes}\etab+\frac 12 \chih\cdot(\etab\widehat{\otimes}\etab)-\frac 12 \trch\div\etab-\frac 12\trch |\etab|^2,\label{eq:null.Bianchi.3}\\
\nab_3\sigmac+\frac 32\trchb\sigmac=&\: -\div ^*\betab+\f 12(\eta - \etab)\wedge\betab-2\eta\wedge
\betab+\frac 12 \chibh\wedge(\nab\widehat{\otimes}\eta)+\frac 12 \chibh\wedge(\eta\widehat{\otimes}\eta),\label{eq:null.Bianchi.4}\\
\nab_3 K+\trchb K=&\: \div\betab-\f 12(\eta - \etab)\cdot\betab+2\eta\cdot\betab+\frac 12 \chibh\cdot\nab\widehat{\otimes}\eta+\frac 12 \chibh\cdot(\eta\widehat{\otimes}\eta)-\frac 12 \trchb\div\eta-\frac 12 \trchb |\eta|^2,\label{eq:null.Bianchi.5}\\
\nab_4\betab+\trch\betab=&\: \slashed{\nabla} K +^*\slashed{\nabla}\sigmac+ 2\omega\betab +2\chibh\cdot\beta+3(\etab K+^*\etab\sigmac)-\frac 1 2(\slashed{\nabla}(\chih\cdot\chibh)-^*\slashed{\nabla}(\chih\wedge\chibh))\nonumber\\
&+\frac 14 (\nab\trch \trchb+\trch\nab\trchb)-\f 32(\etab\chih\cdot\chibh-^*\etab\chih\wedge\chibh)+\frac 34 \etab\trch\trchb. \label{eq:null.Bianchi.6}
\end{align}

\end{proposition}
\begin{proof}
These equations can be derived starting the decomposing the Bianchi equations $\nabla^{\alp}R_{\alp\bt\mu\nu} = 0$ with respect to the null frame and then performing algebraic manipulations; they can be found for instance in \cite{DL}. \qedhere
\end{proof}

\begin{definition}\label{def:Bianchi.integrated}
Let $(\mathcal M = [0,u_*]\times [0,\ub_*]\times \mathbb S^2,g)$ be an angularly regular spacetime in double null coordinates. We say that \textbf{the renormalized Bianchi equations are satisfied} if 
\begin{enumerate}
\item the equations \eqref{eq:null.Bianchi.2}--\eqref{eq:null.Bianchi.5} are satisfied in the integrated sense (Definition~\ref{def:weak.transport}); 
\item the equations \eqref{eq:null.Bianchi.1} and \eqref{eq:null.Bianchi.6} are satisfied in the weak integrated sense (Definition~\ref{def:weaker.transport}).
\end{enumerate}
\end{definition}

\subsection{Auxiliary equations}

In this subsection, we discuss a few auxiliary equations. We will show that when the spacetime is sufficiently regular, they hold as a consequence of the null structure equations (recall Section~\ref{sec:null.structure.eqn}). We will then introduce an appropriate (weak) notion for these solutions in the setting of angularly regular spacetimes; see already Definition~\ref{def:aux.integrated}. 

The equations that we will be interested in are those for the higher derivatives of the metric components, those for the mass aspect functions, and those for the derivatives of $\trch$ and $\trchb$. These are not necessary to make sense of the Einstein equations weakly, but will be useful for the proof of our uniqueness theorem (Theorem~\ref{thm:uniqueness.intro}).

\subsubsection{Higher order transport equation for the metric components} 

\begin{proposition}\label{prop:higher.order.metric.C2}
The following holds for a $C^2$ metric in double null coordinates (see \eqref{double.null.coordinates}):
\begin{equation}\label{eq:nablagamma}
\nab_4 \nab \gamma =  0,
\end{equation}
\begin{equation}\label{eq:nablaOmega}
\nab_4 \nab \log\Omg =  - (\eta + \etab) \omega - \chi \cdot \nab \log\Omg.
\end{equation}
Under the same assumptions, the following equation for $\nab b$, which we write in index notations, also holds
\begin{equation}\label{eq:nablab}
\begin{split}
\nab_4 \nab_B b^A = &\: - (\eta + \etab)_B  (\eta + \etab)^A + \f 12 (\eta +\etab)_B \chi_{AC} b^C -  (\gamma^{-1})^{CD} \chi_{BD}\nab_C b^{A} \\
&\: + \sum_{i=1}^r (\chi_{B}{ }^A \etab_C - \chi_{BC} \etab^{A}+ \in_{AC} { }^*\beta_B ) b^C.
\end{split}
\end{equation}
\end{proposition}
\begin{proof}
By Lemma~7.3.3 in \cite{CK}, the following commutation formula holds:
\begin{equation}\label{eq:CK.Lemma733}
\begin{split}
[ \nab_4,\nab_B] \phi_{A_1\dots A_r}  = &\:  \f 12(\eta_B + \etab_B)\nab_4\phi_{A_1\dots A_r} - (\gamma^{-1})^{CD} \chi_{BD}\nab_C \phi_{A_1\dots A_r} \\
&\: +\sum_{i=1}^r ((\gamma^{-1})^{CD} \chi_{A_i B} \etab_D - (\gamma^{-1})^{CD} \chi_{BD} \etab_{A_i}+ \in_{A_i}{ }^C { }^*\beta_B )\phi_{A_1\dots \hat{A_i} C\dots A_r},
\end{split}
\end{equation}
where ${ }^*$ denotes the Hodge dual on the $2$-sphere, and $\hat{A_i}$ in the indices means that the original $A_i$ is removed.

The conclusion then follows from applying the commutation formula to
\begin{equation}\label{eq:g.eqn.in.nab}
\nab_4 \gamma = 0,\quad \nab_4 b = - 2\Omg (\eta - \etab) + \chi\cdot b,\quad \nab_4 \log \Omg = -2 \omega,
\end{equation}
noting that the regularity assumptions are sufficient to justify the use of the commutation formula. \qedhere
\end{proof}

\subsubsection{Transport equations for the mass aspect functions}

Define the mass aspect functions $\mu$ and $\mub$ by 
\begin{equation}\label{eq:mu.def}
\mu:= - \div \eta + K,\quad \mub:=- \div \etab + K.
\end{equation}
We then have the following equations\footnote{Remark that the key point for introducing $\mu$ and $\mub$ (instead of just considering $\div\eta$ and $\div\etab$) is that in the transport equations they satisfy, no terms of first derivative of curvature appear on the RHS.} for $\mu$ and $\mub$.
\begin{proposition}\label{prop:mu.background}
The following holds for a $C^3$ solution to the Einstein vacuum equations in double null coordinates (see \eqref{double.null.coordinates}):
\begin{equation}\label{eq:mu.0}
\begin{split}
\nab_4 \mu = &\: \div \{\chi\cdot(\eta-\etab) \} + \f 12 (\eta + \etab) \cdot \{\chi\cdot(\eta-\etab) + \bt\} + \chi \cdot \nab \eta - \trch \etab \cdot \eta + \chi\cdot \etab\cdot \eta + \beta \cdot \eta \\
 &\: -\trch K -\f 12(\eta - \etab)\cdot\beta-2\etab\cdot\beta+\frac 12 \chih\cdot\nab\widehat{\otimes}\etab+\frac 12 \chih\cdot(\etab\widehat{\otimes}\etab)-\frac 12 \trch\div\etab-\frac 12\trch |\etab|^2,
 \end{split}
 \end{equation}
 and
 \begin{equation}\label{eq:mub.0}
 \begin{split}
 \nab_3 \mub =&\: \div \{\chi\cdot(\etab-\eta) \} + \f 12 (\eta + \etab) \cdot \{\chi\cdot(\etab-\eta) - \betab\} +\chib \cdot \nab \etab - \trchb \eta \cdot \etab + \chib\cdot \eta\cdot \etab - \betab \cdot \etab \\
 &\: -\trchb K-\f 12(\eta - \etab)\cdot\betab+2\eta\cdot\betab+\frac 12 \chibh\cdot\nab\widehat{\otimes}\eta+\frac 12 \chibh\cdot(\eta\widehat{\otimes}\eta)-\frac 12 \trchb\div\eta-\frac 12 \trchb |\eta|^2.
\end{split}
\end{equation}
\end{proposition} 
\begin{proof}
The following commutation formulae hold for any $C^2$ $S$-tangent $1$-form $\xi$ on a $C^2$ metric in a double null coordinate system (the equation \eqref{eq:commutation.4} can be derived from \eqref{eq:CK.Lemma733}; \eqref{eq:commutation.3} can be achieved similarly starting from Lemma~7.3.3 in \cite{CK}):
\begin{align}
[\nab_4,\div] \xi = &\:  \f 12 (\eta+ \etab)\cdot \nab_4\xi - \chi\cdot \nab\xi + \trch \etab \xi - \chi\cdot \etab\cdot \xi - \bt \cdot \xi, \label{eq:commutation.4}\\
[\nab_3,\div] \xi = &\:  \f 12 (\eta+ \etab)\cdot \nab_3\xi - \chib\cdot \nab\xi + \trchb \eta \xi - \chib\cdot \eta\cdot \xi + \betab \cdot \xi. \label{eq:commutation.3}
\end{align}

Now, by Definition~\ref{def:curv}, the equations \eqref{Ric4A} and \eqref{Ric3A} can be rephrased as
\begin{equation}\label{eq:eta.etab.rephrased}
\nab_4 \eta = -\chi\cdot(\eta-\etab) - \bt,\quad \nab_3\etab = -\chib\cdot (\etab - \eta) + \betab,
\end{equation}
i.e.~the top order derivatives can be grouped in terms of $\bt$ and $\betab$.

Applying \eqref{eq:commutation.4} and \eqref{eq:commutation.3} to \eqref{eq:eta.etab.rephrased}, we thus obtain
\begin{equation*}
\begin{split}
\nab_4 \div \eta =&\:  - \div \{\chi\cdot(\eta-\etab) + \bt\} - \f 12 (\eta + \etab) \cdot \{\chi\cdot(\eta-\etab) + \bt\} - \chi \cdot \nab \eta + \trch \etab \cdot \eta - \chi\cdot \etab\cdot \eta - \beta \cdot \eta, \\
\nab_3 \div \etab =&\:  - \div \{\chi\cdot(\etab-\eta) - \betab\} - \f 12 (\eta + \etab) \cdot \{\chi\cdot(\etab-\eta) - \betab\} - \chib \cdot \nab \etab + \trchb \eta \cdot \etab - \chib\cdot \eta\cdot \etab + \betab \cdot \etab.
\end{split}
\end{equation*}

Recalling the definition of $\mu$ and $\mub$ in \eqref{eq:mu.def}, we can combine the above equations with \eqref{eq:null.Bianchi.3} and \eqref{eq:null.Bianchi.5} to obtain the desired conclusion. \qedhere

\end{proof}

\subsubsection{Weak formulation of transport equations for derivatives of $\trch$ and $\trchb$}

As in Sections~\ref{sec.weak.double.null} and \ref{sec.weak.double.null.dust}, the transport equations for $\trch$ in the $e_4$ direction and for $\trchb$ in the $e_3$ direction are different depending on whether the spacetime satisfies the Einstein vacuum equations or the Einstein--null dust system. We first consider the vacuum case.
\begin{proposition}\label{prop:Xtrch.vac}
Let $(\mathcal M, g)$ be a $C^3$ solution to the Einstein vacuum equations in double null coordinates. Assume $\slashed X$ is a $C^1$ $S$-tangent vector field. Then the following holds:
\begin{equation}\label{eq:Xtrch.0.vac}
 \f{\rd}{\rd\ub} \slashed X(\Omg^{-1} \trch) + \f 12 \slashed X (\trch)^2 = - \slashed X|\chih|_{\gamma}^2 + [\f{\rd}{\rd\ub}, \slashed X] (\Omg^{-1} \trch). 
\end{equation}
\begin{equation}\label{eq:Xtrchb.0.vac}
(\f{\rd}{\rd u} + \nab_{b}) \slashed X(\Omg^{-1}\trchb) +\frac 12 \slashed X(\trchb)^2= -\slashed X|\chibh|_{\gamma}^2 + [\f{\rd}{\rd u} + \nab_{b}, \slashed X](\Omg^{-1}\trchb).
\end{equation}

\end{proposition} 
\begin{proof}
This follows from differentiating \eqref{Ric44} and \eqref{Ric33} by $\slashed{X}$ and using \eqref{Ricci.relation}. \qedhere
\end{proof}

In the presence of dust, \eqref{eq:Xtrch.0.vac} and \eqref{eq:Xtrchb.0.vac} do not hold. Instead, the dust term acts as a source in these transport equations (cf.~Section~\ref{sec.weak.double.null.dust}). We introduce a weak formulation of these equations, which can be considered as the higher derivative version of \eqref{eq:trchb} and \eqref{eq:trch}. 

Let $\slashed X$ be a $C^1$ $S$-tangent vector field. Consider the following equation for $\slashed X (\Omg^{-1} \trch)$  (for all $u\in [0,u_*]$ and all $0\leq \ub_1<\ub_2 \leq \ub_*$)
\begin{equation}\label{eq:Xtrch.0}
\begin{split}
&\: \int_{S_{u,\ub_2}}  \Omg^2 (\slashed X (\Omg^{-1} \trch))^- \,\mathrm{dA}_{\gamma} - \int_{S_{u,\ub_1}} \Omg^2 (\slashed X (\Omg^{-1} \trch))^+ \,\mathrm{dA}_{\gamma} \\
= &\: \int_{\ub_1}^{\ub_2} \int_{S_{u,\ub}} ([\f{\rd}{\rd \ub} , \slashed X](\Omg^{-1}\trch) -4\om \Omg\slashed X(\Omg^{-1}\trch) )\Omg^2\,\mathrm{dA}_{\gamma} \,\ud \ub \\
&\: + \int_{\ub_1}^{\ub_2} \int_{S_{u,\ub}} (2\slashed X(\log\Omg) + \div \slashed X)\Omg^2|\chih|_{\gamma}^2 \,\mathrm{dA}_{\gamma} \,\ud \ub + \int_{\{u\}\times (\ub_1,\ub_2)\times \mathbb S^2} (2\slashed X(\log\Omg) + \div \slashed X)\,\ud{\nu}_{u},
\end{split}
\end{equation}
and the following equation for $\slashed X (\Omg^{-1} \trchb)$ (for all $\ub\in [0,\ub_*]$ and all $0\leq u_1<u_2 \leq u_*$)
\begin{equation}\label{eq:Xtrchb.0}
\begin{split}
&\: \int_{S_{u_2,\ub}}  \Omg^2 (\slashed X (\Omg^{-1} \trchb))^- \,\mathrm{dA}_{\gamma} - \int_{S_{u_1,\ub}} \Omg^2 (\slashed X (\Omg^{-1} \trchb))^+ \,\mathrm{dA}_{\gamma} \\
= &\: \int_{u_1}^{u_2} \int_{S_{u,\ub}}  ([\f{\rd}{\rd u} + \nab_{b}, \slashed X](\Omg^{-1}\trchb) -4\omb \Omg\slashed X(\Omg^{-1}\trchb) )\Omg^2\,\mathrm{dA}_{\gamma} \,\ud u \\
&\: + \int_{u_1}^{u_2} \int_{S_{u,\ub}} (2\slashed X(\log\Omg) + \div \slashed X)\Omg^2|\chibh|_{\gamma}^2 \,\mathrm{dA}_{\gamma} \,\ud u + \int_{(u_1,u_2)\times \{\ub\} \times \mathbb S^2} (2\slashed X(\log\Omg) + \div \slashed X)\,\ud\underline{\nu}_{\ub}.
\end{split}
\end{equation}
Recall here that ${ }^{\pm}$ denotes the trace (see Lemma~\ref{lem:trace}), which is well defined in an angularly regular double null spacetime (Definition~\ref{double.null.def.2}) since $\slashed X (\Omg^{-1}\trch)$ is in $BV(H_u)$ for all $u$ and $\slashed X(\Omg^{-1}\trchb)$ is in $BV(\Hb_{\ub})$ for all $\ub$. 
 
\subsubsection{Weak formulation of all the auxiliary equations}
We are now ready to define an appropriate weak formulation of the equations that we have considered in this section.
\begin{definition}\label{def:aux.integrated}
Let $(\mathcal M = [0,u_*]\times [0,\ub_*]\times \mathbb S^2,g)$ be an angularly regular spacetime in double null coordinates. We say that \textbf{the equations \eqref{eq:nablagamma}, \eqref{eq:nablaOmega}, \eqref{eq:nablab}, \eqref{eq:mu.0}, \eqref{eq:mub.0}, \eqref{eq:Xtrch.0} and \eqref{eq:Xtrchb.0} are satisfied} if 
\begin{enumerate}
\item the equations \eqref{eq:nablagamma}--\eqref{eq:nablab} and \eqref{eq:mu.0}--\eqref{eq:mub.0} are satisfied in the integrated sense (Definition~\ref{def:weak.transport}); 
\item the equation \eqref{eq:Xtrch.0} is satisfied for all $C^1$ $S$-tangent vector field $\slashed X$, for all $u\in [0,u_*]$ and all $0\leq \ub_1<\ub_2 \leq \ub_*$;
\item the equation \eqref{eq:Xtrchb.0} is satisfied for all $C^1$ $S$-tangent vector field $\slashed X$, for all $\ub\in [0,\ub_*]$ and all $0\leq u_1<u_2 \leq u_*$.
\end{enumerate}
\end{definition}

\section{Existence theorems}\label{sec.existence}
In this section, we recall the existence and uniqueness theorem in \cite{LR2}. We will in particular need to use the estimates derived in \cite{LR2} to prove our main theorems.

To simplify the exposition\footnote{The original theorems in \cite{LR, LR2} indeed allow the initial data to be non-smooth. In particular, they allow the initial $\chih$ and $\chibh$ to be discontinuous, which corresponds to impulsive gravitational waves.}, we restrict our attention for smooth initial data. This will already be sufficient for our purposes. Even though the data are smooth, the key point here is that the result in \cite{LR2} guarantees that the region of existence and the estimates that the solutions obey depend only on low-regularity norms of the data.

In \textbf{Section~\ref{sec:existence.data}}, we first give some remarks regarding the characteristic initial value problem in the double null foliation gauge. We then give the statement of the main result of \cite{LR2} in \textbf{Section~\ref{sec:statement.ext.thm}}.

\subsection{The characteristic initial value problem}\label{sec:existence.data}

We will consider a characteristic initial value problem as follows. We impose characteristic initial data on two transversally intersecting null hypersurfaces $H_0 = [0, \underline{I}] \times \mathbb S^2$ and $\Hb_0 = [0, I] \times \mathbb S^2$, where $\{0\}\times \mathbb S^2 \subset H_0$ and $\{0\}\times \mathbb S^2 \subset \Hb_0$ are identified. See Figure~\ref{fig:CIVP}.

\begin{figure}[htbp]\label{fig:CIVP}
\begin{center}
 
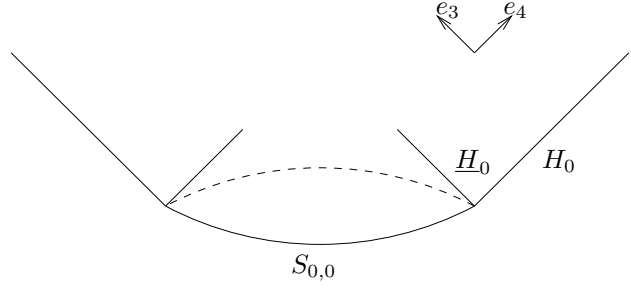
 
\caption{Setup for the characteristic initial value problem}
\end{center}
\end{figure}

We will provide two ways of thinking about the characteristic initial data, which we call the full characteristic initial data and the reduced characteristic initial data; see Sections~\ref{sec:full.char.data} and \ref{sec:reduced.data} respectively.

\subsubsection{The full characteristic initial data}\label{sec:full.char.data}

The first way to prescribe characteristic initial data is to prescribe all of the metric components $(\gamma,\, \Omg,\, b)$ and all of the Ricci coefficients $(\chi,\,\chib,\,\eta,\,\etab,\,\om,\,\omb)$ on both $H_0$ and $\Hb_0$. All of these objects are required to be smooth, $\gamma$ is required to be positive definite, and $\Omg$ is required to be positive. Moreover, the metric components and the Ricci coefficients are required to satisfy the following:
\begin{itemize}
\item The metric components and the Ricci coefficients relate to each other via \eqref{metric.derivative.invar} and \eqref{Ricci.relation}.
\item The null structure equations \eqref{Ric44}--\eqref{Ric34.1} all hold, where it is understood that the $\nab_3$ equations are required to hold on $\Hb_0$ and the $\nab_4$ equations are required to hold on $H_0$.
\end{itemize}

\subsubsection{The reduced characteristic initial data}\label{sec:reduced.data}

As is discussed in \cite{Chr}, there is another way to think about the characteristic initial data. Essentially, this allows one to identify some freely prescribable data and then impose the other constraints by solving appropriate transport equations. From this point of view, the initial data consist of $(\Om,\,\Phi,\,\hat{\gamma},\,b\restriction_{\Hb_0},\,\slashed{\mathcal L}_{\f{\rd}{\rd\ub}} b\restriction_{S_{0,0}})$, where $\Phi$ and $\hat{\gamma}$ are such that $\gamma = \Phi^2 \hat{\gamma}$, and $\hat{\gamma}$ is normalized by the condition $\f{\det\hat{\gamma}}{\det\mathring{\gamma}} = 1$ for some $\mathring{\gamma}$ satisfying $\slashed{\mathcal L}_{\f{\rd}{\rd\ub}} \mathring{\gamma}\restriction_{H_0} = 0 = \slashed{\mathcal L}_{\f{\rd}{\rd u} + b} \mathring{\gamma}\restriction_{\Hb_0}$.

Here, $\Phi$ and $\hat{\gamma}$ are required to satisfy the \emph{constraint equation}:
\begin{equation}\label{eq:constraints.first.time}
\begin{split}
\f{\rd^2 \Phi}{\rd\ub^2}=2\f{\rd \log\Omega}{\rd\ub}\f{\rd\Phi}{\rd\ub}-\f {1}8 |\f{\rd\hat{\gamma}}{\rd \ub}|_{\hat{\gamma}}^2 \Phi \quad \mbox{on $H_0$},\\
\f{\rd^2 \Phi}{\rd u^2}=2\f{\rd \log\Omega}{\rd u}\f{\rd\Phi}{\rd u}-\f {1}8 |\f{\rd\hat{\gamma}}{\rd u}|_{\hat{\gamma}}^2 \Phi \quad \mbox{on $\Hb_0$},
\end{split}
\end{equation}
where
\begin{equation}\label{eq:def.|dubgamma|^2}
|\f{\rd\hat{\gamma}}{\rd \ub}|_{\hat{\gamma}}^2 :=  (\hat{\gamma}^{-1})^{AC}(\hat{\gamma}^{-1})^{BD}(\f{\rd}{\rd\ub}\hat{\gamma}_{AB})(\f{\rd}{\rd\ub}\hat{\gamma}_{CD}),\quad |\f{\rd\hat{\gamma}}{\rd u}|_{\hat{\gamma}}^2 :=  (\hat{\gamma}^{-1})^{AC}(\hat{\gamma}^{-1})^{BD}(\f{\rd}{\rd u}\hat{\gamma}_{AB})(\f{\rd}{\rd u}\hat{\gamma}_{CD}).
\end{equation}

To simplify the exposition, \textbf{we assume from now on $b\restriction_{\Hb_0} \equiv 0$}. 

Once $\Omega$, $\Phi$, $\hat{\gamma}$ and $\slashed{\mathcal L}_{\f{\rd}{\rd\ub}} b\restriction_{S_{0,0}}$ are prescribed, we can derive the full characteristic initial data set as in Section~\ref{sec:full.char.data} by the following procedure (see \cite{Chr}):
\begin{itemize}
\item Given $\Omg$, we can obtain $\om\restriction_{H_0}$ and $\omb_{\Hb_0}$ by (see \eqref{Ricci.relation})
$$\om \restriction_{H_0} = -2 \Omg^{-1} \f{\rd}{\rd\ub}\log\Om,\quad \omb\restriction_{\Hb_0} = -2 \Omg^{-1} \f{\rd}{\rd\ub}\log\Om.$$
\item Given $\Omg$, $\hat{\gamma}$ and $\Phi$, we can obtain $\chi\restriction_{H_0}$ and $\chib \restriction_{\Hb_0}$ by
$$\chih_{AB}\restriction_{H_0} = \f 12 \Omg^{-1} \Phi^2 \f{\rd}{\rd\ub}\gamma_{AB},\quad \trch\restriction_{H_0} = \f{2\Omg^{-2}}{\Phi} \f{\rd\Phi}{\rd\ub},$$
$$\chibh_{AB}\restriction_{\Hb_0} = \f 12 \Omg^{-1} \Phi^2 \f{\rd}{\rd u}\gamma_{AB},\quad \trchb\restriction_{\Hb_0} = \f{2\Omg^{-2}}{\Phi} \f{\rd\Phi}{\rd u}.$$
\item $\eta$ and $\etab$ can be obtained on $H_0$ by solving \eqref{Ric4A} together with the condition $\f 12 (\eta + \etab) = \nab\log\Omg$ (see \eqref{Ricci.relation}); $\eta$ and $\etab$ can be obtained on $\Hb_0$ by solving \eqref{Ric3A} together with the condition $\f 12 (\eta + \etab) = \nab\log\Omg$. The initial condition for both of these ODEs are given in view of $\slashed{\mathcal L}_{\f{\rd}{\rd\ub}} b = -2\Omg^2 (\eta^\sharp -\etab^\sharp)$ (see \eqref{metric.derivative}) and the fact that $\slashed{\mathcal L}_{\f{\rd}{\rd\ub}} b\restriction_{S_{0,0}}$ is prescribed.
\item $b$ on $H_0$ can be obtained by $\slashed{\mathcal L}_{\f{\rd}{\rd\ub}} b = -2\Omg^2 (\eta^\sharp -\etab^\sharp)$ and the condition that $b\restriction_{S_{0,0}}= 0$ (which follows from $b\restriction_{\Hb_0} = 0$).
\item $\om$ on $\Hb_0$ can be obtained solving the transport equations \eqref{Ric34.1}; $\omb$ on $H_0$ can be obtained solving the transport equation \eqref{Ric34}. Note that the initial data for $\om$ for this transport equation is determined by the value of $\om$ on $H_0$ (obtained above); similarly for $\omb$. 
\item $\chi$ on $\Hb_0$ can be obtained solving the transport equations \eqref{trRicAB.1} and \eqref{RicAB.1}; $\chib$ on $H_0$ can be obtained solving the transport equations \eqref{trRicAB} and \eqref{RicAB}. Note that the initial data for $\chi$ for these transport equations is determined by the value of $\chi$ on $H_0$; similarly for $\chib$. 
\end{itemize}

The important point for the above procedure is that all the terms on the RHS of the transport equations are given in the previous steps. The transport equations can therefore all be solved (despite the fact that the procedure involves a loss of derivatives in the sense that to obtain estimates for (say) three derivatives for all the initial Ricci coefficients require prescribing more derivatives on $(\Omg,\,\Phi,\,\hat{\gamma},\,\slashed{\mathcal L}_{\f{\rd}{\rd\ub}} b\restriction_{S_{0,0}})$).

Tracing through the above procedure, we also obtain quantitive estimates for $(b, \, \chi,\, \chibh,\, \eta,\, \etab,\, \om,\, \omb)$. We give the result in the following lemma. The proof is straightforward and omitted (see again \cite{Chr} for details).
\begin{lemma}\label{lem:suff.cond.on.data}
Suppose $(\Omg,\,\Phi,\,\hat{\gamma},\,\slashed{\mathcal L}_{\f{\rd}{\rd\ub}} b\restriction_{S_{0,0}})$ is a reduced characteristic initial data set with the estimates
\begin{equation*}
\begin{split}
\|(\log\Om,\,\hat{\gamma},\,\log\Phi,\,\f{\rd\Phi}{\rd \ub})\restriction_{H_0} \|_{C^0_{\ub} W^{6,2}(S_{u,\ub},\mathring{\gamma})} + \|(\log\Om,\,\hat{\gamma},\,\log\Phi,\,\f{\rd\Phi}{\rd u})\restriction_{\Hb_0} \|_{C^0_{\ub} W^{6,2}(S_{u,\ub},\mathring{\gamma})} \leq C_0,
\end{split}
\end{equation*}
\begin{equation*}
\begin{split}
\|(\f{\rd\log\Om}{\rd\ub},\,\slashed{\mathcal L}_{\f{\rd}{\rd\ub}}\hat{\gamma})\restriction_{H_0} \|_{L^2_{\ub} W^{5,2}(S_{u,\ub},\mathring{\gamma})} + \|(\f{\rd\log\Om}{\rd u},\,\slashed{\mathcal L}_{\f{\rd}{\rd u}}\hat{\gamma})\restriction_{\Hb_0} \|_{L^2_{u} W^{5,2}(S_{u,\ub},\mathring{\gamma})} \leq C_0,
\end{split}
\end{equation*}
\begin{equation*}
\begin{split}
\|\slashed{\mathcal L}_{\f{\rd}{\rd\ub}} b\restriction_{S_{0,0}} \|_{W^{5,2}(S_{0,0},\mathring{\gamma})} \leq C_0.
\end{split}
\end{equation*}

Then, for the metric components and Ricci coefficients derived with the procedure outlined above, there exists $\widetilde{C}_0>0$ depending on $C_0$ such that 
$$\| (\gamma, \,\log\f{\det\gamma}{\det\mathring{\gamma}}) \|_{L^\i_u W^{3,2}(S_{u,0},\mathring{\gamma})} + \| (\gamma, \,\log\f{\det\gamma}{\det\mathring{\gamma}}) \|_{L^\i_{\ub} W^{3,2}(S_{0,\ub},\mathring{\gamma})} \leq \widetilde{C}_0,$$
\begin{equation*}
\begin{split}
\| (\eta,\,\etab,\,\trch,\,\trchb) \|_{L^\i_u W^{3,2}(S_{u,0},\gamma)} + \| K \|_{L^\i_u W^{2,2}(S_{u,0},\gamma)} \leq \widetilde{C}_0,
\end{split}
\end{equation*} 
$$\| (\eta,\,\etab,\,\trch,\,\trchb) \|_{L^\i_{\ub} W^{3,2}(S_{0,\ub},\gamma)} + \| K \|_{L^\i_{\ub} W^{2,2}(S_{0,\ub},\gamma)} \leq \widetilde{C}_0,$$
\begin{equation*}
\begin{split}
\| (\chibh,\,\omb) \|_{L^2_u W^{3,2}(S_{u,0},\gamma)} + \| (\chih,\,\om) \|_{L^2_{\ub} W^{3,2}(S_{0,\ub},\gamma)} \leq \widetilde{C}_0.
\end{split}
\end{equation*}

\end{lemma}

\subsection{The statements of the existence theorems}\label{sec:statement.ext.thm}

We now state the main result in \cite{LR2}. Recall again that we only focus on the case of the result in \cite{LR2} where the initial data are smooth. We will first give the existence statement (Theorem~\ref{ext.thm}) and then give the estimates (Theorem~\ref{thm:ext.est}).

As above, we consider a characteristic initial value problem for the Einstein vacuum equations with smooth initial data given on $H_0 = [0,\underline{I}]\times \mathbb S^2$ and $\Hb_0 = [0, I]\times \mathbb S^2$ where $\{0\}\times \mathbb S^2 \subset H_0$ and $\{0\}\times \mathbb S^2 \subset \Hb_0$ are identified. Suppose we are given a full characteristic initial data set as in Section~\ref{sec:full.char.data} with $b\restriction_{\Hb_0} = 0$. In addition, fix a smooth metric $\mathring{\gamma}$ on $\mathbb S^2$. Extend $\mathring{\gamma}$ to $H_0\cup \Hb_0$ by requiring $\slashed{\mathcal L}_{\f{\rd}{\rd\ub}} \mathring{\gamma}\restriction_{H_0} = 0 = \slashed{\mathcal L}_{\f{\rd}{\rd u}} \mathring{\gamma}\restriction_{\Hb_0}$.

The following is the main existence theorem\footnote{Note that though equivalent, Theorem~\ref{ext.thm}  is slightly differently worded as \cite[Theorem~3]{LR2}. In particular, instead of \eqref{eq:LR.gamma.data}, the estimates in \cite{LR2} are stated in terms of local coordinates instead of with respect to a reference metric $\mathring{\gamma}$.} from \cite{LR2}:

\begin{theorem}[L.--R.,~Theorem~3 in \cite{LR2}]\label{ext.thm}
Consider the characteristic initial value problem for the Einstein vacuum equations with smooth full characteristic initial data as described above.
Suppose there exists a constant $D_0>0$ such that the prescribed data satisfy
\begin{equation}\label{eq:LR.gamma.data}
\| (\gamma, \,\log\f{\det\gamma}{\det\mathring{\gamma}}) \|_{L^\i_u W^{3,2}(S_{u,0},\mathring{\gamma})} + \| (\gamma, \,\log\f{\det\gamma}{\det\mathring{\gamma}}) \|_{L^\i_{\ub} W^{3,2}(S_{0,\ub},\mathring{\gamma})} \leq D_0,
\end{equation}
\begin{equation}
\begin{split}
\| (\eta,\,\etab,\,\trch,\,\trchb) \|_{L^\i_u W^{3,2}(S_{u,0},\gamma)} + \| K \|_{L^\i_u W^{2,2}(S_{u,0},\gamma)} \leq D_0,
\end{split}
\end{equation} 
\begin{equation}
\| (\eta,\,\etab,\,\trch,\,\trchb) \|_{L^\i_{\ub} W^{3,2}(S_{0,\ub},\gamma)} + \| K \|_{L^\i_{\ub} W^{2,2}(S_{0,\ub},\gamma)} \leq D_0,
\end{equation}
\begin{equation}\label{eq:LR.data.4}
\begin{split}
\| (\chibh,\,\omb) \|_{L^2_u W^{3,2}(S_{u,0},\gamma)} + \| (\chih,\,\om) \|_{L^2_{\ub} W^{3,2}(S_{0,\ub},\gamma)} \leq D_0.
\end{split}
\end{equation}

Then there exists $\ep>0$ (sufficiently small) depending only on $D_0$ and $I$ such that the following holds:

Given $u_* \in (0,I]$ and $\ub_* \in (0,\ep]$, there exists a unique smooth solution to the Einstein vacuum equations in double null coordinates in the domain $[0,u_*]\times [0,\ub_*]\times \mathbb S^2$ which achieves the given initial data.
\end{theorem}

We now turn to the estimates for the Ricci coefficients which are obtained in the proof of Theorem~\ref{ext.thm}. We record them in Theorem~\ref{thm:ext.est} below. Most of the following estimates can be directly read off from \cite{LR2}. For the convenience of the reader, we include in Appendix~\ref{app:est} a derivation of these estimates from \cite{LR2}. (For the statement of Theorem~\ref{thm:ext.est}, recall again the definition of $\gamma_{0,0}$ in Definition~\ref{def:gamma00}.)

\begin{theorem}\label{thm:ext.est}
Consider a characteristic initial value problem as in Theorem~\ref{ext.thm}, as well as a spacetime solution given as in the conclusion Theorem~\ref{ext.thm}. 

After choosing $\ep>0$ sufficiently small if necessary, there exists $\widetilde{C}>0$ depending only on $D_0$ and $I$ (in Theorem~\ref{ext.thm}) such that the following estimates hold in $[0,u_*]\times [0,\ub_*]\times \mathbb S^2$:

\begin{align}
\mathbf{I}(S_{u,\ub}, \gamma) \leq \widetilde{C},\quad \widetilde{C}^{-1} \leq \mathrm{Area}(S_{u,\ub}, \gamma) \leq \widetilde{C}, \label{eq:bdd.isoperimetric} \\
\|\log \f{\det\gamma}{\det\gamma_{0,0}} \|_{C^0_u C^0_{\ub} C^0(S)} \leq \widetilde{C}, \label{eq:bdd.density}\\
\sum_{\slashed g \in \{\gamma - \gamma_{0,0},\, b,\,\log\Omg\}}(\|\slashed g\|_{C^0_u C^0_{\ub} W^{3,2}(S)} + \|\slashed{\mathcal L}_{\f{\rd}{\rd \ub}} \slashed g\|_{L^2_{\ub} L^\i_u W^{2,2}(S)} + \|\slashed{\mathcal L}_{\f{\rd}{\rd u}} \slashed g\|_{L^2_{u} L^\i_{\ub} W^{2,2}(S)} )\leq \widetilde{C}, \label{eq:bdd.metric}\\
\sum_{\psi\in \{\eta,\,\etab\}} \|\psi \|_{C^0_u C^0_{\ub} W^{2,2}(S)} + \sum_{\psi\in \{\trch,\,\trchb\}} \|\psi \|_{C^0_u C^0_{\ub} W^{3,2}(S)}  \leq  \widetilde{C}, \label{eq:bdd.psi}\\
\sum_{\psi\in \{\eta,\,\etab\}} (\|\nab^3 \psi \|_{C^0_u L^2_{\ub} L^2(S)} + \|\nab^3 \psi \|_{C^0_{\ub} L^2_{u} L^2(S)}) + \|\nab^2 K \|_{C^0_u L^2_{\ub} L^2(S)} + \|\nab^2 K \|_{C^0_{\ub} L^2_{u} L^2(S)} \leq \widetilde{C}, \label{eq:bdd.psi.top} \\
\sum_{\psi_H \in \{ \chih,\om \}} (\|\psi_H \|_{L^2_{\ub} C^0_u W^{2,2}(S)} + \|\nab^3 \psi_H \|_{C^0_u L^2_{\ub} L^2(S)}) \leq \widetilde{C}, \label{eq:bdd.psiH} \\
\sum_{\psi_{\Hb} \in \{ \chibh,\omb \}} ( \|\psi_{\Hb} \|_{L^2_{u} C^0_{\ub} W^{2,2}(S)} + \|\nab^3 \psi_{\Hb} \|_{C^0_{\ub} L^2_{u} L^2(S)}) \leq \widetilde{C}, \label{eq:bdd.psiHb} \\
\sum_{\psi\in \{\eta,\,\etab \}} (\|\slashed{\mathcal L}_{\f{\rd}{\rd \ub}} \psi\|_{C^0_u L^2_{\ub} W^{2,2}(S)} + \|\slashed{\mathcal L}_{\f{\rd}{\rd u}} \psi\|_{C^0_{\ub} L^2_{u}  W^{2,2}(S)}) \leq \widetilde{C}, \label{eq:bdd.psi.trans} \\
\|\slashed{\mathcal L}_{\f{\rd}{\rd \ub}} \trch \|_{C^0_u L^1_{\ub} W^{3,2}(S)} + \|\slashed{\mathcal L}_{\f{\rd}{\rd u}} \trch \|_{L^2_u C^0_{\ub} W^{2,2}(S)} \leq \widetilde{C}, \label{eq:bdd.trch.trans} \\
 \|\slashed{\mathcal L}_{\f{\rd}{\rd u}} \trchb \|_{C^0_{\ub} L^1_{u} W^{3,2}(S)} + \|\slashed{\mathcal L}_{\f{\rd}{\rd \ub}} \trchb \|_{L^2_{\ub} C^0_{u} W^{2,2}(S)}\leq \widetilde{C}, \label{eq:bdd.trchb.trans} \\
\sum_{\psi_H \in \{ \chih,\om \}} \|\slashed{\mathcal L}_{\f{\rd}{\rd u}} \psi_H\|_{L^2_u L^2_{\ub} W^{2,2}(S)}  + \sum_{\psi_{\Hb} \in \{ \chibh,\omb \}} \|\slashed{\mathcal L}_{\f{\rd}{\rd \ub}} \psi_{\Hb}\|_{L^2_u L^2_{\ub} W^{2,2}(S)}  \leq \widetilde{C}, \label{eq:bdd.psiH.psiHb.trans}
\end{align}
where we have used the shorthand above that $W^{k,2}(S) = W^{k,2}(S_{u,\ub},\gamma)$. 
\end{theorem}

\begin{remark}
Notice that the conditions in Lemma~\ref{lem:suff.cond.on.data} guarantee that the assumptions \eqref{eq:LR.gamma.data}--\eqref{eq:LR.data.4} of Theorem~\ref{ext.thm} hold.
\end{remark}

\section{Main results}\label{sec.main.thm}

\subsection{Existence and characterization of the limiting spacetime}

\begin{theorem}\label{main.thm}
Consider a sequence of smooth characteristic initial data for the vacuum Einstein equations
$$Ric_{\mu\nu}=0.$$
Assume that all the geometric quantities obey the bounds in Theorem \ref{ext.thm} uniformly. Denoting $(\gamma_{0,0})_n = \gamma_n \restriction_{S_{0,0}}$, assume that there is a $C^3$ limit $(\gamma_{0,0})_\i$, i.e.
\begin{equation}\label{eq:gamma.C3.conv}
(\gamma_{0,0})_n \to (\gamma_{0,0})_\i \mbox{ in $C^3$}.
\end{equation}

Then, there exists $\epsilon>0$ sufficiently small (independent of $n$) such that the following hold for $\mathcal M:= [0,u_*]\times [0,\ub_*] \times \mathbb S^2$ with $\ub_* \in (0,\ep]$:
\begin{enumerate}
\item There exists a sequence of spacetimes $(\mathcal M, g_n)$ in double null coordinates 
$$g_n=-2\Omega_n^2(du\otimes d\ub+d\ub \otimes du)+(\gamma_n)_{AB} (d\th^A-b_n^A du)\otimes (d\th^B-b_n^B du)$$
corresponding to the sequence of initial data, which solve the Einstein vacuum equations. 
\item A subsequence of spacetime metrics $(\mathcal M, \,g_{n_k})$ converges in $C^0$ in the null coordinate system to a limiting spacetime $(\mathcal M,\,g_{\infty})$
$$g_\infty=-2\Omega_{\infty}^2(du\otimes d\ub+d\ub \otimes du)+( \gamma_{\infty} )_{AB} (d\th^A-b_{\infty}^A du)\otimes (d\th^B-b_{\infty}^B du).$$
In addition, the weak derivatives of the components of $g_{n_k}$ converge weakly in (spacetime) $L^2$ to the weak derivatives of the components of $g_\infty$; and the Ricci coefficients corresponding to $g_{n_k}$ converge weak in (spacetime) $L^2$ to the Ricci coefficients corresponding to $g_\infty$.
\item After passing to a further subsequence (not relabeled), the following weak-* limits exist:
\begin{equation}\label{eq:dnu.def.thm}
\ud\nu_u:= \mathrm{weak}\mbox{-*}\lim_{k\to +\infty} (\Omg^2_{n_k} |\chih_{n_k}|_{\gamma_{n_k}}^2\, \mathrm{dA}_{\gamma_{n_k}} \,\ud \ub - \Omg^2_{\infty} |\chih_{\infty}|_{\gamma_{\infty}}^2\, \mathrm{dA}_{\gamma_{\infty}} \,\ud \ub),
\end{equation}
\begin{equation}\label{eq:dnub.def.thm}
\ud\nub_{\ub}:= \mathrm{weak}\mbox{-*}\lim_{k\to +\infty} (\Omg^2_{n_k} |\chibh_{n_k}|_{\gamma_{n_k}}^2\, \mathrm{dA}_{\gamma_{n_k}} \,\ud u - \Omg^2_{\infty} |\chibh_{\infty}|_{\gamma_{\infty}}^2\, \mathrm{dA}_{\gamma_{\infty}} \,\ud u).
\end{equation}
Moreover, $(\mathcal M, g_\infty, \{\ud\nu_u\}_{u\in [0,u_*]}, \{\ud\nub_{\ub} \}_{\ub\in [0,\ub_*]})$ is an angularly regular weak solution to the Einstein--null dust system in the sense of Definition~\ref{def:weak.sol.ang.reg}.
\item For $(\mathcal M, g_\infty, \{\ud\nu_u\}_{u\in [0,u_*]}, \{\ud\nub_{\ub} \}_{\ub\in [0,\ub_*]})$, the renormalized Bianchi equations are satisfied in the sense of Definition~\ref{def:Bianchi.integrated}.
\item For $(\mathcal M, g_\infty, \{\ud\nu_u\}_{u\in [0,u_*]}, \{\ud\nub_{\ub} \}_{\ub\in [0,\ub_*]})$, the equations \eqref{eq:nablagamma}, \eqref{eq:nablaOmega}, \eqref{eq:nablab}, \eqref{eq:mu.0}, \eqref{eq:mub.0}, \eqref{eq:Xtrch.0} and \eqref{eq:Xtrchb.0} are satisfied in the sense of Definition~\ref{def:aux.integrated}.
\end{enumerate}
\end{theorem}

Theorem~\ref{main.thm} will be proven in \textbf{Sections~\ref{sec:existence} and \ref{sec:eqns.for.limit}}. See the conclusion of the proof in Section~\ref{sec:proof.of.main.thm}.

Some remarks are in order.

\begin{remark}
To simplify the statements, we only asserted that the limit is achieved in the $C^0$ and the $W^{1,2}$-weak sense. Nevertheless, in fact, various different Ricci coefficients have better convergence properties; see more precise convergence statements in Section~\ref{sec:existence}. 
\end{remark}

\begin{remark}
As we explained in the introduction, Theorem~\ref{main.thm} implies a fortiori that the $(\mathcal M_{\infty},\,g_{\infty})$ is vacuum if and only if $\ud\nu_0= 0$ and $\ud\nub_0 =0 $ (Corollary~\ref{cor:vac.cond.intro}).
\end{remark}

\subsection{Uniqueness of the limiting spacetime}

\begin{theorem}\label{thm:uniqueness}
Let $\mathcal M = [0,u_*]\times [0,\ub_*]\times \mathbb S^2$. Suppose $(\mathcal M,\, g^{(1)}, \,\{\ud\nu_u^{(1)}\}_{u\in [0,u_*]},\,\{\ud \underline{\nu}_{\ub}^{(1)}\}_{\ub\in [0,\ub_*]})$ and $(M,\, g^{(2)}, \,\{\ud\nu_u^{(2)}\}_{u\in [0,u_*]},\, \{\ud \underline{\nu}_{\ub}^{(2)}\}_{\ub \in [0,\ub_*]})$ are such that the following holds:
\begin{enumerate}
\item $(M,\, g^{(1)}, \,\{\ud\nu_u^{(1)}\}_{u\in [0,u_*]},\,\{\ud \underline{\nu}_{\ub}^{(1)}\}_{\ub\in [0,\ub_*]})$ and $(M,\, g^{(2)}, \,\{\ud\nu_u^{(2)}\}_{u\in [0,u_*]},\, \{\ud \underline{\nu}_{\ub}^{(2)}\}_{\ub \in [0,\ub_*]})$ are both angularly regular weak solutions to the Einstein--null dust system in the sense of Definition~\ref{def:weak.sol.ang.reg}.
\item $(M,\, g^{(1)}, \,\{\ud\nu_u^{(1)}\}_{u\in [0,u_*]},\,\{\ud \underline{\nu}_{\ub}^{(1)}\}_{\ub\in [0,\ub_*]})$ and $(M,\, g^{(2)}, \,\{\ud\nu_u^{(2)}\}_{u\in [0,u_*]},\, \{\ud \underline{\nu}_{\ub}^{(2)}\}_{\ub \in [0,\ub_*]})$ both satisfy the renormalized Bianichi equations in the sense of Definition~\ref{def:Bianchi.integrated}.
\item $(M,\, g^{(1)}, \,\{\ud\nu_u^{(1)}\}_{u\in [0,u_*]},\,\{\ud \underline{\nu}_{\ub}^{(1)}\}_{\ub\in [0,\ub_*]})$ and $(M,\, g^{(2)}, \,\{\ud\nu_u^{(2)}\}_{u\in [0,u_*]},\, \{\ud \underline{\nu}_{\ub}^{(2)}\}_{\ub \in [0,\ub_*]})$ both satisfy the equations \eqref{eq:nablagamma}, \eqref{eq:nablaOmega}, \eqref{eq:nablab}, \eqref{eq:mu.0}, \eqref{eq:mub.0}, \eqref{eq:Xtrch.0} and \eqref{eq:Xtrchb.0} in the sense of Definition~\ref{def:aux.integrated}.
\item $(M,\, g^{(1)}, \,\{\ud\nu_u^{(1)}\}_{u\in [0,u_*]},\,\{\ud \underline{\nu}_{\ub}^{(1)}\}_{\ub\in [0,\ub_*]})$ and $(M,\, g^{(2)}, \,\{\ud\nu_u^{(2)}\}_{u\in [0,u_*]},\, \{\ud \underline{\nu}_{\ub}^{(2)}\}_{\ub \in [0,\ub_*]})$ have the same initial data in the sense that\footnote{Note that while we only explicitly assumed that the differences of very specific Ricci coefficients vanish on the initial hypersurfaces, in fact it holds that the all Ricci coefficients coincide on the initial hypersurfaces. This is because the remaining Ricci coefficients can be written as tangential (along the initial hypersurfaces) derivatives of the metric components.} 
$$(\gamma^{(1)} - \gamma^{(2)},\, b^{(1)} - b^{(2)},\, \log\f{\Om^{(1)}}{\Om^{(2)}})\restriction_{S_{0,\ub}}  = 0,\,\,\forall \ub \in [0,\ub_*],$$
$$(\gamma^{(1)} - \gamma^{(2)},\, b^{(1)} - b^{(2)},\, \log\f{\Om^{(1)}}{\Om^{(2)}}) \restriction_{S_{u,0}} = 0,\,\,\forall u\in [0,u_*],$$ 
$$(\trchb^{(1)} - \trchb^{(1)})^+ \restriction_{S_{0,\ub}} = 0,\,\,\forall \ub \in [0,\ub_*],$$
$$((\trch^{(1)} - \trch^{(2)})^+,\,\eta^{(1)}-\etab^{(1)}-\eta^{(2)}+ \etab^{(2)}) \restriction_{S_{u,0}} = 0,\,\,\forall u \in [0,u_*],$$
$$\ud \nu_0^{(1)} - \ud \nu_0^{(2)} = 0,\quad \ud \nub_0^{(1)} - \ud \nub_0^{(2)} = 0.$$
\end{enumerate}
Then the following holds:
\begin{enumerate}
\item $\gamma^{(1)}= \gamma^{(2)}$, $b^{(1)} = b^{(2)}$ and $\Omg^{(1)} = \Omg^{(2)}$ everywhere on $\mathcal M$.
\item $\ud \nu_u^{(1)} = \ud\nu_u^{(2)}$ for all $u\in [0,u_*]$.
\item $\ud \nub_{\ub}^{(1)} = \ud\nub_{\ub}^{(2)}$ for all $\ub\in [0,\ub_*]$.
\end{enumerate}
\end{theorem}

The proof of Theorem~\ref{thm:uniqueness} will be carried out in \textbf{Section~\ref{sec:proof.uniqueness}}.

\subsection{Characteristic initial value problem for the Einstein--null dust system with angularly regular measure-valued null dust}

We now turn to our main results on the characteristic initial value problem for the Einstein--null dust system with angularly regular measure-valued null dust. We first introduce the class of initial data that we consider. For simplicity, we will only consider characteristic initial data for which $b^A \equiv 0$ on $\Hb_0$.

\begin{definition}[Strongly angularly regular reduced characteristic initial data]\label{def:SARCID}
Impose characteristic initial data on $H_0 = [0, \underline{I}] \times \mathbb S^2$ and $\Hb_0 = [0, I] \times \mathbb S^2$, where $\{0\}\times \mathbb S^2 \subset H_0$ and $\{0\}\times \mathbb S^2 \subset \Hb_0$ are identified (see Figure~\ref{fig:CIVP}). 

Let $\mathring{\gamma}$ be an (arbitrary) auxiliary smooth metric on $\mathbb S^2$. Define $\mathring{\gamma}$ on $H_0\cup \Hb_0$ by imposing
$$\slashed{\mathcal L}_{\f{\rd}{\rd \ub}} \mathring{\gamma} = 0 = \slashed{\mathcal L}_{\f{\rd}{\rd u}} \mathring{\gamma} .$$

A \textbf{strongly angularly regular reduced characteristic initial data set} to the Einstein--null dust system consists of a hextuple $(\Omega,\,\Phi,\, \hat{\gamma},\, \slashed{\mathcal L}_{\f{\rd}{\rd\ub}} b\restriction_{S_{0,0}},\, \ud \nu_{\mathrm{init}},\,\ud \nub_{\mathrm{init}})$ with the following properties:
\begin{enumerate}
\item $\Omega>0$ is a smooth\footnote{Smoothness of $\Omg$ can also be dropped and replaced by
$$\| \log\Omg \|_{C^0_{\ub}  W^{6,\infty}(S_{0,\ub},\mathring{\gamma})} +\|\f{\rd \log\Omg}{\rd \ub}\|_{L^2_{\ub}  W^{6,\infty}(S_{0,\ub},\mathring{\gamma})} < +\infty,$$
$$\| \log\Omg \|_{C^0_{u}  W^{6,\infty}(S_{u,0},\mathring{\gamma})} +\|\f{\rd \log\Omg}{\rd u}\|_{L^2_{u}  W^{6,\infty}(S_{u,0},\mathring{\gamma})} < +\infty.$$
Since this will lengthen some arguments in Section~\ref{sec.approx.thm}, we will simply work with the slightly stronger assumption on $\Omg$.} function on $H_0 \cup \Hb_0$.
\item $\Phi>0$ is a Lipschitz function on $H_0\cup \Hb_0$ such that on $\f{\rd\Phi}{\rd\ub}\restriction_{H_0} \in BV(H_0, \mathring{\gamma})$ and $\f{\rd\Phi}{\rd u}\restriction_{\Hb_0} \in BV(\Hb_0, \mathring{\gamma})$ (see Definition~\ref{def:BV}). Moreover,
$$\|\Phi\|_{C^0_{\ub} W^{6,\infty}(S_{0,\ub},\mathring{\gamma})} + \|\Phi^{-1}\|_{C^0_{\ub} W^{6,\infty}(S_{0,\ub},\mathring{\gamma})} + \|\f{\rd\Phi}{\rd\ub}\|_{L^\i_{\ub} W^{6,\infty}(S_{0,\ub},\mathring{\gamma})} <+\infty,$$
$$\|\Phi\|_{C^0_{u} W^{6,\infty}(S_{u,0},\mathring{\gamma})} + \|\Phi^{-1}\|_{C^0_{u} W^{6,\infty}(S_{u,0},\mathring{\gamma})} + \|\f{\rd\Phi}{\rd u}\|_{L^\i_{u} W^{6,\infty}(S_{u,0},\mathring{\gamma})} <+\infty.$$
\item $\hat{\gamma}$ is a continuous covariant $S$-tangent $2$-tensor which restricts to a positive definite metric on each $S_{0,\ub}$ or $S_{u,0}$. Moreover $\hat{\gamma}$ satisfies the following properties:
\begin{enumerate}
\item $$\f{\det{\hat{\gamma}}}{\det{\mathring{\gamma}}} = 1.$$
\item $$\| \hat{\gamma} \|_{C^0_{\ub}  W^{6,\infty}(S_{0,\ub},\mathring{\gamma})} +\|\f{\rd \hat{\gamma}}{\rd \ub}\|_{L^2_{\ub}  W^{6,\infty}(S_{0,\ub},\mathring{\gamma})} < +\infty,$$
$$\| \hat{\gamma} \|_{C^0_{u}  W^{6,\infty}(S_{u,0},\mathring{\gamma})} +\|\f{\rd \hat{\gamma}}{\rd u}\|_{L^2_{u}  W^{6,\infty}(S_{u,0},\mathring{\gamma})} < +\infty.$$
\end{enumerate}
\item $\slashed{\mathcal L}_{\f{\rd}{\rd\ub}} b\restriction_{S_{0,0}}$ is a $W^{5,2}(S_{0,0},\mathring{\gamma})$ vector field.
\item $\ud\nu_{\mathrm{init}}$ is a non-negative Radon measure on $(0,\underline{I})\times \mathbb S^2$ satisfying the following regularity estimate:
\begin{equation}\label{eq:dnu.init.bound.0}
\sup \left\{ \sum_{0\leq k\leq 6}\left| \int_{(0,\underline{I})\times \mathbb S^2} (\mathring{\div}{}^k\varphi^{(k)})(\ub,\vartheta)\, \ud\nu_{\mathrm{init}} \right| : \varphi^{(k)}\in C^\infty,\,\|\varphi^{(k)}\|_{L^\infty_{\ub}L^1(S_{0,\ub},\mathring{\gamma})} \leq 1  \right\}  < +\infty,
\end{equation}
where $\varphi^{(k)}$ is a rank-$k$ tensor field.

$\ud\nub_{\mathrm{init}}$ is a non-negative Radon measure on $(0,I)\times \mathbb S^2$ satisfying the following regularity estimate:
\begin{equation}\label{eq:dnub.init.bound.0}
\sup \left\{ \sum_{0\leq k\leq 6}\left| \int_{(0,I)\times \mathbb S^2} (\mathring{\div}{}^k\varphi^{(k)})(u,\vartheta)\, \ud\nub_{\mathrm{init}} \right| : \varphi^{(k)}\in C^\infty,\,\|\varphi^{(k)}\|_{L^\infty_{u}L^1(S_{u,0},\mathring{\gamma})} \leq 1  \right\}  < +\infty,
\end{equation}
where $\varphi^{(k)}$ is a rank-$k$ tensor field.
\item $(\Omg,\,\Phi,\,\hat{\gamma},\,\ud\nu_{\mathrm{init}})$ together satisfy the constraint equations
\begin{equation}\label{eq:dnu.init}
\begin{split}
&\: - \int_0^{\underline{I}} \int_{\mathbb S^2} (\f{\rd\varphi}{\rd\ub} \Omg^{-2} \f{\rd\Phi}{\rd \ub})(\ub',\vartheta) \,\mathrm{dA}_{\mathring{\gamma}}\,\ud \ub' \\
= &\: -\f 18 \int_0^{\underline{I}} \int_{\mathbb S^2} (\varphi \Omg^{-2} |\f{\rd\hat{\gamma}}{\rd\ub}|^2_{\hat{\gamma}} \Phi)(\ub',\vartheta) \,\mathrm{dA}_{\mathring{\gamma}}\,\ud \ub' - \f 12 \int_{(0,\underline{I})\times \mathbb S^2} \Phi^{-1}\varphi (\ub',\vartheta)\,\ud\nu_{\mathrm{init}}
\end{split}
\end{equation}
for any $\varphi \in C^\infty_c((0,\underline{I})\times \mathbb S^2)$.

Similarly, on $\Hb_0$, $(\Omg,\,\Phi,\,\hat{\gamma},\,\ud\nub_{\mathrm{init}})$ together satisfy the constraint equations
\begin{equation}\label{eq:dnub.init}
\begin{split}
&\:-\int_0^{I} \int_{\mathbb S^2} (\f{\rd\varphi}{\rd u} \Omg^{-2} \f{\rd\Phi}{\rd u})(u',\vartheta) \,\mathrm{dA}_{\mathring{\gamma}}\,\ud u'  \\
= &\: -\f 18 \int_0^{I} \int_{\mathbb S^2} (\varphi \Omg^{-2} |\f{\rd\hat{\gamma}}{\rd u}|^2_{\hat{\gamma}} \Phi)(u',\vartheta) \,\mathrm{dA}_{\mathring{\gamma}}\,\ud u' - \f 12 \int_{(0,I)\times \mathbb S^2} \Phi^{-1}\varphi (u',\vartheta)\,\ud\nub_{\mathrm{init}}
\end{split}
\end{equation}
for any $\varphi \in C^\infty_c((0,I)\times \mathbb S^2)$.
\end{enumerate}
\end{definition}
We emphasize that the initial data in Definition~\ref{def:SARCID} are consistent with the initial null dust being merely a measure, and initial metric being merely $C^0 \cap W^{1,2}$.

The following is our local existence and uniqueness theorem for the Einstein--null dust system with measure-valued null shells (cf.~Theorem~\ref{thm:null.shells.intro}):
\begin{theorem}\label{thm:main.local.dust}
Given a strongly angularly regular reduced characteristic initial data set $(\Omega,\,\Phi,\, \hat{\gamma},\, \slashed{\mathcal L}_{\f{\rd}{\rd\ub}} b\restriction_{S_{0,0}},\, \ud \nu_{\mathrm{init}},\,\ud \nub_{\mathrm{init}})$ as in Definition~\ref{def:SARCID}, there exists $\ep>0$ sufficiently small such that whenever $u_* \in (0,I]$ and $\ub_* \in (0,\ep]$, there exists a unique angularly regular weak solution to the Einstein--null dust system in the sense of Definition~\ref{def:weak.sol.ang.reg} in the domain $[0,u_*]\times [0,\ub_*]\times \mathbb S^2$ which achieves the prescribed initial data.
\end{theorem}

The proof of Theorem~\ref{thm:main.local.dust} will be completed in \textbf{Section~\ref{sec.approx.thm}}.

\subsection{Approximating solutions to the Einstein--null dust system by solutions to the Einstein vacuum equations}

Finally, we prove that (up to some technical assumptions) any angularly regular weak solution to the Einstein--null dust system can be weakly approximately by smooth solutions to the Einstein vacuum equations (cf.~Corollary~\ref{cor:reverse.Burnett.intro}):
\begin{theorem}\label{thm:reverse.Burnett}
Let $\mathcal M = [0,u_*]\times [0,\ub_*]\times \mathbb S^2$. Suppose $(\mathcal M,\, g, \,\{\ud\nu_u\}_{u\in [0,u_*]},\,\{\ud \underline{\nu}_{\ub} \}_{\ub\in [0,\ub_*]})$ is an angularly regular weak solution (see Definition~\ref{def:weak.sol.ang.reg}) to the Einstein--null dust system with strongly angularly regular characteristic initial data (see Definition~\ref{def:SARCID}) and $b \restriction_{\Hb_0} = 0$.
Assume moreover that the strongly angularly regular characteristic initial data set satisfies
$$\mathrm{supp}(\ud\nu_{\mathrm{init}}) \subset [0,\underline{I} ]\times U^c,\quad\mathrm{supp}(\ud\nub_{\mathrm{init}}) \subset [0,I]\times U^c$$ for some non-empty open set $U\subset  \mathbb S^2$.

Let $\widetilde{\mathcal M} = [0,u_*]\times [0,\ub_{**}]\times \mathbb S^2 \subseteq \mathcal M$. Then, for $\ub_{**}\in (0, \ub_*]$ sufficiently small, there exists a sequence of smooth solutions to the Einstein vacuum equations $(\widetilde{\mathcal M},\, g_n)$ in double null coordinates such that $(\widetilde{\mathcal M},\,g_n)$ converges to $(\widetilde{\mathcal M},\,g\restriction_{\widetilde{\mathcal M}})$ in the sense described in Theorem~\ref{main.thm}. 
\end{theorem}

The proof of Theorem~\ref{thm:reverse.Burnett} will be completed in \textbf{Section~\ref{sec.approx.thm}}.

We explain some of the assumptions in the following remarks.

\begin{remark}
Theorem~\ref{thm:reverse.Burnett} requires the initial data to satisfy $b \restriction_{\Hb_0} = 0$. This is however not a serious restriction as given any $(\mathcal M,\, g, \,\{\ud\nu_u\}_{u\in [0,u_*]},\,\{\ud \underline{\nu}_{\ub} \}_{\ub\in [0,\ub_*]})$, one can always change coordinates on the $2$-spheres to achieve the vanishing of $b$ on $\Hb_0$.
\end{remark}

\begin{remark}
Theorem~\ref{thm:reverse.Burnett} requires the vanishing of the (initial) null dust in some angular directions. This restriction is due to a hairy ball theorem-type obstruction in prescribing a symmetric traceless tensor on the $2$-sphere.

A consequence of this assumption that if we consider solutions on $\widetilde{\mathcal M} = [0,u_*]\times [0,\ub_{**}]\times \mathbb S^2$, Theorem~\ref{thm:reverse.Burnett} only concerns a restrictive class of solutions. Nevertheless, by the finite speed of propagation, it does imply that (sufficiently\footnote{Note that we still need assumptions more than that in Definition~\ref{def:weak.sol.ang.reg} as the data are required to be more regular.}) angularly regular weak solutions can always be \underline{locally} weakly approximated by vacuum solutions.
\end{remark}

\section{General compactness results}\label{sec:gen.compactness}

\subsection{Preliminaries}

The following Sobolev embedding result is standard (recall Definition~\ref{def:isoperimetric} for notations).
\begin{proposition}[Sobolev embedding]\label{prop:Sobolev}
\begin{enumerate}
\item (\cite[Lemma 5.1]{Chr}) Let $2< p<\infty$ and $r\in \mathbb N\cup\{0\}$. There exists a constant $C_{p,r}>0$, depending only on $p$ and $r$, such that for any closed Riemannian $2$-manifold $(S,\gamma)$ with a $C^2$ metric\footnote{While the lemma is stated in \cite{Chr} for smooth metrics, the $C^2$ case follows from an easy approximation argument.}, 
$$(\mbox{Area}(S))^{-\f1p}\|\xi\|_{L^p(S)}\leq C_{p,r}\sqrt{\max\{{\bf I}(S),1\}}(\|\nab\xi\|_{L^2(S)}+(\mbox{Area}(S))^{-\f12}\|\xi\|_{L^2(S)})$$
for any covariant tensor $\xi$ of rank $r$.
\item (\cite[Lemma~5.2]{Chr}) Let $2< p<\infty$ and $r\in \mathbb N\cup\{0\}$. There exists a constant $C_{p,r}>0$, depending only on $p$ and $r$, such that for any closed Riemannian $2$-manifold $(S,\gamma)$ with a $C^2$ metric,
$$\|\xi\|_{L^\infty(S)}\leq C_{p,r}\sqrt{\max\{{\bf I}(S),1\}}(\mbox{Area}(S))^{\f12-\f1p}(\|\nab\xi\|_{L^p(S)}+(\mbox{Area}(S))^{-\f12}\|\xi\|_{L^p(S)})$$
for any covariant tensor $\xi$ of rank $r$.
\end{enumerate}
\end{proposition}

\begin{proposition}\label{prop:transport.id}
The following identities hold for any $C^1$ function $f$:
$$\f{\rd}{\rd u} \int_{S_{u,\ub}} f \,\mathrm{dA}_{\gamma} = \int_{S_{u,\ub}} \Omg \left(\nab_3 f + \trchb f \right)\, \,\mathrm{dA}_{\gamma} = \int_{S_{u,\ub}} \left(\f{\rd}{\rd u} f + \Omg \trchb f \right)\, \,\mathrm{dA}_{\gamma},$$
$$\f{\rd}{\rd \ub} \int_{S_{u,\ub}} f \,\mathrm{dA}_{\gamma} = \int_{S_{u,\ub}} \Omg \left(\nab_4 f + \trch f \right)\, \,\mathrm{dA}_{\gamma}.$$
\end{proposition}

\subsection{Compactness theorems}\label{sec:cpt.AA}
Starting from this subsection, we prove various general compactness results. We will \textbf{consider the following setup for the remaining subsections in this section}:
\begin{itemize}
\item We will consider as our domain the manifold (with corners) $\mathcal M = [0,u_*]\times [0,\ub_*]\times \mathbb S^2$.
\item On $\mathcal M$, we fix a $C^3$ metric $\mathring{\gamma}$ (independent of $n$) such that
\begin{equation}\label{eq:compact,HO.DT}
\slashed{\mathcal L}_{\f{\rd}{\rd u}} \mathring{\gamma} = \slashed{\mathcal L}_{\f{\rd}{\rd \ub}} \mathring{\gamma} = 0.
\end{equation}
This will be used for defining the norms.
\end{itemize}

Our first compactness result is the following simple variation of the Arzela--Ascoli theorem.
\begin{proposition}\label{prop:AA.gen}
Consider either the case $(p,q) = (+\infty, 4)$ or $(p,q) = (4,2)$.

Let $\{\psi_n\}_{n=1}^{+\infty}$ be a sequence of covariant rank $r$ smooth $S$-tangent tensors satisfying the following uniform bounds:
\begin{equation}\label{eq:assumption.for.AA}
\sup_n (\|\psi_n\|_{L^\i_u L^\i_{\ub} W^{1,q}(S_{u,\ub}, \mathring{\gamma})} + \|\slashed{\mathcal L}_{\f{\rd}{\rd \ub}} \psi_n \|_{L^\i_u L^2_{\ub}L^p(S_{u,\ub},\mathring{\gamma})} + \|\slashed{\mathcal L}_{\f{\rd}{\rd u}} \psi_n \|_{L^\i_{\ub} L^2_u L^p(S_{u,\ub},\mathring{\gamma})})< +\infty.
\end{equation}
Then, there exists a subsequence $\{\psi_{n_k}\}_{k=1}^{+\infty}$ and a $\psi_\infty\in C^0_u C^0_{\ub} L^p(S_{u,\ub},\mathring{\gamma}) \cap L^\i_u L^\i_{\ub} W^{1,q}(S_{u,\ub}, \mathring{\gamma})$ such that 
\begin{equation}\label{eq:AA.gen}
\lim_{k\to +\infty} \|\psi_{n_k} - \psi_\infty\|_{L^\i_u L^\i_{\ub} L^p(S_{u,\ub},\mathring{\gamma})} = 0.
\end{equation}
Moreover, $\slashed{\mathcal L}_{\f{\rd}{\rd \ub}} \psi_\infty \in L^\i_u L^2_{\ub}L^p(S_{u,\ub},\mathring{\gamma})$ and $\slashed{\mathcal L}_{\f{\rd}{\rd u}} \psi_\infty \in L^\i_{\ub} L^2_u L^p(S_{u,\ub},\mathring{\gamma})$.
\end{proposition}
\begin{proof}
\pfstep{Step~1: Existence of $\psi_\i\in L^\i_u L^\i_{\ub} L^p(S_{u,\ub},\mathring{\gamma})$ and proof of \eqref{eq:AA.gen}} First, since $W^{1,q}(S_{u,\ub}, \mathring{\gamma})\subseteq L^p(S_{u,\ub}, \mathring{\gamma})$ is compact for all $(u,\ub)$, we know that for every $(u,\ub)$, there is a subsequence $n_k$ and a $\psi_\infty\in L^p(S_{u,\ub},\mathring{\gamma})$ such that 
\begin{equation}\label{eq:easy.convergence}
\lim_{k\to +\infty} \|\psi_{n_k} - \psi_\infty\|_{L^p(S_{u,\ub},\mathring{\gamma})} = 0.
\end{equation}

By a standard argument extracting a diagonal subsequence, we thus obtain that for a \emph{fixed} subsequence $n_k$, the \eqref{eq:easy.convergence} holds for all $(u,\ub)$ rational. 

We next show that for this fixed subsequence, $\psi_{n_k}$ is uniformly (in $(u,\ub)$) Cauchy in $L^p(S_{u,\ub}, \mathring{\gamma})$.

Let $\ep>0$. For $(u_\mathbb Q, \ub_\mathbb Q)\in ([0,u_*]\times [0,\ub_*])\cap \mathbb Q^2$, let $\mathcal R(u_\mathbb Q, \ub_\mathbb Q; \ep):= \{(u,\ub)\in [0,u_*]\times [0,\ub_*]: |u-u_{\mathbb Q}|+ |\ub-\ub_{\mathbb Q}|\leq \ep^2\}$. Clearly, we can find a finite set of $\{(u_i,\ub_i)\}_{i=1}^m \subset ([0,u_*]\times [0,\ub_*])\cap \mathbb Q^2$ (depending on $\ep$) such that $\cup_{i=1}^m \mathcal R(u_i, \ub_i; \ep) = [0,u_*]\times [0,\ub_*]$. 

Since \eqref{eq:easy.convergence} holds for $(u,\ub) = (u_i,\ub_i)$ for $i=1,\dots,m$, we can find $K$ sufficiently large such that whenever $1\leq i\leq m$ and $k,\,k'\geq K$, we have
$$\|\psi_{n_k} - \psi_{n_{k'}} \|_{L^p(S_{u_i,\ub_i}, \mathring{\gamma})} \leq \ep.$$

Now for every $(u,\ub) \in [0,u_*]\times [0,\ub_*]$, we can find the closest $(u_i,\ub_i)$ so that we obtain
\begin{equation}\label{eq:AA.Cauchy}
\begin{split}
&\: \|\psi_{n_k} - \psi_{n_{k'}}\|_{L^p(S_{u,\ub}, \mathring{\gamma})} \\
\ls &\: \|\psi_{n_k} - \psi_{n_{k'}} \|_{L^p(S_{u_i,\ub_i}, \mathring{\gamma})} + \left| \int_{u_i}^u (\|\slashed{\mathcal L}_{\f{\rd}{\rd u}}\psi_{n_k}\|_{L^p(S_{u',\ub_i},\mathring{\gamma})}  + \|\slashed{\mathcal L}_{\f{\rd}{\rd u}}\psi_{n_{k'}} \|_{L^p(S_{u',\ub_i}\mathring{\gamma})} ) \,\ud u' \right|\\
&\: + \left| \int_{\ub_i}^{\ub} (\|\slashed{\mathcal L}_{\f{\rd}{\rd \ub}}\psi_{n_k}\|_{L^p(S_{u,\ub'},\mathring{\gamma})}  + \|\slashed{\mathcal L}_{\f{\rd}{\rd \ub}}\psi_{n_{k'}} \|_{L^p(S_{u,\ub'}\mathring{\gamma})} ) \,\ud \ub'\right| \ls \ep
\end{split}
\end{equation}
whenever $k,\,k'\geq K$. This proves that there exists $\psi_\i\in L^\i_u L^\i_{\ub} L^p(S_{u,\ub},\mathring{\gamma})$ and that \eqref{eq:AA.gen} holds. 

\pfstep{Step~2:  $\psi_\infty\in C^0_u C^0_{\ub} L^p(S_{u,\ub},\mathring{\gamma}) \cap L^\i_u L^\i_{\ub} W^{1,q}(S_{u,\ub}, \mathring{\gamma})$ } Since $\psi_{n_k}$ are smooth and $\psi_\i$ is their limit in $L^\i_u L^\i_{\ub} L^p(S_{u,\ub},\mathring{\gamma})$, it immediate follows that $\psi_\i \in C^0_u C^0_{\ub} L^p(S_{u,\ub},\mathring{\gamma})$.

Next, notice that for every fixed $(u,\ub)$, the given uniform $W^{1,q}(S_{u,\ub},\mathring{\gamma})$ estimate in \eqref{eq:assumption.for.AA} and the Banach--Alaoglu theorem imply that there exists $\phi_\infty$ such that for a further subsequence, $\mathring{\nab}\psi_{n_k} \rightharpoonup \phi_\infty$ \emph{weakly} in $L^q(S_{u,\ub},\mathring{\gamma})$. Since the $W^{1,q}(S_{u,\ub},\mathring{\gamma})$ estimate is uniform in $(u,\ub)$, it follows that $\phi_\infty \in L^\i_u L^\i_{\ub} L^{q}(S_{u,\ub}, \mathring{\gamma})$. It is easy to check that $\phi_\infty = \mathring{\nab}\psi_\infty$, proving that $\psi_\infty \in  L^\i_u L^\i_{\ub} W^{1,q}(S_{u,\ub}, \mathring{\gamma})$.

\pfstep{Step~3: $\slashed{\mathcal L}_{\f{\rd}{\rd \ub}} \psi_\infty \in L^\i_u L^2_{\ub}L^p(S_{u,\ub},\mathring{\gamma})$ and $\slashed{\mathcal L}_{\f{\rd}{\rd u}} \psi_\infty \in L^\i_{\ub} L^2_u L^p(S_{u,\ub},\mathring{\gamma})$} For every fixed $u$, the estimate \eqref{eq:assumption.for.AA} and Banach--Alaoglu implies that after passing to a further subsequence, $\slashed{\mathcal L}_{\f{\rd}{\rd \ub}} \psi_{n_{k_\ell}}$ has a \emph{weak} limit in $L^2_{\ub}L^p(S_{u,\ub},\mathring{\gamma})$. It is easy to check that the limit coincides with $\slashed{\mathcal L}_{\f{\rd}{\rd \ub}} \psi_{\i}$, and therefore 
$\slashed{\mathcal L}_{\f{\rd}{\rd \ub}} \psi_\infty \in L^\i_u L^2_{\ub}L^p(S_{u,\ub},\mathring{\gamma})$.

The proof of $\slashed{\mathcal L}_{\f{\rd}{\rd u}} \psi_\infty \in L^\i_{\ub} L^2_u L^p(S_{u,\ub},\mathring{\gamma})$ is similar; we omit the details.
\qedhere
\end{proof}

\begin{proposition}\label{prop:compact.embeddings}
Consider one of the following cases: (1) $m\in \{0,1\}$, $(p,q)= (+\infty, 4)$; (2) $m\in \{0,1,2\}$, $(p,q)= (4,2)$.

Let $\{\psi_n\}_{n=1}^{+\infty}$ be a sequence of covariant rank $r$ smooth $S$-tangent tensors satisfying the following uniform bounds:
$$\sup_n (\|\psi_n\|_{L^\i_u L^\i_{\ub} W^{m+1,q}(S_{u,\ub}, \mathring{\gamma})} + \|\slashed{\mathcal L}_{\f{\rd}{\rd \ub}} \psi_n \|_{L^\i_u L^2_{\ub}W^{m,p}(S_{u,\ub},\mathring{\gamma})} + \|\slashed{\mathcal L}_{\f{\rd}{\rd u}} \psi_n \|_{L^\i_{\ub} L^2_u W^{m,p}(S_{u,\ub},\mathring{\gamma})})< +\infty.$$
Then, there exists a subsequence $\{\psi_{n_k}\}_{k=1}^{+\infty}$ and a $\psi_\infty\in C^0_u C^0_{\ub} W^{m,p}(S_{u,\ub},\mathring{\gamma}) \cap L^\i_u L^\i_{\ub} W^{m+1,q}(S_{u,\ub}, \mathring{\gamma})$ such that 
\begin{equation}\label{eq:AA.gen.2}
\lim_{k\to +\infty} \|\psi_{n_k} - \psi_\infty\|_{L^\i_u L^\i_{\ub} W^{m,p}(S_{u,\ub},\mathring{\gamma})} = 0.
\end{equation}
Moreover, $\slashed{\mathcal L}_{\f{\rd}{\rd \ub}} \psi_\infty \in L^\i_u L^2_{\ub}W^{m,p}(S_{u,\ub},\mathring{\gamma})$ and $\slashed{\mathcal L}_{\f{\rd}{\rd u}} \psi_\infty \in L^\i_{\ub} L^2_u W^{m,p}(S_{u,\ub},\mathring{\gamma})$.
\end{proposition}
\begin{proof}
This is straightforward from Proposition~\ref{prop:AA.gen}. For $m=0$, this is exactly Proposition~\ref{prop:AA.gen}. 

We consider the case $m=1$. By Proposition~\ref{prop:AA.gen}, there exists a subsequence $\psi_{n_k}$ and $\psi_\infty\in C^0_u C^0_{\ub} L^p(S_{u,\ub},\mathring{\gamma}) \cap L^\i_u L^\i_{\ub} W^{1,q}(S_{u,\ub}, \mathring{\gamma})$ such that $\slashed{\mathcal L}_{\f{\rd}{\rd \ub}} \psi_\infty \in L^\i_u L^2_{\ub}L^p(S_{u,\ub},\mathring{\gamma})$, $\slashed{\mathcal L}_{\f{\rd}{\rd u}} \psi_\infty \in L^\i_{\ub} L^2_u L^p(S_{u,\ub},\mathring{\gamma})$ and
\begin{equation*}
\lim_{k\to +\infty} \|\psi_{n_k} - \psi_\infty\|_{L^\i_u L^\i_{\ub} L^p(S_{u,\ub},\mathring{\gamma})} = 0.
\end{equation*}
Moreover, using Proposition~\ref{prop:AA.gen} for $\mathring{\nab} \psi_{n_k}$, we see that after passing to a further subsequence (not relabeled), there exists $\phi_\infty\in C^0_u C^0_{\ub} L^p(S_{u,\ub},\mathring{\gamma}) \cap L^\i_u L^\i_{\ub} W^{1,q}(S_{u,\ub}, \mathring{\gamma})$ such that $\slashed{\mathcal L}_{\f{\rd}{\rd \ub}} \phi_\infty \in L^\i_u L^2_{\ub}L^p(S_{u,\ub},\mathring{\gamma})$, $\slashed{\mathcal L}_{\f{\rd}{\rd u}} \phi_\infty \in L^\i_{\ub} L^2_u L^p(S_{u,\ub},\mathring{\gamma})$ and
\begin{equation*}
\lim_{k\to +\infty} \|\mathring{\nab}\psi_{n_k} - \phi_\infty\|_{L^\i_u L^\i_{\ub} L^p(S_{u,\ub},\mathring{\gamma})} = 0.
\end{equation*}
The conclusion of the proposition in the $m=1$ case then follows after checking that $\phi_\infty = \mathring{\nab}\psi_{\i}$ and using that $[\slashed{\mathcal L}_{\f{\rd}{\rd \ub}}, \mathring{\nab}] = [\slashed{\mathcal L}_{\f{\rd}{\rd u}}, \mathring{\nab}]  = 0$ (due to \eqref{eq:compact,HO.DT}).

The $m=2$ case is similar to the $m=1$ case; we omit the details. \qedhere
\end{proof}

\subsection{Compactness in BV}

\begin{lemma}[Aubin--Lions lemma]\label{lem:AL}
Let $X_0\subseteq X \subseteq X_1$ be three Banach spaces such that $X_0\subseteq X$ is compact and $X \subseteq X_1$ is continuous. For $T>0$ and $q>1$, let
$$W:= \{ v \in L^\infty([0,T]; X_0): \dot{v} \in L^q([0,T]; X_1) \},$$
where $\dot{}$ denotes the (weak) derivative in the variable on $[0,T]$. 

Then $W$ embeds compactly into $C^0([0,T]; X)$.
\end{lemma}

\begin{proposition}[Compactness of BV functions (Theorems~5.2 and 5.5 in \cite{Evans})]\label{prop:BV.general}
Let $U\subset \mathbb R^k$ be open and bounded, with Lipschitz boundary $\rd U$. Assume $\{\psi_n\}_{n=1}^{+\infty}$ is a sequence in (Euclidean) $BV(U)$ satisfying
$$\sup_n \|\psi_n\|_{BV(U)}:= \sup_n \left(\int_{U} |\psi_n|\,\ud x+ \sup\{\int_U  \psi_n \mathrm{div}_{\mathbb R^k} \phi \,\ud x : \phi \in C^1_c(U;\mathbb R^k),\,\|\phi\|_{L^\i(U)} \leq 1\}\right)<+\infty.$$
Then there exists a subsequence $\{\psi_{n_k}\}_{k=1}^{+\infty}$ and a function $\psi_\i \in BV(U)$ such that 
$$\int_U |\psi_{n_k} - \psi_\i|\,\ud x\to 0$$
as $k\to +\infty$. Moreover, 
\begin{equation}\label{eq:BV.general.est}
\|\psi_\i\|_{BV(U)} \leq \liminf_{k\to +\infty} \|\psi_{n_k}\|_{BV(U)}.
\end{equation}
\end{proposition}

We now consider the application of Proposition~\ref{prop:BV.general} to our setting. In particular, in the proposition below, we continue to work in the setting described in the beginning of Section~\ref{sec:cpt.AA}.
\begin{proposition}\label{prop:BV}
\begin{enumerate}
\item Let $\{\psi_n\}_{n=1}^{+\infty}$ be a sequence of smooth functions such that
$$\sup_n (\|\psi_n\|_{C^0_u BV(H_u,\mathring{\gamma})} + \|\slashed{\mathcal L}_{\f{\rd}{\rd u}} \psi_n\|_{L^2_u L^1_{\ub} L^1(S_{u,\ub},\mathring{\gamma})})<+\infty.$$
Then there exists a subsequence $\{\psi_{n_k}\}_{k=1}^{+\infty}$ and a $\psi_\infty \in C^0_uL^1_{\ub} L^1(S_{u,\ub},\mathring{\gamma}) \cap L^\i_u BV(H_u,\mathring{\gamma})$ such that 
\begin{equation}\label{eq:BV.L1.conv}
\|\psi_{n_k} - \psi_\infty\|_{C^0_u L^1_{\ub} L^1(S_{u,\ub}, \mathring{\gamma})} \to 0
\end{equation}
as $k\to +\infty$.
\item An entirely symmetric statement holds after swapping $u$ and $\ub$.
\end{enumerate}
\end{proposition}
\begin{proof}
We will only consider case~(1); case (2) can be treated by swapping $u$ and $\ub$.

First, by Lemma~\ref{lem:AL} (with $X_0 = BV(H_u,\mathring{\gamma})$, $X = X_1 = L^1_{\ub} L^1(S_{u,\ub},\mathring{\gamma})$, $T = u_*$, $q = 2$) and Proposition~\ref{prop:BV.general} (which gives the compactness of $X_0\subseteq X$), it follows that there exists $\psi_\infty \in C^0_u L^1_{\ub} L^1(S_{u,\ub}, \mathring{\gamma})$ such that \eqref{eq:BV.L1.conv} holds.

The fact that $\psi_\infty \in L^\i_u BV(H_u)$ then follows from \eqref{eq:BV.general.est} in Proposition~\ref{prop:BV.general}. \qedhere
\end{proof}

\subsection{Weak compactness theorems}We continue to work in the setting described in the beginning of Section~\ref{sec:cpt.AA}.

\begin{proposition}\label{prop:weak}
\begin{enumerate}
\item Let $\{\psi_n\}_{n=1}^{+\infty}$ be a sequence of rank-$r$ $S$-tangent covariant smooth tensor fields such that
\begin{equation}\label{eq:prop.weak.assumption}
\sup_n (\|\psi_n\|_{C^0_u L^1_{\ub} L^1(S_{u,\ub},\mathring{\gamma})} + \|\slashed{\mathcal L}_{\f{\rd}{\rd u}} \psi_n\|_{L^2_u L^1_{\ub} L^1(S_{u,\ub},\mathring{\gamma})})<+\infty.
\end{equation}
Then there exists a subsequence $\{\psi_{n_k}\}_{k=1}^{+\infty}$ and a rank-$r$ $S$-tangent covariant tensor field-valued Radon measure $\{\ud \mu_{\psi,u}\}_{u\in [0,u_*]}$ such that the following hold:
\begin{enumerate}
\item For every $u\in [0,u_*]$ and every rank-$r$ $S$-tangent bounded contravariant tensor field $\varphi \in C^0$ on $(0,\ub_*)\times \mathbb S^2$,
\begin{equation}\label{eq:weak.conv}
\int_{\{u\}\times (0,\ub_*)\times \mathbb S^2} \langle \varphi(\ub,\vartheta), \psi_{n_k}\rangle \, \mathrm{dA}_{\mathring{\gamma}}\,\ud \ub - \int_{\{u\}\times (0,\ub_*)\times \mathbb S^2} \varphi(\ub,\vartheta)\cdot \ud \mu_{\psi,u} \to 0
\end{equation}
as $k\to +\infty$.
\item $\ud \mu_{\psi,u}$ is continuous in $u$ and the following holds:
\begin{equation}\label{eq:weak.conv.cont}
\begin{split}
&\: \left| \int_{\{u'\}\times (0,\ub_*)\times \mathbb S^2} \langle \varphi(\ub,\vartheta), \psi_{n_k}\rangle \, \mathrm{dA}_{\mathring{\gamma}}\,\ud \ub - \int_{\{u\}\times (0,\ub_*) \times \mathbb S^2} \langle \varphi(\ub,\vartheta), \psi_{n_k}\rangle \, \mathrm{dA}_{\mathring{\gamma}}\,\ud \ub\right| \\
\leq &\: \limsup_{k\to +\infty} |u-u'|^{\f 12} \|\slashed{\mathcal L}_{\f{\rd}{\rd u}}\psi_{n_k}\|_{L^2_u L^1_{\ub}L^1(S_{u,\ub},\mathring{\gamma})} \|\varphi\|_{L^\i_{\ub}L^\i(S_{u,\ub},\mathring{\gamma})}.
\end{split}
\end{equation}
\item $\ud \mu_{\psi,u}$ is uniformly bounded as follows:
\begin{equation}\label{eq:weak.conv.bdd}
\begin{split}
&\: \sup_{u\in [0,u_*]} \sup_{\varphi \in C^0: \|\varphi\|_{L^\i_{\ub}L^\i(S_{u,\ub},\mathring{\gamma})}\leq 1} \left|\int_{\{u\}\times (0,\ub_*) \times \mathbb S^2} \varphi(\ub,\vartheta)\cdot \ud \mu_{\psi,u}\right| \\
\leq &\: \limsup_{k\to +\infty}\|\psi_{n_k}\|_{C^0_u L^1_{\ub} L^1(S_{u,\ub},\mathring{\gamma})}.
\end{split}
\end{equation}
\end{enumerate}
\item An entirely symmetric statement holds after swapping $u$ and $\ub$.
\end{enumerate}
\end{proposition}
\begin{proof}
We only consider case (1); case (2) is similar.

\pfstep{Step~1: Proof of \eqref{eq:weak.conv} for \underline{rational} $u$} For every \emph{fixed} $u\in [0,u_*]$, by the Banach--Alaoglu theorem, there exists a subsequence $n_k$ and a rank-$r$ $S$-tangent covariant tensor field-valued Radon measure $\ud\mu_{\psi,u}$ such that \eqref{eq:weak.conv} holds.

By considering $u\in [0,u_*]\cap \mathbb Q$ and using a standard trick of picking a diagonal subsequence, we can therefore find a \emph{fixed} subsequence $n_k$ such that \eqref{eq:weak.conv} holds for all \emph{rational} $u\in [0,u_*]$.

\textbf{We now fix the subsequence $n_k$.}

\pfstep{Step~2: Proof of \eqref{eq:weak.conv} for \underline{all} $u\in [0,u_*]$} We now show that for the subsequence $n_k$ fixed in Step~1, \eqref{eq:weak.conv} in fact holds for \emph{all} $u\in [0,u_*]$. 

First, we note that for every fixed $u\in [0,u_*]$ (not necessarily rational), we can take a further subsequence $n_{k_\ell}$ so that \eqref{eq:weak.conv} holds (for some $\ud\mu_{\psi,u}$). Thus, in order to obtain weak-* convergence of that full subsequence $n_k$, it suffices to show that for \emph{all} $u\in [0,u_*]$, and all bounded $C^0$ $u$-independent, rank-$r$ contravariant, $S$-tangent tensor field $\varphi$,
$$\left\{ \int_{\{u\}\times (0,\ub_*) \times \mathbb S^2} \langle \varphi(\ub,\vartheta), \psi_{n_k}\rangle \, \mathrm{dA}_{\mathring{\gamma}}\,\ud \ub  \right\}_{k=1}^{+\infty}$$
is a Cauchy sequence.

Fix $u\in [0,u_*]$ (not necessarily rational) and $\varphi \in C^0$ for the remainder of this step. Since $\slashed{\mathcal L}_{\f{\rd}{\rd u}} \mathring{\gamma} = 0$, we can use the fundamental theorem of calculus, H\"older's inequality and \eqref{eq:prop.weak.assumption} to obtain
\begin{equation}\label{eq:prop.weak.1}
\begin{split}
&\: \left| \int_{\{u'\}\times (0,\ub_*) \times \mathbb S^2} \langle \varphi(\ub,\vartheta), \psi_{n_k}\rangle \, \mathrm{dA}_{\mathring{\gamma}}\,\ud \ub - \int_{\{u\}\times (0,\ub_*) \times \mathbb S^2} \langle \varphi(\ub,\vartheta), \psi_{n_k}\rangle \, \mathrm{dA}_{\mathring{\gamma}}\,\ud \ub\right| \\
\leq &\: \left| \int_{u'}^u \int_{\{u''\}\times (0,\ub_*) \times \mathbb S^2} \langle \varphi(\ub,\vartheta), \slashed{\mathcal L}_{\f{\rd}{\rd u}}\psi_{n_k}\rangle \,\mbox{dA}_{\mathring{\gamma}} \,\ud\ub \right| \\
\leq &\: |u-u'|^{\f 12} \|\slashed{\mathcal L}_{\f{\rd}{\rd u}}\psi_{n_k}\|_{L^2_u L^1_{\ub}L^1(S_{u,\ub},\mathring{\gamma})} \|\varphi\|_{L^\i_{\ub}L^\i(S_{u,\ub},\mathring{\gamma})} \ls |u-u'|^{\f 12} \|\varphi\|_{L^\i_{\ub}L^\i(S_{u,\ub},\mathring{\gamma})}.
\end{split}
\end{equation}

Let $\ep >0$. There exists $u'$ rational such that 
\begin{equation}\label{eq:prop.weak.2}
|u-u'|^{\f 12} \|\varphi\|_{L^\infty_{\ub}L^\infty(S)} < \ep.
\end{equation}
By Step~1, for this rational $u'$, there exists $K>0$ such that whenever $k,\,k'\geq K$, we have
\begin{equation}\label{eq:prop.weak.3}
\begin{split}
\left|\int_{\{u'\}\times (0,\ub_*) \times \mathbb S^2} \langle \varphi(\ub,\vartheta), \psi_{n_k}\rangle \, \mathrm{dA}_{\mathring{\gamma}}\,\ud \ub - \int_{\{u'\}\times (0,\ub_*) \times \mathbb S^2} \langle \varphi(\ub,\vartheta), \psi_{n_{k'}}\rangle \, \mathrm{dA}_{\mathring{\gamma}}\,\ud \ub\right|<\ep.
\end{split}
\end{equation}
Hence, by \eqref{eq:prop.weak.1} (for both $k$ and $k'$), \eqref{eq:prop.weak.2}, \eqref{eq:prop.weak.3}, and the triangle inequality, 
\begin{equation*}
\begin{split}
\left| \int_{\{u\}\times (0,\ub_*) \times \mathbb S^2} \langle \varphi(\ub,\vartheta), \psi_{n_k}\rangle \, \mathrm{dA}_{\mathring{\gamma}}\,\ud \ub - \int_{\{u\}\times (0,\ub_*) \times \mathbb S^2} \langle \varphi(\ub,\vartheta), \psi_{n_{k'}}\rangle \, \mathrm{dA}_{\mathring{\gamma}}\,\ud \ub \right|\ls \ep
\end{split}
\end{equation*}
for $k,k'\geq K$, which is what we wanted to prove. 

\pfstep{Step~3: Proof of \eqref{eq:weak.conv.cont} and \eqref{eq:weak.conv.bdd}} The estimate \eqref{eq:weak.conv.cont} follows from \eqref{eq:prop.weak.1} after taking $\limsup_{k\to +\infty}$. Finally, \eqref{eq:weak.conv.bdd} follows from \eqref{eq:weak.conv} and H\"older's inequality. \qedhere
\end{proof}

\begin{proposition}\label{prop:weak.L2}
Suppose, in addition to the assumptions of Proposition~\ref{prop:weak}, there exists $q\in [2,+\infty)$ such that
\begin{equation}\label{eq:weak.L2.assumption}
\sup_n (\|\psi_n\|_{L^{q}_{\ub} L^\i_u L^2(S_{u,\ub},\mathring{\gamma})} + \|\slashed{\mathcal L}_{\f{\rd}{\rd u}} \psi_n\|_{L^2_u L^{q}_{\ub} L^2(S_{u,\ub},\mathring{\gamma})})<+\infty.
\end{equation}
Then there is a rank-$r$ $S$-tangent contravariant tensor field $\psi_\infty \in C^0_u L^{q}_{\ub}L^2(S_{u,\ub},\mathring{\gamma})\cap L^{q}_{\ub}L^\i_u L^2(S_{u,\ub},\mathring{\gamma})$ such that $\ud \mu_{\psi, u} = \psi_\infty\,\mathrm{dA}_{\mathring{\gamma}}\,\ud \ub$ (with $\ud \mu_{\psi, u}$ as in Proposition~\ref{prop:weak}). Moreover, $\slashed{\mathcal L}_{\f{\rd}{\rd u}} \psi_\infty \in L^2_u L^{q}_{\ub} L^2(S_{u,\ub},\mathring{\gamma})$.

A symmetric statement also holds after swapping $u$ and $\ub$.
\end{proposition}
\begin{proof}
\pfstep{Step~1: $\ud \mu_{\psi, u}$ is absolutely continuous with respect to $\mathrm{dA}_{\mathring{\gamma}}\,\ud \ub$} Fix $u\in [0,u_*]$. Suppose $U\subset [0,\ub_*]\times \mathbb S^2$ is an open subset such that $\int_U \, \mathrm{dA}_{\mathring{\gamma}}\,\ud \ub <\ep$. Then, using H\"older's inequality and \eqref{eq:weak.L2.assumption},
\begin{equation*}
\begin{split}
&\: \sup_{\varphi \in C^0: \|\varphi\|_{L^\i_{\ub}L^\i(S_{u,\ub},\mathring{\gamma})}\leq 1} \left|\int_U \varphi\cdot \ud \mu_{\psi, u} \right| \\
\leq &\: \sup_{\varphi \in C^0: \|\varphi\|_{L^\i_{\ub}L^\i(S_{u,\ub},\mathring{\gamma})}\leq 1} \limsup_{k\to +\infty} \int_U |\langle \varphi, \psi_{n_k}\rangle| \,\ud A_{\mathring{\gamma}} \, \ud \ub \\
\leq &\: (\sup_{\varphi \in C^0: \|\varphi\|_{L^\i_{\ub}L^\i(S_{u,\ub},\mathring{\gamma})}\leq 1} \|\varphi\|_{L^\i_{\ub}L^\i(S_{u,\ub},\mathring{\gamma})}) (\sup_{k} \|\psi_{n_k}\|_{L^q_{\ub} L^2(S_{u,\ub})}) (\int_0^{\ub_*} \,\ud \ub)^{\f {q-2}{2q}} (\int_U\,\mathrm{dA}_{\mathring{\gamma}}\,\ud \ub)^{\f 12} \\
\ls &\: \ub_*^{\f {q-2}{2q}}\ep^{\f 12}.
\end{split}
\end{equation*}
It therefore follows that $\ud \mu_{\psi, u} = \psi_\infty\,\mathrm{dA}_{\mathring{\gamma}}\,\ud \ub$ for some rank-$r$ tensor field $\psi_\infty \in L^\i_u L^1_{\ub} L^1(S_{u,\ub})$.

\pfstep{Step~2: Regularity of $\psi_\infty$} In view of Step~1, it remains to prove the regularity statement for $\psi_\infty$.

First, by duality, Fatou's lemma, H\"older's inequality and \eqref{eq:weak.L2.assumption},
\begin{equation}\label{eq:weak.conv.2i2}
\begin{split}
\|\psi_\infty \|_{L^{q}_{\ub} L^\i_u L^2(S_{u,\ub},\mathring{\gamma})}  = &\: \sup_{\varphi\in C^0: \|\varphi\|_{L^{\f{q}{q-1}}_{\ub} L^1_u L^{2}(S_{u,\ub}, \mathring{\gamma})} = 1 } \int_{[0,u_*]\times [0,\ub_*]\times \mathbb S^2} \langle \varphi, \psi_\infty \rangle \, dA_{\mathring{\gamma}} \,du \,d\ub \\
\leq &\:  \sup_{\varphi\in C^0: \|\varphi\|_{L^{\f{q}{q-1}}_{\ub} L^1_u L^{2}(S_{u,\ub}, \mathring{\gamma})} = 1 } \liminf_{k\to \infty} \int_{[0,u_*]\times [0,\ub_*]\times \mathbb S^2} \langle \varphi, \psi_{n_k} \rangle \, dA_{\mathring{\gamma}} \,du \,d\ub <+\infty,
\end{split}
\end{equation}
which proves that $\psi_\infty \in L^q_{\ub} L^\i_u L^2(S_{u,\ub},\mathring{\gamma})$.

Next, we show that $\slashed{\mathcal L}_{\f{\rd}{\rd u}} \psi_\infty \in L^2_u L^{q}_{\ub} L^2(S_{u,\ub},\mathring{\gamma})$. Note that by \eqref{eq:weak.L2.assumption} and the Banach--Alaoglu theorem, after passing to a further subsequence (not relabelled), $\slashed{\mathcal L}_{\f{\rd}{\rd u}} \psi_{n_k}$ has a weak limit  in $L^2_u L^{q}_{\ub} L^2(S_{u,\ub},\mathring{\gamma})$. This limit coincides with $\slashed{\mathcal L}_{\f{\rd}{\rd u}} \psi_\infty$, proving that $\slashed{\mathcal L}_{\f{\rd}{\rd u}} \psi_\infty \in L^2_u L^{q}_{\ub} L^2(S_{u,\ub},\mathring{\gamma})$.

Finally, we show that $\psi_\infty \in C^0_u L^{q}_{\ub} L^2(S_{u,\ub}, \mathring{\gamma})$. By \eqref{eq:weak.conv.2i2}, we already know $\psi_\infty \in L^\i_u L^{q}_{\ub} L^2(S_{u,\ub}, \mathring{\gamma})$. To prove continuity in $u$, we use the fundamental theorem of calculus and the fact (established just above) that $\slashed{\mathcal L}_{\f{\rd}{\rd u}} \psi_\infty \in L^2_u L^{q}_{\ub} L^2(S_{u,\ub},\mathring{\gamma})$.  \qedhere
\end{proof}

\section{Existence of a limiting spacetime}\label{sec:existence}
From now on until the end of Section~\ref{sec:eqns.for.limit}, we work under the assumptions of Theorem~\ref{main.thm}.

In this section, we prove the existence of a limiting spacetime (recall part (2) of Theorem~\ref{main.thm}). We will use the convention that \textbf{all constants $C$ or implicit constants in $\ls$ will depend only on quantities in the assumptions of Theorem~\ref{main.thm}}.

To prove convergence we will continually extract subsequences from $(\mathcal M,g_n)$. Phrases such as ``extracting a further subsequence $n_k$'' will mean that we extract a subsequence from that in the previous lemma, proposition, etc. To simplify notations, we will never relabel the further subsequence.

We begin in \textbf{Section~\ref{sec:existence.prelim}} a preliminary step showing that the norms with respect to different metrics are comparable. We then proceed to show that the limit exists and proving the corresponding regularity statement:  
\begin{enumerate}
 \item The existence of uniform limit of the metric components will be proven in \textbf{Sections~\ref{sec:limit.gamma}, \ref{sec:limit.metric}.}
 \item The existence of uniform limit of $\eta$ and $\etab$ will be proven in \textbf{Section~\ref{sec:limit.eta}}.
  \item The existence of weak limit of $\chih$, $\chibh$, $\trch$, $\trchb$, $\om$ and $\omb$ will be proven in \textbf{Section~\ref{sec:limit.chi}}.
 \item The existence of BV limit of $\trch$ and $\trchb$ will be proven in \textbf{Section~\ref{sec:limit.trch}}.
  \item Finally, the existence of limits in \eqref{eq:dnu.def.thm} and \eqref{eq:dnub.def.thm} will be proven in \textbf{Section~\ref{sec:limit.dust}}.
\end{enumerate}

It will also be important to prove a compensated compactness result, related to some convergence properties of the Ricci coefficients which have the weakest convergence. This will be treated in \textbf{Section~\ref{sec:cc}}.

\subsection{Comparability of norms}\label{sec:existence.prelim}

\begin{proposition}[Comparability of norms]\label{prop:norms.compare}
For every $r\in \mathbb N\cup \{0\}$, there exists a constant $C>0$ (independent of $n$ and $(u,\ub)$) such that for any rank-$r$ $S$-tangent covariant tensor $\xi$,
\begin{align}
C^{-1}\|\xi\|_{L^p(S_{u,\ub}, (\gamma_{0,0}) _n)} \leq \|\xi\|_{L^p(S_{u,\ub}, \gamma_n)} \leq  C\|\xi\|_{L^p(S_{u,\ub}, (\gamma_{0,0})_n)},&\quad 1\leq p\leq +\infty,\label{eq:xi.compare}\\
C^{-1}\|\xi\|_{W^{1,p}(S_{u,\ub}, (\gamma_{0,0})_n)} \leq \|\xi\|_{W^{1,p}(S_{u,\ub}, \gamma_n)} \leq C\|\xi\|_{W^{1,p}(S_{u,\ub}, (\gamma_{0,0})_n)},&\quad 1\leq p\leq +\infty,\label{eq:nab.xi.compare} \\
C^{-1}\|\xi\|_{W^{2,p}(S_{u,\ub}, (\gamma_{0,0})_n)} \leq \|\xi\|_{W^{2,p}(S_{u,\ub}, \gamma_n)} \leq C\|\xi\|_{W^{2,p}(S_{u,\ub}, (\gamma_{0,0})_n)},& \quad 1\leq p\leq 4,\label{eq:nab.2.xi.compare} \\
C^{-1}\|\xi\|_{W^{3,p}(S_{u,\ub}, (\gamma_{0,0})_n)} \leq \|\xi\|_{W^{3,p}(S_{u,\ub}, \gamma_n)} \leq C\|\xi\|_{W^{3,p}(S_{u,\ub}, (\gamma_{0,0})_n)},& \quad 1\leq p\leq 2.\label{eq:nab.3.xi.compare}
\end{align}
Using in addition that $(\gamma_{0,0})_n \to (\gamma_{0,0})_\infty$ in $C^3(S)$ (see~\eqref{eq:gamma.C3.conv}), it follows that
\begin{align*}
C^{-1}\|\xi\|_{L^p(S_{u,\ub}, (\gamma_{0,0})_\infty)} \leq \|\xi\|_{L^p(S_{u,\ub}, \gamma_n)} \leq C\|\xi\|_{L^p(S_{u,\ub}, (\gamma_{0,0})_\infty)},&\quad 1\leq p\leq +\infty, \\
C^{-1}\|\xi\|_{W^{1,p}(S_{u,\ub}, (\gamma_{0,0})_\infty)} \leq \|\xi\|_{W^{1,p}(S_{u,\ub}, \gamma_n)} \leq C\|\xi\|_{W^{1,p}(S_{u,\ub}, (\gamma_{0,0})_\infty)},&\quad 1\leq p\leq +\infty, \\
C^{-1}\|\xi\|_{W^{2,p}(S_{u,\ub}, (\gamma_{0,0})_\infty)} \leq \|\xi\|_{W^{2,p}(S_{u,\ub}, \gamma_n)} \leq C\|\xi\|_{W^{2,p}(S_{u,\ub}, (\gamma_{0,0})_\infty)}, &\quad 1\leq p\leq 4, \\
C^{-1}\|\xi\|_{W^{3,p}(S_{u,\ub}, (\gamma_{0,0})_\infty)} \leq \|\xi\|_{W^{3,p}(S_{u,\ub}, \gamma_n)} \leq C\|\xi\|_{W^{3,p}(S_{u,\ub}, (\gamma_{0,0})_\infty)}, &\quad 1\leq p\leq 2.
\end{align*}
\end{proposition}
\begin{proof}
\pfstep{Step~1: Proof of \eqref{eq:xi.compare}} Given $\xi$ on $S_{u,\ub}$, extend $\xi$ to a $S$-tangent tensor on $[0,u_*]\times [0,\ub_*]\times \mathbb S^2$ (still denoted by $\xi$) by stipulating that 
\begin{equation}\label{eq:u.and.ub.of.xi}
\slashed {\mathcal L}_{\f{\rd}{\rd u}}\xi =  \slashed {\mathcal L}_{\f{\rd}{\rd \ub}} \xi = 0
\end{equation}
(possible since $[\f{\rd}{\rd u}, \f{\rd}{\rd \ub}] = 0$). 

From this (and the fact that $\slashed {\mathcal L}_{\f{\rd}{\rd u}}(\gamma_{0,0})_n =  \slashed {\mathcal L}_{\f{\rd}{\rd \ub}} (\gamma_{0,0})_n = 0$) it follows that 
\begin{equation}\label{eq:comparability.1}
\|\xi\|_{L^p(S_{u,\ub}, (\gamma_{0,0})_n)} = \|\xi\|_{L^p(S_{u',\ub'}, (\gamma_{0,0})_n)},\quad \forall (u',\ub').
\end{equation}

On the other hand, by Proposition~\ref{prop:metric.der} we have
$$\f{\rd}{\rd \ub} (\gamma_n)_{AB} = 2\Omg_n (\chi_n)_{AB},\quad \f{\rd}{\rd u} (\gamma_n)_{AB} = 2\Omg_n (\chib_n)_{AB} - (\gamma_n)_{CA}(\nab_n)_{B} b_n^C - (\gamma_n)_{CB}(\nab_n)_{A} b_n^C.$$
Therefore, using also \eqref{eq:u.and.ub.of.xi} and Proposition~\ref{prop:transport.id}, we have
\begin{equation}\label{eq:transport.compare.norms.1}
\begin{split}
&\: \f{\rd}{\rd \ub} \| \xi\|_{L^p(S_{u,\ub},\gamma_n)}^p = \f{\rd}{\rd \ub} \int_{S_{u,\ub}} |\xi|_{\gamma_n}^p \,\mathrm{dA}_{\gamma_n} \\
=&\:  - p\int_{S_{u,\ub}} |\xi|_{\gamma_n}^{p-2} [\sum_{s=1}^r \f{\Omg_n(\chi^{\sharp\sharp}_n)^{A_{s}B_{s}}\Pi_{t=1}^r (\gamma^{-1}_n)^{A_{t}B_{t}} }{(\gamma^{-1}_n)^{A_{s}B_{s}}} \xi_{A_1\dots A_r} \xi_{B_1\dots B_r}] \, \mathrm{dA}_{\gamma_n} + \int_{S_{u,\ub}} \Omg_n \trch_n |\xi|_{\gamma_n}^p \, \mathrm{dA}_{\gamma_n},
\end{split}
\end{equation}
and similarly,
\begin{equation}\label{eq:transport.compare.norms.2}
\begin{split}
&\: \f{\rd}{\rd u} \| \xi\|_{L^p(S_{u,\ub},\gamma_n)}^p = \f{\rd}{\rd u} \int_{S_{u,\ub}} |\xi|_{\gamma_n}^p \,\mathrm{dA}_{\gamma_n} \\
=&\: - p\int_{S_{u,\ub}} |\xi|_{\gamma_n}^{p-2} [\sum_{s=1}^r \f{(\Omg_n (\chib^{\sharp\sharp}_n)^{A_{s}B_{s}} - 2(\gamma_n^{-1})^{C (A_s}\nab_C b_n^{B_s)})\Pi_{t=1}^r (\gamma^{-1}_n)^{A_{t}B_{t}} }{(\gamma^{-1}_n)^{A_{s}B_{s}}} \xi_{A_1\dots A_r} \xi_{B_1\dots B_r} ] \, \mathrm{dA}_{\gamma_n} \\
&\: + \int_{S_{u,\ub}} (\Omg_n \trchb_n -\div_n b_n) |\xi|_{\gamma_n}^p \, \mathrm{dA}_{\gamma_n}.
\end{split}
\end{equation}

Now using the uniform boundedness for the terms $\|\Omg_n \chi_n\|_{L^2_{\ub} L^\i_u L^\i(S_{u,\ub},\gamma_n)}$, $\|\Omg_n \chib_n\|_{L^2_{u} L^\i_{\ub} L^\i(S_{u,\ub},\gamma_n)}$, $\|\nab_n b_n\|_{L^\i_{u} L^\i_{\ub} L^\i(S_{u,\ub},\gamma_n)}$ (by \eqref{eq:bdd.metric}, \eqref{eq:bdd.psiH} and \eqref{eq:bdd.psiHb}) and Gr\"onwall's inequality, we obtain (for all $(u,\ub)\in [0,u_*]\times [0,\ub_*]$)
\begin{equation}\label{eq:comparability.2}
C^{-1} \|\xi\|_{L^p(S_{0,0}, \gamma_n)} \leq \|\xi\|_{L^p(S_{u,\ub}, \gamma_n)} \leq C \|\xi\|_{L^p(S_{0,0}, \gamma_n)}.
\end{equation}

The estimate \eqref{eq:xi.compare} therefore follows from \eqref{eq:comparability.1}, \eqref{eq:comparability.2} and the fact that $(\gamma_{0,0})_n\restriction_{S_{0,0}} = \gamma_n\restriction_{S_{0,0}}$. 

\pfstep{Step~2: Proof of \eqref{eq:nab.xi.compare}} By \eqref{eq:bdd.metric}, Sobolev embedding, \eqref{eq:xi.compare} and the computation
\begin{equation}
(\slashed{\Gamma}_n - (\slashed{\Gamma}_{0,0})_n)^{A}_{BC} = \f 12 ((\gamma_{0,0})_{n}^{-1})^{AD}(2(\nab_n)_{(B} (\gamma_n - (\gamma_{0,0})_n)_{C)D}  - (\nab_n)_D (\gamma_n - (\gamma_{0,0})_n)_{BC}),
\end{equation} 
we have $\|\slashed{\Gamma}_n - (\slashed{\Gamma}_{0,0})_n\|_{C^0_u C^0_{\ub} C^0(S_{u,\ub}, \gamma_n)} \ls 1$. Using again \eqref{eq:xi.compare}, we then obtain 
\begin{equation}\label{eq:nab.compare.with.00}
\|(\nab)_n \xi -(\nab_{0,0})_n \xi \|_{L^p(S_{u,\ub}, (\gamma_{0,0})_n)} \ls \|\xi\|_{L^p(S_{u,\ub}, (\gamma_{0,0})_n)}. 
\end{equation}
As a result, by the triangle inequality, for any $1\leq p\leq +\infty$.
$$\|(\nab)_n\xi\|_{L^p(S_{u,\ub}, (\gamma_{0,0})_n)} \leq C( \|(\nab_{0,0})_n \xi \|_{L^p(S_{u,\ub}, (\gamma_{0,0})_n)}+ \|\xi\|_{L^p(S_{u,\ub}, (\gamma_{0,0})_n)}),$$
and 
$$\|(\nab_{0,0})_n \xi \|_{L^p(S_{u,\ub}, (\gamma_{0,0})_n)} \leq C( \|(\nab)_n\xi\|_{L^p(S_{u,\ub}, (\gamma_{0,0})_n)} + \|\xi\|_{L^p(S_{u,\ub}, (\gamma_{0,0})_n)}).$$
The estimate \eqref{eq:nab.xi.compare} then follows from \eqref{eq:xi.compare}.

\pfstep{Step~3: Proof of \eqref{eq:nab.2.xi.compare} and \eqref{eq:nab.3.xi.compare}} The proof of \eqref{eq:nab.2.xi.compare} and \eqref{eq:nab.3.xi.compare} is similar to that of  \eqref{eq:nab.xi.compare}. The only difference is that by \eqref{eq:bdd.metric} (and Sobolev embedding using Proposition~\ref{prop:Sobolev} and \eqref{eq:bdd.isoperimetric}), we only have the estimates $\|\nab_n (\slashed{\Gamma}_n - (\slashed{\Gamma}_{0,0})_n)\|_{C^0_u C^0_{\ub} L^4(S_{u,\ub}, \gamma_n)} \ls 1$ and $\|\nab_n^2(\slashed{\Gamma}_n - (\slashed{\Gamma}_{0,0})_n)\|_{C^0_u C^0_{\ub} L^2(S_{u,\ub}, \gamma_n)} \ls 1$, which restrict the range of allowable $p$. \qedhere
\end{proof}

\subsection{Limit of $\gamma$ and its angular derivatives}\label{sec:limit.gamma}

\begin{proposition}\label{prop:gamma}
There exists a subsequence $n_k$ and a limiting metric $\gamma_\infty \in C^0_u C^0_{\ub} W^{2,4}(S_{u,\ub}, (\gamma_{0,0})_\i) \cap L^\i_u L^\i_{\ub} W^{3,2}(S_{u,\ub}, (\gamma_{0,0})_\i)$ such that
\begin{equation}\label{eq:gamma.convergence.1}
\|\gamma_{n_k} - \gamma_\infty\|_{C^0_u C^0_{\ub} C^0(S_{u,\ub}, (\gamma_{0,0})_\infty)} + \|(\nab_{0,0})_\infty(\gamma_{n_k} - \gamma_\infty)\|_{C^0_u C^0_{\ub} C^0(S_{u,\ub}, (\gamma_{0,0})_\infty)} \to 0.
\end{equation}
\begin{equation}\label{eq:gamma.convergence.2}
\|(\nab_{0,0})_\infty^2(\gamma_{n_k} - \gamma_\infty)\|_{C^0_u C^0_{\ub} L^4(S_{u,\ub}, (\gamma_{0,0})_\infty)}  \to 0.\end{equation}
Moreover, $\gamma_\infty$ also satisfies
\begin{equation}\label{eq:gamma.bdd.2}
\|\slashed{\mathcal L}_{\f{\rd}{\rd\ub}} \gamma_\infty\|_{L^\infty_u L^2_{\ub} W^{2,4}(S_{u,\ub}, (\gamma_{0,0})_\infty)} + \|\slashed{\mathcal L}_{\f{\rd}{\rd u}} \gamma_\infty\|_{L^\infty_{\ub} L^2_{u} W^{2,4}(S_{u,\ub}, (\gamma_{0,0})_\infty)} \ls 1.
\end{equation}
\end{proposition}
\begin{proof}
By \eqref{eq:bdd.metric}, Sobolev embedding (Proposition~\ref{prop:Sobolev}) and Proposition~\ref{prop:norms.compare}, we have
\begin{equation*}
\begin{split}
&\: \sup_n \left( \|\gamma_n - (\gamma_{0,0})_n\|_{C^0_u C^0_{\ub} W^{2,4}(S_{u,\ub}, (\gamma_{0,0})_\infty)} + \|\gamma_n - (\gamma_{0,0})_n\|_{C^0_u C^0_{\ub} W^{3,2}(S_{u,\ub}, (\gamma_{0,0})_\infty)} \right) \\
&\: + \sup_n \left( \|\slashed {\mathcal L}_{\f{\rd}{\rd\ub}} (\gamma_n - (\gamma_{0,0})_n)\|_{C^0_u L^2_{\ub} W^{3,2}(S_{u,\ub}, (\gamma_{0,0})_\infty)} + \|\slashed {\mathcal L}_{\f{\rd}{\rd u}} (\gamma_n - (\gamma_{0,0})_n)\|_{C^0_{\ub} L^2_u W^{3,2}(S_{u,\ub}, (\gamma_{0,0})_\infty)} \right) \ls 1.
\end{split}
\end{equation*}
Therefore, by Proposition~\ref{prop:compact.embeddings} (with $\mathring{\gamma} = (\gamma_{0,0})_\i$), there exists $\gamma_\infty$ such that 
$$\gamma_{n_k} - (\gamma_{0,0})_{n_k} \to \gamma_\infty - (\gamma_{0,0})_\infty \mbox{ in $C^0_u C^0_{\ub} W^{2,4}(S_{u,\ub}, (\gamma_{0,0})_\infty)$}.$$
This implies \eqref{eq:gamma.convergence.1} and \eqref{eq:gamma.convergence.2} (using Sobolev embedding). Moreover, by Proposition~\ref{prop:compact.embeddings}, $\gamma_\infty - (\gamma_{0,0})_\infty \in L^\i_u L^\i_{\ub} W^{3,2}(S_{u,\ub}, (\gamma_{0,0})_\i)$, which implies $\gamma_\infty \in L^\i_u L^\i_{\ub} W^{3,2}(S_{u,\ub}, (\gamma_{0,0})_\i)$. Finally, Proposition~\ref{prop:compact.embeddings} gives that $\slashed{\mathcal L}_{\f{\rd}{\rd\ub}} ( \gamma_\infty - (\gamma_{0,0})_\infty ) \in L^\infty_u L^2_{\ub} W^{2,4}(S_{u,\ub}, (\gamma_{0,0})_\infty)$ and $\slashed{\mathcal L}_{\f{\rd}{\rd u}} ( \gamma_\infty - (\gamma_{0,0})_\infty ) \in L^\infty_{\ub} L^2_{u} W^{2,4}(S_{u,\ub}, (\gamma_{0,0})_\infty)$, which imply \eqref{eq:gamma.bdd.2} since $\slashed{\mathcal L}_{\f{\rd}{\rd\ub}} (\gamma_{0,0})_\infty = \slashed{\mathcal L}_{\f{\rd}{\rd u}} (\gamma_{0,0})_\infty =0$. \qedhere
\end{proof}

One immediate consequence of Proposition~\ref{prop:gamma} is the uniform bound of the isoperimetric constants and the area of each of the $2$-sphere $S_{u,\ub}$ with respect to the limiting metric $\gamma_\infty$:
\begin{proposition}\label{prop:isoperimetric}
$$\sup_{u,\ub} {\bf I}(S_{u,\ub},\gamma_\infty) \ls 1,$$
$$1\ls \inf_{u,\ub} \mathrm{Area}(S_{u,\ub},\gamma_\infty) \leq \sup_{u,\ub} \mathrm{Area}(S_{u,\ub},\gamma_\infty) \ls 1,$$
and 
$$\|\log \f{\det\gamma_\i}{\det(\gamma_{0,0})_\i} \|_{C^0_u C^0_{\ub} C^0(S_{u,\ub})} \ls 1.$$
\end{proposition}
\begin{proof}
By the $C^0$ convergence statement \eqref{eq:gamma.convergence.1} in Proposition~\ref{prop:gamma}, it follows that for every $(u,\ub)$, 
$${\bf I}(S_{u,\ub},\gamma_\infty) \leq \liminf_{k\to+\infty} {\bf I}(S_{u,\ub},\gamma_{n_k}),$$
$$\limsup_{k\to +\infty} \mathrm{Area}(S_{u,\ub}, \gamma_{n_k}) \leq \mathrm{Area}(S_{u,\ub}, \gamma_\infty) \leq \liminf_{k\to +\infty} \mathrm{Area}(S_{u,\ub}, \gamma_{n_k}),$$ 
and (using also \eqref{eq:gamma.C3.conv})
$$\|\log \f{\det\gamma_\i}{\det(\gamma_{0,0})_\i} \|_{C^0_u C^0_{\ub} C^0(S_{u,\ub})} \leq \limsup_{k\to +\infty} \|\log \f{\det\gamma_{n_k}}{\det(\gamma_{0,0})_{n_k}} \|_{C^0_u C^0_{\ub} C^0(S_{u,\ub})}.$$

The desired conclusions then follow from \eqref{eq:bdd.isoperimetric}. \qedhere
\end{proof}

Another immediate consequence of Proposition~\ref{prop:gamma} is the following estimates for the angular connections:
\begin{proposition}\label{prop:Christoffel}
The following hold for $\slashed \Gamma_\infty$ being the Christoffel symbols associated to $\gamma_\infty$:
\begin{equation}\label{eq:Christoffel.1}
\|\slashed \Gamma_{n_k} - \slashed \Gamma_{\infty}\|_{C^0_u C^0_{\ub} C^0(S_{u,\ub}, (\gamma_{0,0})_\i)} + \|\slashed \Gamma_{n_k} - \slashed \Gamma_{\infty}\|_{C^0_u C^0_{\ub} W^{1,4}(S_{u,\ub}, (\gamma_{0,0})_\i)}\to 0,
\end{equation}
\begin{equation}\label{eq:Christoffel.2}
\|\slashed \Gamma_\infty - (\slashed \Gamma_{0,0})_\infty\|_{L^\i_u L^\i_{\ub} W^{2,2}(S_{u,\ub}, (\gamma_{0,0})_\i)}\ls 1.
\end{equation}

Moreover, $K_\infty$ (the Gauss curvature of $(S_{u,\ub},\gamma_\infty)$) is a well-defined $C^0_u C^0_{\ub} L^4(S_{u,\ub}, (\gamma_{0,0})_\i)$ function which satisfies
\begin{equation}\label{eq:K.first.limit}
\|K_{n_k} - K_\infty\|_{C^0_u C^0_{\ub} L^4(S_{u,\ub}, (\gamma_{0,0})_\i)} \to 0,\quad \|K_\infty\|_{L^\i_u L^\i_{\ub} W^{1,2}(S_{u,\ub}, (\gamma_{0,0})_\i)}\ls 1.
\end{equation}
\end{proposition}
\begin{proof}
The estimates \eqref{eq:Christoffel.1} and \eqref{eq:Christoffel.2} follow from Proposition~\ref{prop:metric} and the fact
$$\Gamma_{n_k} - \Gamma_\i = \f 12 (\gamma_{\i}^{-1})^{AD} (2(\nab_{n_k})_{(B} (\gamma_{n_k} - \gamma_\i)_{C)D}  - (\nab_{n_k})_D (\gamma_{n_k} - \gamma_\i)_{BC}).$$

The statements about the Gauss curvature in \eqref{eq:K.first.limit} follow immediate from \eqref{eq:Christoffel.1}, \eqref{eq:Christoffel.2} and \eqref{Gauss.def}. \qedhere
\end{proof}

Given Propositions~\ref{prop:metric} and \ref{prop:Christoffel}, we have the following equivalence of norms:
\begin{lemma}\label{lem:equivalent}
Let $0\leq m \leq 3$ be an integer and $p\in [1,p_m]$, where $p_m = \begin{cases}
+\infty & m=0,1 \\
4 & m=2 \\
2 & m=3
\end{cases}$. Then all of the following norms are equivalent (with constants independent of $k$):
$$W^{m,p}(S_{u,\ub}, \gamma_{n_k}),\,W^{m,p}(S_{u,\ub}, \gamma_\infty),\,W^{m,p}(S_{u,\ub}, (\gamma_{0,0})_{n_k}),\,W^{m,p}(S_{u,\ub}, (\gamma_{0,0})_\infty).$$
\end{lemma}
\begin{proof}
The equivalence of $W^{m,p}(S_{u,\ub}, \gamma_{n_k})$, $W^{m,p}(S_{u,\ub}, (\gamma_{0,0})_{n_k})$ and $W^{m,p}(S_{u,\ub}, (\gamma_{0,0})_\infty)$ has been proven in Proposition~\ref{prop:norms.compare}. That $W^{m,p}(S_{u,\ub}, \gamma_{n_k})$ and $W^{m,p}(S_{u,\ub}, \gamma_\infty)$ are equivalent is a consequence of Propositions~\ref{prop:metric} and \ref{prop:Christoffel}. \qedhere
\end{proof}

In view of Lemma~\ref{lem:equivalent}, \textbf{from now on, we will write $L^p(S_{u,\ub})$, $W^{1,p}(S_{u,\ub})$, etc.~without specifying the metric with respect to which the norms are defined.}

Recall that in the definition of angular regularity (Definition~\ref{double.null.def.2}), we need second derivative estimates for $K_\infty$, which does not follow from the estimates for $\gamma_\infty$. We derive them in the following proposition:
\begin{proposition}\label{prop:K.imp}
The limit $K_\infty$ from Proposition~\ref{prop:Christoffel} satisfies
$$ K_\infty  \in L^\i_u L^2_{\ub} W^{2,2}(S_{u,\ub}) \cap L^\i_{\ub} L^2_{u} W^{2,2}(S_{u,\ub}).$$
\end{proposition}
\begin{proof}
We will only prove the $L^\i_u L^2_{\ub} W^{2,2}(S_{u,\ub})$ estimate; the $L^\i_{\ub} L^2_{u} W^{2,2}(S_{u,\ub})$ bound can be treated in a completely identical manner after switching $u$ and $\ub$.

By \eqref{eq:bdd.psi.top},
$$\|K_{n_k} \|_{L^\i_u L^2_{\ub} W^{2,2}(S_{u,\ub})} \ls 1.$$
It follows from the Banach--Alaoglu theorem that for every $u \in [0,, u_*)$, there exists a further subsequence of $\nab_{n_k}^2 K_{n_k}$ which admits a weak $L^2_{\ub} L^2(S_{u,\ub})$ limit $\psi_\infty$ satisfying the estimate
\begin{equation}\label{eq:K.imp.est}
\|\psi_\infty \|_{L^2_{\ub} L^2(S_{u,\ub})} \ls 1.
\end{equation}
The weak convergence (together with Proposition~\ref{prop:Christoffel}) implies that $\psi_\infty = \nab_\infty^2 K$. Since \eqref{eq:K.imp.est} moreover holds independently of $u$, this proves that $K_\infty \in L^\i_u L^2_{\ub} W^{2,2}(S_{u,\ub})$. \qedhere
\end{proof}

Finally, before we end this subsection, it will be convenient to use the equivalence of norms that we have established above to rephrase the compactness theorems:

\begin{lemma}\label{lem:easy.convergence}
Proposition~\ref{prop:compact.embeddings}, \ref{prop:BV}, \ref{prop:weak} and \ref{prop:weak.L2} all apply in the setting of this section with $W^{m+1,q}(S_{u,\ub},\mathring{\gamma})$, $W^{m,p}(S_{u,\ub},\mathring{\gamma})$, etc.~replaced by $W^{m+1,q}(S_{u,\ub})$, $W^{m,p}(S_{u,\ub})$, etc.
\end{lemma}
\begin{proof}
This is an immediate consequence of Proposition~\ref{prop:compact.embeddings} and Lemma~\ref{lem:equivalent}. \qedhere
\end{proof}

\subsection{Limits of $b$ and $\log\Om$}\label{sec:limit.metric}

\begin{proposition}\label{prop:metric.limit}
There exist $b_\infty$ and $\Omg_\infty$ such that the following hold after passing to a further subsequence $n_k$:
$$\|b_{n_k} - b_\infty\|_{C^0_u C^0_{\ub} C^1(S_{u,\ub})} + \|b_{n_k} - b_\infty\|_{C^0_u C^0_{\ub} W^{2,4}(S_{u,\ub})} \to 0,$$
$$\|\log \f{\Om_{n_k}}{\Om_\infty}\|_{C^0_u C^0_{\ub} C^1(S_{u,\ub})} + \|\log \f{\Om_{n_k}}{\Om_\infty}\|_{C^0_u C^0_{\ub} W^{2,4}(S_{u,\ub})} \to 0,$$

Moreover, for $\slashed g_\infty \in \{b_\infty,\,\log\Om_\infty\}$,
$$\|\slashed g_\infty\|_{L^\infty_u L^\infty_{\ub} W^{3,2}(S_{u,\ub})} + \|\slashed{\mathcal L}_{\f{\rd}{\rd\ub}} \slashed g_\infty\|_{L^\infty_u L^2_{\ub} W^{2,4}(S_{u,\ub})} + \|\slashed{\mathcal L}_{\f{\rd}{\rd u}} \slashed g_\infty\|_{L^\infty_{\ub} L^2_{u} W^{2,4}(S_{u,\ub})} \ls 1.$$
\end{proposition}
\begin{proof}
This is an immediate consequence of the bounds \eqref{eq:bdd.metric} and Proposition~\ref{prop:compact.embeddings} (and Lemma~\ref{lem:easy.convergence}). \qedhere
\end{proof}

\begin{remark}
In particular, combining Propositions~\ref{prop:gamma} and \ref{prop:metric.limit}, it follows that $(\mathcal M, g_{n_k})$ has a $C^0$ limit $(\mathcal M, g_\infty)$ given by 
$$g_\infty = -2\Omega_\infty^2(du\otimes d\ub+d\ub\otimes du)+(\gamma_\infty)_{AB}(d\th^A-b_\infty^Adu)\otimes (d\th^B-b_\infty^B du).$$
\end{remark}

\subsection{Limits of $\eta$ and $\protect\underline{\eta}$}\label{sec:limit.eta}

We next consider the limits of $\eta$ and $\etab$. They in particular have uniform limits. More precisely,

\begin{proposition}\label{prop:eta.etab.limit}
There exist $\eta_\infty$ and $\etab_\infty$ such that the following hold after passing to a further subsequence $n_k$:
$$\|\eta_{n_k} - \eta_\infty\|_{C^0_u C^0_{\ub} C^0(S_{u,\ub})} + \|\eta_{n_k} - \eta_\infty\|_{C^0_u C^0_{\ub} W^{1,4}(S_{u,\ub})} \to 0,$$
$$\|\etab_{n_k} - \etab_\infty\|_{C^0_u C^0_{\ub} C^0(S_{u,\ub})} + \|\etab_{n_k} - \etab_\infty\|_{C^0_u C^0_{\ub} W^{1,4}(S_{u,\ub})} \to 0.$$

Moreover, for $\psi_\infty \in \{\eta_\infty,\,\etab_\infty\}$
$$\|\psi_\infty\|_{L^\i_u L^\i_{\ub} W^{2,2}(S_{u,\ub})}  + \|\slashed{\mathcal L}_{\f{\rd}{\rd\ub}} \psi_\infty\|_{L^\infty_u L^2_{\ub} W^{1,4}(S_{u,\ub})} + \|\slashed{\mathcal L}_{\f{\rd}{\rd u}} \psi_\infty\|_{L^\infty_{\ub} L^2_{u} W^{1,4}(S_{u,\ub})} \ls 1.$$
\end{proposition}
\begin{proof}
This is an immediate consequence of the bounds \eqref{eq:bdd.psi}, Sobolev embedding (Proposition~\ref{prop:Sobolev}) and Proposition~\ref{prop:compact.embeddings} (and Lemma~\ref{lem:easy.convergence}). \qedhere
\end{proof}

\begin{proposition}\label{prop:eta.etab.imp}
For $\psi_\infty \in \{\eta_\infty,\,\etab_\infty\}$, where $\eta_\infty$ and $\etab_\infty$ are as in Proposition~\ref{prop:eta.etab.limit}, it holds that
$$\psi_\i \in L^\i_u L^\i_{\ub} W^{2,2}(S_{u,\ub}) \cap L^\i_u L^2_{\ub} W^{3,2}(S_{u,\ub}) \cap L^\i_{\ub} L^2_{u} W^{3,2}(S_{u,\ub}).$$
\end{proposition}
\begin{proof}
\pfstep{Step~1: The $W^{2,2}$ estimate} By \eqref{eq:bdd.psi}, (for $\psi_{n_k} \in \{\eta_{n_k},\,\etab_{n_k}\}$,) $\psi_{n_k} \in W^{2,2}(S_{u,\ub})$ uniformly in $k$, $u$ and $\ub$, it follows that for every $u$, $\ub$, after passing to a further subsequence, $\psi_{n_k}$ converges so some limit in $W^{2,2}(S_{u,\ub})$. It is easy to check (using \eqref{prop:eta.etab.limit} and a density argument) that this limit coincides with $\psi_{\i}$ almost everywhere. It follows that $ \psi_\infty \in L^\i_u L^\i_{\ub} W^{2,2}(S_{u,\ub})$.

\pfstep{Step~2: The $W^{3,2}$ estimate} The proof is similar to that in Proposition~\ref{prop:K.imp}; we omit the details. \qedhere
\end{proof}

\subsection{Weak limits of $\chih$, $\om$, $\protect\trch$, $\protect\chibh$, $\protect\omb$ and $\protect\trchb$}\label{sec:limit.chi}

In this subsection, we now discuss the \underline{weak} limits of the Ricci coefficients $\chih$, $\om$, $\protect\trch$, $\protect\chibh$, $\protect\omb$ and $\protect\trchb$. The Ricci coefficients $\chih$, $\om$, $\chibh$ and $\omb$ indeed \emph{only} admit weak limits. This is related to the fact that they have the weakest regularity estimates. It is also the reason for which the limit spacetime $(\mathcal M, g_\infty)$ is not necessarily vacuum. On the other hand, $\trch$ and $\trchb$ in fact have stronger convergence properties than those proven in Proposition~\ref{prop:trch.weak.limit}. We will return to this in Section~\ref{sec:limit.trch}.

\begin{proposition}\label{prop:chi.limit}
There exist a further subsequence $n_k$ and $S$-tangent tensor fields $\chih_\infty$, $\om_\infty$, $\trch_\infty$, $\chibh_\infty$, $\omb_\infty$ and $\trchb_\i$ such that for every $u\in [0,u_*]$,
$$\nab_{n_k}^i\chih_{n_k} \rightharpoonup \nab_\infty^i\chih_{\infty},\quad \nab_{n_k}^i\om_{n_k}\rightharpoonup \nab_\infty^i\om_{\infty},\quad \nab_{n_k}^i\trch_{n_k}\rightharpoonup \nab_\infty^i\trch_{\infty}$$
\underline{weakly} in $L^2_{\ub} L^2(S_{u,\ub})$ for $i=0,1,2$, and for every $\ub\in [0,\ub_*]$,
$$\nab_{n_k}^i\chibh_{n_k} \rightharpoonup \nab_\infty^i\chibh_{\infty},\quad \nab_{n_k}^i\omb_{n_k}\rightharpoonup \nab_\infty^i\omb_{\infty},\quad \nab_{n_k}^i\trchb_{n_k}\rightharpoonup \nab_\infty^i\trchb_{\infty}$$
\underline{weakly} in $L^2_{u} L^2(S_{u,\ub})$ for $i=0,1,2$.

Moreover, the limits satisfy $\chih_\infty,\,\om_\infty,\,\trch_\i \in L^2_{\ub} L^\i_u W^{2,2}(S_{u,\ub})\cap C^0_u L^2_{\ub} W^{2,2}(S_{u,\ub})\cap L^\i_u L^2_{\ub} W^{3,2}(S_{u,\ub})$, $\slashed{\mathcal L}_{\f{\rd}{\rd u}} \chih_\infty,\,\slashed{\mathcal L}_{\f{\rd}{\rd u}} \om_\infty,\,\slashed{\mathcal L}_{\f{\rd}{\rd u}} \trch_\infty \in L^2_u L^2_{\ub} W^{2,2}(S_{u,\ub})$;
and similarly $\chibh_\infty,\,\omb_\infty,\,\trchb_\i \in L^2_{u} L^\i_{\ub} W^{2,2}(S_{u,\ub})\cap C^0_{\ub} L^2_{u} W^{2,2}(S_{u,\ub})\cap L^\i_{\ub} L^2_{u} W^{3,2}(S_{u,\ub})$, $\slashed{\mathcal L}_{\f{\rd}{\rd \ub}} \chibh_\infty,\,\slashed{\mathcal L}_{\f{\rd}{\rd \ub}} \omb_\infty,\,\slashed{\mathcal L}_{\f{\rd}{\rd \ub}} \trchb_\infty \in L^2_u L^2_{\ub} W^{2,2}(S_{u,\ub})$.
\end{proposition}
\begin{proof}
We will only discuss the theorem for $\chih$. It is easy to see that $\om$ and $\trch$ can be treated in exactly the same way (since $\om$ satisfies similar estimates, and $\trch$ satisfies even stronger bounds); while $\chibh$, $\omb$ and $\trchb$ can be handled similarly after changing $u$ and $\ub$.

\pfstep{Step~1: Existence of weak limit} By Proposition~\ref{prop:weak.L2} with $q=2$ (and Proposition~\ref{prop:weak}, Lemma~\ref{lem:easy.convergence}) and the estimates in \eqref{eq:bdd.psiH} and \eqref{eq:bdd.psiH.psiHb.trans}, for $i=0,1,2$, after passing to a subsequence $n_k$, there exists $\chih_\infty^{(i)}\in L^2_{\ub} L^\i_u L^{2}(S_{u,\ub})\cap C^0_u L^2_{\ub} L^{2}(S_{u,\ub})$ such that  $\nab_{n_k}^i \chih_{n_k} \rightharpoonup \chih_\infty^{(i)}$ weakly in $L^2_{\ub} L^2(S_{u,\ub})$ for every $u$. By Proposition~\ref{prop:Christoffel} and the uniqueness of distribution limits it follows that $\chih_{\infty}^{(i)} = \nab_\i^i \chih_{\infty}$. This shows that $\chih_\infty \in L^2_{\ub} L^\i_u W^{2,2}(S_{u,\ub})\cap C^0_u L^2_{\ub} W^{2,2}(S_{u,\ub})$.

\pfstep{Step~2: Higher regularity of $\chih_\infty$ and $\slashed{\mathcal L}_{\f{\rd}{\rd u}} \chih_\infty$} Step~1 in particular shows that $\chih_\infty \in L^2_{\ub} L^\i_u W^{2,2}(S_{u,\ub})\cap C^0_u L^2_{\ub} W^{2,2}(S_{u,\ub})$. By Proposition~\ref{prop:weak.L2}, we also have  $\slashed{\mathcal L}_{\f{\rd}{\rd u}} \chih_\infty \in L^2_u L^2_{\ub} W^{2,2}(S_{u,\ub})$.

It therefore remains to show that $\chih_\infty \in L^\i_u L^2_{\ub} W^{3,2}(S_{u,\ub})$. To see this, note that for every $u\in [0,u_*]$, the estimate \eqref{eq:bdd.psiH} and the Banach--Alaoglu theorem imply that after passing to a further subsequence, $\nab^3_{n_k}\chih_{n_k}$ converges weakly in $L^2_{\ub} W^{3,2}(S_{u,\ub})$ to a limit $\chih_\infty^{(3)}$ satisfying the bound (independently of $u$)
\begin{equation}\label{eq:chi.limit.top}
\chi_\infty^{(3)} \in L^2_{\ub} L^2(S_{u,\ub}).
\end{equation}
The weak convergence implies that that $\chih_\infty^{(3)} = \nab_\infty^3 \chih_\infty$. Thus \eqref{eq:chi.limit.top} and the fact (established above) that $\chih_\infty \in L^2_{\ub} L^\i_u W^{2,2}(S_{u,\ub})\cap C^0_u L^2_{\ub} W^{2,2}(S_{u,\ub})$ imply $\chi_\infty\in L^\i_u L^2_{\ub} W^{3,2}(S_{u,\ub})$. \qedhere

\end{proof}

\begin{proposition}\label{prop:trch.weak.limit}
Let $q\in [2,+\infty)$. There exist a further subsequence $n_k$ and functions $\trch_\infty$ and $\trchb_\i$ such that for every $u\in [0,u_*]$,
$$\nab_{n_k}^i\trch_{n_k}\rightharpoonup \nab_\infty^i\trch_{\infty}$$
\underline{weakly} in $L^q_{\ub} L^2(S_{u,\ub})$ for $i=0,1,2$, and for every $\ub\in [0,\ub_*]$,
$$\nab_{n_k}^i\trchb_{n_k}\rightharpoonup \nab_\infty^i\trchb_{\infty}$$
\underline{weakly} in $L^q_{u} L^2(S_{u,\ub})$ for $i=0,1,2$.

Moreover, $\trch_\i \in L^q_{\ub} L^\i_u W^{2,2}(S_{u,\ub})\cap C^0_u L^q_{\ub} W^{2,2}(S_{u,\ub})\cap L^\i_u L^q_{\ub} W^{3,2}(S_{u,\ub})$, $\slashed{\mathcal L}_{\f{\rd}{\rd u}} \trch_\infty \in L^2_u L^q_{\ub} W^{2,2}(S_{u,\ub})$;
and similarly $\trchb_\i \in L^q_{u} L^\i_{\ub} W^{2,2}(S_{u,\ub})\cap C^0_{\ub} L^q_{u} W^{2,2}(S_{u,\ub})\cap L^\i_{\ub} L^q_{u} W^{3,2}(S_{u,\ub})$, $\slashed{\mathcal L}_{\f{\rd}{\rd \ub}} \trchb_\infty \in L^2_{\ub} L^q_{u} W^{2,2}(S_{u,\ub})$.
\end{proposition}
\begin{proof}
This can be proven in exactly the same way as Proposition~\ref{prop:chi.limit}, except for using the better bounds that $\trch$ and $\trchb$ obey (\eqref{eq:bdd.psi}, \eqref{eq:bdd.trch.trans} and \eqref{eq:bdd.trchb.trans}), and applying Proposition~\ref{prop:weak.L2} with a general $q$ (as opposed to only $q=2$). \qedhere
\end{proof}

\subsection{Strong limits of $\protect\trch$ and $\protect\trchb$}\label{sec:limit.trch}

In this subsection we further show \emph{strong} convergence for $\trch$ and $\trchb$ (in addition to Proposition~\ref{prop:trch.weak.limit} in Section~\ref{sec:limit.chi}); see the main results in Proposition~\ref{prop:trch.imp}. For this we rely on compactness of BV. First, we need the following

\begin{proposition}[Comparability of the BV norms]\label{prop:BV.compare}
There exists $C>0$ independent of $n$ such that the following holds for all continuous $\phi: [0,u_*]\times [0,\ub_*]\times \mathbb S^2\to \mathbb R$:
$$C^{-1} \|\phi\|_{BV(H_u,(\gamma_\infty)_{0,0})} \leq \|\phi\|_{BV(H_u, \gamma_n)} \leq C \|\phi\|_{BV(H_u, (\gamma_\infty)_{0,0})},$$
and 
$$C^{-1} \|\phi\|_{BV(\Hb_{\ub}, (\gamma_\infty)_{0,0})} \leq \|\phi\|_{BV(\Hb_{\ub}, \gamma_n)} \leq C \|\phi\|_{BV(\Hb_{\ub}, (\gamma_\infty)_{0,0})}.$$
\end{proposition}
\begin{proof}
After recalling Definition~\ref{def:BV}, this is immediate from Proposition~\ref{prop:norms.compare} and \eqref{eq:nab.compare.with.00}. \qedhere
\end{proof}

In view of Proposition~\ref{prop:BV.compare}, \textbf{from now on we will simply write $BV(H_u)$ and $BV(\Hb_{\ub})$ for either of the equivalent norms.}

\begin{proposition}\label{prop:trch.imp}
The functions $\trch_\infty$ and $\trchb_\infty$ in Proposition~\ref{prop:trch.weak.limit} satisfy in addition
\begin{equation}\label{eq:trch.imp}
\begin{split}
&\: \trch_\infty \in C^0_u L^1_{\ub} W^{2,1}(S_{u,\ub}) \cap L^\i_u BV(H_u) \cap L^\i_u L^\i_{\ub} W^{3,2}(S_{u,\ub}),\\
&\: \trchb_\infty \in C^0_{\ub} L^1_{\ub} W^{2,1}(S_{u,\ub}) \cap L^\i_{\ub} BV(\Hb_{\ub}) \cap L^\i_{\ub} L^\i_{u} W^{3,2}(S_{u,\ub}).
\end{split}
\end{equation}
Moreover, after passing to a further subsequence $n_k$, 
\begin{equation}\label{eq:trch.W21.statement}
\lim_{k\to +\infty} (\|\trch_{n_k} -\trch_\infty \|_{C^0_u L^1_{\ub} W^{2,1}(S_{u,\ub})} + \| \trchb_{n_k} - \trchb_\infty\|_{C^0_{\ub} L^1_u W^{2,1}(S_{u,\ub})}) = 0.
\end{equation}
and for every $p\in [1,+\infty)$, 
\begin{equation}\label{eq:trch.imp.conv}
\lim_{k\to +\infty} (\|\trch_{n_k} - \trch_\infty\|_{C^0_u L^p_{\ub} W^{1,p}(S_{u,\ub})},\,\|\trchb_{n_k} - \trchb_\infty\|_{C^0_{\ub} L^p_{u} W^{1,p}(S_{u,\ub})} ) = 0.
\end{equation}
\end{proposition}
\begin{proof}
In this proof, we are concerned with $\trch$; the proofs for statements regarding $\trchb$ are similar.

\pfstep{Step~1: BV compactness}

\pfstep{Step~1(a): $C^0_u L^1_{\ub} L^1(S_{u,\ub})\cap L^\i_u BV(H_u)$ estimates} The bounds \eqref{eq:bdd.psi}, \eqref{eq:bdd.trch.trans}, \eqref{eq:bdd.trchb.trans} give uniform estimates for $\trch_{n_k}$ which are sufficient to apply Proposition~\ref{prop:BV}. Thus, by Proposition~\ref{prop:BV.compare}, the BV compactness theorem in Proposition~\ref{prop:BV} (and Lemma~\ref{lem:easy.convergence}) and the uniqueness of (distributional) limits, we have, after passing to a further subsequence,
\begin{equation}\label{eq:trch.limit.low}
\lim_{k\to +\infty} (\|\trch_{n_k} -\trch_\infty \|_{C^0_u L^1_{\ub} L^1(S_{u,\ub})} + \| \trchb_{n_k} - \trchb_\infty\|_{C^0_{\ub} L^1_u L^1(S_{u,\ub})}) = 0.
\end{equation}
and that $\trch_{\i} \in C^0_u L^1_{\ub} L^1(S_{u,\ub})\cap L^\i_u BV(H_u)$.

\pfstep{Step~1(b): $C^0_u L^1_{\ub} W^{2,1}(S_{u,\ub})$ estimates and proof of \eqref{eq:trch.W21.statement}}  We now apply the same argument as in Step~1(a) but to higher derivatives. 

By \eqref{eq:bdd.psi}, \eqref{eq:bdd.trch.trans} and Proposition~\ref{lem:equivalent}, we have that $(\nab_{0,0})_\infty \trch_{n_k}$ and $(\nab_{0,0})_\infty^2 \trch_{n_k}$ are uniformly bounded in $BV(H_u)$ for all $u$. Using Propositions~\ref{prop:BV.compare} and \ref{prop:BV} (and Lemma~\ref{lem:easy.convergence}), it follows that after passing to a further subsequence, $(\nab_{0,0})_\infty \trch_{n_k}$ and $(\nab_{0,0})_\infty^2 \trch_{n_k}$ both converge in $C^0_u L^1_{\ub} L^1(S_{u,\ub})$ to some limits. It is easy to check that these limits coincide with $(\nab_{0,0})_\infty \trch_\infty$ and $(\nab_{0,0})_\infty \trch_\infty$, which proves \eqref{eq:trch.W21.statement}.

\pfstep{Step~2: Completion of the proof of \eqref{eq:trch.imp}} The only estimate not already established in Step~1 is that $\trch_\infty \in L^\i_u L^\i_{\ub} W^{3,2}(S_{u,\ub})$. This can be proven in a similar manner as Step~1 in the proof of Proposition~\ref{prop:eta.etab.imp}; we omit the details.

\pfstep{Step~3: Proof of \eqref{eq:trch.imp.conv}} Let $q\in [2,+\infty)$. By Proposition~\ref{prop:trch.weak.limit}, $\trch_\infty\in C^0_u L^q_{\ub} W^{2,2}(S_{u,\ub})$. Hence by Sobolev embedding (using Propositions~\ref{prop:Sobolev} and \ref{prop:isoperimetric}), $\trch_\infty\in C^0_u L^q_{\ub} W^{1,q}(S_{u,\ub})$. Combining with the estimate \eqref{eq:bdd.psi}, it follows that 
\begin{equation}\label{eq:trch.W2q}
\sup_k \|\trch_{n_k} - \trch_\infty\|_{C^0_u L^q_{\ub} W^{1,q}(S_{u,\ub})} \ls 1.
\end{equation}
By H\"older's inequality, for any $p \in [1,+\infty)$, $q \in [2,+\infty)$ with $p <q$. 
$$\|\trch_{n_k} - \trch_\infty\|_{C^0_u L^p_{\ub} W^{1,p}(S_{u,\ub})} \ls \|\trch_{n_k} - \trch_\infty\|_{C^0_u L^1_{\ub} W^{1,1}(S_{u,\ub})}^{(\f 1p- \f 1q) \f{q}{q-1}} \|\trch_{n_k} - \trch_\infty\|_{C^0_u L^q_{\ub} W^{1,q}(S_{u,\ub})}^{1 - (\f 1p- \f 1q) \f{q}{q-1}}.$$
Given any $p\in [1,+\infty)$, we can choose $q\in [2,+\infty)$ with $p<q$ so that \eqref{eq:trch.W21.statement} and \eqref{eq:trch.W2q} imply $\|\trch_{n_k} - \trch_\infty\|_{C^0_u L^p_{\ub} W^{1,p}(S_{u,\ub})}\to 0$. We have thus obtained \eqref{eq:trch.imp.conv}. \qedhere

\end{proof}

\begin{remark}
Notice that while $\trch_\infty$ and $\trchb_{\infty}$ are a.e.~bounded (by \eqref{eq:trch.imp} and Sobolev embedding), they are not necessarily continuous. Nonetheless, they have well-defined traces in the sense of Lemma~\ref{lem:trace}.
\end{remark}

\subsection{Compensated compactness}\label{sec:cc}

While $\chih$, $\om$, $\chibh$ and $\omb$ only admit weak limits (see Section~\ref{sec:limit.chi}), it is important that there is a \emph{compensated compactness} phenomenon for some quadratic products of them. Our main result of this subsection in in Proposition~\ref{prop:chihchibh}, after we state a more general compensated compactness lemma (Lemma~\ref{lem:compensated.compactness}). The proof of Lemma~\ref{lem:compensated.compactness} is relegated to Appendix~\ref{app:CC}.

\begin{lemma}\label{lem:compensated.compactness}
Let $B_{\mathbb R^2}(0,R)\subset \mathbb R^2$ be the ball of radius $R$ in $\mathbb R^2$. Suppose there are two sequences of functions $\{f_n\}_{n=1}^{+\infty},\,\{h_n\}_{n=1}^{+\infty}\subset L^2([0,u_*]\times [0,\ub_*]\times B_{\mathbb R^2}(0,R); \mathbb R)$ such that the following hold:
\begin{enumerate}
\item There exist $L^2([0,u_*]\times [0,\ub_*]\times B_{\mathbb R^2}(0,R); \mathbb R)$ functions $f_{\infty}$ and $h_{\infty}$ such that $f_n \rightharpoonup f_\infty$ and $h_n\rightharpoonup h_\infty$ weakly in $L^2([0,u_*]\times [0,\ub_*]\times B_{\mathbb R^2}(0,R); \mathbb R)$.
\item There exists $C_0>0$ such that
\begin{equation}\label{fn.bd.orig}
\sup_n \sum_{i\leq 1\,j\leq 1,\,k\leq 1} \int_0^{\ub_*} \int_0^{u_*} \iint_{B_{\mathbb R^2}(0,R)} \left((\f{\rd}{\rd u})^{i}(\f{\rd}{\rd y^1})^{j}(\f{\rd}{\rd y^2})^{k} f_n\right)^2\,\ud y^1\, \ud y^2\, \ud u\, \ud \ub \leq C_0
\end{equation}
and
\begin{equation}\label{gn.bd.orig}
\sup_n \sum_{i\leq 1\,j\leq 1,\,k\leq 1} \int_0^{\ub_*} \int_0^{u_*} \iint_{B_{\mathbb R^2}(0,R)} \left((\f{\rd}{\rd \ub})^{i}(\f{\rd}{\rd y^1})^{j}(\f{\rd}{\rd y^2})^{k} h_n\right)^2 \,\ud y^1\, \ud y^2\, \ud u\, \ud \ub \leq C_0.
\end{equation}
\end{enumerate}
Then, after passing to subsequences $\{f_{n_k}\}_{k=1}^{+\infty}$ and $\{h_{n_k}\}_{k=1}^{+\infty}$, 
$f_{n_k} h_{n_k} \rightharpoonup f_\infty h_\infty$ weakly in $L^2([0,u_*]\times [0,\ub_*]\times B_{\mathbb R^2}(0,R); \mathbb R)$.
\end{lemma}

\begin{proposition}\label{prop:chihchibh}
After passing to a subsequence $n_k$, $(\chih_{n_k})_{AB}(\chibh_{n_k})_{CD}$ converges weakly in $L^2_uL^2_{\ub}L^2(S)$ to $(\chih_\infty)_{AB}(\chibh_\infty)_{CD}$, i.e.~for any contravariant $S$-tangent $4$-tensor $\varphi^{ABCD}\in L^2_uL^2_{\ub}L^2(S)$,
\begin{equation*}
\begin{split}
&\: \int_{[0,u_*]\times [0,\ub_*]\times \mathbb S^2} \varphi^{ABCD} (\chih_{n_k})_{AB}(\chibh_{n_k})_{CD} \, \mathrm{dA}_{\gamma_\infty}\,\ud u\,\ud \ub \\
\to &\: \int_{[0,u_*]\times [0,\ub_*]\times \mathbb S^2} \varphi^{ABCD} (\chih_{\infty})_{AB}(\chibh_{\infty})_{CD} \, \mathrm{dA}_{\gamma_\infty}\,\ud u\,\ud \ub.
\end{split}
\end{equation*}
Similarly, after passing to a subsequence $n_k$, $(\chibh_{n_k})_{AB}(\nab_{n_k})_C(\chih_{n_k})_{DE}$ and $(\chih_{n_k})_{AB}(\nab_{n_k})_C(\chibh_{n_k})_{DE}$ also respectively converge weakly in $L^2_uL^2_{\ub}L^2(S)$ to $(\chibh_\infty)_{AB}(\nab_\infty)_C(\chih_\infty)_{DE}$ and $(\chih_\infty)_{AB}(\nab_\infty)_C(\chibh_\infty)_{DE}$.
\end{proposition}
\begin{proof}
Suffices to work component-wise and in a local coordinate chart $U$. The bounds \eqref{eq:bdd.psiH} and \eqref{eq:bdd.psiHb} (together with \eqref{eq:bdd.metric}) imply that in local coordinates, the following estimates hold: 
\begin{equation*}
\sup_n \sum_{i\leq 1\,j+k\leq 2} \int_U \left((\f{\rd}{\rd u})^{i}(\f{\rd}{\rd \th^1})^{j}(\f{\rd}{\rd \th^2})^{k} (\chih_n)_{AB}\right)^2\,\ud x^1 \,\ud x^2\, \ud u\, \ud \ub <+\infty
\end{equation*}
and
\begin{equation*}
\sup_n \sum_{i\leq 1\,j + k\leq 2} \int_U \left((\f{\rd}{\rd \ub})^{i}(\f{\rd}{\rd \th^1})^{j}(\f{\rd}{\rd \th^2})^{k} (\chibh_n)_{AB} \right)^2 \,\ud x^1 \,\ud x^2\, \ud u\, \ud \ub < +\infty.
\end{equation*}
The assertion of the proposition therefore follows from Lemma~\ref{lem:compensated.compactness}. \qedhere
\end{proof}

\subsection{Weak limits of $\Omega^2|\protect\chih|_{\gamma}^2$ and $\Omega^2|\protect\chibh|_{\gamma}^2$}\label{sec:limit.dust}

In this subsection, we discuss the weak limits of $|\chih_n|_{\gamma_n}^2$ and $|\chibh_n|_{\gamma_n}^2$. Notice that unlike $\chih_n$ and $\chibh_n$ themselves, $|\chih_n|_{\gamma_n}^2$ and $|\chibh_n|_{\gamma_n}^2$ are only in $L^1$ (and not in $L^2$). We therefore can only hope to obtain weak limits as measures. For this purpose, our main tool will be Proposition~\ref{prop:weak}.

\begin{proposition}\label{prop:nu.convergence}
For every $u\in [0,u_*]$, there exists a non-negative Radon measure $\ud \nu_u$ on $(0,\ub_*)\times \mathbb S^2$, which is uniformly bounded and continuous in $u$ (see \eqref{eq:weak.conv.cont} and \eqref{eq:weak.conv.bdd}), such that after passing to a subsequence $n_k$, the following convergences hold for every bounded $\varphi \in C^0(\{u\} \times (0,\ub_*)\times \mathbb S^2)$:
\begin{equation}\label{eq:nu.convergence}
\begin{split}
&\: \int_{\{u\}\times (0,\ub_*)\times \mathbb S^2} \varphi \,\ud\nu_u  \\
= &\: \lim_{k\to +\infty} (\int_{\{u\}\times (0,\ub_*)\times \mathbb S^2} \varphi \,\Omega_{n_k}^2 |\chih_{n_k}|^2_{\gamma_{n_k}} \,\mathrm{dA}_{\gamma_{n_k}}\, \ud \ub) - \int_{\{u\}\times (0,\ub_*)\times \mathbb S^2} \varphi \,\Omega_{\infty}^2 |\chih_{\infty}|^2_{\gamma_{\infty}} \,\mathrm{dA}_{\gamma_{\infty}}\,\ud\ub.
\end{split}
\end{equation}

Similarly, there exists a non-negative Radon measure $\ud \nub_{\ub}$, which is uniformly bounded and continuous in $\ub$, such that after passing to a further subsequence, the following convergences hold for every bounded $\varphi \in C^0((0,u_*) \times \{\ub\} \times \mathbb S^2)$: 
\begin{equation}\label{eq:nub.convergence}
\begin{split}
&\: \int_{(0,u_*) \times \{\ub\} \times \mathbb S^2} \varphi \,\ud\nub_{\ub}  \\
= &\: \lim_{k\to +\infty} (\int_{(0,u_*) \times \{\ub\}\times \mathbb S^2} \varphi \,\Omega_{n_k}^2 |\chibh_{n_k}|^2_{\gamma_{n_k}} \,\mathrm{dA}_{\gamma_{n_k}}\, \ud u) - \int_{(0,u_*) \times \{\ub\}\times \mathbb S^2} \varphi \,\Omega_{\infty}^2 |\chibh_{\infty}|^2_{\gamma_{\infty}} \,\mathrm{dA}_{\gamma_{\infty}}\,\ud u.
\end{split}
\end{equation}
\end{proposition}
\begin{proof}
We will only prove the statements concerning $\ud\nu_u$; $\ud\nub_{\ub}$ can be handled similarly.

To show \eqref{eq:nu.convergence}, we first use Proposition~\ref{prop:weak} (and Lemma~\ref{lem:easy.convergence}) with $\psi_{n_k}= \Omg_{n_k}^2|\chih_{n_k} - \chih_\infty|^2_{\gamma_{n_k}}\f{\sqrt{\det \gamma_{n_k}}}{\sqrt{\det (\gamma_{0,0})_\infty}}$. Note that by \eqref{eq:bdd.metric} and \eqref{eq:bdd.psiH}, $\{\psi_{n_k}\}_{k=1}^{+\infty}$ indeed obeys the assumptions of Proposition~\ref{prop:weak}. Hence we conclude that there is a further subsequence and a (scalar-valued) Radon measure $\ud \nu_u$ which is uniformly bounded and continuous in $u$ such that for every real-valued function bounded $\varphi \in C^0(\{u\}\times (0,\ub_*)\times \mathbb S^2)$,
$$\int_{\{u\}\times (0,\ub_*)\times \mathbb S^2} \varphi\, \Omg_{n_k}^2|\chih_{n_k} -\chih_\infty|^2_{\gamma_{n_k}} \f{\sqrt{\det \gamma_{n_k}}}{\sqrt{\det (\gamma_{0,0})_\infty}}\,\mathrm{dA}_{(\gamma_{0,0})_\infty} \,\ud \ub \to \int_{\{u\}\times (0,\ub_*)\times \mathbb S^2} \varphi \,\ud \nu_u.$$

Note that $\ud\nu_u$ as defined is manifestly non-negative. Then noticing that $\f{\sqrt{\det \gamma_{n_k}}}{\sqrt{\det (\gamma_{0,0})_\infty}}\,\mathrm{dA}_{(\gamma_{0,0})_\infty} = \mathrm{dA}_{\gamma_{n_k}}$, and using the convergence statements in Propositions~\ref{prop:metric.limit} and \ref{prop:chi.limit}, it follows that $\ud \nu_u$ indeed satisfies \eqref{eq:nu.convergence}. \qedhere

\end{proof}

\begin{proposition}\label{prop:nu.add.reg}
$(\{\ud \nu_u\}_{u\in [0,u_*]},\, \{ \ud\underline{\nu}_{\ub}\}_{\ub \in [0,\ub_*]})$ is angularly regular in the sense of Definition~\ref{def:ang.reg.null.dust}. 
\end{proposition}
\begin{proof}
As in the previous proposition, we consider only $\ud\nu_n$; the case for $\ud\underline{\nu}_{\ub}$ is similar. 

That $\ud\nu_u$ is continuous in $u$ has already been established in Proposition~\ref{prop:nu.convergence}. It thus remains to bound each of the terms in \eqref{eq:nu.add.reg}, which will be carried out in the three steps below.

\pfstep{Step~1: Estimating the first term in \eqref{eq:nu.add.reg}} Let $\varphi \in \accentset{\scalebox{.7}{\mbox{\tiny (0)}}}{{\mathfrak X}}_{u}$. By density we can assume that $\varphi$ is smooth. Hence, by \eqref{eq:nu.convergence}, for every $u$,
\begin{equation}\label{eq:dust.Li.bd}
\begin{split}
&\: \left| \int_{\{u\}\times [0,\ub_*]\times \mathbb S^2} \varphi \,\ud\nu_u \right| \\
\ls &\:  \limsup_{k\to +\infty} \left| \int_{\{u\}\times [0,\ub_*]\times \mathbb S^2} \varphi \,\Omega_{n_k}^2 |\chih_{n_k}|^2_{\gamma_{n_k}} \,\mathrm{dA}_{\gamma_{n_k}}\, \ud \ub \right| + \left| \int_{\{u\}\times [0,\ub_*]\times \mathbb S^2} \varphi \,\Omega_{\infty}^2 |\chih_{\infty}|^2_{\gamma_{\infty}} \,\mathrm{dA}_{\gamma_{\infty}}\,\ud\ub \right| \\
\ls &\: \|\varphi\|_{C^0_{\ub} L^1(S_{u,\ub})} ( \limsup_{k\to +\infty}\|\Omg_{n_k}\|_{C^0_u C^0_{\ub} C^0(S_{u,\ub})} + \|\Omg_{\infty}\|_{C^0_u C^0_{\ub} C^0(S_{u,\ub})}) \\
&\: \quad \times ( \limsup_{k\to +\infty}\|\chih_{n_k}\|_{L^\i_u L^2_{\ub} L^\i(S_{u,\ub}) } + \|\chih_{\infty}\|_{L^\i_u L^2_{\ub} L^\i(S_{u,\ub}) } )\ls \|\varphi\|_{C^0_{\ub} L^1(S_{u,\ub})} \ls 1,
\end{split}
\end{equation}
where in the last line we used the estimates in Propositions~\ref{prop:metric.limit} and \ref{prop:chi.limit}.

\pfstep{Step~2: Estimating the second term in \eqref{eq:nu.add.reg}} Let $\slashed{X}\in \accentset{\scalebox{.7}{\mbox{\tiny (1)}}}{{\mathfrak X}}_{u}$. As in Step~1, we assume by density that $\slashed{X}$ is smooth. The main difference with Step~1 is that we need to integrate by parts in the angular direction to handle the additional derivative. By \eqref{eq:nu.convergence}, for every $u$,
\begin{equation*}
\begin{split}
&\: \left|\int_{\{u\}\times (0,\ub_*) \times \mathbb S^2} \div_\infty \slashed{X}\,\ud\nu_u\right| \\
\leq &\:  \limsup_{k\to +\infty} \left| \int_{\{u\}\times (0,\ub_*) \times \mathbb S^2} (\div_\infty \slashed{X}) \,\Omega_{n_k}^2 |\chih_{n_k}|^2_{\gamma_{n_k}} \,\mathrm{dA}_{\gamma_{n_k}}\, \ud \ub \right| + \left| \int_{\{u\}\times (0,\ub_*) \times \mathbb S^2} (\div_\infty \slashed{X}) \,\Omega_{\infty}^2 |\chih_{\infty}|^2_{\gamma_{\infty}} \,\mathrm{dA}_{\gamma_{\infty}}\,\ud\ub \right|  \\
\ls  &\:  \limsup_{k\to +\infty} \int_{\{u\}\times (0,\ub_*) \times \mathbb S^2} |\{ (\nab_\infty)_A - (\nab_{n_k})_A \} \slashed{X}^A| \,\Omega_{n_k}^2 |\chih_{n_k}|^2_{\gamma_{n_k}} \,\mathrm{dA}_{\gamma_{n_k}}\, \ud \ub \\
&\: + \limsup_{k\to +\infty} \left| \int_{\{u\}\times (0,\ub_*) \times \mathbb S^2} \slashed{X}^A \, (\nab_{n_k})_A (\Omega_{n_k}^2 |\chih_{n_k}|^2_{\gamma_{n_k}}) \,\mathrm{dA}_{\gamma_{n_k}}\, \ud \ub \right|\\
&\: + \left| \int_{\{u\}\times (0,\ub_*) \times \mathbb S^2} \slashed{X}^A \,(\nab_\infty)_A (\Omega_{\infty}^2 |\chih_{\infty}|^2_{\gamma_{\infty}}) \,\mathrm{dA}_{\gamma_{\infty}}\,\ud\ub  \right| \\
\ls &\: \|\slashed X\|_{L^\i_{\ub}L^1(S_{u,\ub})} \times (1 + \|(\slashed\Gamma_\infty)_{AB}^A - (\slashed\Gamma_{n_k})_{AB}^A\|_{C^0_u C^0_{\ub} C^0(S_{u,\ub})})\\
&\: \quad \times ( \limsup_{k\to +\infty}\|\Omg_{n_k}\|_{C^0_u C^0_{\ub} C^1(S_{u,\ub})} + \|\Omg_{\infty}\|_{C^0_u C^0_{\ub} C^1(S_{u,\ub})}) \\
&\:  \times ( \limsup_{k\to +\infty}\|\chih_{n_k}\|_{L^\i_u L^2_{\ub} W^{1,\i}(S_{u,\ub}) } + \|\chih_{\infty}\|_{L^\i_u L^2_{\ub} W^{1,\i}(S_{u,\ub}) } ) \\
\ls &\: \|\slashed X\|_{L^\i_{\ub}L^1(S_{u,\ub})} \ls 1,
\end{split}
\end{equation*}
where in the last line we have used the estimates established in Propositions~\ref{prop:Christoffel}, \ref{prop:metric.limit} and \ref{prop:chi.limit}.

\pfstep{Step~3: Estimating the third term in \eqref{eq:nu.add.reg}} To estimate the third term, we carry out an integration by parts argument as in Step~2. The only difference is that we have higher derivatives, e.g.~a term $(\nab_\infty)^2 \chih_\infty$, which can only be controlled in $L^\i_u L^2_{\ub} L^4(S_{u,\ub})$ by Proposition~\ref{prop:chi.limit} (but not $L^\i_u L^2_{\ub} L^\i(S_{u,\ub})$). It is for this reason that we need to assume control of $\|\slashed X\otimes \slashed Y\|_{L^\i_{\ub}L^{\f 43}(S_{u,\ub})}$ in the definition of $\accentset{\scalebox{.7}{\mbox{\tiny (2)}}}{{\mathfrak X}}_{u}$. We omit the straightforward details, but will just demonstrate this with one of the most difficult terms:
\begin{equation*}
\begin{split}
&\: \left| \int_{\{u\}\times [0,\ub_*]\times \mathbb S^2} \slashed{X}^A \slashed{Y}^B \,(\nab_\infty)^2_{BA} (\Omega_{\infty}^2 |\chih_{\infty}|^2_{\gamma_{\infty}}) \,\mathrm{dA}_{\gamma_{\infty}}\,\ud\ub \right| \\
\ls &\: \| \slashed{X} \otimes \slashed{Y} \|_{L^\i_{\ub} L^{\f 43}(S_{u,\ub})} \| (\nab_\infty)^2_{BA} (\Omega_{\infty}^2 |\chih_{\infty}|^2_{\gamma_{\infty}}) \|_{L^\i_u L^\i_{\ub} L^4(S_{u,\ub})} \ls \| \slashed{X} \otimes \slashed{Y} \|_{L^\i_{\ub} L^{\f 43}(S_{u,\ub})} \ls 1,
\end{split}
\end{equation*}
where we have used Propositions~\ref{prop:metric.limit} and \ref{prop:chi.limit}. \qedhere

\end{proof}

\textbf{At this point, we fix the subsequence $n_k$ such that Propositions~\ref{prop:gamma}, \ref{prop:Christoffel}, \ref{prop:metric.limit}, \ref{prop:eta.etab.limit}, \ref{prop:chi.limit}, \ref{prop:trch.weak.limit}, \ref{prop:trch.imp}, \ref{prop:chihchibh} and \ref{prop:nu.convergence} hold.} Along this subsequence, the spacetime metrics converge uniformly to a limiting spacetime $(\mathcal M, g_\i)$, with additional refined convergence for the Ricci coefficients as described by the propositions above. Moreover, combining the above propositions with Propositions~\ref{prop:K.imp} and \ref{prop:eta.etab.imp}, it also follows that the limit $(\mathcal M,\,g_\infty)$ is angularly regular (see Definition~\ref{double.null.def.2}).

\section{The equations satisfied by the limit spacetime and the proof of Theorem~\ref{main.thm}}\label{sec:eqns.for.limit}

We continue to work under the assumptions of Theorem~\ref{main.thm}, and take $n_k$ as in the end of the last section.

Using the properties of the limits that we showed in the previous section, we now derive the equations satisfied by the various limiting quantities.
\begin{itemize}
\item In \textbf{Section~\ref{sec:equations.for.metric.comp}}, we derive the transport equations for the metric components.
\item In \textbf{Section~\ref{sec:eq.Ricci.trans.1}} and \textbf{Section~\ref{sec:eq.Ricci.trans.2}}, we derive the transport equations for the Ricci coefficients. These equations correspond exactly to a description of the (weak) Ricci curvature.
\item In \textbf{Section~\ref{sec.prop.dust}}, we prove the propagation equation for the null dust.
\item Finally, in \textbf{Section~\ref{sec:higher.Ricci}} and \textbf{Section~\ref{sec:renorm.Bianchi}}, we derive the higher order equations for the Ricci coefficients and renormalized Bianchi equations respectively.
\end{itemize}

Combined with the results in the previous section, we will then complete the proof of Theorem~\ref{main.thm} in \textbf{Section~\ref{sec:proof.of.main.thm}}.

\subsection{Equations for the metric components}\label{sec:equations.for.metric.comp}

In Section~\ref{sec:existence}, we defined 
\begin{itemize}
\item $(\gamma_\infty, \,b_\infty,\, \Omg_\infty)$, which are understood as subsequential uniform limits of the metric components, and 
\item $(\chi_\infty,\,\chib_\infty,\,\eta_\infty,\,\etab_\infty,\,\om_\infty,\,\omb_\infty)$, which are understood as subsequential weak limits of the Ricci coefficients.
\end{itemize} 
It is easy to check that they are related in the expected manner, i.e.~that $(\chi_\infty,\,\chib_\infty,\,\eta_\infty,\,\etab_\infty,\,\om_\infty,\,\omb_\infty)$ are indeed the Ricci coefficients associated to the limiting metric, i.e.

\begin{proposition}\label{prop:Ricci.is.metric.derivative}
The equations \eqref{metric.derivative.invar} and \eqref{Ricci.relation} hold when we 
\begin{itemize}
\item take the metric components to be $(\gamma_\infty, \,b_\infty,\, \Omg_\infty)$ (given by Propositions~\ref{prop:gamma} and \ref{prop:metric.limit}) and 
\item take the Ricci coefficients to be $(\chi_\infty,\,\chib_\infty,\,\eta_\infty,\,\etab_\infty,\,\om_\infty,\,\omb_\infty)$ (given by Propositions~\ref{prop:eta.etab.limit}, \ref{prop:trch.weak.limit} and \ref{prop:chi.limit}), 
\item where all derivatives are to be understood as weak derivatives. 
\end{itemize}
\end{proposition}
\begin{proof}
This follows easily from the uniqueness of limits in the sense of distributions. \qedhere
\end{proof}

Proposition~\ref{prop:Ricci.is.metric.derivative} immediately implies the following transport equations for the metric components.
\begin{proposition}\label{prop:equations.for.g}
The equations
$$\nab_4 \gamma = 0,\quad \nab_4 b = - 2\Omg (\eta - \etab) + \chi\cdot b,\quad \nab_4 \log \Omg = -2 \omega$$
hold in the integrated sense (see Definition~\ref{def:weak.transport}).
\end{proposition}
\begin{proof}
The last equation here is just the first equation in \eqref{Ricci.relation}. The other two equations follow from \eqref{metric.derivative.invar} and the expression of $\nab_4$ in \eqref{nab4.def}. 

Using Proposition~\ref{prop:Ricci.is.metric.derivative}, we have thus shown that all these equations hold when the derivatives are understood as weak derivatives. Together with the regularity of the metric components and the Ricci coefficients we have derived in the previous section, it follows that these equations are also satisfied in the integrated sense. \qedhere 
\end{proof}

In fact, we can also derive transport equations for the derivatives of the metric components.
\begin{proposition}\label{prop:equations.for.nabla.g}
The equations \eqref{eq:nablagamma}--\eqref{eq:nablab} hold for $(\mathcal M, g_\infty)$ in the integrated sense (see Definition~\ref{def:weak.transport}).
\end{proposition} 
\begin{proof}
We start from the fact that (by Proposition~\ref{prop:higher.order.metric.C2}), the equations \eqref{eq:nablagamma}--\eqref{eq:nablab} hold classically, and hence also in the integrated sense, for $(\mathcal M, g_{n_k})$ for all $n_k \in \mathbb N$. Taking \eqref{eq:nablagamma} as an example, we have that for any rank-3 $C^1$ tensor $\varphi$, 
\begin{equation}\label{eq:d4gamma.nk}
\begin{split}
&\: \int_{S_{u,\ub_1}} \langle\varphi, (\nab\gamma)_{n_k} \rangle \Omg_{n_k}  \,\ud A_{\gamma_{n_k} } - \int_{S_{u,\ub_2}} \langle\varphi, (\nab\gamma)_{n_k}  \rangle \Omg_{n_k}  \,\ud A_{\gamma_{n_k} } \\
=&\: \int_{\ub_1}^{\ub_2} \int_{S_{u,\ub'}} (\langle \varphi,  (\trch_{n_k}  - 2\om_{n_k} )(\nab\gamma)_{n_k} \rangle + \langle \nab_4\varphi,(\nab\gamma)_{n_k} \rangle) \Omg_{n_k} ^2 \,\ud A_{\gamma_{n_k} }\, \ud \ub' .
\end{split}
\end{equation}
Now we want to pass to the $k\to +\infty$ limit to obtain \eqref{eq:nablagamma} in the integrated sense. By Propositions~\ref{prop:gamma} and \ref{prop:metric.limit}, $\gamma_{n_k}$, $(\nab\gamma)_{n_k}$, $\Omg_k$ and $\ud A_{\gamma_{n_k}}$ have uniform limits which allow us to take $k\to+\infty$. Similarly, $\trch_{n_k}$ also has a strong limit $C^0_u L^2_{\ub} L^\i(S_{u,\ub})$ by \eqref{eq:trch.imp.conv}, allowing us to pass to $k\to +\infty$. The only term without a strong limit is therefore  $\om_{n_k}$, which, by Proposition~\ref{prop:chi.limit}, only has a weak $L^2_{\ub}L^2(S_{u,\ub})$ limit (for every fixed $u$). Nevertheless, $\om_{n_k}$ is multiplied by $(\nab\gamma)_{n_k} \Omg_{n_k}^2 \, \ud A_{\gamma_{n_k}}$ which has a uniform limit so that
$$\int_{\ub_1}^{\ub_2} \int_{S_{u,\ub'}} \langle \varphi,  \om_{n_k} (\nab\gamma)_{n_k} \rangle  \,\ud A_{\gamma_{n_k} }\, \ud \ub'  \to \int_{\ub_1}^{\ub_2} \int_{S_{u,\ub'}} \langle \varphi,  \om_{\infty} (\nab\gamma)_{\infty} \rangle  \,\ud A_{\gamma_{\infty} }\, \ud \ub'.$$
Combining all these observations, we can pass \eqref{eq:d4gamma.nk} to the $k\to +\infty$ limit to obtain \eqref{eq:nablagamma} in the integrated sense.

The equations \eqref{eq:nablaOmega} and \eqref{eq:nablab} can be proven similarly, noting that in both cases the only terms without a strong limit involve $\chih_{n_k}$ or $\om_{n_k}$, but they are multiplied by terms which have uniform limits. We omit the details. \qedhere
\end{proof}

\subsection{The vanishing (weak) Ricci curvature components of the limit spacetime}\label{sec:eq.Ricci.trans.1}

\begin{proposition}\label{prop:Ricci.easy}
The equations \eqref{Ric4A}--\eqref{Ric3A} hold for $(\mathcal M, g_\infty)$ in the integrated sense (see Definition~\ref{def:weak.transport}).
\end{proposition}
\begin{proof}
Let us only consider \eqref{Ric4A}. The equation \eqref{Ric3A} can be treated in essentially the same manner. According to Definition~\ref{def:weak.transport}, we need to show, for all $S$-tangent vector field $\varphi \in C^1$,
\begin{equation}\label{eq:Ricci4A.1}
\begin{split}
&\: \int_{S_{u,\ub_1}} \langle\varphi, \eta_{\infty} \rangle \Omg_{\infty} \,\ud A_{\gamma_{\infty}} - \int_{S_{u,\ub_2}} \langle\varphi, \eta_{\infty} \rangle \Omg_{\infty} \,\ud A_{\gamma_{\infty}} \\
= &\: \int_{\ub_1}^{\ub_2} \int_{S_{u,\ub'}} \langle \varphi, \div_{\infty}\chih_{\infty} - \f12 \nab_{\infty}\trch_{\infty} - \f 12(\eta-\etab)_{\infty}\cdot_{\infty}\chih_{\infty} + \f 14\trch_{\infty}\eta_{\infty} \rangle \Omg_{\infty}^2 \,\ud A_{\gamma_{\infty}}\, \ud \ub'\\
&\: + \int_{\ub_1}^{\ub_2} \int_{S_{u,\ub'}} \langle \varphi, \f 34 \trch_{\infty}\etab_{\infty} - 2\om_{\infty}\eta_{\infty}\rangle \Omg_{\infty}^2 \,\ud A_{\gamma_{\infty}}\, \ud \ub' +\int_{\ub_1}^{\ub_2} \int_{S_{u,\ub'}}  \langle (\nab_4)_{\infty}\varphi,\eta_{\infty} \rangle \Omg_{\infty}^2 \,\ud A_{\gamma_{\infty}}\, \ud \ub'.
\end{split}
\end{equation}
Since $(\mathcal M, g_{n_k})$ is a smooth solution to the Einstein vacuum equations, by Proposition~\ref{prop:null.structure},
\begin{equation}\label{eq:Ricci4A.2}
\begin{split}
&\: \int_{S_{u,\ub_1}} \langle\varphi, \eta_{n_k} \rangle \Omg_{n_k} \,\ud A_{\gamma_{n_k}} - \int_{S_{u,\ub_2}} \langle\varphi, \eta_{n_k} \rangle \Omg_{n_k} \,\ud A_{\gamma_{n_k}} \\
= &\: \int_{\ub_1}^{\ub_2} \int_{S_{u,\ub'}} \langle \varphi, \div_{n_k}\chih_{n_k} - \f12 \nab_{n_k}\trch_{n_k} - \f 12(\eta-\etab)_{n_k}\cdot_{n_k}\chih_{n_k} + \f 14\trch_{n_k}\eta_{n_k} \rangle \Omg_{n_k}^2 \,\ud A_{\gamma_{n_k}}\, \ud \ub'\\
&\: + \int_{\ub_1}^{\ub_2} \int_{S_{u,\ub'}} \langle \varphi, \f 34 \trch_{n_k}\etab_{n_k} - 2\om_{n_k}\eta_{n_k}\rangle \Omg_{n_k}^2 \,\ud A_{\gamma_{n_k}}\, \ud \ub' +\int_{\ub_1}^{\ub_2} \int_{S_{u,\ub'}}  \langle (\nab_4)_{n_k}\varphi,\eta_{n_k} \rangle \Omg_{n_k}^2 \,\ud A_{\gamma_{n_k}}\, \ud \ub'.
\end{split}
\end{equation}
Our goal now is to pass to the $k\to+\infty$ limit in \eqref{eq:Ricci4A.2} to obtain \eqref{eq:Ricci4A.1}. The terms on the LHS of \eqref{eq:Ricci4A.2} are easy, since by Proposition~\ref{prop:gamma}, \ref{prop:metric.limit} and \ref{prop:eta.etab.limit}, all of $\gamma_{n_k}$, $\Omg_{n_k}$ and $\eta_{n_k}$ converge uniformly to their limits. We therefore have
\begin{equation}\label{eq:Ricci4A.3}
\begin{split}
&\: \int_{S_{u,\ub_1}} \langle\varphi, \eta_{n_k} \rangle \Omg_{n_k} \,\ud A_{\gamma_{n_k}} - \int_{S_{u,\ub_2}} \langle\varphi, \eta_{n_k} \rangle \Omg_{n_k} \,\ud A_{\gamma_{n_k}} \\
\to &\: \int_{S_{u,\ub_1}} \langle\varphi, \eta_{\infty} \rangle \Omg_{\infty} \,\ud A_{\gamma_{\infty}} - \int_{S_{u,\ub_2}} \langle\varphi, \eta_{\infty} \rangle \Omg_{\infty} \,\ud A_{\gamma_{\infty}}.
\end{split}
\end{equation}

The terms on the RHS of \eqref{eq:Ricci4A.2} are not much more difficult. To proceed, let us first recall from Propositions~\ref{prop:gamma} and \ref{prop:metric.limit} that the metric components converge uniformly to their limit so that we can focus on taking the limits of the Ricci coefficients. There are now three types of terms to understand:
\begin{enumerate}
\item Terms in which $\varphi$ is contracted with $\nab_{n_k} \chi_{n_k}$. For these terms, we use that $\nab_{n_k}\chi_{n_k}$ converges weakly in $L^2_{\ub}L^2(S_{u,\ub})$ for all fixed $u$ (Propositions~\ref{prop:chi.limit} and \ref{prop:trch.weak.limit}). Hence,
\begin{equation}\label{eq:Ricci4A.4}
\begin{split}
&\: \int_{\ub_1}^{\ub_2} \int_{S_{u,\ub'}} \langle \varphi, \div_{n_k}\chih_{n_k} - \f12 \nab_{n_k}\trch_{n_k}\rangle  \Omg_{n_k}^2 \,\ud A_{\gamma_{n_k}}\, \ud \ub' \\
\to &\: \int_{\ub_1}^{\ub_2} \int_{S_{u,\ub'}} \langle \varphi, \div_{\infty}\chih_{\infty} - \f12 \nab_{\infty}\trch_{\infty} \rangle \Omg_{\infty}^2 \,\ud A_{\gamma_{\infty}}\, \ud \ub'.
\end{split}
\end{equation}
\item Terms in which $\varphi$ is contracted with a \underline{product} of two Ricci coefficients. There are in turn two types of terms: (a) a product of two Ricci coefficients, each of which converges in the $C^0_u L^2_{\ub}L^2(S_{u,\ub})$ norm, e.g.~$\trch_{n_k} \eta_{n_k}$, $\trch_{n_k} \etab_{n_k}$ (by Propositions~\ref{prop:eta.etab.limit} and \ref{prop:trch.imp}); (b) a product of two Ricci coefficients, one of which converges weakly in $L^2_{\ub}L^2(S_{u,\ub})$ for every $u$ and the other converges in the $C^0_u L^2_{\ub}L^2(S_{u,\ub})$ norm, e.g.~$\eta_{n_k}\cdot\chih_{n_k}$, $\etab_{n_k}\cdot\chih_{n_k}$, $\om_{n_k} \chih_{n_k}$. In either case, we know that the product converges weakly in $L^2_{\ub}L^2(S_{u,\ub})$ to the product of the limits for every $u\in [0,u_*]$. This implies
\begin{equation}\label{eq:Ricci4A.5}
\begin{split}
&\:\int_{\ub_1}^{\ub_2} \int_{S_{u,\ub'}} \langle \varphi, - \f 12(\eta-\etab)_{n_k}\cdot_{n_k}\chih_{n_k} + \f 14\trch_{n_k}\eta_{n_k} +\f 34 \trch_{n_k}\etab_{n_k} - 2\om_{n_k}\eta_{n_k}\rangle  \Omg_{n_k}^2 \,\ud A_{\gamma_{n_k}}\, \ud \ub' \\
\to &\: \int_{\ub_1}^{\ub_2} \int_{S_{u,\ub'}} \langle \varphi,  - \f 12(\eta-\etab)_{\infty}\cdot_{\infty}\chih_{\infty} + \f 14\trch_{\infty}\eta_{\infty} +\f 34 \trch_{\infty}\etab_{\infty} - 2\om_{\infty}\eta_{\infty} \rangle \Omg_{\infty}^2 \,\ud A_{\gamma_{\infty}}\, \ud \ub'.
\end{split}
\end{equation}
\item Terms in which $(\nab_4)_{n_k}\varphi$ is contracted with $\eta_{n_k}$. Expanding $(\nab_4)_{n_k}\varphi$ using \eqref{nab4.def}, we see from Propositions~\ref{prop:gamma}, \ref{prop:metric.limit}, \ref{prop:trch.weak.limit} and \ref{prop:chi.limit} that $(\nab_4)_{n_k}\varphi \rightharpoonup (\nab_4)_{\infty}\varphi$ weakly in $L^2_{\ub}L^2(S_{u,\ub})$ for every $u\in [0,u_*]$. Now this is contracted with $\eta_{n_k}$, which tends to $\eta_\infty$ in the $C^0_u C^0_{\ub}L^2(S_{u,\ub})$ norm by Proposition~\ref{prop:eta.etab.limit}. It follows that
\begin{equation}\label{eq:Ricci4A.6}
\begin{split}
\int_{\ub_1}^{\ub_2} \int_{S_{u,\ub'}}  \langle (\nab_4)_{n_k}\varphi,\eta_{n_k} \rangle \Omg_{n_k}^2 \,\ud A_{\gamma_{n_k}}\, \ud \ub'
\to &\: \int_{\ub_1}^{\ub_2} \int_{S_{u,\ub'}}  \langle (\nab_4)_{\infty}\varphi,\eta_{\infty} \rangle \Omg_{\infty}^2 \,\ud A_{\gamma_{\infty}}\, \ud \ub'.
\end{split}
\end{equation}
\end{enumerate}

Combining \eqref{eq:Ricci4A.3}--\eqref{eq:Ricci4A.6}, we therefore deduce \eqref{eq:Ricci4A.1} from \eqref{eq:Ricci4A.2} after taking $k\to +\infty$, as desired. \qedhere
\end{proof}

\begin{proposition}\label{prop:Ricci.harder}
The equations \eqref{trRicAB}--\eqref{Ric34.1} hold for $(\mathcal M, g_\infty)$ in the weak integrated sense (see Definition~\ref{def:weaker.transport}).
\end{proposition}
\begin{proof}
The proof is in fact quite similar to that of Proposition~\ref{prop:Ricci.easy}, with two exceptions: 
\begin{itemize}
\item The equations are now in weak integrated (as opposed to integrated) form.
\item We also use the compensated compactness property in Proposition~\ref{prop:chihchibh}. 
\end{itemize}
Since the equations \eqref{trRicAB}--\eqref{Ric34.1} can in fact all be treated in a similar fashion, We will only discuss \eqref{RicAB.1} in detail.

Recalling Definition~\ref{def:weaker.transport}, our goal will be to show that for all contravariant $S$-tangent $2$-tensor field $\varphi \in C^1$,
\begin{equation}\label{eq:Ricci.harder.1}
\begin{split}
&\: \int_0^{\ub_*} \int_{S_{u_2,\ub}} \langle\varphi, \chih_\infty\rangle \Omg_\infty\,\ud A_{\gamma_\infty}\, \ud \ub - \int_0^{\ub_*} \int_{S_{u_1,\ub}} \langle\varphi, \chih_\infty\rangle \Omg \,\ud A_{\gamma_\infty}\,\ud \ub \\
= &\: \int_0^{\ub_*}\int_{u_1}^{u_2} \int_{S_{u',\ub}} \langle \varphi, \nab_\infty \widehat{\otimes}_\infty\eta_\infty +\f 12\trchb_\infty \chih_\infty - \f12 \trch_\infty \chibh_\infty + \eta_\infty \widehat{\otimes}_\infty\eta_\infty \rangle  \Omg_\infty^2 \,\ud A_{\gamma_\infty}\, \ud u'\,\, \ud \ub \\
&\: + \int_0^{\ub_*}\int_{u_1}^{u_2} \int_{S_{u',\ub}} \langle (\nab_3)_\infty \varphi, \chih_\infty \rangle \Omg_\infty^2 \,\ud A_{\gamma_\infty}\, \ud u'\,\, \ud \ub.
\end{split}
\end{equation}

In a similar spirit as the proof of Proposition~\ref{prop:Ricci.easy}, we will derive \eqref{eq:Ricci.harder.1} by taking the $k\to +\infty$ limit of the following equation, which holds thanks to Proposition~\ref{prop:null.structure}:
\begin{equation}\label{eq:Ricci.harder.2}
\begin{split}
&\: \int_0^{\ub_*} \int_{S_{u_2,\ub}} \langle\varphi, \chih_{n_k}\rangle \Omg_{n_k} \,\ud A_{\gamma_{n_k}}\, \ud \ub - \int_0^{\ub_*} \int_{S_{u_1,\ub}} \langle\varphi, \chih_{n_k} \rangle \Omg_{n_k} \,\ud A_{\gamma_{n_k}}\,\ud \ub \\
= &\: \int_0^{\ub_*}\int_{u_1}^{u_2} \int_{S_{u',\ub}} \langle \varphi, \nab_{n_k}\widehat{\otimes}_{n_k}\eta_{n_k}+\f 12\trchb_{n_k} \chih_{n_k} - \f12 \trch_{n_k} \chibh_{n_k} + \eta_{n_k} \widehat{\otimes}_{n_k}\eta_{n_k} \rangle \Omg_{n_k}^2 \,\ud A_{\gamma_{n_k}}\, \ud u'\,\, \ud \ub \\
&\: + \int_0^{\ub_*}\int_{u_1}^{u_2} \int_{S_{u',\ub}} \langle (\nab_3)_{n_k}\varphi, \chih_{n_k} \rangle \Omg_{n_k}^2 \,\ud A_{\gamma_{n_k}}\, \ud u'\,\, \ud \ub.
\end{split}
\end{equation}

We proceed, again, as in the proof of Proposition~\ref{prop:Ricci.easy}. First, by Propositions~\ref{prop:gamma}, \ref{prop:metric.limit} and \ref{prop:chi.limit},
\begin{equation}\label{eq:Ricci.harder.3}
\begin{split}
&\: \int_0^{\ub_*} \int_{S_{u_2,\ub}} \langle\varphi, \chih_{n_k}\rangle \Omg_{n_k} \,\ud A_{\gamma_{n_k}}\, \ud \ub - \int_0^{\ub_*} \int_{S_{u_1,\ub}} \langle\varphi, \chih_{n_k} \rangle \Omg_{n_k} \,\ud A_{\gamma_{n_k}}\,\ud \ub \\
\to &\: \int_0^{\ub_*} \int_{S_{u_2,\ub}} \langle\varphi, \chih_\infty\rangle \Omg_\infty\,\ud A_{\gamma_\infty}\, \ud \ub - \int_0^{\ub_*} \int_{S_{u_1,\ub}} \langle\varphi, \chih_\infty\rangle \Omg \,\ud A_{\gamma_\infty}\,\ud \ub.
\end{split}
\end{equation}

For the terms on the RHS of \eqref{eq:Ricci.harder.2}, we split into three types of terms as in Proposition~\ref{prop:Ricci.easy}. Again, we note that we can focus on the limits of the Ricci coefficients in view of Propositions~\ref{prop:gamma} and \ref{prop:metric.limit}.
\begin{enumerate}
\item Terms in which $\varphi$ is contracted with $\nab_{n_k} \eta_{n_k}$. Using Proposition~\ref{prop:eta.etab.limit} (in addition to Propositions~\ref{prop:gamma} and \ref{prop:metric.limit}), 
\begin{equation}\label{eq:Ricci.harder.4}
\begin{split}
&\: \int_0^{\ub_*}\int_{u_1}^{u_2} \int_{S_{u',\ub}} \langle \varphi, \nab_{n_k}\widehat{\otimes}_{n_k}\eta_{n_k} \rangle \Omg_{n_k}^2 \,\ud A_{\gamma_{n_k}}\, \ud u'\,\, \ud \ub \\
\to &\: \int_0^{\ub_*}\int_{u_1}^{u_2} \int_{S_{u',\ub}} \langle \varphi, \nab_\infty \widehat{\otimes}_\infty\eta_\infty  \rangle  \Omg_\infty^2 \,\ud A_{\gamma_\infty}\, \ud u'\,\, \ud \ub.
\end{split}
\end{equation}
\item Terms in which $\varphi$ is contracted with a \underline{product} of two Ricci coefficients. As in the proof of Proposition~\ref{prop:Ricci.easy}, it can be checked that in any such products, at least one term has a \underline{strong} $L^2_uL^2_{\ub}L^2(S_{u,\ub})$ limit while the other term has at least a weak $L^2_uL^2_{\ub}L^2(S_{u,\ub})$ limit. Thus, using Propositions~\ref{prop:gamma}, \ref{prop:metric.limit}, \ref{prop:eta.etab.limit}, \ref{prop:chi.limit}, \ref{prop:trch.weak.limit} and \ref{prop:trch.imp}, we obtain
\begin{equation}\label{eq:Ricci.harder.5}
\begin{split}
&\:\int_0^{\ub_*}\int_{u_1}^{u_2} \int_{S_{u',\ub}} \langle \varphi, \f 12\trchb_{n_k} \chih_{n_k} - \f12 \trch_{n_k} \chibh_{n_k} + \eta_{n_k} \widehat{\otimes}_{n_k}\eta_{n_k} \rangle \Omg_{n_k}^2 \,\ud A_{\gamma_{n_k}}\, \ud u'\,\, \ud \ub\\
\to &\: \int_0^{\ub_*}\int_{u_1}^{u_2} \int_{S_{u',\ub}} \langle \varphi, \f 12\trchb_\infty \chih_\infty - \f12 \trch_\infty \chibh_\infty + \eta_\infty \widehat{\otimes}_\infty\eta_\infty \rangle  \Omg_\infty^2 \,\ud A_{\gamma_\infty}\, \ud u'\,\, \ud \ub .
\end{split}
\end{equation}
\item Terms in which $(\nab_3)_{n_k}\varphi$ is contracted with $\chih_{n_k}$. Expanding $(\nab_3)_{n_k}\varphi$ using \eqref{nab3.def}, we see that 
\begin{equation*}
\begin{split}
&\: \int_0^{\ub_*}\int_{u_1}^{u_2} \int_{S_{u',\ub}} \langle (\nab_3)_{n_k}\varphi, \chih_{n_k} \rangle \Omg_{n_k}^2 \,\ud A_{\gamma_{n_k}}\, \ud u'\,\, \ud \ub \\
= &\: \int_0^{\ub_*}\int_{u_1}^{u_2} \int_{S_{u',\ub}} (\f{\rd\varphi^{AB}}{\rd u} + b_{n_k}^C (\nab_{n_k})_C \varphi^{AB} - 2(\nab_{n_k})_C b_{n_k}^{(A} \varphi^{B)C}) (\chih_{n_k})_{AB}  \Omg_{n_k} \,\ud A_{\gamma_{n_k}}\, \ud u'\,\, \ud \ub \\
&\: + \int_0^{\ub_*}\int_{u_1}^{u_2} \int_{S_{u',\ub}} 2 (\gamma_{n_k}^{-1})^{C(A|}(\chib_{n_k})_{CD} \varphi^{|B)D} (\chih_{n_k})_{AB}  \Omg_{n_k}^2 \,\ud A_{\gamma_{n_k}}\, \ud u'\,\, \ud \ub.
\end{split}
\end{equation*}
For the terms on the first line on the RHS, we use Propositions~\ref{prop:gamma}, \ref{prop:metric.limit} and \ref{prop:chi.limit} to pass to the limit. For the second line, we note that after expressing $\chi_{n_k} = \chih_{n_k} + \f12 \gamma_{n_k} \trch_{n_k}$, we have a term involving $\chih\otimes\chibh$. We therefore pass to the limit using Proposition~\ref{prop:chihchibh}, in addition to Propositions~\ref{prop:gamma}, \ref{prop:metric.limit}, \ref{prop:chi.limit}, \ref{prop:trch.weak.limit} and \ref{prop:trch.imp}. In summary, we obtain
\begin{equation}\label{eq:Ricci.harder.6}
\begin{split}
&\: \int_0^{\ub_*}\int_{u_1}^{u_2} \int_{S_{u',\ub}} \langle (\nab_3)_{n_k}\varphi, \chih_{n_k} \rangle \Omg_{n_k}^2 \,\ud A_{\gamma_{n_k}}\, \ud u'\,\, \ud \ub \\
\to &\: \int_0^{\ub_*}\int_{u_1}^{u_2} \int_{S_{u',\ub}} \langle (\nab_3)_\infty \varphi, \chih_\infty \rangle \Omg_\infty^2 \,\ud A_{\gamma_\infty}\, \ud u'\,\, \ud \ub.
\end{split}
\end{equation}
\end{enumerate}
Combining \eqref{eq:Ricci.harder.3}--\eqref{eq:Ricci.harder.6}, we obtain \eqref{eq:Ricci.harder.1} from \eqref{eq:Ricci.harder.2}.

As we mentioned in the beginning of the proof, the other equations can all be treated similarly. The only additional thing to note is that with the quadratic terms in the Ricci coefficients (Case~(2) above), in general there are also products with no factors having a strong limit, but all of these terms can be handled using compensated compactness in Proposition~\ref{prop:chihchibh}. We omit the details. \qedhere
\end{proof}

\subsection{The non-vanishing (weak) Ricci curvature components of the limit spacetime}\label{sec:eq.Ricci.trans.2}

\begin{proposition}\label{prop:trch.limit.eqn}
$\trchb^{\pm}_{\infty}$ and $\trch^{\pm}_{\infty}$ satisfy \eqref{eq:trchb} and \eqref{eq:trch} respectively.
\end{proposition}
\begin{proof}
We will only derive equation \eqref{eq:trchb} for $\trchb^{\pm}_\infty$; the equation \eqref{eq:trch} for $\trch^{\pm}_\infty$ is similar (and  simpler).

Fix $0\leq u_1< u_2\leq u_*$.

We begin with the fact that \eqref{Ric33} holds for all $(\mathcal M, g_{n_k})$. Take a $C^1$ function $\varphi:[0,u_*]\times \mathbb S^2\to \mathbb R$. Multiply \eqref{Ric33} by $\varphi(u,\vartheta) \xi_\ell(u)$, where for every $\ell\in \mathbb N$ with $\ell^{-1}\leq \f{u_2 - u_1}{2}$, $\xi_\ell:[0,u_*]\to \mathbb R$ is defined to be the following (Lipschitz) cutoff function:
\begin{equation}\label{def:xi}
\xi_\ell(u):= \begin{cases}
0 & \mbox{if $u\in [0, u_1)$}\\
\ell (u-u_1) & \mbox{if $u\in [u_1, u_1+\ell^{-1})$}\\
1 & \mbox{if $u\in [u_1+\ell^{-1}, u_2-\ell^{-1})$}\\
- \ell (u - u_2) & \mbox{if $u\in [u_2-\ell^{-1}, u_2)$}\\
0 & \mbox{if $u\in [u_2, u_*)$}
\end{cases}.
\end{equation}

Therefore, the following holds for every $\ub\in [0,\ub_*]$:
\begin{equation}\label{eq:trch.limit.eqn.1}
\begin{split}
&\: \ell \int_{u_2-\ell^{-1}}^{u_2} \int_{S_{u,\ub}}  \varphi \Omg_{n_k} \trchb_{n_k} \,\mathrm{dA}_{\gamma_{n_k}}\,\ud u - \ell \int^{u_1+\ell^{-1}}_{u_1} \int_{S_{u,\ub}} \varphi \Omg_{n_k} \trchb_{n_k} \,\mathrm{dA}_{\gamma_{n_k}}\,\ud u \\
= &\: \int_{u_1}^{u_2} \int_{S_{u,\ub}} \xi_\ell ((e_3\varphi)\trchb_{n_k} -4\varphi\omb_{n_k} \trchb_{n_k}+ \f12 \varphi (\trchb_{n_k})^2 - \varphi|\chibh_{n_k}|_{\gamma_{n_k}}^2)\Omg_{n_k}^2\,\mathrm{dA}_{\gamma_{n_k}} \,\ud u.
\end{split}
\end{equation}

We now take $k\to +\infty$. Note that we can\underline{not} just replace the $n_k$'s in \eqref{eq:trch.limit.eqn.1} by $\infty$ because of the term $|\chibh_{n_k}|_{\gamma_{n_k}}^2 \Omg_{n_k}^2$ (see~Proposition~\ref{prop:nu.convergence}). On the other hand, in all the \underline{other} terms which are quadratic in the Ricci coefficients, there must be at least one fact which has a strong $L^2_uL^2(S_{u,\ub})$ limit (for every $\ub\in [0,\ub_*]$). Indeed, using Propositions~\ref{prop:gamma}, \ref{prop:metric.limit}, \ref{prop:chi.limit}, \ref{prop:trch.weak.limit}, \ref{prop:trch.imp} and \ref{prop:nu.convergence}, we obtain
\begin{equation}\label{eq:trch.limit.eqn.2}
\begin{split}
&\: \ell \int_{u_2-\ell^{-1}}^{u_2} \int_{S_{u,\ub}}  \varphi \Omg_{\infty} \trchb_{\infty} \,\mathrm{dA}_{\gamma_{\infty}}\,\ud u - \ell \int^{u_1+\ell^{-1}}_{u_1} \int_{S_{u,\ub}} \varphi \Omg_{\infty} \trchb_{\infty} \,\mathrm{dA}_{\gamma_{\infty}}\,\ud u \\
= &\: \int_{u_1}^{u_2} \int_{S_{u,\ub}} \xi_\ell (((e_3)_\i \varphi)\trchb_{\infty} -4 \varphi \omb_{\infty} \trchb_{\infty}+ \f12 \varphi (\trchb_{\infty})^2 - \varphi|\chibh_{\infty}|_{\gamma_{\infty}}^2)\Omg_{\infty}^2\,\mathrm{dA}_{\gamma_{\infty}} \,\ud u \\
&\: -\int_{(u_1, u_2)\times \{\ub\}\times \mathbb S^2} \xi_\ell \varphi\,\ud \underline{\nu}_{\ub}.
\end{split}
\end{equation}

Finally, we take $\ell \to +\infty$ in \eqref{eq:trch.limit.eqn.2}. For the LHS, we use Lemma~\ref{lem:trace}; for the RHS, we use the dominated convergence theorem. We then obtain
\begin{equation*}
\begin{split}
&\: \int_{S_{u_2,\ub}}  \varphi \Omg_{\infty} \trchb^-_{\infty} \,\mathrm{dA}_{\gamma_{\infty}} - \int_{S_{u_1,\ub}} \varphi \Omg_{\infty} \trchb_{\infty}^+ \,\mathrm{dA}_{\gamma_{\infty}} \\
= &\: \int_{u_1}^{u_2} \int_{S_{u,\ub}} (((e_3)_\i \varphi)\trchb_{\infty} -4 \varphi \omb_{\infty} \trchb_{\infty}+ \f12 \varphi (\trchb_{\infty})^2 - \varphi|\chibh_{\infty}|_{\gamma_{\infty}}^2)\Omg_{\infty}^2\,\mathrm{dA}_{\gamma_{\infty}} \,\ud u \\
&\: -\int_{(u_1, u_2)\times \{\ub\}\times \mathbb S^2}  \varphi\,\ud \underline{\nu}_{\ub},
\end{split}
\end{equation*}
which is exactly the equation \eqref{eq:trchb}. \qedhere
\end{proof}

\subsection{Propagation equation for the null dust}\label{sec.prop.dust}
In this subsection, we prove transport equations for $\ud\nu_u$ and $\ud \nub_{\ub}$ (recall Proposition~\ref{prop:nu.convergence}). The main result is the following proposition:

\begin{proposition}\label{prop:nu.transport}
For every $0\leq u_1 <u_2\leq u_*$, and every $C^1_c$ function $\varphi: [0,u_*]\times (0,\ub_*) \times \mathbb S^2\to \mathbb R$,
$$\int_{\{u_2\}\times (0,\ub_*) \times \mathbb S^2} \varphi \,\ud \nu_{u_2} = \int_{\{u_1\}\times (0,\ub_*) \times \mathbb S^2} \varphi\,\ud\nu_{u_1} + \int_{u_1}^{u_2} \int_{\{u\}\times (0,\ub_*) \times \mathbb S^2} (\f{\rd\varphi}{\rd u} + \nab_{b_\infty} \varphi ) \,\ud \nu_{u}\,\ud u.$$
Similarly, for every $0\leq \ub_1 <\ub_2 \leq \ub_*$, and every $C^1$ function $\varphi: (0,u_*) \times [0,\ub_*] \times \mathbb S^2\to \mathbb R$,
$$\int_{(0,u_*) \times \{\ub_2\} \times \mathbb S^2} \varphi \,\ud \nub_{\ub_2} = \int_{(0,u_*) \times \{\ub_1\} \times \mathbb S^2} \varphi\,\ud\nub_{\ub_1} + \int_{\ub_1}^{\ub_2} \int_{(0,u_*) \times \{\ub\} \times \mathbb S^2} \f{\rd\varphi}{\rd \ub}  \,\ud \nub_{\ub}\,\ud \ub.$$
\end{proposition}

The proof of Proposition~\ref{prop:nu.transport} relies on the next two propositions. We refer the reader to the end of this subsection for the conclusion of the proof of Proposition~\ref{prop:nu.transport}.

\begin{proposition}
The following identity holds for $(\mathcal M, g_n)$ for all $n\in \mathbb N$ and for $(\mathcal M, g_\infty)$:
\begin{equation}\label{eq:main.transport.id}
\begin{split}
&\: \int_0^{\ub_*} \int_{S_{u_2,\ub}} \psi \Omega^2 |\chih|_{\gamma}^2 \,\mathrm{dA}_{\gamma}\, \ud \ub  - \int_0^{\ub_*} \int_{S_{u_1,\ub}} \psi \Omega^2 |\chih|_{\gamma}^2 \,\mathrm{dA}_{\gamma}\,\ud \ub \\
=&\: \int_0^{\ub_*}\int_{u_1}^{u_2} \int_{S_{u,\ub}} \chih\cdot (2\nab\widehat{\otimes} \eta  - \trch \chibh + 2\eta\widehat{\otimes} \eta) \Omg^3 \,\mathrm{dA}_{\gamma}\, \ud u\, \ud \ub \\
&\: + \int_0^{\ub_*} \int_{u_1}^{u_2} \int_{S_{u,\ub}}  (\f{\rd\psi}{\rd u} + \nab_b\psi) (\Omega^2 |\chih|_{\gamma}^2) \,\mathrm{dA}_\gamma\,\ud u\,\ud\ub.
\end{split}
\end{equation}
Similarly, $(\mathcal M, g_n)$ for all $n\in \mathbb N$ and for $(\mathcal M, g_\infty)$
\begin{equation}\label{eq:main.transport.id.b}
\begin{split}
&\: \int_0^{u_*} \int_{S_{u,\ub^2}} \psi \Omega^2 |\chibh|_{\gamma}^2 \,\mathrm{dA}_{\gamma}\, \ud u  - \int_0^{u_*} \int_{S_{u,\ub_1}} \psi \Omega^2 |\chibh|_{\gamma}^2 \,\mathrm{dA}_{\gamma}\,\ud u \\
=&\: \int_0^{u_*} \int_{\ub_1}^{\ub_2}  \int_{S_{u,\ub}} \chibh\cdot (2\nab\widehat{\otimes} \etab  - \trchb \chih + 2\etab\widehat{\otimes} \etab) \Omg^3 \,\mathrm{dA}_{\gamma}\, \ud \ub \, \ud u + \int_0^{u_*} \int_{\ub_1}^{\ub_2}  \int_{S_{u,\ub}}  \f{\rd\psi}{\rd \ub} (\Omega^2 |\chibh|_{\gamma}^2) \,\mathrm{dA}_\gamma \, \ud \ub \, \ud u.
\end{split}
\end{equation}
\end{proposition}
\begin{proof}
We will only prove \eqref{eq:main.transport.id}; the proof of \eqref{eq:main.transport.id.b} is slightly simpler. 

Note that by Theorem~\ref{thm:ext.est} and the results in Section~\ref{sec:existence}, $(\mathcal M, g_n)$ and $(\mathcal M, g_\i)$ are angularly regular. Moreover, Propositions~\ref{prop:null.structure} and \ref{prop:Ricci.harder} shows that in these spacetimes,
\begin{equation}\label{eq:nab3chih.in.proof}
\nab_3\chih+\frac 1 2 \trchb \chih =\nab\widehat{\otimes} \eta+2\omegab \chih-\frac 12 \trch \chibh +\eta\widehat{\otimes} \eta
\end{equation}
is satisfied in the weak integrated sense of Definition~\ref{def:weaker.transport}

It therefore suffices to show that if the equation \eqref{eq:nab3chih.in.proof} is satisfied in the weak integrated sense of Definition~\ref{def:weaker.transport} in an angularly regular spacetime, then in fact \eqref{eq:main.transport.id} holds.

According to Definition~\ref{def:weaker.transport}, that \eqref{eq:nab3chih.in.proof} is satisfied in the weak integrated sense means
\begin{equation}\label{eq:weak.chih}
\begin{split}
&\: \int_0^{\ub_*} \int_{S_{u_2,\ub}} \langle\varphi, \chih\rangle \Omg \,\mathrm{dA}_{\gamma}\, \ud \ub  - \int_0^{\ub_*} \int_{S_{u_1,\ub}} \langle\varphi, \chih \rangle \Omg \,\mathrm{dA}_{\gamma}\,\ud \ub \\
=&\: \int_0^{\ub_*}\int_{u_1}^{u_2} \int_{S_{u,\ub}} \langle \varphi, (\nab\widehat{\otimes} \eta + \f 12 \trchb \chih -\frac 12 \trch \chibh +\eta\widehat{\otimes} \eta)\rangle \Omg^2 \,\mathrm{dA}_{\gamma}\, \ud u\,\, \ud \ub \\
&\: + \int_0^{\ub_*} \int_{u_1}^{u_2} \int_{S_{u,\ub}}  \langle \nab_3\varphi, \chih \rangle \Omg^2  \,\mathrm{dA}_\gamma\,\ud u\,\ud\ub
\end{split}
\end{equation}
for every smooth and compactly supported $S$-tangent $2$-tensor $\varphi^{AB}$. By angular regularity and H\"older's inequality, one verifies that 
$$\|\trchb \chih\|_{L^2_u L^2_{\ub} L^2(S)} + \| \nab\widehat{\otimes} \eta - \frac 12 \trch \chibh +\eta\widehat{\otimes} \eta\|_{L^2_u L^2_{\ub} L^2(S)} <+\infty,\quad \|\chih\|_{L^2_{\ub} L^\i_u L^2(S)}<+\infty.$$
It therefore follows from a density argument that \eqref{eq:weak.chih} holds for all $\varphi$ such that $\varphi \in C^0_u L^2_{\ub} L^2(S)$ and $\nab_3\varphi\in L^2_{\ub} L^1_u L^2(S)$. In particular, we can choose $\varphi^{AB} = \psi\Omega \chi^{AB}$, where $\psi$ is a $C^1$ function to obtain \eqref{eq:main.transport.id}. \qedhere
\end{proof}

\begin{proposition}\label{prop:terms.in.nu.eqn}
Taking the subsequence $n_k$ as in the end of Section~\ref{sec:existence}, the following terms in \eqref{eq:main.transport.id} obey
\begin{equation}\label{eq:terms.in meas.prop}
\begin{split}
&\: \int_{[0,u_*]\times [0,\ub_*]\times\mathbb S^2} \psi (\underbrace{2 \chih_{n_k}\cdot_{n_k} \nab_{n_k} \widehat{\otimes} \eta_{n_k}}_{=:\mathrm{I}} \underbrace{- \trch_{n_k} \chih_{n_k} \cdot_{n_k} \chibh_{n_k}}_{=:\mathrm{II}} +  \underbrace{2 \chih_{n_k}\cdot_{n_k} (\eta_{n_k}\widehat{\otimes} \eta_{n_k})}_{=:\mathrm{III}} ) \Omega_{n_k}^{3} \,\mathrm{dA}_{\gamma_{n_k}}\,\ud u\,\ud\ub \\
\to &\: \int_{[0,u_*]\times [0,\ub_*]\times\mathbb S^2} \psi (2 \chih_\infty \cdot_\infty \nab_\infty \widehat{\otimes} \eta_\infty - \trch_\infty \chih_\infty \cdot_\infty \chibh_\infty +  2 \chih_\infty\cdot_\infty (\eta_\infty\widehat{\otimes} \eta_\infty) ) \Omega_\infty^{3} \,\mathrm{dA}_{\gamma_\infty}\,\ud u\,\ud\ub.
\end{split}
\end{equation}

A similar convergence statement holds for the corresponding terms in \eqref{eq:main.transport.id.b}.
\end{proposition}
\begin{proof}
We will only prove \eqref{eq:terms.in meas.prop}; the terms in \eqref{eq:main.transport.id.b} can be treated in the same way.

We will use the following two facts. First of all, $g_{n_k}\to g_\infty$ uniformly, and since the integrand in every term is at least in $L^1_u L^1_{\ub} L^1(S_{u,\ub})$, we can easily pass to the limit $k\to +\infty$ in all occurrences of $\cdot_{n_k}$, $\Omg_{n_k}$ and $\mathrm{dA}_{\gamma_{n_k}}$. Second, we use the standard fact that whenever there is a product of two quantities, say $f_{n_k}$ and $h_{n_k}$, such that $f_{n_k}$ converges weakly to $f_\infty$ in $L^2_u L^2_{\ub} L^2(S)$and $h_{n_k}$ converges to $h_\infty$ in the $L^2_u L^2_{\ub} L^2(S)$ norm, then $f_{n_k} h_{n_k}$ converges to $f_\infty h_\infty$ in the sense of distribution.

\pfstep{Step~1: Term $\mathrm{I}$} Note that $\nab_{n_k} \widehat{\otimes} \eta_{n_k}$ converges in the $C^0_u C^0_{\ub} L^4(S_{u,\ub})$ norm (by Proposition~\ref{prop:eta.etab.limit}), and thus in particular converges in the $L^2_u L^2_{\ub} L^2(S_{u,\ub})$ norm. On the other hand, by Proposition~\ref{prop:chi.limit}, $\chih_{n_k}$ converges weakly in $L^2_{\ub} L^2(S_{u,\ub})$ for all $u$, and thus in particular converges weakly in $L^2_u L^2_{\ub} L^2(S_{u,\ub})$. By the remark in the beginning of the proof, we deduce that $\mathrm{I}$ attains the limit as indicated in \eqref{eq:terms.in meas.prop} as $k\to +\infty$.

\pfstep{Step~2: Term $\mathrm{II}$} For this term we use compensated compactness. Indeed, by Proposition~\ref{prop:chihchibh}, $\chih_{n_k} \cdot_{n_k} \chibh_{n_k} \rightharpoonup \chih_\infty \cdot_\infty \chibh_\infty$ weakly in $L^2_u L^2_{\ub} L^2(S)$. On the other hand, Proposition~\ref{prop:trch.imp} in particular implies that $\trch_{n_k}\to \trch_\infty$ in the $L^2_u L^2_{\ub} L^2(S)$ norm. As in Step~1, we then conclude using the remark in the beginning of the proof.

\pfstep{Step~3: Term $\mathrm{III}$} By Proposition~\ref{prop:eta.etab.limit}, $\eta_{n_k}\widehat{\otimes} \eta_{n_k}$ converges in the $C^0_u C^0_{\ub} C^0(S_{u,\ub})$ norm, and thus also converges in the $L^2_u L^2_{\ub} L^2(S_{u,\ub})$ norm. On the other hand, as argued in Step~1, $\chih_{n_k}$ converges weakly in $L^2_u L^2_{\ub} L^2(S_{u,\ub})$. As before, we then conclude using the remark in the beginning of the proof. \qedhere
\end{proof}

\begin{proof}[Proof of Proposition~\ref{prop:nu.transport}]
We start with the identity \eqref{eq:main.transport.id} for $(\mathcal M, g_{n_k})$. Now take $k\to+\infty$ and use Propositions~\ref{prop:nu.convergence} and \ref{prop:terms.in.nu.eqn} to obtain
\begin{equation}\label{eq:main.transport.id.2}
\begin{split}
&\: \int_0^{\ub_*} \int_{S_{u_2,\ub}} \psi \Omega_\infty^2 |\chih_\infty|_{\gamma_\infty}^2 \,\mathrm{dA}_{\gamma_\infty}\, \ud \ub  - \int_0^{\ub_*} \int_{S_{u_1,\ub}} \psi \Omega_\infty^2 |\chih_\infty|_{\gamma_\infty}^2 \,\mathrm{dA}_{\gamma_\infty}\,\ud \ub \\
&\: + \int_{\{u_2\}\times (0,\ub_*) \times \mathbb S^2} \psi \,\ud\nu_{u_2} - \int_{\{u_1\}\times [0,\ub_*]\times \mathbb S^2} \psi \,\ud\nu_{u_1} \\
=&\: \int_0^{\ub_*}\int_{u_1}^{u_2} \int_{S_{u,\ub}} \chih_\infty\cdot_\infty (2\nab_\infty\widehat{\otimes}_\infty \eta_\infty  - \trch_\infty \chibh_\infty + 2\eta_\infty\widehat{\otimes}_\infty \eta_\infty) \Omg_\infty^3 \,\mathrm{dA}_{\gamma}\, \ud u\,\, \ud \ub \\
&\: + \int_0^{\ub_*} \int_{u_1}^{u_2} \int_{S_{u,\ub}}  (\f{\rd\psi}{\rd u} + \nab_{b_\infty}\psi) (\Omega_\infty^2 |\chih_\infty|_{\gamma_\infty}^2) \,\mathrm{dA}_{\gamma_\infty} \,\ud u\,\ud\ub \\
&\: + \int_{u_1}^{u_2} \int_{\{u\}\times (0,\ub_*)\times \mathbb S^2} (\f{\rd\psi}{\rd u} + \nab_{b_\infty}\psi) \,\ud\nu_{u}\,\ud u.
\end{split}
\end{equation}
Finally, subtract from \eqref{eq:main.transport.id.2} the equation \eqref{eq:main.transport.id} for $(\mathcal M, g_\infty)$. We then obtain the desired transport equation for $\ud\nu_u$. The transport equation for $\ud\nub_{\ub}$ can be derived analogously. \qedhere
\end{proof}

\subsection{Higher order transport equations for the Ricci coefficients}\label{sec:higher.Ricci}

We next derive the equations \eqref{eq:mu.0}, \eqref{eq:mub.0}, \eqref{eq:Xtrch.0} and \eqref{eq:Xtrchb.0} for the derivatives of the Ricci coefficients. 

\begin{proposition}\label{prop:eq.div.eta}
$\mu_\infty$ and $\mub_\infty$  respectively obeys \eqref{eq:mu.0} and \eqref{eq:mub.0} in the integrated sense (Definition~\ref{def:weak.transport}).
\end{proposition}
\begin{proof}
By Proposition~\ref{prop:mu.background}, we know that \eqref{eq:mu.0} and \eqref{eq:mub.0} are satisfied for $(\mathcal M, g_{n_k})$ for all $k$. It therefore suffices to check that we can take the limit $k\to +\infty$. Except for having some cubic terms, this is similar to Proposition~\ref{prop:Ricci.easy}.

Now, \eqref{eq:mu.0} and \eqref{eq:mub.0} can be schematically written as
\begin{align}
\nab_4 \mu= &\: \nab\chi \star (\eta,\etab) + \nab (\eta,\etab) \star \chi + K\star \chi + \chi \star (\eta,\etab) \star (\eta,\etab), \label{eq:mu.1}\\
\nab_3 \mub = &\:  \nab\chib \star (\eta,\etab) + \nab (\eta,\etab) \star \chib + K \star \chib + \chib \star (\eta,\etab) \star (\eta,\etab),\label{eq:mub.1}
\end{align}
where $\nab\chi\star(\eta,\etab)$ denotes a linear combination of contractions (with respect to $\gamma$) of $\nab\chih \eta$, $\nab\chih \etab$, $\nab\trch \eta$ and $\nab\trch\etab$; and similarly for other terms.

Consider now \eqref{eq:mu.1} since \eqref{eq:mub.1} is similar. For $i=1,2$, the terms
$$\int_{S_{u,\ub_i}} \langle\varphi, \mu_{n_k} \rangle \Omg_{n_k} \,\ud A_{\gamma_{n_k}} \to \int_{S_{u,\ub_i}} \langle\varphi, \mu_{\infty} \rangle \Omg_{\infty} \,\ud A_{\gamma_{\infty}}$$
due to \eqref{eq:mu.def}, Propositions~\ref{prop:metric.limit}, \ref{prop:Christoffel}, \ref{prop:metric.limit} and \ref{prop:eta.etab.limit}.

It thus remains to pass to the $k\to +\infty$ limit for the terms which are integrated in $\ub$:
\begin{enumerate}
\item Term with $\nab_4\varphi$. In a completely analogous manner as the proof of \eqref{eq:Ricci4A.6}, we have
$$\int_{\ub_1}^{\ub_2} \int_{S_{u,\ub'}}  \langle (\nab_4)_{n_k} \varphi,\eta_{n_k} \rangle \Omg_{n_k}^2 \,\ud A_{\gamma_{n_k}}\, \ud \ub'\to \int_{\ub_1}^{\ub_2} \int_{S_{u,\ub'}}  \langle (\nab_4)_{\infty}\varphi,\eta_{\infty} \rangle \Omg_{\infty}^2 \,\ud A_{\gamma_{\infty}}\, \ud \ub'.$$
\item Quadratic terms\footnote{Note that in addition to the inhomogeneous terms in \eqref{eq:mu.1}, the quadratic term includes terms $\trch\mu$ and $\om\mu$ in Definition~\ref{def:weak.transport}.}: $\eta \nab \chi$, $\etab \nab \chi$, $\chi \nab\eta$, $\chi\nab \etab$, $\chi K$, $\omega \nabla \eta$, $\omega K$. By Propositions~\ref{prop:chi.limit} and \ref{prop:trch.weak.limit}, all of $\chih_{n_k}$, $\trch_{n_k}$, $\om_{n_k}$, $\nab_{n_k}\chih_{n_k}$, $\nab_{n_k}\trch_{n_k}$ converge weakly to their (weak) limits in $L^2_{\ub}L^2(S_{u,\ub})$ for all $u$. By Propositions~\ref{prop:Christoffel} and \ref{prop:eta.etab.limit}, $\eta_{n_k}$, $\etab_{n_k}$, $\nab\eta_{n_k}$, $\nab\etab_{n_k}$ all converge to their limits in the $L^\i_u L^\i_{\ub} L^2(S_{u,\ub})$ norm. Hence, all the quadratic terms converge to the appropriate limit in the sense of distribution. 
\item Cubic terms: $\eta\star \eta\star \chi$, $\eta\star \etab\star \chi$, $\etab\star \etab\star \chi$. By Proposition~\ref{prop:eta.etab.limit}, $\eta_{n_k}$ and $\etab_{n_k}$ have pointwise uniform limits. By Proposition~\ref{prop:chi.limit}, $\chi_{n_k}$ (i.e.~both $\trch_{n_k}$ and $\chih_{n_k}$) has a weak $L^2_{\ub} L^2(S_{u,\ub})$ limit for every $u$. It therefore follows that all the cubic terms have the desired limits. \qedhere
\end{enumerate}
\end{proof}

\begin{proposition}\label{prop:trch.top.order}
\begin{enumerate}
\item The equation \eqref{eq:Xtrch.0} is satisfied for all $C^1$ $S$-tangent vector field $\slashed X$, for all $u\in [0,u_*]$ and all $0\leq \ub_1<\ub_2 \leq \ub_*$.
\item The equation \eqref{eq:Xtrchb.0} is satisfied for all $C^1$ $S$-tangent vector field $\slashed X$, for all $\ub\in [0,\ub_*]$ and all $0\leq u_1<u_2 \leq u_*$.
\end{enumerate}

\end{proposition}
\begin{proof}
In view of their similarities, we will only prove (the slightly harder) \eqref{eq:Xtrchb.0}.

By \eqref{eq:Xtrchb.0.vac} in Proposition~\ref{prop:Xtrch.vac}.
\begin{equation}\label{eq:Xtrchb}
(\f{\rd}{\rd u} + \nab_{b_{n_k}}) \slashed X(\Omg_{n_k}^{-1}\trchb_{n_k}) +\frac 12 \slashed X(\trchb_{n_k})^2= -\slashed X|\chibh_{n_k}|_{\gamma_{n_k}}^2 + [\f{\rd}{\rd u} + \nab_{b_{n_k}}, \slashed X](\Omg_{n_k}^{-1}\trchb_{n_k}).
\end{equation}

Fix $0\leq u_1<u_2 \leq u_*$ and let $\xi_\ell$ be a cutoff as in \eqref{def:xi} when $\ell^{-1} \leq \f{u_2-u_1}{2}$. Multiplying \eqref{eq:Xtrchb} by $\xi_\ell$, integrating with respect to $\Omg_{n_k}^2 \, \,\mathrm{dA}_{\gamma_{n_k}}\,\ud u$, and integrating by parts, we obtain that for every $\ub\in [0,\ub_*]$:
\begin{equation*}
\begin{split}
&\: \ell \int_{u_2-\ell^{-1}}^{u_2} \int_{S_{u,\ub}}  \Omg_{n_k}^2 \slashed X (\Omg_{n_k}^{-1} \trchb_{n_k}) \,\mathrm{dA}_{\gamma_{n_k}}\,\ud u - \ell \int^{u_1+\ell^{-1}}_{u_1} \int_{S_{u,\ub}} \Omg_{n_k}^2 \slashed X (\Omg_{n_k}^{-1} \trchb_{n_k}) \,\mathrm{dA}_{\gamma_{n_k}}\,\ud u \\
= &\: \int_{u_1}^{u_2} \int_{S_{u,\ub}} \xi_\ell ([\f{\rd}{\rd u} + \nab_{b_{n_k}}, \slashed X](\Omg_{n_k}^{-1}\trchb_{n_k}) -4\omb_{n_k} \Omg_{n_k}\slashed X(\Omg_{n_k}^{-1}\trchb_{n_k}) - \slashed X|\chibh_{n_k}|_{\gamma_{n_k}}^2)\Omg_{n_k}^2\,\mathrm{dA}_{\gamma_{n_k}} \,\ud u.
\end{split}
\end{equation*}
In order to be able to pass to the $k\to +\infty$ limit, we integrate by parts the last term to obtain
\begin{equation}\label{eq:Xtrchb.after.ibp}
\begin{split}
&\: \ell \int_{u_2-\ell^{-1}}^{u_2} \int_{S_{u,\ub}}  \Omg_{n_k}^2 \slashed X (\Omg_{n_k}^{-1} \trchb_{n_k}) \,\mathrm{dA}_{\gamma_{n_k}}\,\ud u - \ell \int^{u_1+\ell^{-1}}_{u_1} \int_{S_{u,\ub}} \Omg_{n_k}^2 \slashed X (\Omg_{n_k}^{-1} \trchb_{n_k}) \,\mathrm{dA}_{\gamma_{n_k}}\,\ud u \\
= &\: \int_{u_1}^{u_2} \int_{S_{u,\ub}} \xi_\ell ([\f{\rd}{\rd u} + \nab_{b_{n_k}}, \slashed X](\Omg_{n_k}^{-1}\trchb_{n_k}) -4\omb_{n_k} \Omg_{n_k}\slashed X(\Omg_{n_k}^{-1}\trchb_{n_k}) )\Omg_{n_k}^2\,\mathrm{dA}_{\gamma_{n_k}} \,\ud u \\
&\: + \int_{u_1}^{u_2} \int_{S_{u,\ub}} \xi_\ell  (2\slashed X(\log\Omg_{n_k}) + \div_{n_k} \slashed X)\Omg_{n_k}^2|\chibh_{n_k}|_{\gamma_{n_k}}^2 \,\mathrm{dA}_{\gamma_{n_k}} \,\ud u.
\end{split}
\end{equation}
We now argue in a similar manner as Proposition~\ref{prop:trch.limit.eqn}. First we pass to the $k\to +\infty$ limit using Propositions~\ref{prop:gamma}, \ref{prop:Christoffel}, \ref{prop:metric.limit}, \ref{prop:chi.limit}, \ref{prop:trch.weak.limit}, \ref{prop:trch.imp}, \ref{prop:nu.convergence}. Note that except for the $|\chibh_{n_k}|^2_{\gamma_{n_k}}$ terms in the last line of \eqref{eq:Xtrchb.after.ibp}, all the other terms have a most one factor which does not admit a strong limit so that we can replace $n_k$ by $\infty$ in the limit. In particular, because $[\f{\rd}{\rd u} + b_{n_k}, \slashed X]$ is an $S$-tangent vector field, $[\f{\rd}{\rd u} + \nab_{b_{n_k}}, \slashed X](\Omg_{n_k}^{-1}\trchb_{n_k}) \to [\f{\rd}{\rd u} + \nab_{b_{n_k}}, \slashed X](\Omg_{n_k}^{-1}\trchb_{n_k})$ in the $L^2_{u} L^2(S_{u,\ub})$ norm for every $\ub \in [0,\ub_*]$. For the terms on the last line of \eqref{eq:Xtrchb.after.ibp}, we obtain an extra term involving $\ud\nub_{\ub}$ in the limit. 

After taking the $k\to +\infty$ limit we then take $\ell\to +\infty$ (using Lemma~\ref{lem:trace} on the LHS and the dominated convergence theorem on the RHS), we obtain that
\begin{equation*}
\begin{split}
&\: \int_{S_{u_2,\ub}}  \Omg_{\infty}^2 (\slashed X (\Omg_{\infty}^{-1} \trchb_{\infty}))^- \,\mathrm{dA}_{\gamma_{\infty}} - \int_{S_{u_1,\ub}} \Omg_{\infty}^2 (\slashed X (\Omg_{\infty}^{-1} \trchb_{\infty}))^+ \,\mathrm{dA}_{\gamma_{\infty}} \\
= &\: \int_{u_1}^{u_2} \int_{S_{u,\ub}} ([\f{\rd}{\rd u} + \nab_{b_{\infty}}, \slashed X](\Omg_{\infty}^{-1}\trchb_{\infty}) -4\omb_{\infty} \Omg_{\infty}\slashed X(\Omg_{\infty}^{-1}\trchb_{\infty}) )\Omg_{\infty}^2\,\mathrm{dA}_{\gamma_{\infty}} \,\ud u \\
&\: + \int_{u_1}^{u_2} \int_{S_{u,\ub}} (2\slashed X(\log\Omg_{\infty}) + \div_\infty \slashed X)\Omg_{\infty}^2|\chibh_{\infty}|_{\gamma_{\infty}}^2 \,\mathrm{dA}_{\gamma_{\infty}} \,\ud u \\
&\: + \int_{(u_1,u_2)\times \{\ub\} \times \mathbb S^2} (2\slashed X(\log\Omg_{\infty}) + \div_{\infty} \slashed X)\,\ud\underline{\nu}_{\ub},
\end{split}
\end{equation*}
as desired. \qedhere
\end{proof}

\subsection{Renormalized Bianchi equations}\label{sec:renorm.Bianchi}

Recall the definition of the curvature components in Definition~\ref{def:curv}. The set of renormalized Bianchi equations Proposition~\ref{prop:Bianchi} are satisfied by the limit spacetime in an appropriate sense:
\begin{proposition}\label{prop:Bianichi.for.limit}
In the limiting spacetime $(\mathcal M, g_\infty)$, the renormalized Bianchi equations are satisfied in the sense of Definition~\ref{def:Bianchi.integrated}.
\end{proposition}
\begin{proof}
The proof is similar to various previous propositions; we will only indicate the main points. As in Propositions~\ref{prop:Ricci.easy}, \ref{prop:Ricci.harder} and \ref{prop:eq.div.eta}, the main goal will be to make use of the fact that all the equations are satisfied for $(\mathcal M, g_{n_k})$ according to Proposition~\ref{prop:Bianchi} and then take limits.

Now taking the limit of (the integrated form of) \eqref{eq:null.Bianchi.2}--\eqref{eq:null.Bianchi.5} can be done in exactly the same way as in the proof of Proposition~\ref{prop:eq.div.eta}. Indeed, it can be easily checked that except for the top derivative terms $\div \bt$, $\div\betab$, etc., \eqref{eq:null.Bianchi.2}--\eqref{eq:null.Bianchi.5} have schematically the same type of terms as \eqref{eq:mu.0} and \eqref{eq:mub.0}. The top derivative terms can be handled using Propositions~\ref{prop:chi.limit} and \ref{prop:trch.weak.limit} (which guarantees the weak convergence of up to second derivatives of $\chih$, $\trch$, $\chibh$ and $\trchb$) together with Propositions~\ref{prop:gamma}, \ref{prop:metric.limit} and \ref{prop:eta.etab.limit}.
 
Finally, in order to take the limit of (the weak integrated form of) \eqref{eq:null.Bianchi.1} and \eqref{eq:null.Bianchi.6}, we need in addition to use the compensated compactness result in Proposition~\ref{prop:chihchibh} to handle the terms $\chi\nab\chib$, $\chibh\nab\chi$, $\eta\chi\chib$, etc.~(cf.~the difference between Propositions~\ref{prop:Ricci.easy} and \ref{prop:Ricci.harder}). We omit the details. \qedhere
\end{proof}

\subsection{Proof of Theorem~\ref{main.thm}}\label{sec:proof.of.main.thm} We now have all the ingredients to complete the proof of Theorem~\ref{main.thm}:

\begin{proof}[Proof of Theorem~\ref{main.thm}]
\begin{enumerate}
\item This assertion follows directly from Theorem~\ref{ext.thm}.
\item The existence of a $C^0$ limit of the metric components in double null coordinates is a consequence of Propositions~\ref{prop:gamma} and \ref{prop:metric.limit}. The same propositions also give the weak $L^2$ convergence statements for the first derivatives of the metric components.

The weak $L^2$ convergence\footnote{Of course we have in fact proven much stronger convergence statements.} of the Ricci coefficient follow as a consequence of Propositions~\ref{prop:eta.etab.limit}, \ref{prop:chi.limit}, \ref{prop:trch.weak.limit}.

Finally, Proposition~\ref{prop:Ricci.is.metric.derivative} shows that the weak limit of the Ricci coefficients coincide with the Ricci coefficients associated with the limit metric.
\item The weak-* limits \eqref{eq:dnu.def.thm} and \eqref{eq:dnub.def.thm} exist as a consequence of Proposition~\ref{prop:nu.convergence}. The angular regularity of $(\mathcal M,\,g_\infty)$ (see Definition~\ref{double.null.def.2}) follows from Propositions~\ref{prop:gamma}, \ref{prop:isoperimetric}, \ref{prop:Christoffel}, \ref{prop:K.imp}, \ref{prop:metric.limit}, \ref{prop:eta.etab.limit}, \ref{prop:eta.etab.imp}, \ref{prop:chi.limit}, \ref{prop:trch.weak.limit} and \ref{prop:trch.imp}. The angular regularity of $(\{\ud\nu\}_{u\in [0,u_*]},\,\{\ud\nub\}_{\ub \in [0,\ub_*]})$ (see Definition~\ref{def:ang.reg.null.dust}) follows from Propositions~\ref{prop:nu.convergence} and \ref{prop:nu.add.reg}. 

Finally, that $(\mathcal M,\,g_\infty,\,\{\ud\nu\}_{u\in [0,u_*]},\,\{\ud\nub\}_{\ub \in [0,\ub_*]})$ is an angularly regular weak solution to the Einstein--null dust system (see Definition~\ref{def:weak.sol.ang.reg}) follows from Propositions~\ref{prop:Ricci.easy}, \ref{prop:Ricci.harder}, \ref{prop:trch.limit.eqn} and \ref{prop:nu.transport}.
\item The renormalized Bianchi equations follow from Proposition~\ref{prop:Bianichi.for.limit}.
\item The auxiliary equations hold because of Propositions~\ref{prop:equations.for.nabla.g}, \ref{prop:eq.div.eta} and \ref{prop:trch.top.order}. \qedhere
\end{enumerate}
\end{proof}

\section{Uniqueness of the limit}\label{sec:proof.uniqueness}

In this section, we prove the uniqueness theorem (Theorem~\ref{thm:uniqueness}). \textbf{For the whole section, we work under the assumptions of Theorem~\ref{thm:uniqueness}}. In particular, we are given two angularly regular weak solutions $(\mathcal M, g^{(1)}, \{\ud\nu^{(1)}_u\}_{u\in [0,u_*]}, \{\ud\nub^{(1)}_{\ub}\}_{\ub\in [0,\ub_*]})$ and $(\mathcal M, g^{(2)}, \{\ud\nu^{(2)}_u\}_{u\in [0,u_*]}, \{\ud\nub^{(2)}_{\ub}\}_{\ub\in [0,\ub_*]})$ to the Einstein--null dust system. We will first define in \textbf{Section~\ref{sec:def.dist}} a distance function (see \eqref{def:dist}) that controls the difference of the two solutions. The remaining subsections are devoted to controlling this distance function. 
\begin{itemize}
\item In \textbf{Section~\ref{sec:uniqueness.aux.est}}, we prove some easy preliminary estimates.  
\item In \textbf{Section~\ref{sec:diff.transport}}, we prove general estimates for differences of transport equations, and apply them to control the differences of metric components and the Ricci coefficients $\eta$, $\etab$, $\chih$, $\chibh$, $\om$ and $\omb$. The transport equations for $\trch$ and $\trchb$ (and their angular derivatives) will be treated separately in \textbf{Section~\ref{sec:diff.trch}}, because they involve the measure-valued $\ud \nu_u$ and $\ud\nub_{\ub}$ on the RHSs.
\item We then treat the top-order estimates. This is dealt with by a combination of energy estimates for the renormalized curvature components (\textbf{Section~\ref{sec:energy.est}}) and elliptic estimates to handle the top-order derivatives of the Ricci coefficients (\textbf{Section~\ref{sec:elliptic.est}}).
\item In \textbf{Section~\ref{sec:diff.null.dust}}, we estimate the difference of the (measure-valued) null dust.
\end{itemize}
Putting all these together in \textbf{Section~\ref{sec:uniqueness.everything}}, we obtain Theorem~\ref{thm:uniqueness}.

\subsection{Distance function}\label{sec:def.dist}

To proceed, we first introduce a reduction so that the analysis is carried out in a small region. We partition the set $[0,u_*]\times [0,\ub_*]$ into $N^2$ rectangles; namely we write $[0,u_*]\times [0,\ub_*] = \cup_{i=0}^{N-1} \cup_{j=0}^{N-1} [u_i, u_{i+1}]\times [\ub_j,\ub_{j+1}]$ where $u_i = \f{i\times u_*}{N}$ and $\ub_j = \f{j\times \ub_*}{N}$. It suffices to show that for $N$ sufficiently large (depending on the size of $(\mathcal M, g^{(1)}, \{\ud\nu^{(1)}_u\}_{u\in [0,u_*]}, \{\ud\nub^{(1)}_{\ub}\}_{\ub\in [0,\ub_*]})$ and $(\mathcal M, g^{(2)}, \{\ud\nu^{(2)}_u\}_{u\in [0,u_*]}, \{\ud\nub^{(2)}_{\ub}\}_{\ub\in [0,\ub_*]})$), if the two sets of data agree on $\{ u_i \} \times [ \ub_j, \ub_{j+1} ] \times \mathbb S^2$ and $[u_i, u_{i+1}] \times \{\ub_j\}\times \mathbb S^2$, then in fact the two solutions agree in $[u_i, u_{i+1}] \times [ \ub_j, \ub_{j+1} ] \times \mathbb S^2$.

The parameter $N$ will be chosen later. \textbf{For the remainder of the section, we fix some $0\leq i,\,j\leq N-1$, and will only concern ourselves with the region $[u_i, u_{i+1}] \times [ \ub_j, \ub_{j+1} ] \times \mathbb S^2$.} In particular, when applying the definitions or equations from the previous sections, we will replace $[0,u_*]$ by $[u_i, u_{i+1}]$ (and respectively $[0,\ub_*]$ by $[\ub_j, \ub_{j+1}]$).

In order to define a distance function between We define a distance between the two measures $\ud \nu^{(1)}$ and $\ud \nu^{(2)}$.
\begin{equation}\label{def:dist.nu}
\begin{split}
 \mathrm{dist}_\nu(\ud\nu^{(1)},\ud\nu^{(2)}) 
:= &\: \sup_{u'\in [u_i, u_{i+1}]} \sup_{\substack{\varphi(\ub,\vartheta)\in C^\infty_c \\ \|\varphi\|_{L^\i_{\ub} L^2(S_{u',\ub},\gamma^{(1)})}\leq 1} } \left| \int_{H_{u'}} \varphi \,(\ud \nu^{(1)}_{u'} - \f{\sqrt{\det\gamma^{(1)}}}{\sqrt{\det\gamma^{(2)}}}\ud \nu^{(2)}_{u'})\right|  \\
&\: + \sup_{u'\in [u_i, u_{i+1}]} \sup_{ \substack{ \slashed{X}(\ub,\vartheta)\in C^\infty_c \\ \|\slashed X\|_{L^\infty_{\ub} L^2(S_{u',\ub},\gamma^{(1)})}\leq 1} }\left| \int_{H_{u'}} \slashed{\div}^{(1)} \slashed X \,(\ud \nu^{(1)}_{u'} - \f{\sqrt{\det\gamma^{(1)}}}{\sqrt{\det\gamma^{(2)}}}\ud \nu^{(2)}_{u'})\right| .
\end{split}
\end{equation}
Here $\varphi$ is a scalar-valued function and $\slashed{X}$ is an $S$-tangent vector field.

Similarly, we define $\mathrm{dist}_{\nub}(\ud\nub^{(1)},\ud\nub^{(2)})$ after flipping all $u$ and $\ub$, i.e.
\begin{equation}\label{def:dist.nub}
\begin{split}
\mathrm{dist}_{\nub} (\ud\nub^{(1)},\ud\nub^{(2)}) 
:= &\: \sup_{\ub'\in [\ub_j, \ub_{j+1}]} \sup_{ \substack{ \varphi(u,\vartheta)\in C^\infty_c \\ \|\varphi\|_{L^\i_u L^2(S_{u,\ub'}, \gamma^{(1)})} \leq 1}} \left| \int_{\Hb_{\ub'}} \varphi \,(\ud \nub^{(1)}_{\ub'} - \f{\sqrt{\det\gamma^{(1)}}}{\sqrt{\det\gamma^{(2)}}}\ud \nub^{(2)}_{\ub'})\right|   \\
&\: + \sup_{\ub'\in [\ub_j, \ub_{j+1}]} \sup_{\substack{ \slashed{X}(u,\vartheta)\in C^\infty_c \\ \|\slashed X\|_{L^\infty_u L^2(S,\gamma^{(1)})}\leq 1 }} \left| \int_{\Hb_{\ub'}} \slashed{\div}^{(1)} \slashed X \,(\ud \nub^{(1)}_{\ub'} - \f{\sqrt{\det\gamma^{(1)}}}{\sqrt{\det\gamma^{(2)}}}\ud \nub^{(2)}_{\ub'})\right| .
\end{split}
\end{equation}

We then define a distance function between two solutions to the Einstein--null dust system:
\begin{equation}\label{def:dist}
\begin{split}
\mathrm{dist}:= &\: \sum_{\slashed g \in \{\gamma,\,\log\det\gamma,\,b,\,\log\Omg\}} \|\slashed g^{(1)} - \slashed g^{(2)}\|_{L^\i_u L^\i_{\ub} W^{1,2}(S_{u,\ub},\gamma^{(1)})} + \sum_{\psi \in \{ \eta, \etab\} } \|\psi^{(1)} - \psi^{(2)}\|_{L^\i_u L^\i_{\ub} L^{2}(S_{u,\ub},\gamma^{(1)})}  \\
&\: + \sum_{\psi \in \{ \eta, \etab\} } \|\psi^{(1)} - \psi^{(2)}\|_{L^\i_u L^2_{\ub} W^{1,2}(S_{u,\ub},\gamma^{(1)})} + \sum_{\psi \in \{ \eta, \etab\} } \|\psi^{(1)} - \psi^{(2)}\|_{L^\i_{\ub} L^2_u W^{1,2}(S_{u,\ub},\gamma^{(1)})} \\
&\: + \sum_{\psi_H \in \{\trch,\, \chih,\,\om\}} \|\psi_H^{(1)} - \psi_H^{(2)} \|_{L^\i_u L^2_{\ub} W^{1,2} (S_{u,\ub}, \gamma^{(1)})} + \sum_{\psi_H \in \{\trchb,\,\chibh,\,\omb\}} \|\psi_{\Hb}^{(1)} - \psi_{\Hb}^{(2)} \|_{L^\i_{\ub} L^2_{u} W^{1,2} (S_{u,\ub}, \gamma^{(1)})} \\
&\: + \mathrm{dist}_{\nu}(\ud \nu^{(1)},\ud \nu^{(2)}) + \mathrm{dist}_{\nub}(\ud \nub^{(1)},\ud \nub^{(2)}).
\end{split}
\end{equation}

In the following subsections of the section, we will control each piece in \eqref{def:dist} and prove an estimate $\mathrm{dist} \ls \f{\mathrm{dist}}{N^{\f 14}}$. We will use the following convention for constants. The angular regularity assumption of Theorem~\ref{thm:uniqueness} (see Definitions~\ref{double.null.def.2} and \ref{def:ang.reg.null.dust}) gives control of the geometric quantities associated to $(\mathcal M, g^{(1)}, \ud\nu^{(1)}, \ud\nub^{(1)})$ and $(\mathcal M, g^{(1)}, \ud\nu^{(2)}, \ud\nub^{(2)})$ (in the full region $[0,u_*]\times [0,\ub_*]\times \mathbb S^2$). \textbf{All implicit constants in $\ls$ in this section will depend only on the estimates for the geometric quantities given in Definitions~\ref{double.null.def.2} and \ref{def:ang.reg.null.dust}. Importantly, they are independent of $N$.} (Moreover, for instance when we say we use Definition~\ref{double.null.def.2}, we mean that we use the corresponding quantitative estimates.)

\subsection{Some auxiliary estimates}\label{sec:uniqueness.aux.est}

\begin{proposition}\label{prop:diff.gamma.comp}
For every $u$, $\ub$,
$$\|\phi\|_{L^2(S_{u,\ub},\gamma^{(1)})} \ls \|\phi\|_{L^2(S_{u,\ub},\gamma^{(2)})} \ls \|\phi\|_{L^2(S_{u,\ub},\gamma^{(1)})}.$$
\end{proposition}
\begin{proof}
As in the proof of Proposition~\ref{prop:norms.compare}, for $i=1,2$, $L^2(S_{u,\ub},\gamma^{(i)})$ is comparable to $L^2(S_{u,\ub},(\gamma^{(i)})_{0,0})$, where $(\gamma^{(i)})_{0,0}$ is the metric that agrees with $\gamma^{(i)}$ on $S_{0,0}$ and satisfies $\slashed{\mathcal L}_{\f{\rd}{\rd \ub}} (\gamma^{(i)})_{0,0} = \slashed{\mathcal L}_{\f{\rd}{\rd u}} (\gamma^{(i)})_{0,0} = 0.$

Therefore, to establish the proposition, it suffices to show that $L^2(S_{0,0},\gamma^{(1)})$ and $L^2(S_{0,0},\gamma^{(2)})$ are comparable, which is obviously the case since by assumption $\gamma^{(1)} \equiv \gamma^{(2)}$ on $S_{0,0}$. \qedhere
\end{proof}

\begin{proposition}\label{prop:gamma.inverse.diff}
\begin{equation}\label{eq:gamma.inverse.diff.main}
\| (\gamma^{(1)})^{-1} - (\gamma^{(2)})^{-1} \|_{L^\i_u L^\i_{\ub} W^{1,2}(S_{u,\ub}, \gamma^{(1)} ) } \ls \| \gamma^{(1)} - \gamma^{(2)} \|_{L^\i_u L^\i_{\ub} W^{1,2}(S_{u,\ub}, \gamma^{(1)} ) },
\end{equation}
and
\begin{equation}\label{eq:gamma.det.diff.main}
\begin{split}
&\: \|\f{\sqrt{\det\gamma^{(1)}}}{\sqrt{\det\gamma^{(2)}}} - 1 \|_{L^\i_u L^\i_{\ub} W^{1,2}(S_{u,\ub},\gamma^{(1)})} + \|\f{\sqrt{\det\gamma^{(2)}}}{\sqrt{\det\gamma^{(1)}}} - 1 \|_{L^\i_u L^\i_{\ub} W^{1,2}(S_{u,\ub},\gamma^{(1)})} \\
 \ls &\: \| \log \det \gamma^{(1)} - \log \det \gamma^{(2)} \|_{L^\i_u L^\i_{\ub} W^{1,2}(S_{u,\ub},\gamma^{(1)})}.
 \end{split}
\end{equation}
In particular, by \eqref{def:dist}, we also have
$$\mbox{LHS of \eqref{eq:gamma.inverse.diff.main}} + \mbox{LHS of \eqref{eq:gamma.det.diff.main}} \ls \mathrm{dist}.$$
\end{proposition}
\begin{proof}
The estimate \eqref{eq:gamma.inverse.diff.main} is a standard statement regarding the continuity of inverses. We omit the details; see for instance \cite[Proposition~9.2]{DL} for the relevant calculations.

To prove \eqref{eq:gamma.det.diff.main}, first note that for every $x\in (0,+\infty)$, we have the calculus inequality
$ |x-1| \leq \max\{ x,\, \f 1x\} |\log x|.$ 
It follows (by setting $x = \f{\sqrt{\det\gamma^{(1)}}}{\sqrt{\det\gamma^{(2)}}}$) that 
$$|\f{\sqrt{\det\gamma^{(1)}}}{\sqrt{\det\gamma^{(2)}}} -1| \leq \max\{|\f{\sqrt{\det\gamma^{(1)}}}{\sqrt{\det\gamma^{(2)}}}|, \, |\f{\sqrt{\det\gamma^{(2)}}}{\sqrt{\det\gamma^{(1)}}}| \} |\log \f{\sqrt{\det\gamma^{(1)}}}{\sqrt{\det\gamma^{(2)}}}|.$$
By Definition~\ref{double.null.def.2}, we have uniform $L^\i$ bounds for $|\f{\sqrt{\det\gamma^{(1)}}}{\sqrt{\det\gamma^{(2)}}}|$ and $|\f{\sqrt{\det\gamma^{(2)}}}{\sqrt{\det\gamma^{(1)}}}|$. Therefore, taking the $L^\i_u L^\i_{\ub} L^2(S_{u,\ub},\gamma^{(1)})$ norm of the above inequality yields
\begin{equation}\label{eq:gamma.diff.aux.1}
\|\f{\sqrt{\det\gamma^{(1)}}}{\sqrt{\det\gamma^{(2)}}} -1 \|_{L^\i_u L^\i_{\ub} L^2(S_{u,\ub},\gamma^{(1)})} \ls \| \log \det\gamma^{(1)} - \log\det\gamma^{(2)} \|_{L^\i_u L^\i_{\ub} L^2(S_{u,\ub},\gamma^{(1)})}.
\end{equation}

For the first derivative, we compute
$$\nab (\f{\sqrt{\det\gamma^{(1)}}}{\sqrt{\det\gamma^{(2)}}} -1) =  \f 12\f{\sqrt{\det\gamma^{(1)}}}{\sqrt{\det\gamma^{(2)}}} \nab(\log\det\gamma^{(1)} - \log\det\gamma^{(2)}).$$
Using the bounds in Definition~\ref{double.null.def.2}, we obtain
\begin{equation}\label{eq:gamma.diff.aux.2}
\|\nab (\f{\sqrt{\det\gamma^{(1)}}}{\sqrt{\det\gamma^{(2)}}} -1)\|_{L^\i_u L^\i_{\ub} L^2(S_{u,\ub},\gamma^{(1)})} \ls \|\nab(\log\det\gamma^{(1)} - \log\det\gamma^{(2)})\|_{L^\i_u L^\i_{\ub} L^2(S_{u,\ub},\gamma^{(1)})}.
\end{equation}

Combining \eqref{eq:gamma.diff.aux.1} and \eqref{eq:gamma.diff.aux.2} yields the bound for $\f{\sqrt{\det\gamma^{(1)}}}{\sqrt{\det\gamma^{(2)}}} - 1$ in \eqref{eq:gamma.det.diff.main}; the bound for $\f{\sqrt{\det\gamma^{(2)}}}{\sqrt{\det\gamma^{(1)}}} - 1$ can be proven in an entirely analogous manner. \qedhere

\end{proof}

\begin{proposition}\label{prop:Gamma.diff}
$$\| \slashed \Gamma^{(1)} - \slashed \Gamma^{(2)} \|_{L^\i_u L^\i_{\ub} L^{2}(S_{u,\ub}, \gamma^{(1)} ) } \ls \| \gamma^{(1)} - \gamma^{(2)} \|_{L^\i_u L^\i_{\ub} W^{1,2}(S_{u,\ub}, \gamma^{(1)} ) },$$
and
$$\| \slashed \Gamma^{(1)} - \slashed \Gamma^{(2)} \|_{L^\i_u L^\i_{\ub} L^{2}(S_{u,\ub}, \gamma^{(1)} ) } \ls \mathrm{dist}.$$
\end{proposition}
\begin{proof}
The first statement is immediate from \eqref{Gamma.def} and Proposition~\ref{prop:gamma.inverse.diff}. The second statement then follows after applying also \eqref{def:dist}. \qedhere
\end{proof}

\begin{proposition}\label{prop:Omg.diff.aux}
$$\|1 - \f{\Omg^{(1)}}{\Omg^{(2)}} \|_{L^\i_u L^\i_{\ub} W^{1,2}(S_{u,\ub},\gamma^{(1)})} + \|1 - \f{\Omg^{(2)}}{\Omg^{(1)}} \|_{L^\i_u L^\i_{\ub} W^{1,2}(S_{u,\ub},\gamma^{(1)})} \ls \mathrm{dist}.$$
\end{proposition}
\begin{proof}
It suffices to control $1 - \f{\Omg^{(1)}}{\Omg^{(2)}}$, $1 - \f{\Omg^{(2)}}{\Omg^{(1)}}$ and their derivatives by $\log\f{\Omg^{(1)}}{\Omg^{(2)}}$ and its derivatives. This is a computation almost exactly the same as the proof of \eqref{eq:gamma.det.diff.main}; we omit the details. \qedhere

\end{proof}

\subsection{Transport estimates for the metric coefficients and the Ricci coefficients}\label{sec:diff.transport}

In this subsection, we prove some estimates which are derivable using transport equations. We first prove some general estimates regarding general transport equations in Propositions~\ref{prop:transport} and \ref{prop:operator.diff}. These will then be applied in Propositions~\ref{prop:metric}--\ref{prop:chih} to control the differences of the metric coefficients and the Ricci coefficients. 

\begin{proposition}\label{prop:transport}
Suppose\footnote{For all the statements in this proposition, we allow $\phi$ and $F$ to be of arbitrary (but the same) rank. The implicit constants in the estimates may depend on the rank.} $\nab_3^{(1)} \phi = F$ holds in the integrated sense (Definition~\ref{def:weak.transport}) such that $\phi\restriction_{\{u = u_i\}} = 0$. Then
\begin{equation}\label{eq:ii2}
\|\phi\|_{L^\i_u L^\i_{\ub} L^2(S_{u,\ub},\gamma^{(1)})} \ls \|F \|_{L^\i_{\ub} L^1_u L^2(S_{u,\ub},\gamma^{(1)})}.
\end{equation}
Suppose $\nab_4^{(1)} \phi = F$ holds in the integrated sense (Definition~\ref{def:weak.transport}) such that $\phi\restriction_{\{\ub = \ub_j\}} = 0$. Then
\begin{equation}\label{eq:ii2.4}
\|\phi\|_{L^\i_u L^\i_{\ub} L^2(S_{u,\ub},\gamma^{(1)})} \ls \|F \|_{L^\i_u L^1_{\ub} L^2(S_{u,\ub},\gamma^{(1)})}.
\end{equation}
Suppose $\nab_3^{(1)} \phi = F$ holds in the weak integrated sense (Definition~\ref{def:weaker.transport}) such that $\phi\restriction_{\{u = u_i\}} = 0$. Then
\begin{equation}\label{eq:2i2}
\|\phi\|_{L^\i_u L^2_{\ub} L^2(S_{u,\ub},\gamma^{(1)})} \ls \| F \|_{L^1_u L^2_{\ub} L^2(S_{u,\ub},\gamma^{(1)})}.
\end{equation}
Suppose $\nab_4^{(1)} \phi = F$ holds in the weak integrated sense (Definition~\ref{def:weaker.transport}) such that $\phi\restriction_{\{\ub = \ub_j \}} = 0$. Then
\begin{equation}\label{eq:2i2.4}
\|\phi\|_{L^\i_{\ub} L^2_u L^2(S_{u,\ub},\gamma^{(1)})} \ls \| F \|_{L^1_{\ub} L^2_u L^2(S_{u,\ub},\gamma^{(1)})}.
\end{equation}
\end{proposition}
\begin{proof}
We will prove \eqref{eq:ii2} and \eqref{eq:2i2}. The proofs for \eqref{eq:ii2.4} and \eqref{eq:2i2.4} are similar and omitted.

\pfstep{Step~1: Proof of \eqref{eq:ii2}} Fix $(U,\underline{U}) \in [u_i,u_{i+1}] \times [\ub_j,\ub_{j+1}]$. Let $\varphi\in C^1$ satisfy
\begin{equation}\label{eq:varphi.transported}
\nab_3\varphi = 0
\end{equation}
and 
\begin{equation}\label{eq:phi.transported.leq1}
\|\varphi\|_{L^2(S_{U,\underline{U}},\gamma^{(1)})} \leq 1.
\end{equation}
By Proposition~\ref{prop:transport.id} and \eqref{eq:varphi.transported},
\begin{equation}\label{eq:phi.transported.Gronwall}
\f{\rd}{\rd u} \int_{S_{u,\underline{U}}} |\varphi|_{\gamma^{(1)}}^2 \,\mathrm{dA}_{\gamma^{(1)}} = \int_{S_{u,\underline{U}}} \Omg^{(1)}\left(\nab_3^{(1)} |\varphi|_{\gamma^{(1)}}^2 + \trchb^{(1)} |\varphi|_{\gamma^{(1)}}^2 \right)\, \,\mathrm{dA}_{\gamma^{(1)}} = \int_{S_{u,\underline{U}}} \Omg^{(1)} \trchb^{(1)} |\varphi|_{\gamma^{(1)}}^2 \, \,\mathrm{dA}_{\gamma^{(1)}}.
\end{equation}
A simple application of the Gr\"onwall's inequality implies that $\|\varphi\|_{L^\i_u L^2(S_{u,\underline{U}},\gamma^{(1)})} \ls 1$.

By Definition~\ref{def:weak.transport}, \eqref{eq:varphi.transported} and the assumption on the initial data, the following holds for all $u \in [u_i,u_{i+1}]$:
\begin{equation}\label{eq:weak.for.uniqueness}
\begin{split}
\int_{S_{u,\underline{U}}} \langle\varphi, \phi\rangle \Omg^{(1)} \,\ud A_{\gamma^{(1)}}
+ \int_{u_i}^{u} \int_{S_{u',\underline{U}}} \langle \varphi, F + (\trchb^{(1)} - 2\omb^{(1)})\phi \rangle  (\Omg^{(1)})^2 \,\ud A_{\gamma^{(1)}}\, \ud u' =0.
\end{split}
\end{equation}
Applying H\"older's inequality and \eqref{eq:phi.transported.Gronwall} to \eqref{eq:weak.for.uniqueness}, and using the estimates in Definition~\ref{double.null.def.2}, we see that for every $u \in [u_i,u_{i+1}] $,
\begin{equation}\label{eq:phi.dual.formulation}
\begin{split}
\left| \int_{S_{u,\underline{U}}} \langle\varphi, \phi\rangle \Omg^{(1)} \,\ud A_{\gamma^{(1)}}\right| \ls &\: \| F\|_{L^1_u L^2(S_{u,\underline{U}},\gamma^{(1)})} + \|\trchb^{(1)} - 2\omb^{(1)}\|_{L^1_u L^2(S_{u,\underline{U}},\gamma^{(1)})} \| \phi\|_{L^\i_u L^2(S_{u,\underline{U}},\gamma^{(1)})} \\
\ls &\: \| F\|_{L^1_u L^2(S_{u,\underline{U}},\gamma^{(1)})} + \f 1{N^{\f 12}} \| \phi\|_{L^\i_u L^2(S_{u,\underline{U}},\gamma^{(1)})}.
\end{split}
\end{equation}
In particular, it follows from \eqref{eq:phi.dual.formulation} by duality and the boundedness of $\log\Omg^{(1)}$ that
\begin{equation}
\begin{split}
\|\phi\|_{L^2(S_{U,\underline{U}},\gamma^{(1)})}\ls &\: \sup_{\|\varphi\|_{L^2(S_{U,\underline{U}}, \gamma^{(1)})}\leq 1} \left| \int_{S_{U,\underline{U}}} \langle\varphi, \phi\rangle \Omg^{(1)} \,\ud A_{\gamma^{(1)}}\right| \\
\ls &\: \| F\|_{L^1_u L^2(S_{u,\underline{U}},\gamma^{(1)})} + \f 1{N^{\f 12}} \| \phi\|_{L^\i_u L^2(S_{u,\underline{U}},\gamma^{(1)})}.
\end{split}
\end{equation}

In view of the arbitrariness of $(U,\underline{U})$, we then obtain
$$\|\phi\|_{L^\i_u L^\i_{\ub} L^2(S_{u,\ub},\gamma^{(1)})} \ls \| F\|_{L^\i_{\ub} L^1_u L^2(S_{u,\ub},\gamma^{(1)})} + \f 1{N^{\f 12}} \|\phi\|_{L^\i_u L^\i_{\ub} L^2(S_{u,\ub},\gamma^{(1)})},$$
which, after choosing $N$ sufficiently large, implies \eqref{eq:ii2}.

\pfstep{Step~2: Proof of \eqref{eq:2i2}} Fix $U\in [u_i,u_{i+1}]$. Pick $\varphi \in C^1$ satisfying \eqref{eq:varphi.transported}, but instead of \eqref{eq:phi.transported.leq1}, assume
\begin{equation}\label{eq:phi.transported.L2.leq1}
\|\varphi\|_{L^2_{\ub} L^2(S_{U,\ub},\gamma^{(1)})} \leq 1.
\end{equation}
Integrating \eqref{eq:phi.transported.Gronwall} in $\ub$ and applying Gr\"onwall's inequality, we obtain
\begin{equation}\label{eq:phi.transported.L2.leq1.propagated}
\|\varphi\|_{L^\i_u L^2_{\ub} L^2(S_{u,\ub},\gamma^{(1)})} \ls 1.
\end{equation}

By Definition~\ref{def:weaker.transport} and \eqref{eq:varphi.transported}, we have, for all $u \in [u_i,u_{i+1}]$,
\begin{equation}\label{eq:transport.weaker.transport}
\begin{split}
\int_{\ub_j}^{\ub_{j+1}} \int_{S_{u,\ub}} \langle\varphi, \phi\rangle \Omg^{(1)} \,\ud A_{\gamma}\, \ud \ub 
+ \int_{\ub_j}^{\ub_{j+1}}\int_{u_i}^{u} \int_{S_{u',\ub}} (\langle \varphi, F + (\trchb - 2\omb)^{(1)}\phi \rangle ) (\Omg^{(1)})^2 \,\ud A_{\gamma}\, \ud u'\,\, \ud \ub =0.
\end{split}
\end{equation}
Applying \eqref{eq:transport.weaker.transport} when $u=U$ and using H\"older's inequality together with \eqref{eq:phi.transported.L2.leq1.propagated}, we obtain
\begin{equation}\label{eq:transport.weaker.transport.2}
\begin{split}
\|\phi\|_{L^2_{\ub} L^2(S_{U,\ub},\gamma^{(1)})} \ls &\: \sup_{\|\varphi\|_{L^2_{\ub} L^2(S_{U,\ub},\gamma^{(1)})} \leq 1} |\int_{\ub_j}^{\ub_{j+1}} \int_{S_{U,\ub}} \langle\varphi, \phi\rangle \Omg \,\ud A_{\gamma}\, \ud \ub| \\
\ls &\: \| F\|_{L^1_u L^2_{\ub} L^2(S_{u,\ub},\gamma^{(1)})} + \f 1{N^{\f 12}}\|\trchb^{(1)} - 2\omb^{(1)} \|_{L^2_u L^\i(S_{u,\ub},\gamma^{(1)})} \|\phi\|_{L^\i_u L^2_{\ub} L^2(S_{u,\ub},\gamma^{(1)})} \\ 
\ls &\: \| F\|_{L^1_u L^2_{\ub} L^2(S_{u,\ub},\gamma^{(1)})} + \f 1{N^{\f 12}} \|\phi\|_{L^\i_u L^2_{\ub} L^2(S_{u,\ub},\gamma^{(1)})}. 
\end{split}
\end{equation}

Since $U$ is arbitrary, it follows from \eqref{eq:transport.weaker.transport.2} that 
\begin{equation*}
\begin{split}
\|\phi\|_{L^\i_u L^2_{\ub} L^2(S_{u,\ub},\gamma^{(1)})} \ls &\: \| F\|_{L^1_u L^2_{\ub} L^2(S_{u,\ub},\gamma^{(1)})} + \f 1{N^{\f 12}} \|\phi\|_{L^\i_u L^2_{\ub} L^2(S_{u,\ub},\gamma^{(1)})},
\end{split}
\end{equation*}
which implies \eqref{eq:2i2} after taking $N$ sufficiently large. \qedhere
\end{proof}

\begin{proposition}\label{prop:operator.diff}
The following holds for every $\phi$ an $S$-tangent tensorfield of arbitrary rank (with the implicit constant depending on the rank):
\begin{equation}\label{eq:nab4.diff.1}
\begin{split}
 \|( \nab_4^{(1)} - \nab_4^{(2)}) \phi\|_{L^\i_{u} L^1_{\ub} L^2(S_{u,\ub},\gamma^{(1)}) } 
\ls &\: \f{\mathrm{dist}}{N^{\f 12}} (\|\phi \|_{L^\i_{\ub} L^\i_u W^{1,4}(S_{u,\ub},\gamma^{(1)}) } + \|\nab_4^{(2)} \phi \|_{L^\i_u L^\i_{\ub} L^{4}(S_{u,\ub},\gamma^{(1)}) }),
\end{split}
\end{equation}
\begin{equation}\label{eq:nab4.diff.2}
\begin{split}
&\: \|( \nab_4^{(1)} - \nab_4^{(2)}) \phi\|_{ L^1_{\ub} L^2_{u} L^2(S_{u,\ub},\gamma^{(1)}) } \\
\ls &\: \f{\mathrm{dist}}{N^{\f 12}} (\|\phi \|_{L^2_u L^\i_{\ub} L^{4}(S_{u,\ub},\gamma^{(1)}) } + \|\phi \|_{L^\i_{\ub} L^2_u W^{1,4}(S_{u,\ub},\gamma^{(1)}) } + \|\nab_4^{(2)} \phi \|_{ L^\i_{\ub} L^2_u L^{4}(S_{u,\ub},\gamma^{(1)}) }),
\end{split}
\end{equation}
\begin{equation}\label{eq:nab3.diff.1}
\begin{split}
 \|( \nab_3^{(1)} - \nab_3^{(2)}) \phi\|_{L^\i_{\ub} L^1_{u} L^2(S_{u,\ub},\gamma^{(1)}) } 
\ls &\: \f{\mathrm{dist}}{N^{\f 12}} ( \|\phi \|_{L^\i_{u} L^\i_{\ub} W^{1,4}(S_{u,\ub},\gamma^{(1)}) } + \|\nab_3^{(2)} \phi \|_{L^\i_{u} L^\i_{\ub} L^4(S_{u,\ub},\gamma^{(1)}) }).
\end{split}
\end{equation}
and
\begin{equation}\label{eq:nab3.diff.2}
\begin{split}
&\: \|( \nab_3^{(1)} - \nab_3^{(2)}) \phi\|_{L^1_{u} L^2_{\ub} L^2(S_{u,\ub},\gamma^{(1)}) } \\
\ls &\: \f{\mathrm{dist}}{N^{\f 12}} (\|\phi \|_{L^2_{\ub}  L^\i_{u} L^{4}(S_{u,\ub},\gamma^{(1)}) } + \|\phi \|_{L^\i_{u} L^2_{\ub}  W^{1,4}(S_{u,\ub},\gamma^{(1)}) } + \|\nab_3^{(2)} \phi \|_{L^\i_{u} L^2_{\ub}  L^4(S_{u,\ub},\gamma^{(1)}) }).
\end{split}
\end{equation}
\end{proposition}
\begin{proof}
We will only prove \eqref{eq:nab3.diff.1} and \eqref{eq:nab3.diff.2}; the estimates \eqref{eq:nab4.diff.1} and \eqref{eq:nab4.diff.2} are slightly easier.

Before we proceed, first note that by H\"older's inequality and Fubini's theorem, it suffices to\footnote{Note in particular that a smallness factor $N^{-\f 12}$ arises from the difference between $L^1_u$ on the LHS of \eqref{eq:nab3.diff.1}, \eqref{eq:nab3.diff.2} and $L^2_u$ on the LHS of \eqref{eq:nab3.diff.3}.} prove that for $p \in \{2, +\infty\}$, 
 \begin{equation}\label{eq:nab3.diff.3}
\begin{split}
&\: \|( \nab_3^{(1)} - \nab_3^{(2)}) \phi\|_{L^p_{\ub} L^2_{u} L^2(S_{u,\ub},\gamma^{(1)}) } \\
\ls &\: \mathrm{dist}\times (\|\phi \|_{L^p_{\ub}  L^\i_{u} L^{4}(S_{u,\ub},\gamma^{(1)}) } + \|\phi \|_{L^\i_{u} L^p_{\ub}  W^{1,4}(S_{u,\ub},\gamma^{(1)}) } + \|\nab_3^{(2)} \phi \|_{L^\i_{u} L^p_{\ub}  L^4(S_{u,\ub},\gamma^{(1)}) }).
\end{split}
\end{equation}
From now on, we take $p\in \{2, +\infty\}$ and our goal will be to prove \eqref{eq:nab3.diff.3}.

By \eqref{nab3.def}, 
\begin{align}
&\: (( \nab_3^{(1)} - \nab_3^{(2)})  \phi)_{A_1 A_2 ... A_r} \notag\\
= &\: \f{\Om^{(2)} - \Om^{(1)}}{\Om^{(1)}} (\nab_3^{(2)} \phi)_{A_1 A_2 ... A_r} +(\Om^{(1)})^{-1} (b^{(1)}- b^{(2)})^C \nab_C^{(1)} \phi_{A_1 A_2 ... A_r} \label{eq:nab3.diff.line1}\\
&\:  + \f{\Om^{(2)} - \Om^{(1)}}{\Om^{(1)}} \sum_{i=1}^r ((\gamma^{-1})^{BC}\chib_{CA_i})^{(2)} \phi_{A_1\dots\hat{A_i}B\dots A_r} \label{eq:nab3.diff.line2}\\
&\: -\sum_{i=1}^r( ((\gamma^{-1})^{BC}\chib_{CA_i})^{(1)} - ((\gamma^{-1})^{BC}\chib_{CA_i})^{(2)}) \phi_{A_1\dots\hat{A_i}B\dots A_r} \label{eq:nab3.diff.line3}\\
&\:  + \sum_{i=1}^r (\Om^{(1)})^{-1} \nab_{A_i}^{(1)} ( b^{(1)} - b^{(2)})^B \phi_{A_1\dots\hat{A_i}B\dots A_r} \label{eq:nab3.diff.line4}.
\end{align}

For the first term in \eqref{eq:nab3.diff.line1}, we first note that by Proposition~\ref{prop:Omg.diff.aux} and Sobolev embedding (Proposition~\ref{prop:Sobolev}),
\begin{equation}\label{eq:Om.diff.Sobolev}
\| \f{\Om^{(2)} - \Om^{(1)}}{\Om^{(1)}}\|_{L^\i_u L^\i_{\ub} L^4(S_{u,\ub}, \gamma^{(1)})} \ls \| \f{\Om^{(2)} - \Om^{(1)}}{\Om^{(1)}}\|_{L^\i_u L^\i_{\ub} W^{1,2}(S_{u,\ub}, \gamma^{(1)})} \ls \mathrm{dist}.
\end{equation}
Therefore, using the fact $p\geq 2$ and H\"older's inequality, we obtain
\begin{equation*}
\begin{split}
\| \f{\Om^{(2)} - \Om^{(1)}}{\Om^{(1)}} \nab_3^{(2)} \phi \|_{L^p_{\ub} L^2_{u} L^2(S_{u,\ub},\gamma^{(1)})}  \ls &\: \| \f{\Om^{(2)} - \Om^{(1)}}{\Om^{(1)}} \nab_3^{(2)} \phi \|_{ L^2_{u} L^p_{\ub} L^2(S_{u,\ub},\gamma^{(1)})} \\
\ls &\: \f{\mathrm{dist}}{N^{\f 12}} \| \nab_3^{(2)} \phi\|_{L^\i_{u} L^p_{\ub}  L^4(S_{u,\ub},\gamma^{(1)})}.
\end{split}
\end{equation*}

The second term in \eqref{eq:nab3.diff.line1} can be treated in a similar manner. We note that by \eqref{def:dist} and Proposition~\ref{prop:Sobolev}, $\| b^{(1)} - b^{(2)} \|_{L^\i_u L^\i_{\ub} L^4(S_{u,\ub},\gamma^{(1)})} \ls \mathrm{dist}$ and so using also the bounds provided by Definition~\ref{double.null.def.2}, we have
$$\| (\Om^{(1)})^{-1} (b^{(1)}- b^{(2)})\cdot  \nab^{(1)} \phi \|_{L^p_{\ub} L^1_{u} L^2(S_{u,\ub},\gamma^{(1)})} \ls \f{\mathrm{dist}}{N} \| \phi\|_{L^\i_{u} L^p_{\ub}  W^{1,4}(S_{u,\ub},\gamma^{(1)})}.$$

For \eqref{eq:nab3.diff.line2}, we use \eqref{eq:Om.diff.Sobolev}, bounds in Definition~\ref{double.null.def.2} and H\"older's inequality to obtain
$$\|\mbox{\eqref{eq:nab3.diff.line2}}\|_{L^p_{\ub} L^2_{u} L^2(S_{u,\ub},\gamma^{(1)}) } \ls  \f{\mathrm{dist}}{N^{\f 12}} \|\chib\|_{L^2_u L^\i_{\ub} L^\i(S_{u,\ub},\gamma^{(1)})} \|\phi\|_{L^\i_{u} L^p_{\ub}  L^2(S_{u,\ub},\gamma^{(1)})} \ls  \mathrm{dist} \|\phi\|_{L^\i_{u} L^p_{\ub}  L^4(S_{u,\ub},\gamma^{(1)})}.$$

We next consider \eqref{eq:nab3.diff.line3}. For this, we note that by \eqref{def:dist} and Sobolev embedding (Proposition~\ref{prop:Sobolev})m
$$\| ((\gamma^{-1})^{BC}\chib_{CA_i})^{(1)}- ((\gamma^{-1})^{BC}\chib_{CA_i})^{(2)} \|_{ L^\i_{\ub} L^2_{u} L^4(S_{u,\ub},\gamma^{(1)}) }\ls \mathrm{dist}.$$
Thus, H\"older's inequality implies that
$$\|\mbox{\eqref{eq:nab3.diff.line3}} \|_{L^p_{\ub} L^2_{u} L^2(S_{u,\ub},\gamma^{(1)}) } \ls  \mathrm{dist} \|\phi\|_{L^p_{\ub} L^\i_{u}  L^4(S_{u,\ub},\gamma^{(1)})}.$$

Finally, we consider \eqref{eq:nab3.diff.line4}. By \eqref{def:dist},
$$\| (\Om^{(1)})^{-1} \nab_{A_i}^{(1)} ( b^{(1)} - b^{(2)})^B \|_{ L^\i_{\ub} L^\i_{u} L^2(S_{u,\ub},\gamma^{(1)}) }\ls \mathrm{dist}.$$
Therefore, by H\"older's inequality, Sobolev embedding (Proposition~\ref{prop:Sobolev}) and the fact that $p\geq 2$,
$$\|\mbox{\eqref{eq:nab3.diff.line4}}\|_{L^p_{\ub} L^2_{u} L^2(S_{u,\ub},\gamma^{(1)}) } \ls  \mathrm{dist} \|\phi\|_{L^p_{\ub} L^2_{u}  L^\i(S_{u,\ub},\gamma^{(1)})} \ls \f{\mathrm{dist}}{N^{\f 12}} \|\phi\|_{ L^\i_{u} L^p_{\ub}  W^{1,4}(S_{u,\ub},\gamma^{(1)})}.$$

Combining the above estimates, we have thus achieved \eqref{eq:nab3.diff.3}. This concludes the argument. \qedhere
\end{proof}

\begin{proposition}\label{prop:metric}
\begin{equation*}
\begin{split}
&\: \|\gamma^{(1)} - \gamma^{(2)}\|_{L^\i_u L^\i_{\ub} W^{1,2}(S_{u,\ub},\gamma^{(1)})}  + \| \log\f{\det\gamma^{(1)}}{\det\gamma^{(2)}} \|_{L^\i_u L^\i_{\ub} W^{1,2}(S_{u,\ub},\gamma^{(1)})} \\
&\: + \|b^{(1)} - b^{(2)} \|_{L^\i_u L^\i_{\ub} W^{1,2}(S_{u,\ub},\gamma^{(1)})} 
 + \| \log\f{\Omg^{(1)}}{\Omg^{(2)}} \|_{L^\i_u L^\i_{\ub} W^{1,2}(S_{u,\ub},\gamma^{(1)})} \ls \f{\mathrm{dist}}{N^{\f 12}} .
\end{split}
\end{equation*}
\end{proposition}
\begin{proof}
\pfstep{Step~1: Proof of $L^\i_u L^\i_{\ub} L^{2}(S_{u,\ub},\gamma^{(1)})$ estimates} By Propositions~\ref{prop:transport} and \ref{prop:operator.diff}, it suffices to bound $(\nab_4\gamma)^{(1)} - (\nab_4\gamma)^{(2)}$, $(\nab_4\log\det\gamma)^{(1)} - (\nab_4\log\det\gamma)^{(2)}$, $(\nab_4 b)^{(1)} - (\nab_4 b)^{(2)}$ and $(\nab_4\log\Om)^{(1)} - (\nab_4\log\Om)^{(2)}$ in $L^\i_u L^1_{\ub} L^2(S_{u,\ub}, \gamma^{(1)})$.

By \eqref{metric.derivative.invar}, we have
$$(\nab_4\gamma)^{(1)} - (\nab_4\gamma)^{(2)} = 0,$$
$$(\nab_4\log\det\gamma)^{(1)} - (\nab_4\log\det\gamma)^{(2)} = 2\trch^{(1)} - 2\trch^{(2)},$$
$$(\nab_4 b)^{(1)} - (\nab_4 b)^{(2)}  =  - 2(\Omg^{(1)} (\eta - \etab)^{(1)} - \Omg^{(2)} (\eta - \etab)^{(2)}) + (\chi\cdot b)^{(1)} - (\chi\cdot b)^{(2)},$$
$$(\nab_4\log\Om)^{(1)} - (\nab_4\log\Om)^{(2)} = - 2(\om^{(1)} - \om^{(2)}).$$

Now by the estimates in Definition~\ref{double.null.def.2} and \eqref{def:dist}, the RHS of each of these equations is bounded above in $L^\i_u L^2_{\ub} L^2(S_{u,\ub}, \gamma^{(1)})$ by $\mathrm{dist}$. In particular, using the Cauchy--Schwarz inequality, we obtain
\begin{equation*}
\begin{split}
&\: \|(\nab_4\gamma)^{(1)} - (\nab_4\gamma)^{(2)}\|_{L^\i_u L^1_{\ub} L^2(S_{u,\ub}, \gamma^{(1)})} + \|(\nab_4\log\det\gamma)^{(1)} - (\nab_4\log\det\gamma)^{(2)}\|_{L^\i_u L^1_{\ub} L^2(S_{u,\ub}, \gamma^{(1)})} \\
&\: + \|(\nab_4 b)^{(1)} - (\nab_4 b)^{(2)}\|_{L^\i_u L^1_{\ub} L^2(S_{u,\ub}, \gamma^{(1)})} +\| (\nab_4\log\Om)^{(1)} - (\nab_4\log\Om)^{(2)}\|_{L^\i_u L^1_{\ub} L^2(S_{u,\ub}, \gamma^{(1)})} \ls \f{\mathrm{dist}}{N^{\f 12}}.
\end{split}
\end{equation*}

Therefore, we obtain the desired $L^\i_u L^\i_{\ub} L^{2}(S_{u,\ub},\gamma^{(1)})$ estimates.

\pfstep{Step~2: Proof of $L^\i_u L^\i_{\ub} W^{1,2}(S_{u,\ub},\gamma^{(1)})$ estimates} This is similar to Step~1, except that we use the equations \eqref{eq:nablagamma}--\eqref{eq:nablab} instead; we omit the details. \qedhere
\end{proof}

\begin{proposition}\label{prop:eta}
$$\|\eta^{(1)} - \eta^{(2)}\|_{L^\i_u L^\i_{\ub} L^2(S_{u,\ub},\gamma^{(1)})} + \|\etab^{(1)} - \etab^{(2)}\|_{L^\i_u L^\i_{\ub} L^2(S_{u,\ub},\gamma^{(1)})}\ls \f{\mathrm{dist}}{N^{\f 12}}.$$
\end{proposition}
\begin{proof}
We will prove the estimate for $\eta^{(1)} - \eta^{(2)}$; the estimate for $\etab^{(1)} - \etab^{(2)}$ is similar.

By Propositions~\ref{prop:transport}, we need to bound $\nab_4^{(1)} (\eta^{(1)} - \eta^{(2)})$. We write
$$\nab_4^{(1)} (\eta^{(1)} - \eta^{(2)}) = - (\nab_4^{(1)} - \nab_4^{(2)}) \eta^{(2)} + (\nab_4 \eta)^{(1)} - (\nab_4 \eta)^{(2)}.$$

The term $(\nab_4^{(1)} - \nab_4^{(2)}) \eta^{(2)}$ can be estimated using \eqref{eq:nab4.diff.1} in Proposition~\ref{prop:operator.diff} and the bounds for $\eta^{(2)}$ given by Definition~\ref{double.null.def.2}, the equation \eqref{Ric4A} satisfied by $\eta^{(2)}$, together with Propositions~\ref{prop:diff.gamma.comp} and \ref{prop:Gamma.diff} so that we have
$$\|(\nab_4^{(1)} - \nab_4^{(2)}) \eta^{(2)}\|_{L^\i_u L^1_{\ub} L^2(S_{u,\ub}, \gamma^{(1)})} \ls \f{\mathrm{dist}}{N^{\f 12}}.$$

Therefore, according to Proposition~\ref{prop:transport}, it suffices to bound the terms in $(\nab_4 \eta)^{(1)} - (\nab_4 \eta)^{(2)}$. By \eqref{Ric4A}, these terms are either of the form $(\div \chih)^{(1)} - (\div \chih)^{(2)}$ or $(\nab\trch)^{(1)} - (\nab\trch)^{(2)}$ or $(\chi\star \eta)^{(1)} - (\chi\star \eta)^{(2)}$ or $(\chi \star \etab)^{(1)} - (\chi \star \etab)^{(2)}$ (where $\star$ denotes some contraction with respect to $\gamma$).

We first handle the term $(\div \chih)^{(1)} - (\div \chih)^{(2)}$. A direct computation gives
\begin{equation*}
\begin{split}
&\: (\div\chih)^{(1)} - (\div\chih)^{(2)} \\
= &\:\{ (\gamma^{(1)})^{-1} - (\gamma^{(2)})^{-1} \}^{AB} \nab^{(1)}_A \chih^{(1)}_{BC} + \{(\gamma^{(2)})^{-1} \}^{AB} \nab^{(1)}_A (\chih^{(1)}_{BC} - \chih^{(2)}_{BC}) + \{(\gamma^{(2)})^{-1} \}^{AB} (\nab^{(1)}_A - \nab^{(2)}_A)\chih^{(2)}_{BC}.
\end{split}
\end{equation*}
To proceed, note that by Definition~\ref{double.null.def.2} and Sobolev embedding (Proposition~\ref{prop:Sobolev}), $\|\chih^{(1)}\|_{L^\i_u L^2_{\ub} L^\i(S_{u,\ub},\gamma^{(1)})} \ls 1$. Therefore, using Definition~\ref{double.null.def.2}, \eqref{def:dist} and Propositions~\ref{prop:gamma.inverse.diff} and \ref{prop:Gamma.diff}, we obtain 
\begin{equation*}
\begin{split}
&\: \| (\div\chih)^{(1)} - (\div\chih)^{(2)} \|_{L^\i_u L^1_{\ub} L^2(S_{u,\ub}, \gamma^{(1)})} \\
\ls &\: \f{\mathrm{dist}}{N^{\f 12}} + N^{-\f 12} \|\chih^{(1)} - \chih^{(2)}\|_{L^\i_u L^2_{\ub} W^{1,2}(S_{u,\ub}, \gamma^{(1)})} + N^{-\f 12}\|\slashed \Gamma^{(1)} - \slashed \Gamma^{(2)}\|_{L^\i_u L^\i_{\ub} L^2(S_{u,\ub}, \gamma^{(1)})} \ls \f{\mathrm{dist}}{N^{\f 12}}.
\end{split}
\end{equation*}

We next consider $(\nab\trch)^{(1)} - (\nab\trch)^{(2)}$. Again, we use Definition~\ref{double.null.def.2} and \eqref{def:dist} to obtain
\begin{equation*}
\begin{split}
&\: \| (\nab\trch)^{(1)} - (\nab\trch)^{(2)} \|_{L^\i_u L^1_{\ub} L^2(S_{u,\ub}, \gamma^{(1)})} \\
\ls &\: N^{-\f 12} \|\trch^{(1)} - \trch^{(2)}\|_{L^\i_u L^2_{\ub} W^{1,2}(S_{u,\ub}, \gamma^{(1)})} \ls \f{\mathrm{dist}}{N^{\f 12}}.
\end{split}
\end{equation*}

Finally, the terms not involving derivatives of $\chi$. We will look at $(\chi\star \eta)^{(1)} - (\chi\star \eta)^{(2)}$; the term $(\chi \star \etab)^{(1)} - (\chi \star \etab)^{(2)}$ is similar. By Definition~\ref{double.null.def.2} and \eqref{def:dist}, we have
\begin{equation*}
\begin{split}
 &\: \| (\chi\star \eta)^{(1)} - (\chi\star \eta)^{(2)} \|_{L^\i_u L^1_{\ub} L^2(S_{u,\ub}, \gamma^{(1)})} \\
\ls &\: N^{-\f 12} \|\chih^{(1)} - \chih^{(2)}\|_{L^\i_u L^2_{\ub} L^{2}(S_{u,\ub}, \gamma^{(1)})} \\
&\: + N^{-\f 12} (\|(\gamma^{(1)})^{-1} - (\gamma^{(2)})^{-1}\|_{L^\i_u L^\i_{\ub} L^{2}(S_{u,\ub}, \gamma^{(1)})} + \|\eta^{(1)} - \eta^{(2)}\|_{L^\i_u L^\i_{\ub} L^{2}(S_{u,\ub}, \gamma^{(1)})}) \ls \f{\mathrm{dist}}{N^{\f 12}}.
\end{split}
\end{equation*}

Combining all the above estimates, we thus obtain
\begin{equation*}
\begin{split}
\| (\nab_4 \eta)^{(1)} - (\nab_4 \eta)^{(2)}\|_{L^\i_u L^1_{\ub} L^2(S_{u,\ub}, \gamma^{(1)})} \ls \f{\mathrm{dist}}{N^{\f 12}}.
\end{split}
\end{equation*}

As argued above, this estimate concludes the proof. \qedhere

\end{proof}

\begin{proposition}\label{prop:mu}
The following estimate holds:
$$\|\mu^{(1)} - \mu^{(2)}\|_{L^\i_u L^\i_{\ub} L^2(S_{u,\ub},\gamma^{(1)})} + \|\mub^{(1)} - \mub^{(2)}\|_{L^\i_u L^\i_{\ub} L^2(S_{u,\ub},\gamma^{(1)})}\ls \mathrm{dist}.$$
\end{proposition}
\begin{proof}
We will only prove the estimate for $\mu^{(1)}-\mu^{(2)}$; the other term can be treated similarly.

Arguing as in Proposition~\ref{prop:eta}, we first write 
$$\nab_4^{(1)} (\mu^{(1)} - \mu^{(2)}) = - (\nab_4^{(1)} - \nab_4^{(2)}) \mu^{(2)} + (\nab_4 \mu)^{(1)} - (\nab_4 \mu)^{(2)}.$$
Again as in Proposition~\ref{prop:eta}, we use \eqref{eq:nab4.diff.1} in Proposition~\ref{prop:operator.diff} to obtain
$$\|(\nab_4^{(1)} - \nab_4^{(2)}) \mu^{(2)}\|_{L^\i_u L^1_{\ub} L^2(S_{u,\ub},\gamma^{(1)})} \ls \f{\mathrm{dist}}{N^{\f 12}}.$$ 

It thus remains to bound $(\nab_4 \mu)^{(1)} - (\nab_4 \mu)^{(2)}$ in $L^\i_u L^1_{\ub} L^2(S_{u,\ub},\gamma^{(1)})$. To this end, we consider the terms in the equation \eqref{eq:mu.0}. Note that schematically, we need to consider the following terms: $\nab\chi\star (\eta,\etab)$, $\nab(\eta,\etab)\star \chi$, $\chi\star (\eta,\etab)\star (\eta,\etab)$, where $(\eta, \etab)$ means we take either $\eta$ or $\etab$, and $\star$ denotes some contraction with respect to $\gamma$.

We now consider each of these types of terms. For simplicity of the exposition, we will take $\eta$ as a representative of $(\eta, \etab)$.

We begin with the term $\nab\chi\star \eta$. This can be treated as the terms in Proposition~\ref{prop:eta}.
\begin{equation*}
\begin{split}
 &\: \| (\nab \chi\star \eta)^{(1)} - (\nab \chi\star \eta)^{(2)} \|_{L^\i_u L^1_{\ub} L^2(S_{u,\ub}, \gamma^{(1)})} \\
\ls &\: N^{-\f 12} \|\nab^{(1)} (\chih^{(1)} - \chih^{(2)})\|_{L^\i_u L^2_{\ub} L^{2}(S_{u,\ub}, \gamma^{(1)})} + N^{-\f 12} \|\slashed\Gamma^{(1)} - \slashed\Gamma^{(2)}\|_{L^\i_u L^\i_{\ub} L^2(S_{u,\ub},\gamma^{(1)})} \\
&\: + N^{-\f 12} (\|(\gamma^{(1)})^{-1} - (\gamma^{(2)})^{-1}\|_{L^\i_u L^\i_{\ub} L^{2}(S_{u,\ub}, \gamma^{(1)})} + \|\eta^{(1)} - \eta^{(2)}\|_{L^\i_u L^\i_{\ub} L^{2}(S_{u,\ub}, \gamma^{(1)})}) \ls \f{\mathrm{dist}}{N^{\f 12}}.
\end{split}
\end{equation*}

We next consider $\chi\star \nab\eta$. Note that there is a contribution $\chi^{(2)} \nab^{(1)} (\eta^{(1)} - \eta^{(2)})$ for which we need to put $\chi^{(2)}$ in $L^2_{\ub} L^\i(S_{u,\ub},\gamma^{(1)})$ and put $\nab^{(1)} (\eta^{(1)} - \eta^{(2)})$ in $L^2_{\ub} L^2(S_{u,\ub},\gamma^{(1)})$. Therefore we will not be able to obtain a smallness factor of $N^{-\f 12}$. 
\begin{equation*}
\begin{split}
 &\: \| (\chi\star \nab\eta)^{(1)} - (\chi\star \nab \eta)^{(2)} \|_{L^\i_u L^1_{\ub} L^2(S_{u,\ub}, \gamma^{(1)})} \\
\ls &\: N^{-\f 12} \|\chih^{(1)} - \chih^{(2)}\|_{L^\i_u L^2_{\ub} L^{2}(S_{u,\ub}, \gamma^{(1)})} \\
&\: + N^{-\f 12} \|(\gamma^{(1)})^{-1} - (\gamma^{(2)})^{-1}\|_{L^\i_u L^\i_{\ub} L^{2}(S_{u,\ub}, \gamma^{(1)})} + \|\nab^{(1)} (\eta^{(1)} - \eta^{(2)}) \|_{L^\i_u L^2_{\ub} L^{2}(S_{u,\ub}, \gamma^{(1)})} \ls \mathrm{dist}.
\end{split}
\end{equation*}

Finally, we handle $\chi\star \eta \star \eta$. This can again be treated as the terms in Proposition~\ref{prop:eta}.
\begin{equation*}
\begin{split}
 &\: \| (\chi\star \eta \star \eta)^{(1)} - (\chi\star \eta \star \eta)^{(2)} \|_{L^\i_u L^1_{\ub} L^2(S_{u,\ub}, \gamma^{(1)})} \\
\ls &\: N^{-\f 12} \|\chih^{(1)} - \chih^{(2)}\|_{L^\i_u L^2_{\ub} L^{2}(S_{u,\ub}, \gamma^{(1)})} \\
&\: + N^{-\f 12} (\|(\gamma^{(1)})^{-1} - (\gamma^{(2)})^{-1}\|_{L^\i_u L^\i_{\ub} L^{2}(S_{u,\ub}, \gamma^{(1)})} + \|\eta^{(1)} - \eta^{(2)}\|_{L^\i_u L^\i_{\ub} L^{2}(S_{u,\ub}, \gamma^{(1)})}) \ls \f{\mathrm{dist}}{N^{\f 12}}.
\end{split}
\end{equation*}

This concludes the proof. \qedhere
\end{proof}

\begin{proposition}\label{prop:chih}
$$\|\chih^{(1)} - \chih^{(2)}\|_{L^\i_u L^2_{\ub} L^2(S_{u,\ub},\gamma^{(1)})} + \|\chibh^{(1)} - \chibh^{(2)}\|_{ L^\i_{\ub} L^2_u L^2(S_{u,\ub},\gamma^{(1)})} \ls \f{\mathrm{dist}}{N^{\f 12}},$$
$$\|\om^{(1)} - \om^{(2)}\|_{L^\i_u L^2_{\ub} L^2(S_{u,\ub},\gamma^{(1)})} + \|\omb^{(1)} - \omb^{(2)}\|_{ L^\i_{\ub} L^2_u L^2(S_{u,\ub},\gamma^{(1)})}\ls \f{\mathrm{dist}}{N^{\f 12}}.$$
\end{proposition}
\begin{proof}
All of the estimates can be obtained in a similar way, we will consider only $\chih^{(1)} - \chih^{(2)}$ in the proof.

By \eqref{eq:2i2} in Proposition~\ref{prop:transport}, it suffices to bound $\nab_3^{(1)} (\chih^{(1)} - \chih^{(2)})$ in $L^1_u L^2_{\ub} L^2(S_{u,\ub}, \gamma^{(1)})$. We write
$$\nab_3^{(1)} (\chih^{(1)} - \chih^{(2)}) = - (\nab_3^{(1)} - \nab_3^{(2)})\chih^{(2)} + (\nab_3 \chih)^{(1)} - (\nab_3 \chih)^{(2)}.$$

By Proposition~\ref{prop:operator.diff}, the estimates in Definition~\ref{double.null.def.2}, Propositions~\ref{prop:diff.gamma.comp} and \ref{prop:Gamma.diff}, we have
$$\|(\nab_3^{(1)} - \nab_3^{(2)})\chih^{(2)}\|_{L^1_u L^2_{\ub} L^2(S_{u,\ub}, \gamma^{(1)})} \ls \f{\mathrm{dist}}{N^{\f 12}}.$$

It therefore suffices to bound $(\nab_3 \chih)^{(1)} - (\nab_3 \chih)^{(2)}$ in $L^1_u L^2_{\ub} L^2(S_{u,\ub},\gamma^{(1)})$. We now look at the corresponding terms in \eqref{RicAB.1}.

We begin with the $\nab\otimes \eta$ term:
\begin{equation*}
\begin{split}
&\: \|(\nab\otimes \eta)^{(1)} - (\nab\otimes \eta)^{(2)}\|_{L^1_u L^2_{\ub} L^2(S_{u,\ub},\gamma^{(1)})} \\
\ls &\: \|\gamma^{(1)} - \gamma^{(2)} \|_{L^1_u L^2_{\ub} L^2(S_{u,\ub},\gamma^{(1)})} + \|\slashed \Gamma^{(1)} - \slashed \Gamma^{(2)} \|_{L^1_u L^2_{\ub} L^2(S_{u,\ub},\gamma^{(1)})} + \| \eta^{(1)} - \eta^{(2)} \|_{L^1_u L^2_{\ub} W^{1,2}(S_{u,\ub},\gamma^{(1)})} \\
\ls &\: \f{\mathrm{dist}}{N^{\f 32}}.
\end{split}
\end{equation*}

Next, we consider the quadratic term is $\eta$.
\begin{equation*}
\begin{split}
&\: \|(\eta \otimes \eta)^{(1)} - (\eta \otimes \eta)^{(2)}\|_{L^1_u L^2_{\ub} L^2(S_{u,\ub},\gamma^{(1)})} \\
\ls &\: \|\gamma^{(1)} - \gamma^{(2)} \|_{L^1_u L^2_{\ub} L^2(S_{u,\ub},\gamma^{(1)})}  + \| \eta^{(1)} - \eta^{(2)} \|_{L^1_u L^2_{\ub} W^{1,2}(S_{u,\ub},\gamma^{(1)})} \ls \f{\mathrm{dist}}{N}.
\end{split}
\end{equation*}

Finally, we consider the terms $\trchb \chih$, $\omb \chih$ and $\trch\chibh$. We need to be more careful since some of the terms involved cannot be bounded in $L^\i_u L^\i_{\ub}$ type norms. We consider $\omb \chih$ as an example (the other terms are similar) . By H\"older's inequality, 
\begin{equation}\label{eq:omb.chih.diff}
\begin{split}
&\: \|(\omb \chih)^{(1)} - (\omb \chih)^{(2)}\|_{L^1_u L^2_{\ub} L^2(S_{u,\ub},\gamma^{(1)})} \\
\ls &\: N^{-\f 12} \|\omb^{(1)} - \omb^{(2)} \|_{L^\i_{\ub} L^2_u L^2(S_{u,\ub},\gamma^{(1)})} \|\chih^{(1)}  \|_{L^2_{\ub} L^\i_u L^\i(S_{u,\ub},\gamma^{(1)})} \\
&\: + N^{-\f 12} \| \omb^{(2)} \|_{ L^2_u L^\i_{\ub} L^\i(S_{u,\ub},\gamma^{(1)})} \|\chih^{(1)} - \chih^{(2)} \|_{L^\i_u L^2_{\ub} L^2(S_{u,\ub},\gamma^{(1)})} \ls   \f{\mathrm{dist}}{N^{\f 12}}.
\end{split}
\end{equation}
We note explicitly that in the above estimate, while the differences $\omb^{(1)} - \omb^{(2)}$ has to be controlled by taking $L^2_u$ first before taking $L^\i_{\ub}$, it is important that according to Definition~\ref{double.null.def.2}, $\chih^{(1)}$ can be bounded by taking $L^\i_{\ub}$ first before taking $L^2_u$. A similar comment applies to the product $\omb^{(2)} (\chih^{(1)} - \chih^{(2)})$.

Putting all these together and using \eqref{eq:2i2} in Proposition~\ref{prop:transport}, we obtain the desired estimate. \qedhere
\end{proof}

\subsection{Estimates for $\trch$, $\protect\trchb$ and their derivatives}\label{sec:diff.trch}

\begin{proposition}\label{prop:trch}
$$\|\trch^{(1)} - \trch^{(2)} \|_{L^\i_u L^\i_{\ub} L^2(S_{u,\ub},\gamma^{(1)})} + \|\trchb^{(1)} - \trchb^{(2)} \|_{L^\i_u L^\i_{\ub} L^2(S_{u,\ub},\gamma^{(1)} )}\ls \mathrm{dist}.$$
In particular, 
$$\|\trch^{(1)} - \trch^{(2)} \|_{L^2_{\ub} L^\i_u L^2(S_{u,\ub},\gamma^{(1)})} + \|\trchb^{(1)} - \trchb^{(2)} \|_{L^2_u L^\i_{\ub} L^2(S_{u,\ub},\gamma^{(1)} )}\ls \f{\mathrm{dist}}{N^{\f 12}}.$$
\end{proposition}
\begin{proof}
We will only prove the estimate for $\trchb$; the estimate for $\trch$ is similar.

Fix $\ub \in [\ub_j, \ub_{j+1}]$ and $U \in [u_i,u_{i+1}]$ for the remainder of the proof.

Let $\varphi$ be a function on $S_{U,\ub}$ satisfying 
\begin{equation}\label{eq:trchb.diff.duality}
\|\varphi\|_{L^2(S_{U,\ub},\gamma^{(1)})} \leq 1.
\end{equation}
Extend $\varphi$ on $[u_i, u_{i+1}] \times \mathbb S^2$ by $e_3^{(1)} \varphi = 0$. Proposition~\ref{prop:transport.id}, Gr\"onwall's inequality, and the estimates for $\trchb^{(1)}$ and $\Omg^{(1)}$ together imply that
\begin{equation}\label{eq:trchb.diff.duality.2}
\|\varphi\|_{L^\i_u L^2(S_{u,\ub},\gamma^{(1)})} \ls 1.
\end{equation}

Using $e_3^{(1)} \varphi = 0$, we also obtain
\begin{equation}\label{eq:trchb.diff.e32.varphi}
\begin{split}
e_3^{(2)} \varphi = &\: -(e_3^{(1)} - e_3^{(2)}) \varphi 
= - (\Omg^{(2)})^{-1} \nab_{b^{(1)} - b^{(2)}} \varphi.
\end{split}
\end{equation}

To proceed, we use the equation \eqref{eq:trchb} for both $(\mathcal M, g^{(1)})$ and $(\mathcal M, g^{(2)})$. For $(\mathcal M, g^{(1)})$ we will use $\varphi$ as the test function; while for $(\mathcal M, g^{(2)})$ we will use $\varphi \f{(\Omg^{(1)})^2}{(\Omg^{(2)})^2} \f{\sqrt{\det \gamma^{(1)}}}{\sqrt{\det \gamma^{(2)}}}$ (instead of $\varphi$) as the test function. Taking the difference between the two identities, using the fact that the initial data coincide, and applying \eqref{eq:trchb.diff.e32.varphi}, we then obtain
\begin{align}
&\: \int_{S_{U,\ub}}  \varphi \Omg^{(1)}( (\trchb^{(1)})^- - \f{\Omg^{(1)}}{\Omg^{(2)}} (\trchb^{(2)})^-) \,\mathrm{dA}_{\gamma^{(1)}}  \notag \\
= &\: - \int_{u_i}^{U} \int_{S_{u,\ub}} (\nab_{b^{(1)} - b^{(2)}} \varphi)\, \trchb^{(2)}  \f{(\Omg^{(1)})^2}{\Omg^{(2)}}\,\mathrm{dA}_{\gamma^{(1)}} \,\ud u  \label{eq:trchb.diff.main.1}\\
&\: + \int_{u_i}^{U} \int_{S_{u,\ub}} \varphi \, e_3^{(2)} [\f{(\Omg^{(1)})^2}{(\Omg^{(2)})^2} \f{\sqrt{\det\gamma^{(1)}}}{\sqrt{\det\gamma^{(2)}}}] \, \ud A_{\gamma^{(2)}}\, \ud u \label{eq:trchb.diff.main.later.addition}\\
&\: + \int_{u_i}^{U} \int_{S_{u,\ub}} (-4 \varphi (\omb^{(1)} \trchb^{(1)} - \omb^{(2)} \trchb^{(2)}) + \f12 \varphi ((\trchb^{(1)})^2 -(\trchb^{(2)})^2) )(\Omg^{(1)})^2\,\mathrm{dA}_{\gamma^{(1)}} \,\ud u \label{eq:trchb.diff.main.2} \\
&\: - \int_{u_i}^{U} \int_{S_{u,\ub}} \varphi ( |\chibh^{(1)}|_{\gamma^{(1)}}^2 - |\chibh^{(2)}|_{\gamma^{(2)}}^2) (\Omg^{(1)})^2\,\mathrm{dA}_{\gamma^{(1)}} \,\ud u \label{eq:trchb.diff.main.3}\\
&\: - \int_{(u_i, U)\times \{\ub\}\times \mathbb S^2}  \varphi\,(\ud \underline{\nu}_{\ub}^{(1)} - \f{\sqrt{\det \gamma^{(1)}}}{\sqrt{\det \gamma^{(2)}}} \ud \underline{\nu}_{\ub}^{(2)}). \label{eq:trchb.diff.main.4}
\end{align}

We now estimate each of the terms \eqref{eq:trchb.diff.main.1}--\eqref{eq:trchb.diff.main.4}.

For \eqref{eq:trchb.diff.main.1}, we integrate by parts, use H\"older's inequality and use the estimates in Definition~\ref{double.null.def.2} and \eqref{def:dist} and \eqref{eq:trchb.diff.duality.2} to obtain
\begin{equation}\label{eq:trchb.diff.main.1.est}
\begin{split}
|\eqref{eq:trchb.diff.main.1}| 
\leq &\: \left| \int_{u_i}^{U} \int_{S_{u,\ub}}  \varphi\, \nab_{b^{(1)} - b^{(2)}} (\trchb^{(2)}  \f{(\Omg^{(1)})^2}{\Omg^{(2)}}) \,\mathrm{dA}_{\gamma^{(1)}} \,\ud u \right|  \\
&\: + \left| \int_{u_i}^{U} \int_{S_{u,\ub}}  \varphi\, [\div^{(1)}(b^{(1)} - b^{(2)})]  (\trchb^{(2)}  \f{(\Omg^{(1)})^2}{\Omg^{(2)}}) \,\mathrm{dA}_{\gamma^{(1)}} \,\ud u \right| \\
\ls &\: \|\varphi\|_{L^\i_u L^2(S_{u,\ub},\gamma^{(1)} ) } \|b^{(1)} - b^{(2)}\|_{L^\i_{\ub} L^\i_u W^{1,2}(S_{u,\ub}, \gamma^{(1)} ) } \|\trch^{(2)} \f{(\Omg^{(1)})^2}{\Omg^{(2)}} \|_{L^\i_{\ub} L^2_u W^{1,\i}(S_{u,\ub},\gamma^{(1)})} \ls \f{\mathrm{dist}}{N^{\f 12}}.
\end{split}
\end{equation}

For \eqref{eq:trchb.diff.main.later.addition}, we compute 
\begin{equation*}
\begin{split}
&\: e_3^{(2)} [\f{(\Omg^{(1)})^2}{(\Omg^{(2)})^2} \f{\sqrt{\det\gamma^{(1)}}}{\sqrt{\det\gamma^{(2)}}}] \\
= &\: [\f{(\Omg^{(1)})^2}{(\Omg^{(2)})^2} \f{\sqrt{\det\gamma^{(1)}}}{\sqrt{\det\gamma^{(2)}}}] \{ -2 \om^{(1)} + 2\om^{(2)} + (\Omg^{(2)})^{-1} b^{(2)} - (\Omg^{(1)})^{-1} b^{(1)} + \trch^{(1)} - \trch^{(2)} \\
&\: \qquad \qquad \qquad \qquad   - (\Omg^{(1)})^{-1} \div^{(1)} b^{(1)} + (\Omg^{(2)})^{-1} \div^{(2)} b^{(2)} + (\Omg^{(2)})^{-1} \nab_{b^{(2)}} \log \f{\sqrt{\det\gamma^{(1)}}}{\sqrt{\det\gamma^{(2)}}} \}.
\end{split}
\end{equation*}
Therefore, using Definition~\ref{double.null.def.2}, \eqref{def:dist}, H\"older's inequality and \eqref{eq:trchb.diff.duality.2},
\begin{equation}\label{eq:trchb.diff.main.later.addition.est}
|\eqref{eq:trchb.diff.main.later.addition}| \ls \mathrm{dist}.
\end{equation}

Using Definition~\ref{double.null.def.2}, \eqref{def:dist}, H\"older's inequality and \eqref{eq:trchb.diff.duality.2}, 
\begin{equation}\label{eq:trchb.diff.main.2.3.est}
|\eqref{eq:trchb.diff.main.2}| + |\eqref{eq:trchb.diff.main.3}| \ls \mathrm{dist}.
\end{equation}

Finally, \eqref{eq:trchb.diff.main.4} can be directly estimate using the definition \eqref{def:dist.nu} and \eqref{eq:trchb.diff.duality.2}:
\begin{equation}\label{eq:trchb.diff.main.4.est}
| \eqref{eq:trchb.diff.main.4} | \ls \mathrm{dist}(\ud \nu^{(1)},\,\ud \nu^{(2)}) \ls \mathrm{dist}.
\end{equation}

Combining \eqref{eq:trchb.diff.main.1.est}, \eqref{eq:trchb.diff.main.later.addition.est}, \eqref{eq:trchb.diff.main.2.3.est} and \eqref{eq:trchb.diff.main.4.est} and plugging into the identity preceding these estimates, we obtain
\begin{equation*}
\left| \int_{S_{U,\ub}}  \varphi \Omg^{(1)}( (\trchb^{(1)})^- - \f{\Omg^{(1)}}{\Omg^{(2)}} (\trchb^{(2)})^-) \,\mathrm{dA}_{\gamma^{(1)}}  \right| \ls \mathrm{dist}.
\end{equation*}
By duality and the bounds of $\Omg^{(1)}$ in Definition~\ref{double.null.def.2}, it follows that
\begin{equation}\label{eq:trchb.diff.almost}
\|(\trchb^{(1)})^- - \f{\Omg^{(1)}}{\Omg^{(2)}} (\trchb^{(2)})^- \|_{L^2(S_{U,\ub},\gamma^{(1)})} \ls \mathrm{dist}.
\end{equation}

Since $(U,\ub)$ is arbitrary, by \eqref{eq:trchb.diff.almost}, we have 
\begin{equation}\label{eq:trchb.diff.almost.2}
\|(\trchb^{(1)})^- - \f{\Omg^{(1)}}{\Omg^{(2)}} (\trchb^{(2)})^- \|_{L^\i_u L^\i_{\ub} L^2(S_{u,\ub},\gamma^{(1)})} \ls \mathrm{dist}.
\end{equation}
Finally, we compute
$$(\trchb^{(1)})^- - (\trchb^{(2)})^- = (\trchb^{(1)})^-  - \f{\Omg^{(1)}}{\Omg^{(2)}} (\trchb^{(2)})^- - (1- \f{\Omg^{(1)}}{\Omg^{(2)}}) (\trchb^{(2)})^-$$
and observe that the desired conclusion follows from \eqref{eq:trchb.diff.almost.2}, the bounds for $(1- \f{\Omg^{(1)}}{\Omg^{(2)}})$ in Proposition~\ref{prop:Omg.diff.aux} and the bounds for $(\trchb^{(2)})^-$ in Definition~\ref{double.null.def.2}. \qedhere
\end{proof}

\begin{proposition}\label{prop:nabla.trch}
\begin{equation}\label{eq:nabla.trch.main.1}
\|\nab^{(1)} (\trch^{(1)} - \trch^{(2)})\|_{L^\i_u L^\i_{\ub} L^2(S_{u,\ub},\gamma^{(1)})} + \|\nab^{(1)} (\trchb^{(1)} - \trchb^{(2)})\|_{L^\i_u L^\i_{\ub} L^2(S_{u,\ub},\gamma^{(1)})}\ls \mathrm{dist}.
\end{equation}
In particular,
\begin{equation}\label{eq:nabla.trch.main.2}
\|\nab^{(1)} (\trch^{(1)} - \trch^{(2)})\|_{L^2_{\ub} L^\i_u L^2(S_{u,\ub},\gamma^{(1)})} + \|\nab^{(1)} (\trchb^{(1)} - \trchb^{(2)})\|_{L^2_u L^\i_{\ub} L^2(S_{u,\ub},\gamma^{(1)})}\ls \f{\mathrm{dist}}{N^{\f 12}}.
\end{equation}
\end{proposition}
\begin{proof}
It clearly suffices to prove \eqref{eq:nabla.trch.main.1}; as \eqref{eq:nabla.trch.main.2} follows from \eqref{eq:nabla.trch.main.1} and H\"older's inequality. We will only prove the estimate for $\trchb$; the estimate for $\trch$ is similar. To this end, we will use the equation for $\slashed X(\Omg^{-1}\trchb)$ in \eqref{eq:Xtrchb.0}. 

Fix $\ub \in [\ub_j, \ub_{j+1}]$ and $U \in [u_i,u_{i+1}]$ for the remainder of the proof.

\pfstep{Step~1: Definition of $\slashed X$} Let $\mathring{\slashed X}$ be a smooth vector field on $S_{U, \ub}$ satisfying
\begin{equation}\label{eq:Xtrch.duality}
\|\mathring{\slashed X}\|_{L^2(S_{U,\ub},\gamma^{(1)})} \leq 1.
\end{equation} 
Extend $\mathring{\slashed X}$ to an $S$-tangent vector field $\slashed X$ on $[u_i, U]\times \mathbb S^2$ by stipulating  to be the unique solution to 
\begin{equation}\label{eq:Xtrch.Xtransport}
\begin{cases}
[\f{\rd}{\rd u} + \nab_{b^{(1)}}, \slashed X] =  0 \\
\slashed X(U,\vartheta) = \mathring{\slashed X}(\vartheta).
\end{cases}
\end{equation}
It is easy to check using \eqref{eq:Xtrch.duality}, \eqref{eq:Xtrch.Xtransport} and estimates in Definition~\ref{double.null.def.2} that for every $u\in [u_i,U]$,
\begin{equation}\label{eq:Xtrch.X.est}
\|\slashed X\|_{L^\i_u L^2(S_{u,\ub},\gamma^{(1)})} \ls 1.
\end{equation}

\pfstep{Step~2: Preliminary computations} Before we proceed, we first carry out a couple of computations for terms that will arise later in Step~3 when we apply the equation \eqref{eq:Xtrchb.0} for $\slashed X(\Omg^{-1} \trchb)$. The main point of these computations is to rewrite the expression in a form so that any term involving derivatives of $\slashed X$ can be regrouped as a total divergence.

First, by \eqref{eq:Xtrch.Xtransport}, for any $f\in L^\i_uL^\i_{\ub} W^{2,2}(S_{u,u},\gamma^{(1)})$,
\begin{equation}\label{eq:computation.for.Xtrchb.1}
\begin{split}
&\: [\f{\rd}{\rd u} + \nab_{b^{(2)}}, \slashed X] f =  \nab_{[b^{(2)} - b^{(1)}, \slashed X]}  f \\
= &\: \div^{(2)} ((\nab_{\slashed X} f) (b^{(2)} - b^{(1)})) - (\nab_{\slashed X} f) (\div^{(2)} (b^{(2)} - b^{(1)})) - \nab_{\nab_{\slashed X} (b^{(2)} - b^{(1)})} f - \slashed{\mathrm{Hess}}(f)(\slashed X, b^{(2)} - b^{(1)}). 
\end{split}
\end{equation}

Notice that, after integrating by parts and using \eqref{Ricci.relation}, the first term also satisfies
\begin{equation}\label{eq:computation.for.Xtrchb.2}
\int_{S_{u,\ub}} \div^{(2)} ((\nab_{\slashed X} f) (b^{(2)} - b^{(1)})) (\Omg^{(2)})^2 \, \ud A_{\gamma^{(2)}}  = - 2\int_{S_{u,\ub}} (\nab_{\slashed X} f)  (\nab_{(b^{(2)} - b^{(1)})} \log \Omg^{(2)}) (\Omg^{(2)})^2 \, \ud A_{\gamma^{(2)}}. 
\end{equation}

On the other hand, we also have
\begin{equation}\label{eq:computation.for.Xtrchb.3}
\begin{split}
&\: (\div^{(1)} \slashed X) \f{\sqrt{\det \gamma^{(1)}}}{\sqrt{\det \gamma^{(2)}}} - \div^{(2)} \slashed X \\
=&\: (\div^{(2)}\slashed X) (\f{\sqrt{\det \gamma^{(1)}}}{\sqrt{\det \gamma^{(2)}}} - 1) + (\div^{(1)}\slashed X - \div^{(2)} \slashed X) \f{\sqrt{\det \gamma^{(1)}}}{\sqrt{\det \gamma^{(2)}}}\\
=&\: \div^{(2)} ((\f{\sqrt{\det \gamma^{(1)}}}{\sqrt{\det \gamma^{(2)}}} - 1) \slashed X) - \nab_{\slashed X} (\f{\sqrt{\det \gamma^{(1)}}}{\sqrt{\det \gamma^{(2)}}} - 1) + \f{\sqrt{\det \gamma^{(1)}}}{\sqrt{\det \gamma^{(2)}}} \nab_{\slashed X} \log  \f{\sqrt{\det \gamma^{(1)}}}{\sqrt{\det \gamma^{(2)}}}.
\end{split}
\end{equation}

\pfstep{Step~3: Application of \eqref{eq:Xtrchb.0}} We now apply \eqref{eq:Xtrchb.0} to both $(\mathcal M, g^{(1)})$ and $(\mathcal M, g^{(2)})$, using the fact that $\slashed X$ satisfies \eqref{eq:Xtrch.Xtransport} and that $\Omg^{(1)} = \Omg^{(2)}$, $\gamma^{(1)} = \gamma^{(2)}$ and $(\trchb^{(1)})^+ = (\trchb^{(2)})^+$ when $\ub = 0$.  Using also the computations \eqref{eq:computation.for.Xtrchb.1}--\eqref{eq:computation.for.Xtrchb.3}, we obtain
\begin{align}
&\: \int_{S_{U,\ub}}  \left( (\Omg^{(1)})^2 (\slashed X ((\Omg^{(1)})^{-1} \trchb^{(1)}))^- - (\Omg^{(2)})^2 \f{\sqrt{\det \gamma^{(2)}}}{\sqrt{\det \gamma^{(1)}}}(\slashed X ((\Omg^{(2)})^{-1} \trchb^{(2)}))^- \right) \,\mathrm{dA}_{\gamma^{(1)}} \label{eq:Xtrchb.main.term.to.control} \\
=&\: \int_{u_i}^{U} \int_{S_{u,\ub}}  \{ 2\nab_{b^{(2)} - b^{(1)}}\log \Omg^{(2)} + \div^{(2)} (b^{(2)} - b^{(1)}) \} \slashed X((\Omg^{(2)})^{-1}\trchb^{(2)}) (\Omg^{(2)})^2\,\mathrm{dA}_{\gamma^{(2)}} \,\ud u \label{eq:Xtrchb.main.1} \\
&\: + \int_{u_i}^{U} \int_{S_{u,\ub}}  \left( (\nab_{\nab_{\slashed X} (b^{(2)} - b^{(1)})}) \Omg^{(2)}\trchb^{(2)} - (\Omg^{(2)})^2 \slashed{\mathrm{Hess}}((\Omg^{(2)})^{-1}\trchb^{(2)})(\slashed X, b^{(2)} - b^{(1)}) \right) \,\mathrm{dA}_{\gamma^{(2)}} \,\ud u \label{eq:Xtrchb.main.2}\\
&\: -4 \int_{u_i}^{U} \int_{S_{u,\ub}} [ \omb^{(1)} (\Omg^{(1)})^3\slashed X((\Omg^{(1)})^{-1}\trchb^{(1)}) - \omb^{(2)} (\Omg^{(2)})^3\slashed X((\Omg^{(2)})^{-1}\trchb^{(2)})\f{\sqrt{\det \gamma^{(2)}}}{\sqrt{\det \gamma^{(1)}}} ]\,\mathrm{dA}_{\gamma^{(1)}} \,\ud u \label{eq:Xtrchb.main.3}\\
&\: + \int_{u_i}^{U} \int_{S_{u,\ub}} (2\slashed X(\log\Omg^{(1)}) (\Omg^{(1)})^2|\chibh^{(1)}|_{\gamma^{(1)}}^2 - 2\slashed X(\log\Omg^{(2)}) (\Omg^{(2)})^2|\chibh^{(2)}|_{\gamma^{(2)}}^2 \f{\sqrt{\det \gamma^{(2)}}}{\sqrt{\det \gamma^{(1)}}}) \,\mathrm{dA}_{\gamma^{(1)}} \,\ud u \label{eq:Xtrchb.main.4}\\
&\: +  \int_{u_i}^{U} \int_{S_{u,\ub}}  ((\div^{(1)} \slashed X) (\Omg^{(1)})^2|\chibh^{(1)}|_{\gamma^{(1)}}^2  - (\div^{(2)} \slashed X)(\Omg^{(2)})^2|\chibh^{(2)}|_{\gamma^{(2)}}^2 \f{\sqrt{\det \gamma^{(2)}}}{\sqrt{\det \gamma^{(1)}}}) \,\mathrm{dA}_{\gamma^{(1)}} \,\ud u \label{eq:Xtrchb.main.5}\\
&\: + \int_{(u_i,U)\times \{\ub\}\times \mathbb S^2} (2\slashed X(\log\Omg^{(1)}) \f{\sqrt{\det \gamma^{(1)}}}{\sqrt{\det \gamma^{(2)}}} - 2 \slashed X(\log\Omg^{(2)}) )\,\ud\underline{\nu}_{\ub}^{(2)} \label{eq:Xtrchb.main.6} \\
&\: + \int_{(u_i,U)\times \{\ub\}\times \mathbb S^2} (- \nab_{\slashed X} (\f{\sqrt{\det \gamma^{(1)}}}{\sqrt{\det \gamma^{(2)}}} - 1) + \f{\sqrt{\det \gamma^{(1)}}}{\sqrt{\det \gamma^{(2)}}} \nab_{\slashed X} \log  \f{\sqrt{\det \gamma^{(1)}}}{\sqrt{\det \gamma^{(2)}}})\,\ud\underline{\nu}_{\ub}^{(2)} \label{eq:Xtrchb.main.7}\\
&\: + \int_{(u_i,U)\times \{\ub\}\times \mathbb S^2} \div^{(2)} ((\f{\sqrt{\det \gamma^{(1)}}}{\sqrt{\det \gamma^{(2)}}} - 1) \slashed X) \,\ud\underline{\nu}_{\ub}^{(2)} \label{eq:Xtrchb.main.8}\\
&\: + \int_{(u_i,U)\times \{\ub\}\times \mathbb S^2} (2 \slashed X(\log\Omg^{(1)}) + \div^{(1)} \slashed X ) \,(\ud\underline{\nu}_{\ub}^{(1)} - \f{\sqrt{\det \gamma^{(1)}}}{\sqrt{\det \gamma^{(2)}}} \ud\underline{\nu}_{\ub}^{(2)}). \label{eq:Xtrchb.main.9}
\end{align}

Now \eqref{eq:Xtrchb.main.1}--\eqref{eq:Xtrchb.main.4} do not involve derivatives of $\slashed X$ and we can directly estimate then using Definition~\ref{double.null.def.2}, \eqref{def:dist}, \eqref{eq:Xtrch.X.est}, Propositions~\ref{prop:gamma.inverse.diff} and \ref{prop:Omg.diff.aux}, and H\"older's inequality to get
\begin{equation}\label{eq:Xtrchb.main.est.1}
|\eqref{eq:Xtrchb.main.1}| + |\eqref{eq:Xtrchb.main.2}| + |\eqref{eq:Xtrchb.main.3}| + |\eqref{eq:Xtrchb.main.4}|\ls \mathrm{dist}.
\end{equation}

For \eqref{eq:Xtrchb.main.est.2}, we first integrate by parts away the derivative on $\slashed X$ and then argue as above to obtain
\begin{equation}\label{eq:Xtrchb.main.est.2}
\begin{split}
&\: |\eqref{eq:Xtrchb.main.5}| \\
= &\:  \left|  \int_{u_i}^{U} \int_{S_{u,\ub}} \nab_{\slashed X} \{(\Omg^{(1)})^2|\chibh^{(1)}|_{\gamma^{(1)}}^2\} \,\mathrm{dA}_{\gamma^{(1)}} \,\ud u  - \int_{u_i}^{U} \int_{S_{u,\ub}} \nab_{\slashed X} \{ (\Omg^{(2)})^2|\chibh^{(2)}|_{\gamma^{(2)}}^2\}  \f{\sqrt{\det \gamma^{(2)}}}{\sqrt{\det \gamma^{(1)}}} \,\mathrm{dA}_{\gamma^{(1)}} \,\ud u\right| \\
\ls &\: \mathrm{dist}.
\end{split}
\end{equation}

We then consider the terms involving the measure $\ud\nu^{(2)}_{\ub}$. To handle the terms \eqref{eq:Xtrchb.main.6}--\eqref{eq:Xtrchb.main.8}, we first use the regularity of $\ud\nu^{(2)}_{\ub}$ given in \eqref{eq:nub.add.reg}, and then argue as above (with Definition~\ref{double.null.def.2}, \eqref{def:dist}, \eqref{eq:gamma.det.diff.main}, \eqref{eq:Xtrch.X.est}, Propositions~\ref{prop:gamma.inverse.diff} and \ref{prop:Omg.diff.aux}, and H\"older's inequality) to obtain 
\begin{equation}\label{eq:Xtrchb.main.est.3}
\begin{split}
&\: |\eqref{eq:Xtrchb.main.6}| + |\eqref{eq:Xtrchb.main.7}| + |\eqref{eq:Xtrchb.main.8}| \\
\ls &\: \|2\slashed X(\log\Omg^{(1)}) \f{\sqrt{\det \gamma^{(1)}}}{\sqrt{\det \gamma^{(2)}}} - 2 \slashed X(\log\Omg^{(2)}) \|_{L^\i_u L^\i_{\ub} L^1(S_{u,\ub},\gamma^{(2)})} + \|(\f{\sqrt{\det \gamma^{(1)}}}{\sqrt{\det \gamma^{(2)}}} - 1) \slashed X\|_{L^\i_u L^\i_{\ub} L^1(S_{u,\ub},\gamma^{(2)})}\\
&\: + \| (- \nab_{\slashed X} (\f{\sqrt{\det \gamma^{(1)}}}{\sqrt{\det \gamma^{(2)}}} - 1) + \f{\sqrt{\det \gamma^{(1)}}}{\sqrt{\det \gamma^{(2)}}} \nab_{\slashed X} \log  \f{\sqrt{\det \gamma^{(1)}}}{\sqrt{\det \gamma^{(2)}}})\|_{L^\i_u L^\i_{\ub} L^1(S_{u,\ub},\gamma^{(2)})} \\
\ls &\: \mathrm{dist}.
\end{split}
\end{equation}

Finally, we bound the term \eqref{eq:Xtrchb.main.9}. Using \eqref{def:dist.nu}, \eqref{def:dist}, H\"older's inequality, \eqref{eq:Xtrch.X.est} and Definition~\ref{double.null.def.2}, we obtain
\begin{equation}\label{eq:Xtrchb.main.est.4}
\begin{split}
|\eqref{eq:Xtrchb.main.9}| \ls &\: \mathrm{dist} (\|\slashed X \log \Omg^{(1)} \|_{L^\i_u L^\i_{\ub} L^2(S_{u,\ub}, \gamma^{(1)})} + \|\slashed X \|_{L^\i_u L^\i_{\ub} L^2(S_{u,\ub}, \gamma^{(1)})}) \\
\ls &\: \mathrm{dist} (1 + \|\nab \log \Omg^{(1)} \|_{L^\i_u L^\i_{\ub} L^\i (S_{u,\ub}, \gamma^{(1)})}) \ls \mathrm{dist}.
\end{split}
\end{equation}

\pfstep{Step~4: Conclusion} 
Plugging the estimates \eqref{eq:Xtrchb.main.est.1}--\eqref{eq:Xtrchb.main.est.4} back into \eqref{eq:Xtrchb.main.1}--\eqref{eq:Xtrchb.main.9}, we obtain
\begin{equation}\label{eq:Xtrchb.prefinal}
\begin{split}
 \sup_{\|\slashed X\|_{L^2(S_{U,\ub},\gamma^{(1)})} \leq 1}\eqref{eq:Xtrchb.main.term.to.control} 
\ls &\: \mathrm{dist}.
\end{split}
\end{equation}
By duality, Definition~\ref{double.null.def.2}, \eqref{def:dist}, Propositions~\ref{prop:gamma.inverse.diff} and \ref{prop:Omg.diff.aux}, and \eqref{eq:Xtrchb.prefinal} that we have just established,
\begin{equation}\label{eq:Xtrchb.final}
\begin{split}
&\: \|\nab (\trchb^{(1)} - \trchb^{(2)})\|_{L^2(S_{U,\ub},\gamma^{(1)})} \\
= &\: \sup_{\|\slashed X\|_{L^2(S_{U,\ub},\gamma^{(1)})} \leq 1} \int_{S_{U,\ub}}  \left( (\slashed X (\trchb^{(1)}))^- - (\slashed X ( \trchb^{(2)}))^- \right) \,\mathrm{dA}_{\gamma^{(1)}}  \\
\ls &\: \sup_{\|\slashed X\|_{L^2(S_{U,\ub},\gamma^{(1)})} \leq 1} \int_{S_{U,\ub}}  \Omg^{(1)} \left( (\slashed X (\trchb^{(1)}))^- - (\slashed X ( \trchb^{(2)}))^- \right) \,\mathrm{dA}_{\gamma^{(1)}}  \\
= &\: \sup_{\|\slashed X\|_{L^2(S_{U,\ub},\gamma^{(1)})} \leq 1}  \eqref{eq:Xtrchb.main.term.to.control} + \mathrm{dist} \ls \mathrm{dist}.
\end{split}
\end{equation}

Since $U\in [u_i, u_{i+1}]$ and $\ub \in [\ub_j, \ub_{j+1}]$ are both arbitrary, the desired conclusion follows from \eqref{eq:Xtrchb.final}. \qedhere
\end{proof}

\subsection{Energy estimates for the renormalized curvature components}\label{sec:energy.est}
In this subsection, we control the differences of the renormalized curvature components using energy estimates. We first bound the terms appearing in the equations for the difference of the renormalized curvature components.

\begin{proposition}\label{prop:bianchi.diff}
The following equations, when viewed as transport equations, hold in the weak integrated sense\footnote{We remark that \eqref{eq:bianchi.diff.2}, \eqref{eq:bianchi.diff.3}, \eqref{eq:bianchi.diff.5} and \eqref{eq:bianchi.diff.6} in fact hold in the integrated sense, and therefore a fortiori hold in the weak integrated sense.} of Definition~\ref{def:weak.transport}:
\begin{align}
\label{eq:bianchi.diff.1}
(\nab_3)^{(1)} (\bt^{(1)} - \bt^{(2)}) + \slashed{\nabla} (K^{(1)} -K^{(2)})  - (^*)^{(1)} \slashed{\nabla}(\sigmac^{(1)} - \sigmac^{(2)}) =&\: \mathrm{error}_{\bt,\,K,\,\sigmac},\\
\label{eq:bianchi.diff.2} (\nab_4)^{(1)} (\sigmac^{(1)} - \sigmac^{(2)}) +(\div^*)^{(1)}(\bt^{(1)} - \bt^{(2)})=&\:\mathrm{error}_{\sigmac,\,\bt},\\
\label{eq:bianchi.diff.3} (\nab_4)^{(1)} (K^{(1)} - K^{(2)}) + \div^{(1)} (\bt^{(1)} - \bt^{(2)})=&\:\mathrm{error}_{K,\,\bt},\\
\label{eq:bianchi.diff.4} (\nab_3)^{(1)}(\sigmac^{(1)} - \sigmac^{(2)})+(\div ^*)^{(1)}(\betab^{(1)} - \betab^{(2)})=&\:\mathrm{error}_{\betab,\,K,\,\sigmac},\\
\label{eq:bianchi.diff.5} (\nab_3)^{(1)} (K^{(1)} - K^{(2)})-\div^{(1)} (\betab^{(1)} - \betab^{(2)}) =&\:\mathrm{error}_{K,\,\betab},\\
\label{eq:bianchi.diff.6} (\nab_4)^{(1)}(\betab^{(1)} - \betab^{(2)}) - \slashed{\nabla} (K^{(1)} -K^{(2)}) -(^*)^{(1)}\slashed{\nabla}(\sigmac^{(1)} - \sigmac^{(2)})=&\: \mathrm{error}_{\sigmac,\,\betab},
\end{align}
where
\begin{equation}\label{eq:null.Bianchi.error.1}
\|\mathrm{error}_{\bt,\,K,\,\sigmac}\|_{L^1_u L^2_{\ub} L^2(S_{u,\ub},\gamma^{(1)})}  + \|\mathrm{error}_{K,\,\betab}\|_{L^1_u L^2_{\ub} L^2(S_{u,\ub},\gamma^{(1)})} + \|\mathrm{error}_{\sigmac,\,\betab}\|_{L^1_u L^2_{\ub} L^2(S_{u,\ub},\gamma^{(1)})} \ls \f{\mathrm{dist}}{N^{\f 12}},
\end{equation}
\begin{equation}\label{eq:null.Bianchi.error.2}
\|\mathrm{error}_{\sigmac,\,\bt}\|_{L^1_{\ub} L^2_u L^2(S_{u,\ub},\gamma^{(1)})} + \|\mathrm{error}_{K,\,\beta}\|_{L^1_{\ub} L^2_u L^2(S_{u,\ub},\gamma^{(1)})} + \|\mathrm{error}_{\betab,\,K,\,\sigmac}\|_{L^1_{\ub} L^2_u L^2(S_{u,\ub},\gamma^{(1)})} \ls \f{\mathrm{dist}}{N^{\f 12}}.
\end{equation}
\end{proposition}
\begin{proof}
We only consider a subset of the equations in view of their similarities. More precisely, we will consider \eqref{eq:bianchi.diff.1} and \eqref{eq:bianchi.diff.2}. The equation \eqref{eq:bianchi.diff.3} is similar to \eqref{eq:bianchi.diff.2}. On the other hand, the equations \eqref{eq:bianchi.diff.4}--\eqref{eq:bianchi.diff.6} are similar to \eqref{eq:bianchi.diff.1}--\eqref{eq:bianchi.diff.3} after changing $u$, $\ub$, $e_3$, $e_4$ etc.~appropriately.

\pfstep{Step~1: The equation \eqref{eq:bianchi.diff.1}} We compute
\begin{equation*}
\begin{split}
\mbox{LHS of \eqref{eq:bianchi.diff.1}} =&\:  \underbrace{(\nab_3 \bt + \nab K - ^* \nab\sigmac)^{(1)} - (\nab_3 \bt + \nab K - ^* \nab\sigmac)^{(2)} }_{=:\mathrm{I}}\\
&\: \underbrace{- (\nab_3^{(1)} - \nab_3^{(2)})\bt^{(2)}}_{=:\mathrm{II}} + \underbrace{\{ (^*)^{(1)} - (^*)^{(2)}\} \nab \sigmac^{(2)}}_{=:III}.
\end{split}
\end{equation*}
We now control the RHS in $L^1_u L^2_{\ub} L^2(S_{u,\ub},\gamma^{(1)})$.

To bound the term $\mathrm{I}$ we use the equation \eqref{eq:null.Bianchi.1}. Schematically, there are three types of terms: 
\begin{enumerate}
\item $\nab\chi\star \chib$, $\nab\chi\star \omb$, $\chi\star \nab\chib$, 
\item $\eta\star\chi\star\chib$, $\eta\star\chi\star\omb$, 
\item $\eta K$, $\eta \sigmac$,
\end{enumerate}
where $\star$ denotes some arbitrary contraction with respect to the metric.

We will take one example from each group of terms above; the other terms can be treated in exactly the same manner. 

For the first group, we consider $\nab\chi\star \chib$, which can be treated as \eqref{eq:omb.chih.diff}. More precisely, we use H\"older's inequality, Definition~\ref{double.null.def.2}, \eqref{def:dist}, Propositions~\ref{prop:gamma.inverse.diff} and \ref{prop:Gamma.diff} to obtain
\begin{equation*}
\begin{split}
&\: \|(\nab\chi\star \chib)^{(1)} - (\nab\chi\star \chib)^{(2)} \|_{L^1_u L^2_{\ub} L^2(S_{u,\ub},\gamma^{(1)})} \\
\ls &\: N^{-\f 12} \|(\gamma^{(1)})^{-1} - (\gamma^{(2)})^{-1} \|_{L^\i_u L^\i_{\ub} L^2(S_{u,\ub},\gamma^{(1)})} (\max_{i=1,2} \|\nab\chi^{(i)}\|_{L^\i_u L^2_{\ub} L^\i(S_{u,\ub}, \gamma^{(1)})})(\max_{i=1,2} \|\chib^{(i)}\|_{L^2_{u} L^\i_{\ub} L^\i(S_{u,\ub}, \gamma^{(1)})}) \\
&\: + N^{-\f 12} \|\chib^{(1)} - \chib^{(2)} \|_{L^\i_{\ub} L^2_u L^2(S_{u,\ub},\gamma^{(1)})} \|(\nab \chi)^{(1)}  \|_{L^2_{\ub} L^\i_u L^\i(S_{u,\ub},\gamma^{(1)})} \\
&\: + N^{-\f 12} \| \chib^{(2)} \|_{ L^2_u L^\i_{\ub} L^\i(S_{u,\ub},\gamma^{(1)})} \|\nab^{(1)}(\chi^{(1)} - \chi^{(2)}) \|_{L^\i_u L^2_{\ub} L^2(S_{u,\ub},\gamma^{(1)})} \\
&\:  + N^{-\f 12} \| \chib^{(2)} \|_{ L^2_u L^\i_{\ub} L^\i(S_{u,\ub},\gamma^{(1)})} \|\slashed{\Gamma}^{(1)} - \slashed{\Gamma}^{(2)}\|_{ L^\i_u L^\i_{\ub} L^2(S_{u,\ub},\gamma^{(1)})} \|\chi^{(1)}  \|_{L^\i_u L^2_{\ub} L^\i(S_{u,\ub},\gamma^{(1)})}  
\ls  \f{\mathrm{dist}}{N^{\f 12}}.
\end{split}
\end{equation*}

For the second group, we consider $\eta\star\chi\star\omb$. We use H\"older's inequality, Definition~\ref{double.null.def.2}, \eqref{def:dist}, Proposition~\ref{prop:gamma.inverse.diff} and \eqref{eq:omb.chih.diff} to obtain
\begin{equation*}
\begin{split}
&\: \|(\eta\star\chi\star\omb)^{(1)} - (\eta\star\chi\star\omb)^{(2)} \|_{L^1_u L^2_{\ub} L^2(S_{u,\ub},\gamma^{(1)})} \\
\ls &\: N^{-\f 12} \|\eta^{(1)} - \eta^{(2)}\|_{L^\i_u L^\i_{\ub} L^2(S_{u,\ub},\gamma^{(1)})} \|\chi^{(1)}\|_{L^\i_u L^2_{\ub} L^\i(S_{u,\ub}, \gamma^{(1)})} \|\omb^{(1)}\|_{L^2_{u} L^\i_{\ub} L^\i(S_{u,\ub}, \gamma^{(1)})} \\
&\: + N^{-\f 12}  \|(\gamma^{(1)})^{-1} - (\gamma^{(2)})^{-1} \|_{L^\i_u L^\i_{\ub} L^2(S_{u,\ub},\gamma^{(1)})} \|\chi^{(1)}\|_{L^\i_u L^2_{\ub} L^\i(S_{u,\ub}, \gamma^{(1)})} \|\omb^{(1)}\|_{L^2_{u} L^\i_{\ub} L^\i(S_{u,\ub}, \gamma^{(1)})} \\
&\: + (\|\eta^{(1)}\|_{L^\i_u L^\i_{\ub} L^\i(S_{u,\ub}, \gamma^{(1)})} + \|\eta^{(2)}\|_{L^\i_u L^\i_{\ub} L^\i(S_{u,\ub}, \gamma^{(1)})})  \|(\chi \omb)^{(1)} - (\chi\omb)^{(2)}\|_{L^1_u L^2_{\ub} L^2(S_{u,\ub},\gamma^{(1)})} \\
\ls &\: \f{\mathrm{dist}}{N^{\f 12}}.
\end{split}
\end{equation*}

For the third group, we consider $\eta K$. Using H\"older's inequality, Definition~\ref{double.null.def.2} and \eqref{def:dist}, we obtain
\begin{equation*}
\begin{split}
&\: \|(\eta K)^{(1)} - (\eta K)^{(2)}\|_{L^1_u L^2_{\ub} L^2(S_{u,\ub},\gamma^{(1)})} \\
\ls &\:  N^{-1} \|\eta^{(1)} - \eta^{(2)}\|_{L^\i_u L^\i_{\ub} L^2(S_{u,\ub},\gamma^{(1)})} \|K^{(1)}\|_{L^\i_u L^2_{\ub} L^\i(S_{u,\ub},\gamma^{(1)})}  \\
&\: + N^{-1}\|\eta^{(2)}\|_{L^\i_u L^\i_{\ub} L^\i(S_{u,\ub},\gamma^{(1)}) } \|K^{(1)} - K^{(2)}\|_{L^\i_u L^2_{\ub} L^2(S_{u,\ub},\gamma^{(1)})} 
\ls  \f{\mathrm{dist}}{N}.
\end{split}
\end{equation*}

To handle $\mathrm{II}$ we use \eqref{eq:nab3.diff.2} in Proposition~\ref{prop:operator.diff}. Using the estimates for $\bt^{(2)}$ given by Definition~\ref{double.null.def.2} (and the estimate for $\nab_3^{(2)}\bt^{(2)}$ obtained after using  \eqref{eq:null.Bianchi.1}), we obtain 
$$\|\mathrm{II} \|_{L^1_u L^2_{\ub} L^2(S_{u,\ub},\gamma^{(1)})} \ls \f{\mathrm{dist}}{N^{\f 12}}$$

Finally, the term $\mathrm{III}$ can be bounded by H\"older's inequality, Sobolev embedding (Proposition~\ref{prop:Sobolev}), Definition~\ref{double.null.def.2} and \eqref{def:dist},
\begin{equation*}
\begin{split}
\|\mathrm{III} \|_{L^1_u L^2_{\ub} L^2(S_{u,\ub}, \gamma^{(1)})} \ls &\: N^{-1} \| \gamma^{(1)} - \gamma^{(2)} \|_{L^\i_u L^\i_{\ub} L^4(S_{u,\ub},\gamma^{(1)})}  \|\sigmac\|_{L^\i_u L^2_{\ub} W^{1,4}(S_{u,\ub},\gamma^{(1)})} \\
\ls &\: N^{-1}\| \gamma^{(1)} - \gamma^{(2)} \|_{L^\i_u L^\i_{\ub} W^{1,2}(S_{u,\ub},\gamma^{(1)})}  \|\sigmac^{(2)} \|_{L^\i_u L^2_{\ub} W^{2,2}(S_{u,\ub},\gamma^{(1)})} \ls \f{\mathrm{dist}}{N}.
\end{split}
\end{equation*}

Combining the above considerations, we thus obtain \eqref{eq:bianchi.diff.1}.

\pfstep{Step~2: The equation \eqref{eq:bianchi.diff.2}}
We compute as in Step~1 to obtain
\begin{equation*}
\begin{split}
\mbox{LHS of \eqref{eq:bianchi.diff.2}} =&\:  \underbrace{(\nab_4 \sigmac + \div ^* \bt)^{(1)} - (\nab_4 \sigmac + \div ^* \bt)^{(2)} }_{=:\mathrm{I}}\\
&\: \underbrace{- (\nab_4^{(1)} - \nab_4^{(2)})\sigmac^{(2)}}_{=:\mathrm{II}} - \underbrace{\{ (\div^*)^{(1)} - (\div^*)^{(2)}\} \bt^{(2)}}_{=:III}.
\end{split}
\end{equation*} 
According to \eqref{eq:null.Bianchi.2}, $(\nab_4 \sigmac + \div ^* \bt)^{(1)} - (\nab_4 \sigmac + \div ^* \bt)^{(2)}$ consists of terms similar to those in $(\nab_4\mu)^{(1)} - (\nab_4\mu)^{(2)}$ and therefore $\mathrm{I}$ can be treated in a similar manner as in Proposition~\ref{prop:mu}. We thus obtain
$$\|I\|_{L^1_{\ub} L^2_u L^2(S_{u,\ub},\gamma^{(1)})} \ls N^{-1} \|I\|_{L^\i_{\ub} L^2_u L^2(S_{u,\ub},\gamma^{(1)})}  \ls \f{\mathrm{dist}}{N}.$$
We omit the details.

The term $\mathrm{II}$ can be controlled using \eqref{eq:nab4.diff.2} in Proposition~\ref{prop:operator.diff} after using the estimates for $K$ given by combining Definition~\ref{double.null.def.2} and the equation \eqref{eq:null.Bianchi.2} so that
$$\|\mathrm{II}\|_{L^1_{\ub}L^2_u L^2(S_{u,\ub},\gamma^{(1)})} \ls \f{\mathrm{dist}}{N^{\f 12}}.$$

Finally, the term $\mathrm{III}$ can be bounded by H\"older's inequality, Sobolev embedding (Proposition~\ref{prop:Sobolev}), Definition~\ref{double.null.def.2}, \eqref{def:dist}, Propositions~\ref{prop:gamma.inverse.diff} and \ref{prop:Gamma.diff} so that
\begin{equation*}
\begin{split}
\|\mathrm{III} \|_{L^1_{\ub} L^2_{u} L^2(S_{u,\ub}, \gamma^{(1)})} 
\ls &\: N^{-1} \| \gamma^{(1)} - \gamma^{(2)} \|_{L^\i_u L^\i_{\ub} L^4(S_{u,\ub},\gamma^{(1)})}  \|\bt^{(2)}\|_{L^\i_u L^2_{\ub} W^{1,4}(S_{u,\ub},\gamma^{(1)})} \\
&\: +  N^{-1} \| \slashed\Gamma^{(1)} - \slashed\Gamma^{(2)} \|_{L^\i_u L^\i_{\ub} L^2(S_{u,\ub},\gamma^{(1)})}  \|\bt^{(2)}\|_{L^\i_u L^2_{\ub} L^{\i}(S_{u,\ub},\gamma^{(1)})}  \\
\ls &\: N^{-1}\| \gamma^{(1)} - \gamma^{(2)} \|_{L^\i_u L^\i_{\ub} W^{1,2}(S_{u,\ub},\gamma^{(1)})}  \|\bt^{(2)} \|_{L^\i_u L^2_{\ub} W^{2,2}(S_{u,\ub},\gamma^{(1)})} \ls \f{\mathrm{dist}}{N}.
\end{split}
\end{equation*}
This concludes the proof of \eqref{eq:bianchi.diff.2}. \qedhere
\end{proof}

\begin{proposition}\label{prop:energy.est}
\begin{align}
 \|\beta^{(1)} - \beta^{(2)}\|_{L^\i_u L^2_{\ub} L^2(S_{u,\ub},\gamma^{(1)})} + \|(K^{(1)} - K^{(2)},
\,\sigmac^{(1)} - \sigmac^{(2)}) \|_{(L^\i_{\ub} L^2_u L^2(S_{u,\ub},\gamma^{(1)}))^2} \ls &\: \f{\mathrm{dist}}{N^{\f 14}}, \label{eq:energy.est.1} \\
\|\betab^{(1)} - \betab^{(2)}\|_{L^\i_{\ub} L^2_u L^2(S_{u,\ub},\gamma^{(1)})} + \|(K^{(1)} - K^{(2)},\,\sigmac^{(1)} - \sigmac^{(2)})\|_{(L^\i_u L^2_{\ub} L^2(S_{u,\ub},\gamma^{(1)}))^2}  \ls &\: \f{\mathrm{dist}}{N^{\f 14}}. \label{eq:energy.est.2}
\end{align}
\end{proposition}
\begin{proof}
The proofs of \eqref{eq:energy.est.1} and \eqref{eq:energy.est.2} are similar; we will only treat \eqref{eq:energy.est.1}. This will be achieved by considering \eqref{eq:bianchi.diff.1}--\eqref{eq:bianchi.diff.3}.

\pfstep{Step~1: Derivation of the main energy identities} We begin with \eqref{eq:bianchi.diff.1}, which is satisfied in the weak integrated sense (see~Definition~\ref{def:weaker.transport}). Uinsg Definition~\ref{def:weaker.transport} the fact that $\bt^{(1)} = \bt^{(2)}$ a.e.~on $\{u = u_i\}$, this means that
\begin{equation*}
\begin{split}
&\: \int_{\ub_j}^{\ub_{j+1}} \int_{S_{u,\ub}} \langle\varphi, \bt^{(1)} - \bt^{(2)} \rangle \Omg^{(1)} \,\ud A_{\gamma^{(1)}}\, \ud \ub  \\
&\: + \int_{\ub_j}^{\ub_{j+1}} \int_{u_i}^{u} \int_{S_{u',\ub}} \langle \varphi, -\slashed{\nabla} (K^{(1)} -K^{(2)})  + (^*)^{(1)} \slashed{\nabla}(\sigmac^{(1)} - \sigmac^{(2)}) + \mathrm{error}_{\bt,\,K,\,\sigmac} \rangle (\Omg^{(1)}) \,\ud A_{\gamma^{(1)}}\, \ud u'\,\, \ud \ub \\
&\: + \int_{\ub_j}^{\ub_{j+1}} \int_{u_i}^{u} \int_{S_{u',\ub}} (\langle \varphi,(\trchb^{(1)} - 2\omb^{(1)})(\bt^{(1)} - \bt^{(2)}) \rangle + \langle \nab_3\varphi, (\bt^{(1)}- \bt^{(2)})\rangle) (\Omg^{(1)})^2 \,\ud A_{\gamma^{(1)}}\, \ud u'\,\, \ud \ub =0
\end{split}
\end{equation*}
A priori this holds for $\varphi \in C^1$, but using the bounds in Definition~\ref{double.null.def.2} and \eqref{eq:null.Bianchi.error.1}, we can apply a density argument to show that in fact it suffices to have $\varphi \in C^0_u L^2_{\ub} L^2(S_{u,\ub},\gamma^{(1)})$ and $\nab_3 \varphi \in L^1_u L^2_{\ub} L^2(S_{u,\ub},\gamma^{(1)})$. Therefore, for every fixed $\ub\in [\ub_j,\ub_{j+1}]$, we can choose 
$$\varphi^A(u,\ub',\vartheta) := (((\gamma^{(1)})^{-1})^{AB} (\bt^{(1)} - \bt^{(2)})_B)(u,\ub',\vartheta) \mathbbm 1_{\ub'\in [\ub_j, \ub)}(\ub'),$$
where $\mathbbm 1$ denotes the indicator function. We then obtain
\begin{equation}\label{eq:EE.1}
\begin{split}
&\: \int_{\ub_j}^{\ub} \int_{S_{u,\ub'}} |\bt^{(1)} - \bt^{(2)}|_{\gamma^{(1)}}^2  \Omg^{(1)} \,\ud A_{\gamma^{(1)}}\, \ud \ub'  \\
&\: + 2 \int_{\ub_j}^{\ub} \int_{u_i}^{u} \int_{S_{u',\ub'}} \langle \bt^{(1)} - \bt^{(2)}, -\slashed{\nabla} (K^{(1)} -K^{(2)})  + (^*)^{(1)} \slashed{\nabla}(\sigmac^{(1)} - \sigmac^{(2)})  \rangle_{\gamma^{(1)}} (\Omg^{(1)})^2 \,\ud A_{\gamma^{(1)}}\, \ud u'\,\, \ud \ub' \\
&\: + \int_{\ub_j}^{\ub} \int_{u_i}^{u} \int_{S_{u',\ub'}} (2\langle \bt^{(1)} - \bt^{(2)}, \mathrm{error}_{\bt,\,K,\,\sigmac} \rangle_{\gamma^{(1)}} + (\trchb^{(1)} - 2\omb^{(1)}) |\bt^{(1)} - \bt^{(2)}|_{\gamma^{(1)}}^2) (\Omg^{(1)})^2 \,\ud A_{\gamma^{(1)}}\, \ud u'\,\, \ud \ub' =0.
\end{split}
\end{equation}

In a completely similar manner, but using \eqref{eq:bianchi.diff.2} and \eqref{eq:bianchi.diff.3} instead of \eqref{eq:bianchi.diff.1}, we obtain
\begin{equation}\label{eq:EE.2}
\begin{split}
&\: \int_{u_i}^{u} \int_{S_{u',\ub}} |K^{(1)} - K^{(2)}|^2  \Omg^{(1)} \,\ud A_{\gamma^{(1)}}\, \ud u'  \\
&\: - 2 \int_{\ub_j}^{\ub} \int_{u_i}^{u} \int_{S_{u',\ub'}} (K^{(1)} - K^{(2)})\, \div^{(1)} (\bt^{(1)} - \bt^{(2)})  (\Omg^{(1)})^2 \,\ud A_{\gamma^{(1)}}\, \ud u'\,\, \ud \ub' \\
&\: + \int_{\ub_j}^{\ub} \int_{u_i}^{u} \int_{S_{u',\ub'}} (2 (K^{(1)} - K^{(2)}) \mathrm{error}_{K,\,\bt} + (\trch^{(1)} - 2\om^{(1)}) |K^{(1)} - K^{(2)}|^2) (\Omg^{(1)})^2 \,\ud A_{\gamma^{(1)}}\, \ud u'\,\, \ud \ub' =0,
\end{split}
\end{equation}
and
\begin{equation}\label{eq:EE.3}
\begin{split}
&\: \int_{u_i}^{u} \int_{S_{u',\ub}} |\sigmac^{(1)} - \sigmac^{(2)}|^2  \Omg^{(1)} \,\ud A_{\gamma^{(1)}}\, \ud u'  \\
&\: - 2 \int_{\ub_j}^{\ub} \int_{u_i}^{u} \int_{S_{u',\ub'}} (\sigmac^{(1)} - \sigmac^{(2)})\, (\div^*)^{(1)} (\bt^{(1)} - \bt^{(2)})  (\Omg^{(1)})^2 \,\ud A_{\gamma^{(1)}}\, \ud u'\,\, \ud \ub' \\
&\: + \int_{\ub_j}^{\ub} \int_{u_i}^{u} \int_{S_{u',\ub'}} (2 (\sigmac^{(1)} - \sigmac^{(2)}) \mathrm{error}_{\sigmac,\,\bt} + (\trch^{(1)} - 2\om^{(1)}) |\sigmac^{(1)} - \sigmac^{(2)}|^2) (\Omg^{(1)})^2 \,\ud A_{\gamma^{(1)}}\, \ud u'\,\, \ud \ub' =0.
\end{split}
\end{equation}

Our goal will be to derive an estimate after summing \eqref{eq:EE.1}, \eqref{eq:EE.2} and \eqref{eq:EE.3}. In the next two steps, we will estimate terms in \eqref{eq:EE.1}--\eqref{eq:EE.3}. We will then return to summing \eqref{eq:EE.1}, \eqref{eq:EE.2} and \eqref{eq:EE.3} in Step~4.

\pfstep{Step~2: Handling the main angular terms}  We note that the highest order terms in \eqref{eq:EE.1}--\eqref{eq:EE.3} involving angular derivatives of $\bt^{(1)} - \bt^{(2)}$, $K^{(1)} - K^{(2)}$ and $\sigmac^{(1)} - \sigmac^{(2)}$ cannot be directly controlled by $\mathrm{dist}$. Instead, we need to integrate by parts. Using also \eqref{Ricci.relation} and H\"older's inequality, we obtain
\begin{equation}\label{eq:EE.error.1}
\begin{split}
&\: \left| 2 \int_{\ub_j}^{\ub} \int_{u_i}^{u} \int_{S_{u',\ub'}} \langle \bt^{(1)} - \bt^{(2)}, -\slashed{\nabla} (K^{(1)} -K^{(2)})  + (^*)^{(1)} \slashed{\nabla}(\sigmac^{(1)} - \sigmac^{(2)})  \rangle_{\gamma^{(1)}} (\Omg^{(1)})^2 \,\ud A_{\gamma^{(1)}}\, \ud u'\,\, \ud \ub' \right. \\
&\: \left. - 2 \int_{\ub_j}^{\ub} \int_{u_i}^{u} \int_{S_{u',\ub'}} (K^{(1)} - K^{(2)})\, \div^{(1)} (\bt^{(1)} - \bt^{(2)})  (\Omg^{(1)})^2 \,\ud A_{\gamma^{(1)}}\, \ud u'\,\, \ud \ub' \right.\\
&\:\left. - 2 \int_{\ub_j}^{\ub} \int_{u_i}^{u} \int_{S_{u',\ub'}} (\sigmac^{(1)} - \sigmac^{(2)})\, (\div^*)^{(1)} (\bt^{(1)} - \bt^{(2)})  (\Omg^{(1)})^2 \,\ud A_{\gamma^{(1)}}\, \ud u'\,\, \ud \ub' \right|\\
= &\: \left|2 \int_{\ub_j}^{\ub} \int_{u_i}^{u} \int_{S_{u',\ub'}} (K^{(1)} -K^{(2)}) (\bt^{(1)} - \bt^{(2)}) \cdot \nab(\eta+\etab) \, (\Omg^{(1)})^2 \,\ud A_{\gamma^{(1)}}\, \ud u'\,\, \ud \ub' \right.\\
&\: \left. +2 \int_{\ub_j}^{\ub} \int_{u_i}^{u} \int_{S_{u',\ub'}} (\sigmac^{(1)} - \sigmac^{(2)}) { }^*(\bt^{(1)} - \bt^{(2)}) \cdot \nab(\eta+\etab) \, (\Omg^{(1)})^2 \,\ud A_{\gamma^{(1)}}\, \ud u'\,\, \ud \ub'.\right| \\
\ls &\: \f 1N \|\bt^{(1)}-\bt^{(2)}\|_{L^\i_u L^2_{\ub}L^2(S_{u,\ub},\gamma^{(1)})} \|(K^{(1)} -K^{(2)},\, \sigmac^{(1)} - \sigmac^{(2)})\|_{(L^\i_{\ub} L^2_{u}L^2(S_{u,\ub},\gamma^{(1)}))^2}  \ls \f{\mathrm{dist}^2}{N}.
\end{split}
\end{equation}

\pfstep{Step~3: Estimating the error terms} We now handle the remaining terms in \eqref{eq:EE.1}, \eqref{eq:EE.2} and \eqref{eq:EE.3}:
\begin{equation}\label{eq:EE.error.2}
\begin{split}
&\: \left| \int_{\ub_j}^{\ub} \int_{u_i}^{u} \int_{S_{u',\ub'}} (2\langle \bt^{(1)} - \bt^{(2)}, \mathrm{error}_{\bt,\,K,\,\sigmac} \rangle_{\gamma^{(1)}} + (\trchb^{(1)} - 2\omb^{(1)}) |\bt^{(1)} - \bt^{(2)}|_{\gamma^{(1)}}^2) (\Omg^{(1)})^2 \,\ud A_{\gamma^{(1)}}\, \ud u'\,\, \ud \ub'\right| \\
&\: + \left| \int_{\ub_j}^{\ub} \int_{u_i}^{u} \int_{S_{u',\ub'}} (2 (K^{(1)} - K^{(2)}) \mathrm{error}_{K,\,\bt} + (\trch^{(1)} - 2\om^{(1)}) |K^{(1)} - K^{(2)}|^2) (\Omg^{(1)})^2 \,\ud A_{\gamma^{(1)}}\, \ud u'\,\, \ud \ub' \right| \\
&\: + \left| \int_{\ub_j}^{\ub} \int_{u_i}^{u} \int_{S_{u',\ub'}} (2 (\sigmac^{(1)} - \sigmac^{(2)}) \mathrm{error}_{\sigmac,\,\bt} + (\trch^{(1)} - 2\om^{(1)}) |\sigmac^{(1)} - \sigmac^{(2)}|^2) (\Omg^{(1)})^2 \,\ud A_{\gamma^{(1)}}\, \ud u'\,\, \ud \ub' \right| \\
\ls &\: \|\bt^{(1)}-\bt^{(2)}\|_{L^\i_u L^2_{\ub}L^2(S_{u,\ub},\gamma^{(1)})} \|\mathrm{error}_{\bt,\,K,\,\sigmac}\|_{L^1_u L^2_{\ub}L^2(S_{u,\ub},\gamma^{(1)})} \\
&\: + \|\bt^{(1)}-\bt^{(2)}\|_{L^\i_u L^2_{\ub}L^2(S_{u,\ub},\gamma^{(1)})}^2 \|\trchb^{(1)} - 2\omb^{(1)} \|_{L^1_uL^\i_{\ub} L^\i(S_{u,\ub},\gamma^{(1)})} \\
&\: +\|(K^{(1)} -K^{(2)},\, \sigmac^{(1)} - \sigmac^{(2)})\|_{(L^\i_{\ub} L^2_{u}L^2(S_{u,\ub},\gamma^{(1)}))^2} \|(\mathrm{error}_{K,\,\bt}, \mathrm{error}_{\sigmac,\,\bt})\|_{(L^1_{\ub} L^2_{u}L^2(S_{u,\ub},\gamma^{(1)}))^2} \\
&\: +\|(K^{(1)} -K^{(2)},\, \sigmac^{(1)} - \sigmac^{(2)})\|_{(L^\i_{\ub} L^2_{u}L^2(S_{u,\ub},\gamma^{(1)}))^2}^2  \|\trch^{(1)} - 2\om^{(1)}\|_{L^1_{\ub} L^\i_{u} L^\i(S_{u,\ub},\gamma^{(1)})}
\ls \f{\mathrm{dist}^2}{N^{\f 12}},
\end{split}
\end{equation}
where we have used \eqref{eq:null.Bianchi.error.1} and \eqref{eq:null.Bianchi.error.2}, as well as the estimates
$$\|\trchb^{(1)} - 2\omb^{(1)} \|_{L^1_uL^\i_{\ub} L^\i(S_{u,\ub},\gamma^{(1)})} + \|\trch^{(1)} - 2\om^{(1)}\|_{L^1_{\ub} L^\i_{u} L^\i(S_{u,\ub},\gamma^{(1)})} \ls N^{-\f 12},$$
which in turn follow from the bounds in Definition~\ref{double.null.def.2} and H\"older's inequality.

\pfstep{Step~4: Putting everything together} Summing \eqref{eq:EE.1}, \eqref{eq:EE.2}, \eqref{eq:EE.3}, and using \eqref{eq:EE.error.1} and \eqref{eq:EE.error.2}, we obtain that for every $(u,\ub)\in [u_i,u_{i+1}]\times [\ub_j,\ub_{j+1}]$,
\begin{equation*}
\begin{split}
 \int_{u_i}^{u} \int_{S_{u',\ub}} &\: (|K^{(1)} - K^{(2)}|^2 + |\sigmac^{(1)} - \sigmac^{(2)}|^2)  \Omg^{(1)} \,\ud A_{\gamma^{(1)}}\, \ud u'  \\
&\:+ \int_{\ub_j}^{\ub} \int_{S_{u,\ub'}} |\bt^{(1)} - \bt^{(2)}|_{\gamma^{(1)}}^2  \Omg^{(1)} \,\ud A_{\gamma^{(1)}}\, \ud \ub' \ls \f{\mathrm{dist}^2}{N^{\f 12}}.
\end{split}
\end{equation*}
Since $(u,\ub)$ is arbitrary, we obtain \eqref{eq:energy.est.1} after taking square roots. \qedhere
\end{proof}

\subsection{Elliptic estimates for the Ricci coefficients}\label{sec:elliptic.est}
We recall the following standard elliptic estimate for div-curl systems:
\begin{proposition}\label{prop:elliptic}
Let $(\mathbb S^2, \gamma)$ be a Riemmanian manifold such that $\gamma \in C^1$ and the Gauss curvature $K_\gamma \in L^2(\mathbb S^2,\gamma)$. Then the following holds for all covariant symmetric tensor $\xi$ of rank $(r+1)$ belonging to $W^{1,2}(\mathbb S^2,\gamma)$:
\begin{equation}\label{eq:Bochner}
\int_{\mathbb S^2} (|\nab\xi|_\gamma^2 + (r+1) K_\gamma |\xi|_{\gamma}^2) \,\mathrm{dA}_{\gamma} = \int_{\mathbb S^2} (|\div_\gamma\xi|_\gamma^2 + |\curl_\gamma\xi|_\gamma^2 + rK_\gamma |\tr_\gamma \xi|_\gamma^2) \,\mathrm{dA}_{\gamma}.
\end{equation}

In particular, there exists a constant $C>0$ depending only on $r$, $\|K_\gamma\|_{L^2(\mathbb S^2,\gamma)}$, the area $\mathrm{Area}(\mathbb S^2,\gamma)$, and the isoperimetric constant ${\bf I}(\mathbb S^2,\gamma)$ such that
\begin{equation}\label{eq:elliptic.est}
\|\nab\xi\|_{L^2(\mathbb S^2, \gamma)} \leq C (\|\div_\gamma\xi\|_{L^2(\mathbb S^2, \gamma)} + \|\curl_\gamma\xi\|_{L^2(\mathbb S^2, \gamma)} + \|\xi\|_{L^2(\mathbb S^2, \gamma)} + \|\tr_\gamma \xi\|_{L^{\f 83}(\mathbb S^2, \gamma)}).
\end{equation}
\end{proposition}
\begin{proof}
A proof of \eqref{eq:Bochner} can be found in Lemma~7.1 in \cite{Chr}.

In order to prove \eqref{eq:elliptic.est}, we first note that by H\"older's inequality and the Sobolev inequality (Proposition~\ref{prop:Sobolev}),
\begin{equation*}
\begin{split}
&\: \int_{\mathbb S^2}  ((r+1) |K_\gamma| |\xi|_{\gamma}^2 + r |K_\gamma| |\tr_\gamma \xi|_\gamma^2) \mathrm{dA}_{\gamma} \\
\ls &\: \|K_\gamma\|_{L^4(\mathbb S^2,\gamma)} (\|\xi\|_{L^2(\mathbb S^2, \gamma)} \|\xi\|_{L^4(\mathbb S^2, \gamma)} +  \|\tr_\gamma \xi\|_{L^{\f 83}(\mathbb S^2, \gamma)}^2) \\
\ls &\: \|\xi\|_{L^2(\mathbb S^2, \gamma)}^2 + \|\xi\|_{L^2(\mathbb S^2, \gamma)} \|\nab\xi\|_{L^2(\mathbb S^2, \gamma)} + \|\tr_\gamma \xi\|_{L^{\f 83}(\mathbb S^2, \gamma)}^2.
\end{split}
\end{equation*}
Plugging this into \eqref{eq:Bochner} and applying Young's inequality to absorb $\|\nab\xi\|_{L^2(\mathbb S^2, \gamma)}$, we obtain \eqref{eq:elliptic.est}. \qedhere
\end{proof}

\begin{proposition}\label{prop:nabla.eta}
$$\|\nab^{(1)} (\eta^{(1)} - \eta^{(2)})\|_{L^\i_u L^2_{\ub} L^2(S_{u,\ub},\gamma^{(1)})} + \|\nab^{(1)} (\eta^{(1)} - \eta^{(2)})\|_{L^\i_{\ub} L^2_u L^2(S_{u,\ub},\gamma^{(1)})} \ls \f{\mathrm{dist}}{N^{\f 14}},$$
$$\|\nab^{(1)} (\etab^{(1)} - \etab^{(2)})\|_{L^\i_u L^2_{\ub} L^2(S_{u,\ub},\gamma^{(1)})} + \|\nab^{(1)} (\etab^{(1)} - \etab^{(2)})\|_{L^\i_{\ub} L^2_u L^2(S_{u,\ub},\gamma^{(1)})} \ls \f{\mathrm{dist}}{N^{\f 14}}.$$
\end{proposition}
\begin{proof}
By Definition~\ref{def:curv} and \eqref{eq:mu.def}, 
$$\div^{(1)} (\eta^{(1)} - \eta^{(2)}) = -(\mu^{(1)} - \mu^{(2)}) + K^{(1)} - K^{(2)} - (\div^{(1)} - \div^{(2)})\eta^{(2)},$$
$$\curl^{(1)}(\eta^{(1)} - \eta^{(2)}) = \sigmac^{(1)}-\sigmac^{(2)} - (\curl^{(1)} - \curl^{(2)}) \eta^{(2)},$$

Applying Proposition~\ref{prop:elliptic} with $(\mathbb S^2,\gamma) = (S_{u,\ub}, \gamma^{(1)})$ for every $(u,\ub)$, and using H\"older's inequality, we then obtain
\begin{equation}\label{eq:eta.top}
\begin{split}
&\: \|\nab^{(1)} (\eta^{(1)} - \eta^{(2)})\|_{L^\i_u L^2_{\ub} L^2(S_{u,\ub},\gamma^{(1)})} + \|\nab^{(1)} (\eta^{(1)} - \eta^{(2)})\|_{L^\i_{\ub} L^2_u L^2(S_{u,\ub},\gamma^{(1)})} \\
\ls &\: \f 1{N^{\f 12}} \|\eta^{(1)} - \eta^{(2)}\|_{L^\i_u L^\i_{\ub} L^2(S_{u,\ub},\gamma^{(1)})}  + \f 1{N^{\f 12}} \|\mu^{(1)} - \mu^{(2)}\|_{L^\i_u L^\i_{\ub} L^2(S_{u,\ub},\gamma^{(1)})} \\ 
&\: + \|K^{(1)} - K^{(2)}\|_{L^\i_u L^2_{\ub} L^2(S_{u,\ub},\gamma^{(1)})} + \|K^{(1)} - K^{(2)}\|_{L^\i_{\ub} L^2_u L^2(S_{u,\ub},\gamma^{(1)})} \\
&\: + \f 1{N^{\f 12}}\|(\div^{(1)} - \div^{(2)})\eta^{(2)}\|_{L^\i_u L^\i_{\ub} L^2(S_{u,\ub},\gamma^{(1)})} + \f 1{N^{\f 12}} \|(\curl^{(1)} - \curl^{(2)})\eta^{(2)}\|_{L^\i_u L^\i_{\ub} L^2(S_{u,\ub},\gamma^{(1)})}.
\end{split}
\end{equation}
The first four terms in \eqref{eq:eta.top} can be bounded above by $\ls \f{\mathrm{dist}}{N^{\f 14}}$ using Propositions~\ref{prop:eta}, \ref{prop:mu} and \ref{prop:energy.est}.

For the last two terms in \eqref{eq:eta.top}, we use Proposition~\ref{prop:gamma.inverse.diff} to bound the difference of the inverse metrics and use Proposition~\ref{prop:Gamma.diff} to control the difference of the connections, and combine them with the estimate in Proposition~\ref{prop:metric}. Using also the bound for $\eta^{(2)}$ given by Definition~\ref{double.null.def.2}, H\"older's inequality and Sobolev embedding (Proposition~\ref{prop:Sobolev}), we obtain
\begin{equation*}
\begin{split}
&\: \|(\div^{(1)} - \div^{(2)})\eta^{(2)} \|_{L^2(S_{u,\ub}, \gamma^{(1)})} + \|(\curl^{(1)} - \curl^{(2)})\eta^{(2)} \|_{L^2(S_{u,\ub}, \gamma^{(1)})}\\
\ls &\: \|\gamma^{(1)} - \gamma^{(2)} \|_{L^4(S_{u,\ub}, \gamma^{(1)})} \|\nab \eta^{(2)}\|_{L^4(S_{u,\ub}, \gamma^{(1)})} + \|\slashed\Gamma^{(1)} - \slashed\Gamma^{(2)} \|_{L^2(S_{u,\ub}, \gamma^{(1)})} \|\eta^{(2)}\|_{L^\i(S_{u,\ub},\gamma^{(1)})} \\
\ls &\: \f{\mathrm{dist}}{N^{\f 12}}.
\end{split}
\end{equation*}

Combining all the above bounds gives the desired estimates for $\nab^{(1)} (\eta^{(1)} - \eta^{(2)})$.

The estimate for $\etab^{(1)} - \etab^{(2)}$ can be derived in a similar manner but using instead the following equations\footnote{To derive the equation for $\curl^{(1)}(\etab^{(1)} - \etab^{(2)})$, we use \eqref{Ricci.relation}, and $\curl^{(1)} \nab \log\Omg^{(1)} = \curl^{(2)} \nab \log\Omg^{(2)} = 0$, in addition to \eqref{eq:mu.def}.}:
$$\div^{(1)} (\etab^{(1)} - \etab^{(2)}) = -(\mub^{(1)} - \mub^{(2)}) + K^{(1)} - K^{(2)} - (\div^{(1)} - \div^{(2)})\etab^{(2)},$$
$$\curl^{(1)}(\etab^{(1)} - \etab^{(2)}) = -(\sigmac^{(1)}-\sigmac^{(2)}) - (\curl^{(1)} - \curl^{(2)}) \etab^{(2)}.$$
We omit the details. \qedhere
\end{proof}

\begin{proposition}\label{prop:nabla.chih}
\begin{equation*}
\begin{split}
\|\nab^{(1)} (\chih^{(1)} - \chih^{(2)})\|_{L^\i_u L^2_{\ub} L^2(S_{u,\ub},\gamma^{(1)})} + \|\nab^{(1)}(\chibh^{(1)} - \chibh^{(2)})\|_{ L^\i_{\ub} L^2_u L^2(S_{u,\ub},\gamma^{(1)})}
\ls \f{\mathrm{dist}}{N^{\f 14}}. 
\end{split}
\end{equation*}
\end{proposition}
\begin{proof}
We will only handle the estimate for $\chih^{(1)} - \chih^{(2)}$; that for $\chibh^{(1)} - \chibh^{(2)}$ can be treated similarly.

\pfstep{Step~1: Estimates for $\div^{(1)} (\chih^{(1)} - \chih^{(2)})$} By Definition~\ref{def:curv},
\begin{equation}\label{eq:div.chih.diff}
\begin{split}
\div^{(1)} (\chih^{(1)} - \chih^{(2)}) = &\: - (\div^{(1)} - \div^{(2)})\chih^{(2)} - (\bt^{(1)} - \bt^{(2)}) + \f 12 (\nab^{(1)}\trch^{(1)} - \nab^{(2)}\trch^{(2)})\\
&\: - \f12 \{ [(\eta - \etab)\cdot (\chi - \trch\gamma)]^{(1)} - [(\eta - \etab)\cdot (\chi - \trch\gamma)]^{(2)} \}.
\end{split}
\end{equation}

We now estimate each term in \eqref{eq:div.chih.diff}. By Propositions~\ref{prop:gamma.inverse.diff}, \ref{prop:Gamma.diff} and \ref{prop:metric}
\begin{equation*}
\begin{split}
\| (\div^{(1)} - \div^{(2)})\chih^{(2)} \|_{L^\i_u L^2_{\ub} L^2(S_{u,\ub},\gamma^{(1)})} \ls \f{\mathrm{dist}}{N^{\f 12}}.
\end{split}
\end{equation*}
The $\bt^{(1)} - \bt^{(2)}$ term can be estimated using Proposition~\ref{prop:energy.est} by
\begin{equation*}
\begin{split}
\| \bt^{(1)} - \bt^{(2)} \|_{L^\i_u L^2_{\ub} L^2(S_{u,\ub},\gamma^{(1)})} \ls \f{\mathrm{dist}}{N^{\f 14}}.
\end{split}
\end{equation*}
 The remaining terms can be controlled using the bounds in Definition~\ref{double.null.def.2}, the estimates obtained in Propositions~\ref{prop:eta}, \ref{prop:trch}, \ref{prop:nabla.trch} and H\"older's inequality as follows:
\begin{equation*}
\begin{split}
&\: \|  \f 12 (\nab^{(1)}\trch^{(1)} - \nab^{(2)}\trch^{(2)}) - \f12 \{ [(\eta - \etab)\cdot (\chi - \trch\gamma)]^{(1)} - [(\eta - \etab)\cdot (\chi - \trch\gamma)]^{(2)} \} \|_{L^\i_u L^2_{\ub} L^2(S_{u,\ub},\gamma^{(1)})} \\
\ls &\: \f{\mathrm{dist}}{N^{\f 12}}.
\end{split}
\end{equation*}

Combining all the above estimates, we thus obtain
\begin{equation}\label{eq:div.chih.est}
\|\div^{(1)} (\chih^{(1)} - \chih^{(2)}) \|_{L^\i_u L^2_{\ub} L^2(S_{u,\ub},\gamma^{(1)})} \ls \f{\mathrm{dist}}{N^{\f 14}}.
\end{equation}

\pfstep{Step~2: Estimates for $\tr_{\gamma^{(1)}} (\chih^{(1)} - \chih^{(2)})$} Next, we compute, using $\tr_{\gamma^{(1)}} \chih^{(1)} = \tr_{\gamma^{(2)}} \chih^{(2)} = 0$, that
\begin{equation}\label{eq:elliptic.trch.exp}
\tr_{\gamma^{(1)}} (\chih^{(1)} - \chih^{(2)}) = -((\gamma^{(1)})^{-1} - (\gamma^{(2)})^{-1})\cdot \chih^{(2)}.
\end{equation}
Combining \eqref{eq:elliptic.trch.exp} with Propositions~\ref{prop:gamma.inverse.diff} and \ref{prop:metric}, we obtain
\begin{equation}\label{eq:elliptic.trch.est}
\|\tr_{\gamma^{(1)}} (\chih^{(1)} - \chih^{(2)}) \|_{L^\i_u L^2_{\ub} L^2(S_{u,\ub},\gamma^{(1)})} \ls \f{\mathrm{dist}}{N^{\f 12}}.
\end{equation}

\pfstep{Step~3: Estimates for $\curl^{(1)} (\chih^{(1)} - \chih^{(2)})$ and conclusion of argument} To proceed, note that for any symmetric rank-$2$ tensor $\xi$, we have (see for instance computations leading up to (7.63) in \cite{Chr})
$$\curl^{(1)} \xi = -(^*\nab)^{(1)} \mathrm{tr}_{\gamma^{(1)}} \xi + ({ }^*\div)^{(1)} \xi.$$
Using \eqref{eq:div.chih.est} and  \eqref{eq:elliptic.trch.est}, this implies
\begin{equation}\label{eq:curl.chih.est}
\|\curl^{(1)} (\chih^{(1)} - \chih^{(2)}) \|_{L^\i_u L^2_{\ub} L^2(S_{u,\ub},\gamma^{(1)})} \ls \f{\mathrm{dist}}{N^{\f 14}}.
\end{equation}

Finally, combining \eqref{eq:div.chih.est}, \eqref{eq:elliptic.trch.est} and \eqref{eq:curl.chih.est}, and using Propositions~\ref{prop:elliptic} and \ref{prop:chih}, we obtain the desired conclusion. \qedhere
\end{proof}

\subsection{Estimates for the measure-valued null dusts}\label{sec:diff.null.dust}

\begin{proposition}\label{prop:diff.nu.0}
The following estimates hold:
\begin{equation}\label{eq:diff.nu.0.statement.1}
\sup_{u'\in [u_i, u_{i+1}]} \sup_{\substack{\varphi(\ub,\vartheta)\in C^\infty_c \\ \|\varphi\|_{L^\i_{\ub} L^2(S_{u',\ub},\gamma^{(1)})}\leq 1} } \left| \int_{H_{u'}} \varphi \,(\ud \nu^{(1)}_{u'} - \f{\sqrt{\det\gamma^{(1)}}}{\sqrt{\det\gamma^{(2)}}}\ud \nu^{(2)}_{u'})\right|
\ls  \f{\mathrm{dist}}{N} ,
\end{equation}
\begin{equation}\label{eq:diff.nu.0.statement.2}
\sup_{\ub'\in [\ub_j, \ub_{j+1}]} \sup_{ \substack{ \varphi(u,\vartheta)\in C^\infty_c \\ \|\varphi\|_{L^\i_u L^2(S_{u,\ub'}, \gamma^{(1)})} \leq 1}} \left| \int_{\Hb_{\ub'}} \varphi \,(\ud \nub^{(1)}_{\ub'} - \f{\sqrt{\det\gamma^{(1)}}}{\sqrt{\det\gamma^{(2)}}}\ud \nub^{(2)}_{\ub'})\right| \ls \f{\mathrm{dist}}{N}.
\end{equation}
\end{proposition}
\begin{proof}
In view of their similarities we will only prove \eqref{eq:diff.nu.0.statement.1}. 

Let $\varphi(\ub,\vartheta)$ be a smooth function compactly supported in $(\ub_j,\ub_{j+1})\times \mathbb S^2$ which satisfies $\|\varphi\|_{L^\i_{\ub} L^2(S_{U,\ub},\gamma^{(1)})} \leq 1$.

For every fixed $U\in [0,u_*]$, define $\varphi_U (u,\ub,\vartheta)$ to be the unique solution to the following initial value problem:
\begin{equation}\label{eq:diff.nu.varphi.eq}
\begin{cases}
\rd_u \varphi_U + \nab_{b^{(1)}} \varphi_U = 0 \\
\varphi_U(U,\ub,\vartheta) = \varphi(\ub,\vartheta)
\end{cases}.
\end{equation}
It then follows from integrating \eqref{eq:diff.nu.varphi.eq} and using the estimates in Definition~\ref{double.null.def.2} that
\begin{equation}\label{eq:varphi.U.global.est}
\|\varphi_U\|_{L^\i_u L^\i_{\ub} L^2(S_{u,\ub},\gamma^{(1)})} \ls \|\varphi\|_{L^\i_{\ub} L^2(S_{U,\ub},\gamma^{(1)})} \ls 1.
\end{equation}

We compute
\begin{equation}\label{eq:compute.e3.frac.area.density}
\begin{split}
&\: (\rd_u  + \nab_{b^{(2)}}) (\varphi_U \f{\sqrt{\det\gamma^{(1)}}}{\sqrt{\det\gamma^{(2)}}}) \\
= &\: \f{\sqrt{\det\gamma^{(1)}}}{\sqrt{\det\gamma^{(2)}}} \{ \nab_{b^{(2)} - b^{(1)}} \varphi_U + \varphi_U(\rd_u + \nab_{b^{(2)}})(\log \f{\sqrt{\det\gamma^{(1)}}}{\sqrt{\det\gamma^{(2)}}}) \} \\
=&\: \div^{(2)} [\f{\sqrt{\det\gamma^{(1)}}}{\sqrt{\det\gamma^{(2)}}}  (b^{(2)} - b^{(1)}) \varphi_U] - \f{\sqrt{\det\gamma^{(1)}}}{\sqrt{\det\gamma^{(2)}}} \div^{(1)}[ (b^{(2)} - b^{(1)})] \varphi_U \\
&\: + \f{\sqrt{\det\gamma^{(1)}}}{\sqrt{\det\gamma^{(2)}}} \varphi_U(\rd_u + \nab_{b^{(2)}})(\log \f{\sqrt{\det\gamma^{(1)}}}{\sqrt{\det\gamma^{(2)}}}) \\
=&\: \div^{(2)} [ \f{\sqrt{\det\gamma^{(1)}}}{\sqrt{\det\gamma^{(2)}}} (b^{(2)} - b^{(1)}) \varphi_U] + [(\Omega \trchb)^{(2)} - (\Omega \trchb)^{(1)}]\f{\sqrt{\det\gamma^{(1)}}}{\sqrt{\det\gamma^{(2)}}} \varphi_U,
\end{split}
\end{equation}
where in the last line we have used (by \eqref{metric.derivative})
\begin{equation*}
\begin{split}
&\: (\rd_u + \nab_{b^{(2)}})(\log \f{\sqrt{\det\gamma^{(1)}}}{\sqrt{\det\gamma^{(2)}}}) - \div^{(1)} (b^{(2)} - b^{(1)}) =  (\Omega \trchb)^{(2)} - (\Omega \trchb)^{(1)}.
\end{split}
\end{equation*}
Therefore, using \eqref{eq:diff.nu.varphi.eq}, the transport equation \eqref{eq:nu} corresponding to $\ud\nu^{(1)}$ and $\ud\nu^{(2)}$, the fact that $\ud\nu^{(1)}_{u_i} = \ud\nu^{(2)}_{u_i}$, and the estimates in Definition~\ref{def:ang.reg.null.dust},
\begin{equation*}
\begin{split}
&\: \left| \int_{H_U} \varphi \,(\ud \nu^{(1)}_{U} - \f{\sqrt{\det\gamma^{(1)}}}{\sqrt{\det\gamma^{(2)}}}\, \ud \nu^{(2)}_{U}) \right| \\
= &\: \left| - \int_{u_i}^U \int_{H_u} (\rd_u  + \nab_{b^{(2)}}) (\varphi_U \f{\sqrt{\det\gamma^{(1)}}}{\sqrt{\det\gamma^{(2)}}}) \, \ud \nu^{(2)}_{u} \,\ud u \right| \\
=&\: \left| - \int_{u_i}^U \int_{H_u} \left\{ \div^{(2)} [ \f{\sqrt{\det\gamma^{(1)}}}{\sqrt{\det\gamma^{(2)}}} (b^{(2)} - b^{(1)}) \varphi_U] + [(\Omega \trchb)^{(2)} - (\Omega \trchb)^{(1)}]\f{\sqrt{\det\gamma^{(1)}}}{\sqrt{\det\gamma^{(2)}}} \varphi_U\right\} \, \ud \nu^{(2)}_{u} \,\ud u \right| \\
\ls &\: \|\f{\sqrt{\det\gamma^{(1)}}}{\sqrt{\det\gamma^{(2)}}} (b^{(2)} - b^{(1)}) \varphi_U\|_{L^1_u L^\i_{\ub} L^1(S_{u,\ub}, \gamma^{(2)})} + \|[(\Omega \trchb)^{(2)} - (\Omega \trchb)^{(1)}]\f{\sqrt{\det\gamma^{(1)}}}{\sqrt{\det\gamma^{(2)}}} \varphi_U\|_{L^1_u L^\i_{\ub} L^1(S_{u,\ub}, \gamma^{(2)})} \\
\ls &\: \f 1N \|b^{(2)} - b^{(1)}\|_{L^\i_u L^\i_{\ub} L^2(S_{u,\ub}, \gamma^{(2)})} + \f 1{N^{\f 12}}\|(\Omega \trchb)^{(2)} - (\Omega \trchb)^{(1)}\|_{L^2_u L^\i_{\ub} L^2(S_{u,\ub}, \gamma^{(2)})}) \ls \f {\mathrm{dist}}{N},
\end{split}
\end{equation*}
where in the second to last inequality we have used \eqref{eq:varphi.U.global.est} and the Cauchy--Schwarz inequality; and in the last inequality we have used \eqref{def:dist} and Proposition~\ref{prop:trch}. \qedhere

\end{proof}

\begin{proposition}\label{prop:diff.nu.1}
The following estimates hold:
\begin{equation}\label{eq:diff.nu.statement.1}
 \sup_{u'\in [u_i, u_{i+1}]} \sup_{ \substack{ \slashed{X}(\ub,\vartheta)\in C^\infty_c \\ \|\slashed X\|_{L^\infty_{\ub} L^2(S_{u',\ub},\gamma^{(1)})}\leq 1} }\left| \int_{H_{u'}} \slashed{\div}^{(1)} \slashed X \,(\ud \nu^{(1)}_{u'} - \f{\sqrt{\det\gamma^{(1)}}}{\sqrt{\det\gamma^{(2)}}}\ud \nu^{(2)}_{u'})\right|,
  \end{equation}
\begin{equation}\label{eq:diff.nu.statement.2}
 \sup_{\ub'\in [\ub_j, \ub_{j+1}]} \sup_{\substack{ \slashed{X}(u,\vartheta)\in C^\infty_c \\ \|\slashed X\|_{L^\infty_u L^2(S,\gamma^{(1)})}\leq 1 }} \left| \int_{\Hb_{\ub'}} \slashed{\div}^{(1)} \slashed X \,(\ud \nub^{(1)}_{\ub'} - \f{\sqrt{\det\gamma^{(1)}}}{\sqrt{\det\gamma^{(2)}}}\ud \nub^{(2)}_{\ub'})\right| \ls \f{\mathrm{dist}}{N^{\f 12}}.
 \end{equation}
\end{proposition}
\begin{proof}
We will only prove \eqref{eq:diff.nu.statement.1}; the estimate \eqref{eq:diff.nu.statement.2} is similar. 

Fix $U\in [u_i,\, u_{i+1}]$.

\pfstep{Step~1: Choice of $\slashed{X}$} Let $\mathring{\slashed{X}}(\ub,\vartheta)$ be a $C^\infty_c$ $S$-tangent vector field on $(\ub_j, \ub_{j+1})$. Extend $\mathring{X}$ to $\slashed{X} (u,\ub,\vartheta)$, which is defined to be the unique solution to the following initial value problem:
$$\begin{cases}
\nab_3^{(1)} \slashed{X} = 0 \\
\slashed{X}(U,\ub,\vartheta) = \mathring{\slashed{X}}(\ub,\vartheta)
\end{cases}$$
so that we in particular have
\begin{equation}\label{eq:L2controlofX}
\|\slashed{X}\|_{L^\i_u L^\i_{\ub} L^2(S_{u,\ub}, \gamma^{(1)})}\leq 1.
\end{equation}

We now compute (using\footnote{Notice that we have enough regularity to justify this commutation, where the derivatives on the LHS of \eqref{eq:du.1.divX} are understood as weak derivatives.} \eqref{eq:commutation.3}) that
\begin{equation}\label{eq:du.1.divX}
\begin{split}
&\: (\rd_u + \nab_{b^{(1)}}) \div^{(1)}\slashed{X} \\
= &\: \Omg^{(1)} \left(-\f 12 \trchb^{(1)} \div^{(1)} \slashed{X} - \chibh^{(1)}\cdot^{(1)} \nab^{(1)} \slashed{X} +\betab^{(1)} \cdot^{(1)} \slashed{X} - \chibh^{(1)} \cdot^{(1)} \eta^{(1)}\cdot^{(1)} \slashed{X} \right) =: F_{\slashed X}.
\end{split}
\end{equation}

On the other hand, computing as in \eqref{eq:compute.e3.frac.area.density} and using \eqref{eq:du.1.divX}, we obtain
\begin{equation}\label{eq:du.2.divX}
\begin{split}
&\: (\rd_u + \nab_{b^{(2)}}) (\div^{(1)}\slashed{X} \f{\sqrt{\det\gamma^{(1)}}}{\sqrt{\det\gamma^{(2)}}}) \\
= &\: \div^{(2)} [ \f{\sqrt{\det\gamma^{(1)}}}{\sqrt{\det\gamma^{(2)}}} (b^{(2)} - b^{(1)}) (\div^{(1)}\slashed{X}) ] + [(\Omega \trchb)^{(2)} - (\Omega \trchb)^{(1)}]\f{\sqrt{\det\gamma^{(1)}}}{\sqrt{\det\gamma^{(2)}}} (\div^{(1)}\slashed{X}) + F_{\slashed X} \f{\sqrt{\det\gamma^{(1)}}}{\sqrt{\det\gamma^{(2)}}}\\
=: &\: G_{\slashed X} + F_{\slashed X} \f{\sqrt{\det\gamma^{(1)}}}{\sqrt{\det\gamma^{(2)}}},
\end{split}
\end{equation}
where $F_{\slashed X}$ is as in \eqref{eq:du.1.divX}. 

\pfstep{Step~2: Application of \eqref{eq:nu}} We now apply the transport equation\footnote{Note that a priori, to apply \eqref{eq:nu} requires the test function to be $C^1_c$, but it can easily be checked by an approximation argument that $\div^{(1)}\slashed X$ and $(\div^{(1)}\slashed X) \f{\sqrt{\det\gamma^{(1)}}}{\sqrt{\det\gamma^{(2)}}}$ are also admissible test functions.} for $\ud \nu_u$ (in \eqref{eq:nu}) with \eqref{eq:du.1.divX} and \eqref{eq:du.2.divX}, and the fact that $\ud\nu_{u_i}^{(1)} = \ud\nu_{u_i}^{(2)}$ to obtain
\begin{align}
&\: \left| \int_{H_U} (\div^{(1)}\slashed X)\,(\ud \nu_U^{(1)} - \f{\sqrt{\det\gamma^{(1)}}}{\sqrt{\det\gamma^{(2)}}} \,\ud\nu_U^{(2)}) \right| \nonumber\\
= &\: \left| - \int_{u_i}^U \int_{H_u} (\rd_u + \nab_{b^{(1)}})(\div^{(1)}\slashed X) \, \ud \nu^{(1)}_u \,\ud u\right. \nonumber\\
&\qquad \left. + \int_{u_i}^U \int_{H_u} (\rd_u + \nab_{b^{(2)}}) ((\div^{(1)}\slashed X)\f{\sqrt{\det\gamma^{(1)}}}{\sqrt{\det\gamma^{(2)}}}) \, \ud \nu^{(2)}_u \,\ud u\right| \nonumber\\
\leq &\: \left|\int_{u_i}^U \int_{H_u} F_{\slashed X} \, (\ud \nu_u^{(1)} - \f{\sqrt{\det\gamma^{(1)}}}{\sqrt{\det\gamma^{(2)}}} \,\ud\nu_u^{(2)})\, \ud u \right| \label{eq:nu.diff.error.1}\\
&\: + \left|\int_{u_i}^U \int_{H_u}  G_{\slashed X}\, \ud \nu^{(2)}_u \,\ud u\right|. \label{eq:nu.diff.error.2}
\end{align}

For the remainder of the proof we will estimate \eqref{eq:nu.diff.error.1} and \eqref{eq:nu.diff.error.2}.

\pfstep{Step~3: Estimating \eqref{eq:nu.diff.error.1}} For the term \eqref{eq:nu.diff.error.1}, we recall the definition of $F_{\slashed X}$ in \eqref{eq:du.1.divX} and compute
\begin{equation}
\begin{split}
F_{\slashed X} = &\: -\f 12 \div^{(1)} (\Omg^{(1)} \trchb^{(1)} \slashed{X}) + \f 12 (\nab_{\slashed X} (\Omg^{(1)}\trchb^{(1)})) - \div^{(1)}(\Omg^{(1)}\chibh^{(1)}\cdot \slashed{X}) + \slashed{X}\cdot \div^{(1)}(\Omg^{(1)}\chibh^{(1)}) \\
&\: +\Omg^{(1)} \betab^{(1)} \cdot \slashed{X} - \Omg^{(1)} \chibh^{(1)} \cdot^{(1)} \eta^{(1)}\cdot^{(1)} \slashed{X} \\
=&\: -\div^{(1)}(F_{\slashed X,1}) + F_{\slashed X,2},
\end{split}
\end{equation} 
where $F_{\slashed X,1}$ and $F_{\slashed X,2}$, given by
$$F_{\slashed X,1} := \f 12\Omg^{(1)} \trchb^{(1)} \slashed{X} + \Omg^{(1)}\chibh^{(1)}\cdot \slashed{X},$$
$$F_{\slashed X,2} :=\f 12 (\nab_{\slashed X} (\Omg^{(1)}\trchb^{(1)})) + \slashed{X}\cdot \div^{(1)}(\Omg^{(1)}\chibh^{(1)}) +\Omg^{(1)} \betab^{(1)} \cdot \slashed{X} - \Omg^{(1)} \chibh^{(1)} \cdot^{(1)} \eta^{(1)}\cdot^{(1)} \slashed{X},$$
obey the estimates 
$$\|F_{\slashed X,1} \|_{L^2_u L^\i_{\ub} L^2(S)} + \| F_{\slashed X,2} \|_{L^2_u L^\i_{\ub} L^2(S)} \ls 1.$$
(To see that, we use Definition~\ref{double.null.def.2}, \eqref{eq:L2controlofX} and H\"older's inequality.)

Recalling then the definition of $\mathrm{dist}_\nu(\ud \nu^{(1)}, \ud\nu^{(2)})$ (see \eqref{def:dist.nu}), we obtain 
\begin{equation}\label{eq:nu.diff.error.1.final}
\begin{split}
\mbox{\eqref{eq:nu.diff.error.1}} \ls &\: \mathrm{dist}_\nu(\ud \nu^{(1)}, \ud\nu^{(2)}) \int_{u_i}^U (\|F_{\slashed X,1} \|_{L^\i_{\ub} L^2(S_{u,\ub}, \gamma^{(1)})} + \| F_{\slashed X,2} \|_{L^\i_{\ub} L^2(S_{u,\ub}, \gamma^{(1)})}) \,\ud u \\
\ls &\: \f 1{N^{\f 12}} \mathrm{dist}_\nu(\ud \nu^{(1)}, \ud\nu^{(2)}) \ls \f{\mathrm{dist}}{N^{\f 12}}.
\end{split}
\end{equation}

\pfstep{Step~4: Estimating \eqref{eq:nu.diff.error.2}} We further compute the $G_{\slashed X}$ terms (recall \eqref{eq:du.2.divX}). Using $(\div^{(1)} -\div^{(2)}) \slashed Y = \slashed Y \log \f{\sqrt{\det \gamma^{(1)}}}{\sqrt{\det \gamma^{(2)}}}$, we obtain
\begin{equation}\label{eq:du.2.divX.1}
\begin{split}
&\: \div^{(2)} [ \f{\sqrt{\det\gamma^{(1)}}}{\sqrt{\det\gamma^{(2)}}} (b^{(2)} - b^{(1)}) (\div^{(1)}\slashed{X}) ] \\
= &\: \div^{(2)}\div^{(2)} [ \f{\sqrt{\det\gamma^{(1)}}}{\sqrt{\det\gamma^{(2)}}} (b^{(2)} - b^{(1)}) \otimes \slashed{X} ]  - \div^{(2)} [ \slashed{X} \div^{(2)}(\f{\sqrt{\det\gamma^{(1)}}}{\sqrt{\det\gamma^{(2)}}} (b^{(2)} - b^{(1)}))  ] 
\\
&\: + \div^{(2)} [ \f{\sqrt{\det\gamma^{(1)}}}{\sqrt{\det\gamma^{(2)}}} (b^{(2)} - b^{(1)}) ( \slashed X \log \f{\sqrt{\det \gamma^{(1)}}}{\sqrt{\det \gamma^{(2)}}} ) ],
\end{split}
\end{equation}
and
\begin{equation}\label{eq:du.2.divX.2}
\begin{split}
&\: [(\Omega \trchb)^{(2)} - (\Omega \trchb)^{(1)}]\f{\sqrt{\det\gamma^{(1)}}}{\sqrt{\det\gamma^{(2)}}} (\div^{(1)}\slashed{X}) \\
=&\: \div^{(2)}\{ [(\Omega \trchb)^{(2)} - (\Omega \trchb)^{(1)}]\f{\sqrt{\det\gamma^{(1)}}}{\sqrt{\det\gamma^{(2)}}} \slashed{X}\} - \nab_{\slashed X} \{ [(\Omega \trchb)^{(2)} - (\Omega \trchb)^{(1)}]\f{\sqrt{\det\gamma^{(1)}}}{\sqrt{\det\gamma^{(2)}}}\} \\
&\: + [(\Omega \trchb)^{(2)} - (\Omega \trchb)^{(1)}]\f{\sqrt{\det\gamma^{(1)}}}{\sqrt{\det\gamma^{(2)}}}( \slashed X \log \f{\sqrt{\det \gamma^{(1)}}}{\sqrt{\det \gamma^{(2)}}} ).
\end{split}
\end{equation}
Therefore, \eqref{eq:du.2.divX}, \eqref{eq:du.2.divX.1} and \eqref{eq:du.2.divX.2} together imply that we have the following decomposition\footnote{Note that $G_{\slashed X, r}$ is a tensor of rank $r$.}
$$G_{\slashed X} = \div^{(2)}\div^{(2)}G_{\slashed X,2} + \div^{(2)} G_{\slashed X, 1} + G_{\slashed X, 0},$$
where
$$G_{\slashed X,2} = \f{\sqrt{\det\gamma^{(1)}}}{\sqrt{\det\gamma^{(2)}}} (b^{(2)} - b^{(1)}) \otimes \slashed{X},$$
\begin{equation*}
\begin{split}
&\: G_{\slashed X,1} = -\slashed{X} \div^{(2)}(\f{\sqrt{\det\gamma^{(1)}}}{\sqrt{\det\gamma^{(2)}}} (b^{(2)} - b^{(1)})) + \f{\sqrt{\det\gamma^{(1)}}}{\sqrt{\det\gamma^{(2)}}} (b^{(2)} - b^{(1)}) ( \slashed X \log \f{\sqrt{\det \gamma^{(1)}}}{\sqrt{\det \gamma^{(2)}}} ) \\
&\: + [(\Omega \trchb)^{(2)} - (\Omega \trchb)^{(1)}]\f{\sqrt{\det\gamma^{(1)}}}{\sqrt{\det\gamma^{(2)}}} \slashed{X},
\end{split}
\end{equation*}
$$G_{\slashed X,0} = - \nab_{\slashed X} \{ [(\Omega \trchb)^{(2)} - (\Omega \trchb)^{(1)}]\f{\sqrt{\det\gamma^{(1)}}}{\sqrt{\det\gamma^{(2)}}}\} + [(\Omega \trchb)^{(2)} - (\Omega \trchb)^{(1)}]\f{\sqrt{\det\gamma^{(1)}}}{\sqrt{\det\gamma^{(2)}}}( \slashed X \log \f{\sqrt{\det \gamma^{(1)}}}{\sqrt{\det \gamma^{(2)}}} ). $$
By Definition~\ref{double.null.def.2} (and the definition of $\slashed X$), $G_{\slashed X,2} \in C^0_{\ub} L^{\f 43}(S_{u,\ub},\gamma^{(1)})$, $G_{\slashed X,1},\, G_{\slashed X,0} \in C^0_{\ub} L^2(S_{u,\ub},\gamma^{(1)})$. Moreover, the following estimates are satisfied (using H\"older's inequality, Definition~\ref{double.null.def.2}, \eqref{eq:L2controlofX} and Sobolev embedding (Proposition~\ref{prop:Sobolev})):
\begin{align*}
\|G_{\slashed X,2}\|_{L^\i_{\ub}L^{\f 43}(S_{u,\ub},\gamma^{(1)})} \ls &\: \|b^{(1)} - b^{(2)} \|_{L^\i_{\ub} L^4(S_{u,\ub},\gamma^{(1)})} \\
\ls &\: \|\nab^{(1)} (b^{(1)} - b^{(2)}) \|_{L^\i_{\ub} L^2(S_{u,\ub},\gamma^{(1)})} + \|b^{(1)} - b^{(2)} \|_{L^\i_{\ub} L^2(S_{u,\ub},\gamma^{(1)})}, \\
\|G_{\slashed X,1}\|_{L^\i_{\ub}L^{1}(S_{u,\ub},\gamma^{(1)})} \ls &\: \|\nab^{(1)} (b^{(1)} - b^{(2)}) \|_{L^\i_{\ub} L^2(S_{u,\ub},\gamma^{(1)})} + \|b^{(1)} - b^{(2)} \|_{L^\i_{\ub} L^2(S_{u,\ub},\gamma^{(1)})} \\
&\: + \| (\Omega \trchb)^{(2)} - (\Omega \trchb)^{(1)}\|_{L^\i_{\ub} L^2(S_{u,\ub},\gamma^{(1)})}, \\
\|G_{\slashed X,0}\|_{L^\i_{\ub}L^{1}(S_{u,\ub},\gamma^{(1)})} \ls &\: \|\nab ((\Omega \trchb)^{(2)} - (\Omega \trchb)^{(1)})\|_{L^\i_{\ub} L^2(S_{u,\ub},\gamma^{(1)})} \\
&\: +\|(\Omega \trchb)^{(2)} - (\Omega \trchb)^{(1)}\|_{L^\i_{\ub} L^2(S_{u,\ub},\gamma^{(1)})} .
\end{align*}
Using Definition~\ref{double.null.def.2}, \eqref{def:dist}, Proposition~\ref{prop:trch} and the Cauchy--Schwarz inequality, this implies that
\begin{equation}\label{eq:integralofG}
\int_{u_i}^{u_{i+1}} ( \|G_{\slashed X,2}\|_{L^\i_{\ub}L^{\f 43}(S_{u,\ub},\gamma^{(1)})} + \|G_{\slashed X,1}\|_{L^\i_{\ub}L^{1}(S_{u,\ub},\gamma^{(1)})} + \|G_{\slashed X,0}\|_{L^\i_{\ub}L^{1}(S_{u,\ub},\gamma^{(1)})} )\,\ud u \ls \f {\mathrm{dist}}{N^{\f 12}}.
\end{equation}

Hence, using the regularity estimate for $\ud\nu_u$ in Definition~\ref{def:ang.reg.null.dust} together with \eqref{eq:integralofG}, we obtain
\begin{equation}\label{eq:nu.diff.error.2.final}
\begin{split}
\mbox{\eqref{eq:nu.diff.error.2}} \ls \int_{u_i}^U ( \|G_{\slashed X,2}\|_{L^\i_{\ub}L^{\f 43}(S_{u,\ub},\gamma^{(1)})} + \|G_{\slashed X,1}\|_{L^\i_{\ub}L^{1}(S_{u,\ub},\gamma^{(1)})} + \|G_{\slashed X,0}\|_{L^\i_{\ub}L^{1}(S_{u,\ub},\gamma^{(1)})} )\,\ud u \ls \f {\mathrm{dist}}{N^{\f 12}}.
\end{split}
\end{equation}

\pfstep{Step~5: Putting everything together} Using the estimates \eqref{eq:nu.diff.error.1.final} and \eqref{eq:nu.diff.error.2.final} for \eqref{eq:nu.diff.error.1} and \eqref{eq:nu.diff.error.2}, and returning to Step~2, we obtain
$$\left| \int_{H_U} (\div^{(1)}\slashed X)\,(\ud \nu_U^{(1)} - \f{\sqrt{\det\gamma^{(1)}}}{\sqrt{\det\gamma^{(2)}}} \,\ud\nu_U^{(2)}) \right| \ls \f{\mathrm{dist}}{N^{\f 12}}.$$
In view of (1) the arbitrariness of the prescription of $\slashed X$ at $u=U$ (subject to $\|\slashed X\|_{L^\i_{\ub}L^2(S_{U,\ub},\gamma^{(1)})}\leq 1$) and (2) the arbitrariness of $U \in [u_i, u_{i+1}]$, we have obtain \eqref{eq:diff.nu.statement.1}. \qedhere
\end{proof}

\subsection{Putting everything together: Proof of Theorem~\ref{thm:uniqueness}}\label{sec:uniqueness.everything}

\begin{proof}[Proof of Theorem~\ref{thm:uniqueness}]
Recalling the definition of the distance function in \eqref{def:dist}, we see that by Propositions~\ref{prop:metric}, \ref{prop:eta}, \ref{prop:chih}, \ref{prop:trch}, \ref{prop:nabla.trch}, \ref{prop:energy.est}, \ref{prop:nabla.eta}, \ref{prop:nabla.chih}, \ref{prop:diff.nu.0} and \ref{prop:diff.nu.1},
$$\mathrm{dist} \ls \f{\mathrm{dist}}{N^{\f 14}}.$$
Choosing $N$ large enough gives $\mathrm{dist} = 0$, which implies the desired uniqueness result. \qedhere
\end{proof}

\section{Weak approximation theorem}\label{sec.approx.thm}

In this section we prove our final two theorems, Theorem~\ref{thm:main.local.dust} and Theorem~\ref{thm:reverse.Burnett}. Given what we have obtained so far (Theorems~\ref{main.thm} and \ref{thm:uniqueness}), the key step to both Theorem~\ref{thm:main.local.dust} and Theorem~\ref{thm:reverse.Burnett} is an approximation result which allows us to approximate any null dust initial data set (with merely measure-valued null dust) by a smooth vacuum initial data set; see already Proposition~\ref{prop:final.approx}. This result will be carried out in the following steps:
\begin{enumerate}
\item We first show that all \emph{smooth null dust initial data sets} with two families of null dusts in double null foliation can be approximated by \emph{smooth vacuum initial data sets} (Proposition~\ref{prop.data.approx} in \textbf{Section~\ref{sec:approx.smooth.dust.by.vac}}).
\item Our next step is to approximate \emph{measure-valued null dust initial data} by \emph{smooth null dust initial data} (Proposition~\ref{prop:f.approx} in \textbf{Section~\ref{sec:f.approx}}).
\item Combining the previous steps, we then achieve an approximation of \emph{measure-valued null dust initial data} by \emph{smooth null dust initial data sets} (Proposition~\ref{prop:final.approx} in \textbf{Section~\ref{sec:final.approx}}).
\end{enumerate}
Once we obtain the main approximation result, we conclude the proofs of Theorem~\ref{thm:main.local.dust} and Theorem~\ref{thm:reverse.Burnett}. This will be carried out in \textbf{Section~\ref{sec:approx.final}}.

Since the approximation results are the same on $H_0$ and $\Hb_0$, in what follows we will focus on $H_0$. The case for $\Hb_0$ is the same; see Proposition~\ref{prop:final.approx.1}.

\subsection{Approximating smooth null dust data by vacuum data}\label{sec:approx.smooth.dust.by.vac}

Before we proceed, recall Definition~\ref{def:SARCID} only gives a notion of the null dust where the null dust is a measure. We now introduce a convention for smooth null dust data in order to carry out the approximation procedure. For this we will stipulate that $\ud\nu_{\mathrm{init}}$ is absolutely continuous with respect to $\ud A_\gamma \, \ud \ub$ and for a smooth function $f$,
$$\ud \nu_{\mathrm{init}} = \Phi^{-2} \Omg^{-2} f\, \ud A_\gamma \, \ud \ub.$$
In this smooth setting, the constraint equation \eqref{eq:constraints.first.time} is replaced by 
\begin{equation}\label{eq:constraint.ODE}
\f{\rd^2 \Phi}{\rd\ub^2}=2\f{\rd \log\Omega}{\rd\ub}\f{\rd\Phi}{\rd\ub}-\f {1}8 |\f{\rd\hat{\gamma}}{\rd \ub}|_{\hat{\gamma}}^2 \Phi-\f 12\Phi^{-1} f \mbox{ on $H_0$},
\end{equation}
where $f$ is non-negative and $|\f{\rd\hat{\gamma}}{\rd \ub}|_{\hat{\gamma}}^2$ is defined as in \eqref{eq:def.|dubgamma|^2}, (with obviously modifications on $\Hb_0$). 

\begin{proposition}\label{prop.data.approx}
Let $K\in \mathbb N$. Fix  an arbitrary smooth metric $\mathring{\gamma}$  on $S_{0,0}$. Extend $\mathring{\gamma}$ to $[0,\ub_*]\times \mathbb S^2$ by
\begin{equation}\label{eq:data.approx.Lgamma}
\slashed{\mathcal L}_{\f{\rd}{\rd\ub}} \mathring{\gamma} = 0.
\end{equation}
Assume the following:
\begin{itemize}
\item Let $\Omega$ be a given smooth positive function on $\{0 \}\times [0,\ub_*]\times \mathbb S^2$.
\item Suppose $\hat{\gamma}^{(dust)}$ is a smooth $S$-tangent covariant $2$-tensor on $\{0\} \times [0,\ub_*]\times \mathbb S^2$, which is a Riemannian metric on $S_{0,\ub}$ for every $\ub\in [0, \ub_*]$. 
\item Suppose $\Phi^{(dust)}$ is a positive smooth function on $\{0\}\times [0,\ub_*]\times \mathbb S^2$.
\item Assume that the following constraint equation is satisfied
\begin{equation}\label{eq:data.approx.dust.constraint}
\f{\rd^2 \Phi^{(dust)}}{\rd\ub^2}=2\f{\rd \log\Omega}{\rd\ub}\f{\rd\Phi^{(dust)}}{\rd\ub}-\f {1}8 |\f{\rd\hat{\gamma}^{(dust)}}{\rd \ub}|_{\hat{\gamma}^{(dust)}}^2 \Phi^{(dust)}-\f 12 \f{f}{\Phi^{(dust)}}
\end{equation}
for a non-negative smooth function $f(\ub,\vartheta)$. 
\item Suppose moreover that  there exists $\emptyset \neq U\subset \mathbb S^2$ such that for every $\vartheta\in U$, $f(\ub,\vartheta) = 0$ for all $\ub\in [0,\ub_*]$.
\end{itemize}

Then there exists a sequence of smooth $\{(\hat{\gamma}_n, \,\Phi_n) \}_{n=1}^{+\infty}$ on $\{0\}\times [0,\ub_*]\times \mathbb S^2$ such that the following holds:
\begin{enumerate}
\item (Positivity of $\hat{\gamma}_n$ and $\Phi_n$) For every $n$ sufficiently large, $\hat{\gamma}_n$ is a smooth $S$-tangent tensor on $\{0\}\times [0,\ub_*]\times \mathbb S^2$ which is a Riemannian metric on $S_{0,\ub}$  satisfying
$\f{\det {\hat{\gamma}_n}}{\det\mathring{\gamma}} = 1$,
and $\Phi_n$ is a positive smooth function on $\{0\}\times [0,\ub_*]\times \mathbb S^2$.
\item (Vacuum constraint) For every $n\geq 1$,
\begin{equation}\label{eq:data.approx.vac.constraint}
\begin{cases}
\f{\rd^2 \Phi_n}{\rd\ub^2}=2\f{\rd \log\Omega}{\rd\ub}\f{\rd\Phi_n}{\rd\ub}-\f {1}8 |\f{\rd\hat{\gamma}_n}{\rd \ub}|_{\hat{\gamma_n}}^2 \Phi_n,\\
\Phi_n(\ub=0) = \Phi^{(dust)}(\ub =0),\quad \f{\rd\Phi_n}{\rd\ub} (\ub = 0) = \f{\rd\Phi^{(dust)}}{\rd\ub}(\ub = 0).
\end{cases}
\end{equation}
\item (Uniform estimates and convergence) The following hold for some implicit constants depending only on $f$, $\mathring{\gamma}$, $\hat{\gamma}^{(dust)}$, $\Phi^{(dust)}$ and $\Omega$ but \underline{independent} of $n$:
\begin{equation}\label{eq:data.approx.gamma.easy}
\|\hat{\gamma}_n - \hat{\gamma}^{(dust)}\|_{L^\i_{\ub}W^{K,\infty}(S_{0,\ub},\mathring{\gamma})} \ls n^{-1},\quad \|\f{\rd \hat{\gamma}_n}{\rd\ub} \|_{L^\i_{\ub}W^{K,\infty}(S_{0,\ub},\mathring{\gamma})}\ls 1,
\end{equation}
\begin{equation}\label{eq:data.approx.Phi.final}
\|\Phi_n - \Phi^{(dust)}\|_{L^\i_{\ub}W^{K,\infty}} \ls n^{-1}, \quad \|\f{\rd (\Phi_n - \Phi^{(dust)})}{\rd \ub}\|_{L^\i_{\ub}W^{K,\infty}(S_{0,\ub},\mathring{\gamma})}\ls n^{-1}.
\end{equation}
\item (Refined\footnote{Note that Part~(3) already contains the statement $\|\f{\rd \hat{\gamma}_n}{\rd\ub} \|_{L^2_{\ub}W^{K,\infty}(S_{0,\ub},\mathring{\gamma})}\ls 1$ (after using H\"older's inequality). The key point here is that we further analyze the dependence of the implicit constant. In particular, the constant $C>0$ in the estimate can be chosen uniformly for all $\hat{\gamma}^{(dust)}$ satisfying \eqref{eq:data.approx.uniform.assumption}.} uniform estimates for $\f{\rd \hat{\gamma}_n}{\rd\ub}$) For any fixed $\mathring{\gamma}$ (satisfying \eqref{eq:data.approx.Lgamma}) and $\hat{\gamma}_0$ (satisfying $\f{\det {\hat{\gamma}_0}}{\det\mathring{\gamma}} = 1$), there exist $C>0$ and $\ep>0$ (both depending only on $\mathring{\gamma}$ and $\hat{\gamma}_0$) such that if 
\begin{equation}\label{eq:data.approx.uniform.assumption}
\|\hat{\gamma}^{(dust)} - \hat{\gamma}_0\|_{L^\i_{\ub} W^{K,\infty}(S_{0,\ub},\mathring{\gamma})} <\ep,
\end{equation}
then for all $n$ sufficiently large
\begin{equation}\label{eq:data.approx.uniform.consequence}
\|\f{\rd \hat{\gamma}_n}{\rd\ub} \|_{L^2_{\ub}W^{K,\infty}(S_{0,\ub},\mathring{\gamma})} \leq C (1 + \|\f{\rd\hat{\gamma}^{(dust)}}{\rd\ub}\|_{L^2_{\ub} W^{K,\infty}(S_{0,\ub},\mathring{\gamma})} + \|f\|_{L^1_{\ub} W^{K,\infty}(S_{0,\ub},\mathring{\gamma})}).
\end{equation}
\item (Oscillatory estimates) 
There exist smooth functions $\{F_n\}_{n=1}^{+\infty}$ with uniform in $n$ estimates 
\begin{equation}\label{eq:data.approx.Fn.est}
\|F_n\|_{L^\i_{\ub} W^{K,\infty}(S_{0,\ub},\mathring{\gamma})} \ls 1
\end{equation}
such that
\begin{equation}\label{eq:data.approx.dgamma.weak}
\| [|\f{\rd\hat{\gamma}_n}{\rd \ub}|_{\hat{\gamma}_n}^2  - |\f{\rd\hat{\gamma}^{(dust)}}{\rd \ub}|_{\hat{\gamma}^{(dust)}}^2] (\Phi^{(dust)})^2 - 4 f -\f 1n \f{\rd F_n}{\rd \ub} \|_{L^\i_{\ub}W^{K,\infty}(S_{0,\ub},\mathring{\gamma})} \ls n^{-1},
\end{equation}
In both \eqref{eq:data.approx.Fn.est} and \eqref{eq:data.approx.dgamma.weak}, the implicit constants depend only on $f$, $\mathring{\gamma}$, $\hat{\gamma}^{(dust)}$ and $\Phi^{(dust)}$ but are \underline{independent} of $n$.
\end{enumerate}
\end{proposition}
\begin{proof}
Since there exists $\emptyset \neq U\subset \mathbb S^2$ such that $f(\ub,\vartheta) = 0$ for $\vartheta\in U$ and $\ub \in [0,\ub_*]$, we can work with a local coordinate system on $\mathbb S^2$. 

\pfstep{Step~1: Construction of $\hat{\gamma}_n$} Suppose that in local coordinates $\hat{\gamma}$ be given by
$$\hat{\gamma}^{(dust)}=\left( \begin{array}{cc}
a & b \\
b & d \end{array} \right)$$
where $a$, $b$ and $d$ are smooth functions of $\ub$ and $\vartheta$. By assumption $\hat{\gamma}^{(dust)} = \mathring{\gamma}$ and thus we have 
\begin{equation}\label{det.condition}
ad-b^2=\det \mathring{\gamma}.
\end{equation} 
Notice moreover that since $\hat{\gamma}$ is positive definite, we must have $a>0$ and $d>0$. 

We now define the sequence $\hat{\gamma}_n$ by
\begin{equation}\label{eq:data.approx.gamma.def}
\hat{\gamma}_n:=\left( \begin{array}{cc}
a+\f{(ad - b^2)}{d} \f{2 f^{\f 12}}{\Phi^{(dust)}}\f{1}{k n}\sin(k n\ub) & b \\
b & d - (ad - b^2)\f{\f{2 f^{\f 12}}{\Phi^{(dust)}}\f{1}{k n}\sin(k n\ub)}{a+\f{(ad-b^2)}{d}\f{2 f^{\f 12}}{\Phi^{(dust)}}\f{1}{k n}\sin(k n\ub)} \end{array} \right),
\end{equation}
where $k$ is some large but fixed real parameter chosen so that $\hat{\gamma}_n$ is positive definite for all $n\geq 1$. Note that \eqref{eq:data.approx.gamma.def} is well-defined since\footnote{It is for this step that we have used the condition of non-negativity of $\Phi^{(dust)}$ and $f$.} $a>0$, $d > 0$, $\Phi^{(dust)} \geq 0$ and $f\geq 0$. Furthermore,  \eqref{eq:data.approx.gamma.def} is chosen so that indeed by \eqref{det.condition},
\begin{equation*}
\begin{split}
\det(\hat{\gamma}_n)=&\:(ad-b^2) \{ 1+ \f{2 f^{\f 12}}{\Phi^{(dust)}} \f{1}{k n}\sin(k n\ub)-\f{2 a  f^{\f 12}\f{1}{k n}\sin(k n\ub)}{\Phi^{(dust)}  [a+\f{(ad-b^2)}{d}\f{2 f^{\f 12}}{\Phi^{(dust)}} \f{1}{k n}\sin(k n\ub)]}\\
&\: \qquad\qquad -\f{(ad-b^2)\f{4 f}{ d (k n)^2 (\Phi^{(dust)})^2}\sin^2(k n\ub)}{a + \f{(ad-b^2)}{d}\f{2 f^{\f 12}}{\Phi^{(dust)}}\f{1}{k n}\sin(k n\ub)} \}=ad-b^2 = \det\hat{\gamma}^{(dust)} = \det\mathring{\gamma}.
\end{split}
\end{equation*}

From the formula \eqref{eq:data.approx.gamma.def}, the smoothness of $(a,\,b,\,d,\,f)$, and the positivity of $a$ and $d$, it immediate follows that  \eqref{eq:data.approx.gamma.easy} holds.

To obtain \eqref{eq:data.approx.uniform.consequence} and \eqref{eq:data.approx.dgamma.weak}, we compute
\begin{equation}\label{eq:data.approx.dgamma.prelim}
\f{\rd}{\rd\ub}\hat\gamma_n=\left( \begin{array}{cc}
\f{\rd a}{\rd\ub} + \f{(ad-b^2)}{d} \f 2{\Phi^{(dust)}} (\f{f}{ d})^{\f 12}\cos(k n\ub) & \f{\rd b}{\rd\ub} \\
\f{\rd b}{\rd\ub} & \f{\rd d}{\rd\ub} - \f{(ad-b^2)}{a}\f{2 f^{\f 12}}{ \Phi^{(dust)}} \cos(k n\ub) \end{array} \right)+O_K(\f 1n),
\end{equation}
where $O_K(\f 1n)$ denotes terms whose $L^\i_{\ub} W^{K,\infty}(S_{u,\ub})$ norm is bounded above up to a constant by $\f 1n$ (where the constant depends on $f$, $\mathring{\gamma}$, $\hat{\gamma}^{(dust)}$ and $\Phi^{(dust)}$).

From this formula we obtain \eqref{eq:data.approx.uniform.consequence}. Indeed, 
\begin{itemize}
\item the $\f{\rd a}{\rd\ub}$, $\f{\rd b}{\rd \ub}$, $\f{\rd d}{\rd \ub}$ terms in the $L^2_{\ub} W^{K,\infty}(S_{0,\ub},\mathring{\gamma})$  norm can be controlled by $\|\f{\rd\hat{\gamma}^{(dust)}}{\rd\ub}\|_{L^2_{\ub} W^{K,\infty}(S_{0,\ub},\mathring{\gamma})}$;
\item the $O_K(\f 1n)$ terms in the $L^2_{\ub} W^{K,\infty}(S_{0,\ub},\mathring{\gamma})$  norm can be bounded by $1$ after choosing $n$ to be sufficient large;
\item the $\f{(ad-b^2)}{d} (\f{4 f}{\Phi^{(dust)} d})^{\f 12}\cos(k n\ub)$ and $\f{(ad-b^2)}{a}(\f{4 f}{ \Phi^{(dust)}})^{\f 12}\cos(k n\ub)$ terms in the $L^2_{\ub}W^{K,\infty}(S_{0,\ub},\mathring{\gamma})$  norm can be controlled by the $\| f\|_{L^1_{\ub} W^{K,\infty}(S_{0,\ub},\mathring{\gamma}) }$ norm. Moreover, the constant of the estimate can be chosen to depend continuously on $\hat{\gamma}_n$. In particular, the constant can be chosen to be uniform under the assumption \eqref{eq:data.approx.uniform.assumption}.
\end{itemize}

To establish \eqref{eq:data.approx.dgamma.weak}, we compute using \eqref{eq:data.approx.gamma.def} and \eqref{eq:data.approx.dgamma.prelim} that
\begin{equation}\label{eq:data.approx.dgamma.weak.2}
\begin{split}
|\f{\rd\hat{\gamma}_n}{\rd \ub}|_{\hat{\gamma}_n}^2 =&\: \left((\hat{\gamma}_n^{-1})^{AC}(\hat{\gamma}_n^{-1})^{BD}(\f{\rd}{\rd\ub}\hat{\gamma}_n)_{AB}(\f{\rd}{\rd\ub}\hat{\gamma}_n)_{CD}\right)(\ub,\vartheta)\\
=&\: \left(((\hat{\gamma}^{(dust)})^{-1})^{AC} ((\hat{\gamma}^{(dust)})^{-1})^{BD}(\f{\rd}{\rd\ub}\hat{\gamma}_n)_{AB}(\f{\rd}{\rd\ub}\hat{\gamma}_n)_{CD}\right)(\ub,\vartheta) +O_K(\f 1n) \\
=&\: \f{d^2}{(ad-b^2)^2} [\f{\rd a}{\rd\ub} + \f{(ad-b^2)}{d} \f{2 f^{\f 12}}{ \Phi^{(dust)}} \cos(k n\ub)]^2 - \f{2b^2}{(ad-b^2)^2}(\f{\rd b}{\rd\ub})^2 \\
&\: + \f{a^2}{(ad-b^2)^2} [ \f{\rd d}{\rd\ub} - \f{(ad-b^2)}{a} \f{2 f^{\f 12}}{ \Phi^{(dust)}} \cos(k n\ub) ]^2 + O_K(\f 1n) \\
= &\: \f{1}{(ad-b^2)^2} [ (d\f{\rd a}{\rd\ub})^2 -2 (b \f{\rd b}{\rd\ub})^2 + a (\f{\rd d}{\rd\ub})^2] + \f{4 f }{(\Phi^{(dust)})^2} (1+ \cos(2kn\ub)) \\
&\:  + \f{2 d}{ad-b^2} \f{\rd a}{\rd\ub} \f{2 f^{\f 12}}{ \Phi^{(dust)}}\cos(k n\ub) - \f{2 a}{ad-b^2}\f{\rd d}{\rd\ub}\f{2 f^{\f 12}}{ \Phi^{(dust)}}\cos(k n\ub) + O_K(\f 1n).
\end{split}
\end{equation}
We now analyze these terms in \eqref{eq:data.approx.dgamma.weak.2}. First, 
$$\f{1}{(ad-b^2)^2} [ (d\f{\rd a}{\rd\ub})^2 -2 (b \f{\rd b}{\rd\ub})^2 + a (\f{\rd d}{\rd\ub})^2] = |\f{\rd\hat{\gamma}^{(dust)}}{\rd \ub}|_{\hat{\gamma}^{(dust)}}^2.$$
Next, for the highly oscillatory terms, we can write 
\begin{equation*}
\begin{split}
&\: (\Phi^{(dust)})^2 \times \{ \f{4 f }{(\Phi^{(dust)})^2} \cos(2kn\ub) + \f{2 }{ad-b^2} (d\f{\rd a}{\rd\ub} - a\f{\rd d}{\rd\ub}) \f{2 f^{\f 12}}{ \Phi^{(dust)}}\cos(k n\ub) \}\\
=&\: \f 1 n \f{\rd}{\rd \ub} [\f{2 f}{k } \sin(2kn\ub) + \f{4 }{k(ad-b^2)} (d\f{\rd a}{\rd\ub} - a\f{\rd d}{\rd\ub})  f^{\f 12} \Phi^{(dust)} \sin(k n\ub)] \\
&\: - \f 1 n [\f{\rd}{\rd \ub} \f{2 f}{k }] \sin(2kn\ub) - \f 1n [\f{\rd}{\rd\ub} (\f{4 }{k(ad-b^2)} (d\f{\rd a}{\rd\ub} - a\f{\rd d}{\rd\ub})  f^{\f 12}\Phi^{(dust)}) ]\sin(k n\ub) \\
=&\: \f 1 n \f{\rd}{\rd \ub} [\f{2 f}{k } \sin(2kn\ub) + \f{4 }{k(ad-b^2)} (d\f{\rd a}{\rd\ub} - a\f{\rd d}{\rd\ub})  f^{\f 12} \Phi^{(dust)} \sin(k n\ub)]  + O_K(\f 1n).
\end{split}
\end{equation*}
Thus, defining 
$$F_n:= \f{2 f}{k } \sin(2kn\ub) + \f{4 }{k(ad-b^2)} (d\f{\rd a}{\rd\ub} - a\f{\rd d}{\rd\ub})  f^{\f 12} \Phi^{(dust)} \sin(k n\ub),$$
which clearly satisfies \eqref{eq:data.approx.Fn.est}, it follows that \eqref{eq:data.approx.dgamma.weak} holds.

\pfstep{Step~2: Estimates for $\Phi_n$} Define now $\Phi_n$ by the initial value problem \eqref{eq:data.approx.vac.constraint}. Since $\Omg$ and $\hat{\gamma}_n$ are smooth, this linear ODE has a unique solution. Our goal now is to prove \eqref{eq:data.approx.Phi.final}, from which the positivity of $\Phi_n$ (for $n$ large) would also follow.

By \eqref{eq:data.approx.dust.constraint} and \eqref{eq:data.approx.vac.constraint}, $(\Phi_n-\Phi^{(dust)})$ satisfies the ODE
\begin{equation}\label{eqn:Phi.diff}
\begin{split}
&\:\f{\rd^2(\Phi_n-\Phi^{(dust)})}{\rd\ub^2}\\
=&\: 2\f{\rd \log\Omega}{\rd\ub}\f{\rd(\Phi_n-\Phi)}{\rd\ub}-\f 18(|\f{\rd\hat{\gamma}_n}{\rd \ub}|_{\hat{\gamma}_n}^2 - |\f{\rd\hat{\gamma}^{(dust)}}{\rd \ub}|_{\hat{\gamma}^{(dust)}}^2)\Phi^{(dust)} \\
&\:- \f 18 |\f{\rd\hat{\gamma}_n}{\rd \ub}|_{\hat{\gamma}_n}^2 (\Phi_n-\Phi^{(dust)}) +\f 12 \f{f}{\Phi^{(dust)}},
\end{split}
\end{equation}
where the initial data for both $\Phi_n-\Phi^{(data)}$ and $\f{\rd}{\rd\ub}(\Phi_n-\Phi^{(data)})$ vanish.

Using \eqref{eq:data.approx.dgamma.weak} and then integrating by parts and using \eqref{eq:data.approx.Fn.est}, 
\begin{equation}\label{eq:data.approx.Phi.est.osc}
\begin{split}
&\: \int_0^{\ub} [-\f 18(|\f{\rd\hat{\gamma}_n}{\rd \ub}|_{\hat{\gamma}_n}^2  - |\f{\rd\hat{\gamma}^{(dust)}}{\rd \ub}|_{\hat{\gamma}^{(dust)}}^2)\Phi^{(dust)} +\f 12 \f{f}{\Phi^{(dust)}}](\ub',\vartheta) \,\ud \ub' \\
=&\: - \f 1{8n} \int_0^{\ub} \f{1}{\Phi^{(dust)}}\f{\rd F_n}{\rd \ub} (\ub',\vartheta) \,\ud \ub' + O_K(\f 1n) = O_K(\f 1n).
\end{split}
\end{equation}

Integrating \eqref{eqn:Phi.diff} and using \eqref{eq:data.approx.Phi.est.osc}, we obtain
\begin{equation}\label{eqn:Phi.diff.int}
\begin{split}
&\: \f{\rd(\Phi_n-\Phi^{(dust)})}{\rd\ub} (\ub, \vartheta) \\
= &\: \int_0^{\ub} [2\f{\rd \log\Omega}{\rd\ub}\f{\rd(\Phi_n-\Phi)}{\rd\ub}(\ub',\vartheta) - \f 18 |\f{\rd\hat{\gamma}_n}{\rd \ub}|_{\hat{\gamma}_n}^2 (\Phi_n-\Phi^{(dust)})(\ub',\vartheta)] \,\ud \ub' + O_K(\f 1n).
\end{split}
\end{equation}
Taking the $W^{K,\infty}$ norm along the $2$-spheres, \eqref{eqn:Phi.diff.int} implies
\begin{equation}\label{eq:data.approx.Phi.almost}
\begin{split}
&\: \sup_{\ub' \in [0,\ub]} \|\f{\rd(\Phi_n-\Phi^{(dust)})}{\rd\ub} \|_{W^{K,\infty}(S_{0,\ub},\mathring{\gamma})} \\
\ls &\: \int_0^{\ub} [\sup_{\ub'' \in [0, \ub'] }\|\f{\rd(\Phi_n-\Phi^{(dust)})}{\rd\ub} \|_{W^{K,\infty}(S_{0,\ub''},\mathring{\gamma})} + \sup_{\ub'' \in [0, \ub'] } \|\Phi_n-\Phi^{(dust)} \|_{W^{K,\infty}(S_{0,\ub''},\mathring{\gamma})}] \,\ud \ub' + n^{-1} \\
\ls &\: \int_0^{\ub} \sup_{\ub'' \in [0, \ub'] }\|\f{\rd(\Phi_n-\Phi^{(dust)})}{\rd\ub} \|_{W^{K,\infty}(S_{0,\ub''},\mathring{\gamma})}  \,\ud \ub' + n^{-1},
\end{split}
\end{equation}
where in the last line we used the fundamental theorem of calculus and the initial vanishing of $\Phi_n - \Phi^{(data)}$ to control $\sup_{\ub'' \in [0, \ub'] } \|\Phi_n-\Phi^{(dust)} \|_{W^{K,\infty}(S_{0,\ub''},\mathring{\gamma})}$.

The estimate \eqref{eq:data.approx.Phi.final} therefore follows from first applying Gr\"onwall's inequality to \eqref{eq:data.approx.Phi.almost}, and then using the fundamental theorem of calculus to estimate $\|\Phi_n-\Phi^{(dust)} \|_{L^\i_{\ub}W^{K,\infty}(S_{0,\ub},\mathring{\gamma})}$. \qedhere
\end{proof}

\subsection{Approximating measure-valued null dust data by smooth null dust data}\label{sec:f.approx}

\begin{proposition}\label{prop:f.approx}
Let $K\in \mathbb N$. Assume that we are given $\mathring{\gamma}$, $\Omg$ and $\ud\nu_{\mathrm{init}}$ on $[0,\ub_*]\times \mathbb S^2$ as follows:
\begin{enumerate}
\item Let $\mathring{\gamma}$ be an arbitrary smooth metric on $S_{0,0}$. Extend $\mathring{\gamma}$ to $[0,\ub_*]\times \mathbb S^2$ by
$$\slashed{\mathcal L}_{\f{\rd}{\rd\ub}} \mathring{\gamma} = 0.$$
\item Let $\Omg$ be an arbitrary positive smooth function on $[0,\ub_*]\times \mathbb S^2$.
\item Let $\ud\nu_{\mathrm{init}}$ be a non-negative Radon measure on $(0,\ub_*) \times \mathbb S^2$ such that for  $\varphi^{(k)}$ being a rank-$k$ tensor $k$-th differentiable along the $\mathbb S^2$ directions,
\begin{equation}\label{eq:dnu.init.bound}
\sup \left\{ \sum_{0\leq k\leq K}\left| \int_{[0,\ub_*]\times \mathbb S^2} (\mathring{\div}{}^k\varphi^{(k)})(\ub,\vartheta)\, \ud\nu_{\mathrm{init}} \right| : \|\varphi^{(k)}\|_{L^\infty_{\ub}L^1(S_{0,\ub},\mathring{\gamma})} \leq 1  \right\} =:\Lambda < +\infty,
\end{equation}
\end{enumerate}

Then there exists a sequence of smooth functions $\{f_m\}_{m=1}^{+\infty}$, with $f_m:[0,\ub_*]\times \mathbb S^2\to [0,\infty)$ such that
\begin{enumerate}
\item $\Omg^{-2} f_m\,\mathrm{dA}_{\mathring{\gamma}}\,\ud\ub$ converges to $\ud\nu_{\mathrm{init}}$ in the weak-* topology as $m\to +\infty$.
\item For every $m\in \mathbb N$, $f_m$ satisfies the bound
\begin{equation}\label{eq:fm.uniform}
\|f_m\|_{L^1_{\ub}W^{K,\infty}(S_{0,\ub},\mathring{\gamma})}\ls 1.
\end{equation}
\item For $0\leq k\leq K$, $f_m$ satisfies the quantitative convergence estimate
\begin{equation}\label{eq:fm.quantitative}
\begin{split}
&\: \left| \int_0^{\ub_*}\int_{S_{0,\ub}} (\mathring{\div}{}^k\varphi^{(k)}) \Omg^{-2} f_m \,\mathrm{dA}_{\mathring{\gamma}}\,\ud\ub - \int_{(0,\ub_*) \times \mathbb S^2} \mathring{\div}{}^k\varphi^{(k)} \,\ud\nu_{\mathrm{init}} \right|\\
\ls &\: 2^{-m}\|\f{\rd\varphi^{(k)}}{\rd\ub}\|_{L^2_{\ub}L^1(S_{0,\ub},\mathring{\gamma})} + 2^{-2m} \|\varphi^{(k)}\|_{L^\i_{\ub} L^1(S_{0,\ub}, \mathring{\gamma})},
\end{split}
\end{equation}
for every continuous rank-$k$ $S$-tangent tensor field $\varphi^{(k)}$ such that $\f{\rd\varphi^{(k)}}{\rd\ub} \in L^2_{\ub}L^1(S_{0,\ub},\mathring{\gamma})$.
\end{enumerate}
Here, the implicit constants in \eqref{eq:fm.uniform} and \eqref{eq:fm.quantitative} depend only on $\ub_*$, $\mathring{\gamma}$, $\Omg$ and $\Lambda$.

Finally, if $\ud\nu_{\mathrm{init}}$ is supported on $[0,\ub_*]\times U^c$ for some $U\subset  \mathbb S^2$, then for any open $V\subseteq U$ with $\overline{V} \subset U$, $f_m$ can be chosen so that $\mathrm{supp}(f_m)\subset [0,\ub_*]\times V^c$.
\end{proposition}
\begin{proof}
In this proof, we will suppress the explicit dependence on $\mathring{\gamma}$ in the norms if there is no risk of confusion.

\pfstep{Step~1: Definition of $\{\widetilde{f}_\ep\}_{\ep\in (0,10^{-10}\ub_*]}$} We first define an auxiliary one parameter family of (not necessarily smooth) functions $\{\widetilde{f}_\ep\}_{\ep\in (0,10^{-10}\ub_*]}$. The functions $\widetilde{f}_\ep$ are obtained by mollifying $\ud\nu_{\mathrm{init}}$ in the $\ub$-direction and then taking the density with respect to $\Omg^{-2} \, \ud A_{\mathring{\gamma}}\, \ud \ub$.  

Let 
\begin{itemize}
\item $\varphi:[0,\ub_*]\times \mathbb S^2\to \mathbb R$ be a smooth tensor field; and
\item $\{\zeta_i\}_{i=1}^3$ be a smooth partition of unity of $[0,\ub_*]$ such that $\mathrm{supp}(\zeta_1) \subset [0,\f{\ub_*}{3}]$, $\mathrm{supp}(\zeta_2) \subset [\f{\ub_*}{4}, \f{3\ub_*}{4}]$, $\mathrm{supp}(\zeta_3)\subset [\f{2\ub_*}3, \ub_*]$ and $\sum_{i=1}^3 \zeta_i \equiv 1$; and
\item $\varrho:\mathbb R\to \mathbb R$ be a non-negative smooth even cutoff function with $\mathrm{supp}\varrho \subseteq [-1,1]$ and $\int_{\mathbb R} \varrho = 1$. 
\end{itemize}
Given $\ep\in (0,10^{-10}\ub_*]$, define\footnote{Note that $\varphi_\ep$ is defined on all of $[0,\ub_*]$ because we have shifted $\varphi$ on the support of $\zeta_1$ and $\zeta_3$ near the boundary.} $\varphi_\ep:[0,\ub_*]\to \mathbb R$ by
\begin{equation}\label{eq:varphi.ep.def.in.data}
\begin{split}
\varphi_\ep(\ub',\vartheta):= &\: \int_{\mathbb R} (\zeta_1\varphi)(\ub+\ep,\vartheta) \f 1{\ep} \varrho(\f{\ub-\ub'}{\ep}) \,\ud \ub + \int_{\mathbb R} (\zeta_2\varphi)(\ub,\vartheta) \f 1{\ep} \varrho(\f{\ub-\ub'}{\ep}) \,\ud \ub \\
&\: + \int_{\mathbb R} (\zeta_3\varphi)(\ub-\ep,\vartheta) \f 1{\ep} \varrho(\f{\ub-\ub'}{\ep}) \,\ud \ub \\
=&\: \sum_{i=1}^3 \int_{\mathbb R} (\zeta_i\varphi)(\ub + \alp_i\ep,\vartheta) \f 1{\ep} \varrho(\f{\ub-\ub'}{\ep}) \,\ud \ub,
\end{split}
\end{equation}
where $\alp_1 = 1$, $\alp_2 = 0$ and $\alp_3 = -1$. It is easy to check that $\varphi_\ep \to \varphi$ uniformly.

Note that (for an implicit constant independent of $\varphi$)
$
\sup_{\ub' \in [0, \ub_*]} \| \varphi_\ep(\ub', \cdot)\|_{L^1(\mathbb S^2)} \ls \ep^{-\f 12}.
$ 
Therefore, using \eqref{eq:dnu.init.bound}, the map
$$\varphi \mapsto \int_{[0,\ub_*]\times \mathbb S^2} \varphi_\ep(\ub',\vartheta) \, \ud \nu_{\mathrm{init}}(\ub',\vartheta) $$
extends to a bounded linear map $:L^2_{\ub}L^1(S_{0,\ub}) \to \mathbb R$. It follows by duality that there exists $\widetilde{f}_{\ep}\in L^2_{\ub}L^\i(S_{0,\ub})$ such that 
\begin{equation}\label{eq:f.ep.def}
\begin{split}
 \int_{[0,\ub_*]\times \mathbb S^2} \varphi_\ep(\ub',\vartheta) \, \ud \nu_{\mathrm{init}}(\ub',\vartheta) = \int_0^{\ub_*}\int_{S_{0,\ub}} \varphi(\ub,\vartheta)\Omg^{-2} \widetilde{f}_{\ep}(\ub,\vartheta)\,\ud A_{\mathring{\gamma}} \, \ud \ub.
\end{split}
\end{equation}
By \eqref{eq:f.ep.def}, it follows that $\Omg^{-2} \widetilde{f}_\ep\,\ud A_{\mathring{\gamma}}\,\ud\ub\rightharpoonup \ud \nu_{\mathrm{init}}$ in the weak-* topology as $\ep \to 0$.

\pfstep{Step~2: Uniform estimates for $\{\widetilde{f}_\ep\}_{\ep\in (0,10^{-10}\ub_*]}$} Consider the class of rank-$k$ tensor fields
$$\mathcal D^{(k)} = \{\varphi^{(k)}\in \Gamma(T^k ([0,\ub_*]\times \mathbb S^2)): \varphi \in L^\i_{\ub} C^\i(S_{0,\ub}) \}.$$


Since $\mathcal D^{(k)}$ is dense in $L^\i_{\ub}L^1(S_{0,\ub})$, we can compute using duality \eqref{eq:f.ep.def} and \eqref{eq:dnu.init.bound} that
\begin{equation}\label{eq:ftildeep.est}
\begin{split}
\|\widetilde{f}_{\ep}\|_{L^1_{\ub}W^{K,\i}(S_{0,\ub})} = &\: \sup_{\{\varphi^{(0)}\in \mathcal D^{(0)}: \|\varphi^{(0)}\|_{L^\i_{\ub}L^1(S_{0,\ub})} = 1\}} \int_0^{\ub_*}\int_{S_{0,\ub}} \varphi^{(0)}(\ub,\vartheta)\widetilde{f}_{\ep}(\ub,\vartheta)\,\ud A_{\mathring{\gamma}} \, \ud \ub \\
&\: + \sup_{\{\varphi^{(K)}\in \mathcal D^{(K)}: \|\varphi^{(K)}\|_{L^\i_{\ub}L^1(S_{0,\ub})} = 1\}} \int_0^{\ub_*}\int_{S_{0,\ub}} (\mathring{\div}{}^K\varphi^{(K)})(\ub,\vartheta)\widetilde{f}_{\ep}(\ub,\vartheta)\,\ud A_{\mathring{\gamma}} \, \ud \ub \\
\ls &\: 1.
\end{split}
\end{equation}

\pfstep{Step~3: Speed of convergence} Using additional regularity of the test function $\varphi$, we show a quantitative speed for the weak-* convergence of $\Omg^{-2} \widetilde{f}_\ep\,\ud A_{\mathring{\gamma}}\,\ud\ub \rightharpoonup \ud \nu_{\mathrm{init}}$. We first compute, for $i = 1,2,3$,
\begin{equation}\label{eq:convolutions.and.so.on}
\begin{split}
 &\:  \sup_{\ub' \in [0,\ub_*]} \| \int_{\mathbb R} (\zeta_i \varphi)(\ub+\alp_i \ep,\vartheta) \f 1{\ep} \varrho(\f{\ub-\ub'}{\ep}) \,\ud \ub - (\zeta_i\varphi)(\ub',\vartheta)\|_{L^1(\mathbb S^2)} \\
\leq &\: \sup_{\ub'\in [0,\ub_*]}  \int_{\mathbb R} \f 1{\ep} \rho(\f{\ub-\ub'}{\ep})\|(\zeta_i\varphi)(\ub+\alp_i \ep,\vartheta) - (\zeta_i\varphi)(\ub',\vartheta)\|_{L^1(\mathbb S^2)} \,\ud \ub \\
\leq &\: \sup_{\substack{ |\ub' - \ub''| \leq 2\ep \\ \ub',\,\ub'' \in [0, \ub_*]}}  \|(\zeta_i\varphi)(\ub'',\vartheta) - (\zeta_i\varphi)(\ub',\vartheta)\|_{L^1(\mathbb S^2)} \\
= &\:\sup_{\substack{ |\ub' - \ub''| \leq 2\ep \\ \ub',\,\ub'' \in [0, \ub_*]}} \|\int_{\ub''}^{\ub'} \f{\rd(\zeta_i\varphi)}{\rd\ub}(\ub''',\vartheta) \,\ud\ub'''\|_{L^1(\mathbb S^2)} 
\ls  \ep^{\f 12}\|\f{\rd\varphi}{\rd\ub}\|_{L^2_{\ub}L^1(\mathbb S^2)} + \ep \|\varphi\|_{L^\i_{\ub} L^1(\mathbb S^2)}.
\end{split}
\end{equation}

Let $0\leq k \leq K$. Using \eqref{eq:convolutions.and.so.on}, \eqref{eq:dnu.init.bound} and the definition of $\widetilde{f}_\ep$ in \eqref{eq:varphi.ep.def.in.data} and \eqref{eq:f.ep.def}, we obtain
\begin{equation}\label{eq:fm.quantitative.prelim}
\begin{split}
&\: \left| \int_0^{\ub_*}\int_{S_{0,\ub}} (\mathring{\div}{}^k\varphi^{(k)}) \Omg^{-2} \widetilde{f}_\ep \,\mathrm{dA}_{\mathring{\gamma}}\,\ud\ub - \int_{(0,\ub_*) \times \mathbb S^2} \mathring{\div}{}^k\varphi^{(k)}(\ub', \vartheta) \,\ud\nu_{\mathrm{init}}(\ub',\vartheta) \right|\\
= &\: \left| \sum_{i=1}^3 \int_{(0,\ub_*)\times \mathbb S^2} \{ \mathring{\div}{}^k [ \int_{\mathbb R}  (\zeta_i \varphi^{(k)})(\ub + \alp_i \ep,\vartheta) \f 1{\ep} \varrho(\f{\ub-\ub'}{\ep}) \,\ud \ub - (\zeta_i \varphi^{(k)})(\ub',\vartheta)] \} \,\ud\nu_{\mathrm{init}}(\ub',\vartheta) \right|\\
\ls &\:  \max_{i=1,2,3} \sup_{\ub' \in [0,\ub_*]} \| \int_{\mathbb R} (\zeta_i \varphi)(\ub+\alp_i \ep,\vartheta) \f 1{\ep} \varrho(\f{\ub-\ub'}{\ep}) \,\ud \ub - (\zeta_i\varphi)(\ub',\vartheta)\|_{L^1(\mathbb S^2)} \\
\ls &\: \ep^{\f 12} \|\f{\rd\varphi^{(k)}}{\rd\ub}\|_{L^2_{\ub}L^1(\mathbb S^2)} + \ep  \|\varphi\|_{L^\i_{\ub} L^1(\mathbb S^2)}.
\end{split}
\end{equation}


\pfstep{Step~4: Definition of $\{f_m\}_{m=1}^{+\infty}$ and conclusion of proof} Finally, let $\{f_m\}_{m=1}^{+\infty}$be smooth and such that 
\begin{equation}\label{eq:f.&.ft.very.close}
\|f_m - \widetilde{f}_{2^{-2m}}\|_{L^1_{\ub}W^{K,\i}(S_{0,\ub})}\leq 2^{-2m},
\end{equation}
(where by $\widetilde{f}_{2^{-2m}}$ we mean $\widetilde{f}_{\ep}$ with $\ep = 2^{-2m}$). Since $\widetilde{f}_\ep$ is non-negative by \eqref{eq:f.ep.def}, it is easy to see that $f_m$ can be arranged to be non-negative. In the case $\ud\nu_{\mathrm{init}}$ is supported on $[0,\ub_*]\times U^c$ for some $U\subset  \mathbb S^2$, then for any open $V\subseteq U$, we can moreover impose $\mathrm{supp}(f_m)\subset [0,\ub_*]\times V^c$.

We need to check that the three properties asserted in the statement of the proposition hold.
\begin{enumerate}
\item Since $\Omg^{-2}\widetilde{f}_\ep\,\ud A_{\mathring{\gamma}}\,\ud\ub\rightharpoonup \ud \nu_{\mathrm{init}}$ in the weak-* topology, it follows from \eqref{eq:f.&.ft.very.close} that $\Omg^{-2} f_m\,\ud A_{\mathring{\gamma}}\,\ud\ub\rightharpoonup \ud \nu_{\mathrm{init}} $ in the weak-* topology.
\item The estimate \eqref{eq:fm.uniform} follows from \eqref{eq:ftildeep.est}, \eqref{eq:f.&.ft.very.close} and the triangle inequality.
\item Integrating by parts and using \eqref{eq:f.&.ft.very.close}, smoothness of $\Omg$ and H\"older's inequality, we obtain
$$\left| \int_0^{\ub_*}\int_{S_{0,\ub}} (\mathring{\div}{}^k\varphi^{(k)}) \Omg^{-2} (f_m - \widetilde{f}_{2^{-2m}}) \,\mathrm{dA}_{\mathring{\gamma}}\,\ud\ub \right| \ls 2^{-2m} \|\varphi^{(k)}\|_{L^\i_{\ub} L^1(S_{0,\ub})}.$$
Combining this with \eqref{eq:fm.quantitative.prelim}, we then obtain \eqref{eq:fm.quantitative} using the triangle inequality. \qedhere
\end{enumerate}

\end{proof}

\begin{proposition}\label{prop:f.approx.Phi.conv}
Let $\ud\nu_{\mathrm{init}}$, $\{f_m\}_{m=1}^{+\infty}$ and $\mathring{\gamma}$ be as in Proposition~\ref{prop:f.approx}. Suppose there exist continuous functions $(\Phi,\log\Om)$ and a continuous $S$-tangent $2$-tensor $\hat{\gamma}$ such that the following holds for some $K\in \mathbb N$:
\begin{itemize}
\item $\Phi$ is positive, Lipschitz and $L^\i_{\ub}W^{K,\infty}$, $\f{\rd\Phi}{\rd\ub}$ is BV and $L^\i_{\ub} W^{K,\infty}$, and the following estimates hold:
\begin{equation}\label{eq:f.approx.Phi.bkg.est}
\|\Phi\|_{L^\i_{\ub}W^{K,\infty}(S_{0,\ub},\mathring{\gamma})} + \|\Phi^{-1}\|_{L^\i_{\ub}W^{K,\infty}(S_{0,\ub},\mathring{\gamma})} + \|\f{\rd\Phi}{\rd\ub}\|_{L^\i_{\ub}W^{K,\infty}(S_{0,\ub},\mathring{\gamma})} \ls 1.
\end{equation}
\item $\Omega$ is smooth, positive and satisfy \begin{equation}\label{eq:f.approx.Omg.bkg.est}
\|\log\Omg\|_{L^\i_{\ub}W^{K,\i}(S_{0,\ub},\mathring{\gamma})} + \|\f{\rd}{\rd\ub} \log\Omg\|_{L^2_{\ub}W^{K,\i}(S_{0,\ub},\mathring{\gamma}) } \ls 1.
\end{equation}
\item $\hat{\gamma}$ satisfies $\f{\det(\hat{\gamma})}{\det(\mathring{\gamma})} = 1$ and obeys the following bounds:
\begin{equation}\label{eq:f.approx.gamma.bkg.est}
\|\hat{\gamma}\|_{L^\i_{\ub} W^{K,\i}(S_{0,\ub})} + \|\f{\rd}{\rd\ub} \hat{\gamma}\|_{L^2_{\ub} W^{K,\i}(S_{0,\ub},\mathring{\gamma})} \ls 1.
\end{equation}
\end{itemize}
Assume also that $(\ud\nu_{\mathrm{init}},\Phi,\log\Om, \hat{\gamma})$ satisfies the equation \eqref{eq:dnu.init} for any $\varphi \in C^\infty_c((0,\ub_*)\times \mathbb S^2)$.

Consider the initial value problem for $\Phi_m$ with smooth data:
\begin{equation}\label{Phim.ODE}
\begin{cases}
\f{\rd^2\Phi_m}{\rd \ub^2}-2\f{\rd \log\Omega}{\rd\ub}\f{\rd\Phi_m}{\rd\ub}+\f 18 |\f{\rd\hat{\gamma}_m}{\rd\ub}|^2_{\hat{\gamma}_m}\Phi_m + \f 12 f_m \Phi_m^{-1} =0, \\
\Phi_m(0,\vartheta) = \overline{\Phi}_m(\vartheta),\quad \f{\rd\Phi_m}{\rd \ub} (0,\vartheta) = \overline{\Psi}_m(\vartheta),
\end{cases}
\end{equation}
where $(\hat{\gamma}_m,\, \overline{\Phi}_m, \, \overline{\Psi}_m)$ are such that the following holds for all $m\in \mathbb N$:
\begin{itemize}
\item The tensor $\hat{\gamma}_m$ is smooth and such that
\begin{equation}\label{eq:f.approx.gamma.est}
\f{\det\hat{\gamma}_m}{\det \mathring{\gamma}} = 1,\quad \|\hat{\gamma}_m- \hat{\gamma}\|_{L^\i_{\ub}W^{K,\i}(S_{0,\ub},\mathring{\gamma})} + \|\f{\rd}{\rd\ub}(\hat{\gamma}_m- \hat{\gamma})\|_{L^2_{\ub}W^{K,\i}(S_{0,\ub},\mathring{\gamma})}\leq 2^{-m}.
\end{equation}
\item The data $(\overline{\Phi}_m,\overline{\Psi}_m)$ of \eqref{Phim.ODE} are smooth and such that\footnote{Here ${}^-$ denote the trace of BV functions, see Lemma~\ref{lem:trace}.}
\begin{equation}\label{eq:f.approx.data.est}
\| \overline{\Phi}_m - \Phi(0,\cdot) \|_{W^{K,\i}(S_{0,0},\mathring{\gamma})} + \|\overline{\Psi}_m  - (\f{\rd \Phi}{\rd\ub})^-(0,\cdot)\|_{W^{K,\i}(S_{0,0},\mathring{\gamma})} \leq 2^{-m}.
\end{equation}
\end{itemize}

Then the following holds for $m\in \mathbb N$ sufficiently large:
\begin{enumerate}
\item The solution $\Phi_m$ to \eqref{Phim.ODE} is defined on all of $[0,\ub_*]\times \mathbb S^2$.
\item $\Phi_m$ converges to $\Phi$ in the following sense\footnote{We emphasize that while we have uniform $L^\i_{\ub}$ bounds for $\f{\rd\Phi_m}{\rd\ub}$ in \eqref{eq:duphi.uniform}, the convergence estimate in \eqref{eq:phi.diff.est} for $\f{\rd\Phi_m}{\rd\ub}$ is only in $L^2_{\ub}$. In fact, it can be easily checked that the convergence in general does \underline{not} hold in $L^\i_{\ub}$.}:
\begin{equation}\label{eq:phi.diff.est}
\|\Phi_m - \Phi \|_{L^\i_{\ub} W^{K,\i}(S_{0,\ub},\mathring{\gamma})} + \|\f{\rd}{\rd\ub}(\Phi_m - \Phi)\|_{L^2_{\ub} W^{K,\i}(S_{0,\ub},\mathring{\gamma})} \ls 2^{-m}.
\end{equation}
\item $\f{\rd\Phi_m}{\rd\ub}$ satisfies uniformly the estimate
\begin{equation}\label{eq:duphi.uniform}
\sup_m \|\f{\rd\Phi_m}{\rd\ub}\|_{L^{\i}_{\ub} W^{K,\i}(S_{0,\ub},\mathring{\gamma})} \ls 1.
\end{equation}
\end{enumerate}
Here, the implicit constants depend only on $\ub_*$, $\mathring{\gamma}$, $\Omg$ and $\Lambda$ in Proposition~\ref{prop:f.approx}.
\end{proposition}

\begin{proof}
We will only prove the estimates which will then in particular imply existence of solution on all of $[0,\ub_*]\times \mathbb S^2$.

Since all norms on the $2$-spheres in this proof will be taken with respect to $\mathring{\gamma}$, when there is no risk of confusion we will suppress explicit references to $\mathring{\gamma}$.

For Steps~1 and 2 of the argument, let us fix $\Ub\in [0,\ub_*]$. We will first be deriving estimates for $\Phi_m - \Phi$ and $\f{\rd}{\rd\ub} (\Phi_m - \Phi)$ on the interval $[0,\Ub]$. Thus, \textbf{in Steps~1 and 2, $L^2_{\ub}$ and $L^\i_{\ub}$ mean $L^2_{\ub}([0,\Ub])$ and $L^\i_{\ub}([0,\Ub])$. Moreover, all the implicit constants in the estimates will be independent of $\Ub$.}

\pfstep{Step~1: Equation for $\Phi_m - \Phi$} We first derive an equation for $\Phi_m - \Phi$ using \eqref{eq:dnu.init} and \eqref{Phim.ODE}. Let $k\in \{0,\dots,K\}$. We will consider smooth rank-$k$ covariant tensor $\widetilde{\varphi} = \widetilde{\varphi}_{A_1\dots A_k}$ such that 
\begin{equation}\label{eq:varphi.cond.constraints}
\|e^{A\ub} \f{\rd\widetilde{\varphi}}{\rd\ub}\|_{L^2_{\ub}L^1(S_{0,\ub})} = 1,\quad \widetilde{\varphi}\restriction_{\ub = \Ub} = 0,
\end{equation}
for some large $A>1$ to be determined\footnote{The largeness of $A$ will be used to facilitate the proof of the estimates in Step~2. One could think of the weight $e^{A\ub}$ as a device to prove a Gr\"onwall-like estimate in the setting involving the measure $\ud\nu_{\mathrm{init}}$.}.

For every $\ell \in \mathbb N$ with $\ell^{-1} \leq \f{\Ub}{2}$, define $\widetilde{\xi}_\ell:[0,\ub_*]\to \mathbb R_{\geq 0}$ by
\begin{equation}\label{eq:def.xi.t}
\widetilde{\xi}_\ell(\ub):= \begin{cases}
\ell \ub & \mbox{if $\ub\in [0, \ell^{-1})$}\\
1 & \mbox{if $\ub\in [\ell^{-1}, \Ub]$}\\
0 & \mbox{if $\ub \in (\Ub,\ub_*]$}
\end{cases}.
\end{equation}

Note that given $\widetilde{\varphi}$ as above and $\widetilde{\xi}_\ell$ as in \eqref{eq:def.xi.t}, $\widetilde{\xi}_\ell\mathring{\div}{}^k\widetilde{\varphi} \in C^0_c((0,\ub_*)\times \mathbb S^2)$ for all $\ell \in \mathbb N$. After an easy limiting argument, we can thus apply \eqref{eq:dnu.init} with $\varphi = \widetilde{\xi}_\ell \mathring{\div}{}^k\widetilde{\varphi}$. Together with \eqref{Phim.ODE}, we thus obtain that, for every $\ell \in \mathbb N$ with $\ell^{-1} \leq \f{\ub_*}{2}$,
\begin{equation}\label{eq:diff.Phi.der.prelim}
\begin{split}
&\: \underbrace{-\int_0^{\ub_*} \int_{\mathbb S^2} \xi_\ell(\ub')((\mathring{\div}{}^k\f{\rd\widetilde{\varphi}}{\rd\ub}) \Omg^{-2} \f{\rd(\Phi-\Phi_m)}{\rd \ub})(\ub',\vartheta) \,\mathrm{dA}_{\mathring{\gamma}}\,\ud \ub'}_{\mathrm{O}'} \\
&\: - \underbrace{\ell \int_0^{\ell^{-1}}\int_{\mathbb S^2} ((\mathring{\div}{}^k\widetilde{\varphi}) \Omg^{-2} \f{\rd(\Phi-\Phi_m)}{\rd \ub})(\ub',\vartheta) \,\mathrm{dA}_{\mathring{\gamma}}\, \ud \ub'}_{=:\mathrm{I}'} \\
= &\: 
 -\underbrace{\f 18 \int_0^{\ub_*} \int_{\mathbb S^2} \xi_\ell(\ub') (\mathring{\div}{}^k\widetilde{\varphi}) \Omg^{-2}( |\f{\rd\hat{\gamma}}{\rd\ub}|^2_{\hat{\gamma}} \Phi - |\f{\rd\hat{\gamma}_m}{\rd\ub}|^2_{\hat{\gamma}_m} \Phi_m)(\ub',\vartheta) \,\mathrm{dA}_{\mathring{\gamma}}\,\ud \ub'}_{=:\mathrm{II}'} \\
&\: - \underbrace{ \f 12 \int_{(0,\ub_*)\times \mathbb S^2} \xi_\ell(\ub')((\mathring{\div}{}^k\widetilde{\varphi})\Phi^{-1}) (\ub',\vartheta)\,\ud\nu_{\mathrm{init}}(\ub',\vartheta)}_{=:\mathrm{III}_A'} \\
&\: + \underbrace{\f 12  \int_0^{\ub_*} \int_{\mathbb S^2} \xi_\ell(\ub')((\mathring{\div}{}^k\widetilde{\varphi}) \Omg^{-2} \Phi_m^{-1} f_m)(\ub',\vartheta) \,\mathrm{dA}_{\mathring{\gamma}}\,\ud\ub'}_{=:\mathrm{III}_B'}.
\end{split}
\end{equation}

Taking $\ell \to +\infty$, applying the dominated convergence theorem for $\mathrm{O}'$, $\mathrm{II}'$, $\mathrm{III}_A'$, $\mathrm{III}_B'$ and using the fact that $\f{\rd\Phi_m}{\rd \ub}$ is smooth and $\lim_{\ub\to 0^+}\f{\rd\Phi}{\rd \ub}$ is well-defined in the trace sense for $\mathrm{I}'$, we obtain
\begin{equation}\label{eq:diff.Phi.der}
\begin{split}
&\: -\int_0^{\Ub} \int_{\mathbb S^2} ((\mathring{\div}{}^k\f{\rd\widetilde{\varphi}}{\rd\ub}) \Omg^{-2} \f{\rd(\Phi-\Phi_m)}{\rd \ub})(\ub',\vartheta) \,\mathrm{dA}_{\mathring{\gamma}}\,\ud \ub'  - \underbrace{\int_{\mathbb S^2} (\mathring{\div}{}^k\widetilde{\varphi}) \Omg^{-2} [(\f{\rd\Phi}{\rd \ub})^- - \f{\rd\Phi_m}{\rd \ub}](0,\vartheta) \,\mathrm{dA}_{\mathring{\gamma}}}_{=:\mathrm{I}} \\
= &\:  -\underbrace{\f 18 \int_0^{\Ub} \int_{\mathbb S^2} (\mathring{\div}{}^k\widetilde{\varphi}) \Omg^{-2}( |\f{\rd\hat{\gamma}}{\rd\ub}|^2_{\hat{\gamma}} \Phi - |\f{\rd\hat{\gamma}_m}{\rd\ub}|^2_{\hat{\gamma}_m} \Phi_m)(\ub',\vartheta) \,\mathrm{dA}_{\mathring{\gamma}}\,\ud \ub'}_{=:\mathrm{II}} \\
&\: - \underbrace{\f 12 \int_{(0,\Ub)\times \mathbb S^2} ((\mathring{\div}{}^k\widetilde{\varphi}) \Phi^{-1}) (\ub',\vartheta)\,\ud\nu_{\mathrm{init}}(\ub',\vartheta) + \f 12  \int_0^{\Ub} \int_{\mathbb S^2} ((\mathring{\div}{}^k\widetilde{\varphi}) \Omg^{-2} \Phi_m^{-1} f_m)(\ub',\vartheta) \,\mathrm{dA}_{\mathring{\gamma}}\,\ud\ub'}_{=:\mathrm{III}}.
\end{split}
\end{equation}

\pfstep{Step~2: Estimating the terms in \eqref{eq:diff.Phi.der}} We now estimate the terms from \eqref{eq:diff.Phi.der}. Before we proceed, note that it follows easily from \eqref{eq:varphi.cond.constraints} that for every $\ub \in [0, \Ub]$,
\begin{equation}\label{eq:varphi.also.pointwise}
\|\widetilde{\varphi}\|_{L^1(S_{0,\ub})}(\ub) \leq \int_{\ub}^{\Ub} \|\f{\rd\widetilde{\varphi}}{\rd\ub}\|_{L^1(S_{0,\ub})}(\ub') \,\ud \ub'\leq \|e^{A\ub} \f{\rd\widetilde{\varphi}}{\rd\ub}\|_{L^2_{\ub}L^1(S_{0,\ub})} (\int_{\ub}^{\Ub*} e^{-2A\ub'} \,\ud\ub')^{\f 12}\leq \f{1}{\sqrt{2A}}e^{-A\ub}.
\end{equation}

Now for all of the terms in \eqref{eq:diff.Phi.der}, we first integrate by parts in the angular variables and then control the resulting terms with \eqref{eq:varphi.also.pointwise} and the bounds in the assumption of the proposition.

By H\"older's inequality, \eqref{eq:varphi.also.pointwise}, \eqref{eq:f.approx.Omg.bkg.est}, \eqref{Phim.ODE} and \eqref{eq:f.approx.data.est},
\begin{equation}\label{eq:main.Phi.m.est.1}
\begin{split}
|\mathrm{I}| \ls &\: \sum_{k_1+k_2 = k} \|\widetilde{\varphi} \|_{L^1(S_{0,0})}\|\mathring{\nab}{}^{k_1}\Omg^{-2}\|_{L^\i(S_{0,0})}\|\mathring{\nab}{}^{k_2}[(\f{\rd\Phi}{\rd \ub})^-(0,\vartheta) - \f{\rd\Phi_m}{\rd \ub}(\vartheta)]\|_{L^\i(S_{0,0})} \ls \f{2^{-m}}{\sqrt{A}}.
\end{split}
\end{equation}


By H\"older's inequality, \eqref{eq:varphi.also.pointwise}, \eqref{eq:f.approx.Phi.bkg.est},  \eqref{eq:f.approx.Omg.bkg.est}, \eqref{eq:f.approx.gamma.bkg.est} and \eqref{eq:f.approx.gamma.est}, 
\begin{equation}\label{eq:main.Phi.m.est.2}
\begin{split}
&\: |\mathrm{II}| \\
\ls &\: \|e^{A\ub}\widetilde{\varphi}\|_{L^\i_{\ub} L^1(S_{0,\ub})} \times \sum_{k_1+k_2 =k} (\|\mathring{\nab}{}^{k_1}(\Omg^{-2} (|\f{\rd\hat{\gamma}}{\rd\ub}|^2_{\hat{\gamma}}  - |\f{\rd\hat{\gamma}_m}{\rd\ub} |^2_{\hat{\gamma}_m} ) )\|_{L^1_{\ub}L^\i(S_{0,\ub})} \|e^{-A\ub}\mathring{\nab}{}^{k_2}\Phi\|_{L^\i_{\ub}L^\i(S_{0,\ub})} \\
&\: \qquad + \|\mathring{\nab}{}^{k_1} (\Omg^{-2} |\f{\rd\hat{\gamma}_m}{\rd\ub}|^2_{\hat{\gamma}_m})\|_{L^1_{\ub}L^\i(S_{0,\ub})} \|e^{-A\ub} \mathring{\nab}{}^{k_2}(\Phi-\Phi_m)\|_{L^\i_{\ub}L^\i(S_{0,\ub})}) \\
\ls &\: \f{1}{\sqrt{A}} (2^{-m} + \sum_{0\leq k_1\leq k} \|e^{-A\ub} \mathring{\nab}{}^{k_1}(\Phi-\Phi_m)\|_{L^\i_{\ub}L^\i(S_{0,\ub})}).
\end{split}
\end{equation}

We now turn to term $\mathrm{III}$, which we will split into two terms (see \eqref{eq:f.approx.III.1} and \eqref{eq:f.approx.III.2} below).

First we observe that by \eqref{eq:f.approx.Phi.bkg.est},
\begin{equation}\label{eq:diff.Phi-1.Phim-1}
\begin{split}
\| \Phi_m^{-1} - \Phi^{-1} \|_{L^\i_{\ub} L^\i(S_{0,\ub})} =&\:  \| \Phi^{-1} \Phi_m^{-1} (\Phi  -\Phi_m) \|_{L^\i_{\ub} L^\i(S_{0,\ub})} \\
\ls &\: \|\Phi_m^{-1}\|_{L^\i_{\ub} L^\i(S_{0,\ub})} \|\Phi - \Phi_m \|_{L^\i_{\ub} L^\i(S_{0,\ub})}. 
\end{split}
\end{equation}
Therefore, integrating by parts the $\mathring{\div}{ }^k$ and using the estimates \eqref{eq:fm.uniform}, \eqref{eq:f.approx.Omg.bkg.est}, \eqref{eq:diff.Phi-1.Phim-1} and \eqref{eq:varphi.also.pointwise}, we obtain
\begin{equation}\label{eq:f.approx.III.1}
\begin{split}
&\: \f 12 \left| \int_0^{\ub_*} \int_{\mathbb S^2} ((\mathring{\div}{}^k\widetilde{\varphi}) \Omg^{-2} (\Phi_m^{-1} - \Phi^{-1}) f_m)(\ub',\vartheta) \,\mathrm{dA}_{\mathring{\gamma}}\,\ud\ub' \right| \\
\ls &\: \sum_{k_1+k_2+k_3 = k} \|\mathring{\nab}^{k_1} \Phi_m^{-1}\|_{L^\i_{\ub} L^\i(S_{0,\ub})} \|e^{-A\ub}\mathring{\nab}^{k_2}(\Phi - \Phi_m) \|_{L^\i_{\ub} L^\i(S_{0,\ub})} \| \mathring{\nab}^{k_3} f_m \|_{L^\i_{\ub} L^\i(S_{0,\ub})} \|e^{A\ub} \widetilde{\varphi} \|_{L^\i_{\ub} L^1(S_{0,\ub})}  \\
\ls &\: \f{1}{\sqrt{A}}\|\Phi_m^{-1}\|_{L^\i_{\ub} W^{k,\infty}(S_{0,\ub})} (\sum_{0\leq k'\leq k} \|e^{-A\ub} \mathring{\nab}^{k'} (\Phi - \Phi_m) \|_{L^\i_{\ub} L^{\infty}(S_{0,\ub})}).
\end{split}
\end{equation}

On the other hand, notice that we can write $(\mathring{\div}{}^k\widetilde{\varphi})\Phi^{-1}$ as a linear combination of terms of the form
$$\mathring{\div}{}^{k_1} (\widetilde{\varphi} \mathring{\nab}^{k_2}\Phi^{-1})$$
for $k_1 + k_2 = k$. Therefore, using \eqref{eq:fm.quantitative} in Proposition~\ref{prop:f.approx}, and then using \eqref{eq:varphi.cond.constraints}, \eqref{eq:varphi.also.pointwise} (and using $A\ub \geq 0$) and \eqref{eq:f.approx.Phi.bkg.est}, we obtain
\begin{equation}\label{eq:f.approx.III.2}
\begin{split}
&\: -\f 12 \int_{(0,\ub_*)\times \mathbb S^2} ((\mathring{\div}{}^k\widetilde{\varphi})\Phi^{-1}) (\ub',\vartheta)\,\ud\nu_{\mathrm{init}}(\ub',\vartheta) + \f 12  \int_0^{\ub_*} \int_{\mathbb S^2} ((\mathring{\div}{}^k\widetilde{\varphi}) \Omg^{-2} \Phi^{-1} f_m)(\ub',\vartheta) \,\mathrm{dA}_{\mathring{\gamma}}\,\ud\ub' \\
\ls &\: 2^{-m} \sum_{0\leq k'\leq k} \|\f{\rd(\widetilde{\varphi}\mathring{\nab}^{k'}\Phi^{-1})}{\rd\ub} \|_{L^2_{\ub} L^1(S_{0,\ub})} + 2^{-2m} \sum_{0\leq k'\leq k} \|\varphi\mathring{\nab}^{k'}\Phi^{-1} \|_{L^2_{\ub} L^1(S_{0,\ub})} \\
\ls &\: 2^{-m} \|\f{\rd\widetilde{\varphi}}{\rd\ub}\|_{L^2_{\ub} L^1(S_{0,\ub})}  \|\Phi^{-1} \|_{L^\i_{\ub} W^{k,\i}(S_{0,\ub})}  + 2^{-m} \|\widetilde{\varphi}\|_{L^\i_{\ub} L^1(S_{0,\ub})}  \|\Phi^{-2} \f{\rd\Phi}{\rd\ub} \|_{L^\i_{\ub} W^{k,\i}(S_{0,\ub})} \\
&\: + 2^{-2m} \| \widetilde{\varphi} \|_{L^2_{\ub} L^1(S_{0,\ub})} \|\Phi^{-1} \|_{L^\i_{\ub} W^{k,\i}(S_{0,\ub})} \\
\ls &\: 2^{-m}.
\end{split}
\end{equation}

Combining \eqref{eq:f.approx.III.1} and \eqref{eq:f.approx.III.2} and using the triangle inequality, we thus obtain
\begin{equation}\label{eq:main.Phi.m.est.3}
\begin{split}
|\mathrm{III}|\ls \f{1}{\sqrt{A}}\|\Phi_m^{-1}\|_{L^\i_{\ub} W^{k,\infty}(S_{0,\ub})} (\sum_{0\leq k'\leq k} \|e^{-A\ub} \mathring{\nab}^{k'} (\Phi - \Phi_m) \|_{L^\i_{\ub} L^{\infty}(S_{0,\ub})}) + 2^{-m}.
\end{split}
\end{equation}

Starting with duality and using \eqref{eq:diff.Phi.der}, \eqref{eq:main.Phi.m.est.1}, \eqref{eq:main.Phi.m.est.2} and \eqref{eq:main.Phi.m.est.3}, we obtain
\begin{equation}\label{eq:main.Phi.m.est.5}
\begin{split}
&\: \sum_{0\leq k \leq K} \|e^{-A\ub} \mathring{\nab}{}^k(\Omg^{-2}\f{\rd(\Phi - \Phi_m)}{\rd\ub}) \|_{L^2_{\ub}L^\i(S_{0,\ub})} \\
=&\: \sum_{0\leq k \leq K} \sup_{\widetilde{\varphi} \,\mathrm{ satisfying \eqref{eq:varphi.cond.constraints}}} \int_0^{\ub_*} \int_{S_{0,\ub}} (\f{\rd\widetilde{\varphi}}{\rd\ub} \mathring{\nab}{}^k(\Omg^{-2} \f{\rd(\Phi-\Phi_m)}{\rd \ub}))(\ub',\vartheta)\,\mathrm{dA}_{\mathring{\gamma}}\,\ud \ub' \\
\ls &\: 2^{-m} + (1+ \|\Phi_m^{-1}\|_{L^\i_{\ub} W^{K,\i}(S_{0,\ub})})\sum_{0\leq k\leq K} \f{1}{\sqrt{A}} \|e^{-A\ub}\mathring{\nab}{}^{k}(\Phi_m - \Phi)\|_{L^{\i}_{\ub}L^{\i}(S_{0,\ub})}.
\end{split}
\end{equation}

We then complement the estimate \eqref{eq:main.Phi.m.est.5} of $\f{\rd(\Phi - \Phi_m)}{\rd\ub}$ with the following estimate on $\Phi - \Phi_m$ for $\ub \in [0,\Ub]$, which is derived using the fundamental theorem of calculus, \eqref{Phim.ODE}, \eqref{eq:f.approx.data.est} and H\"older's inequality:
\begin{equation}\label{eq:main.Phi.m.est.6}
\begin{split}
\|\mathring{\nab}{}^{k}(\Phi - \Phi_m)\|_{L^\i(S_{0,\ub})}(\ub) \leq &\: \|\mathring{\nab}{}^{k}(\Phi\restriction_{\ub = 0} - \overline{\Phi}_m)\|_{L^\i(S_{0,\ub})} + \int_0^{\ub} \|\mathring{\nab}{}^{k}\f{\rd(\Phi - \Phi_m)}{\rd \ub}\|_{L^\i(S_{0,\ub'})} \,\ud \ub' \\
\ls &\: 2^{-m} + (\int_0^{\ub} e^{2A\ub'}\,\ud\ub')^{\f 12}\|e^{-A\ub} \mathring{\nab}{}^{k}\f{\rd(\Phi-\Phi_m)}{\rd \ub}\|_{L^2_{\ub}L^\i(S_{0,\ub})} \\
\ls &\: 2^{-m} + \sum_{0\leq k_1\leq k}\f{e^{A\ub}}{\sqrt{A}} \|e^{-A\ub} \mathring{\nab}{}^{k_1}(\Omg^{-2}\f{\rd(\Phi-\Phi_m)}{\rd \ub})\|_{L^2_{\ub}L^\i(S_{0,\ub})}.
\end{split}
\end{equation}
The estimate \eqref{eq:main.Phi.m.est.6} implies (using $A\ub\geq 0$) that for every $\ub \in [0,\Ub]$,
\begin{equation}\label{eq:main.Phi.m.est.7}
\sum_{0\leq k\leq K} \|e^{-A\ub}\mathring{\nab}{}^{k}(\Phi - \Phi_m)\|_{L^\i_{\ub}L^\i(S_{0,\ub})}\ls  2^{-m} + \sum_{0\leq k\leq K} \f{1}{\sqrt{A}} \|e^{-A\ub} \mathring{\nab}{}^{k}(\Omg^{-2}\f{\rd(\Phi-\Phi_m)}{\rd \ub})\|_{L^2_{\ub}L^\i(S_{0,\ub})}.
\end{equation}

Adding \eqref{eq:main.Phi.m.est.5} and \eqref{eq:main.Phi.m.est.7}, we obtain 
\begin{equation}\label{eq:main.Phi.m.est.8}
\begin{split}
&\: \sum_{0\leq k\leq K}(\|e^{-A\ub} \mathring{\nab}{}^{k}(\Omg^{-2}\f{\rd(\Phi - \Phi_m)}{\rd\ub}) \|_{L^2_{\ub}L^\i(S_{0,\ub})} + \|e^{-A\ub}\mathring{\nab}{}^{k}(\Phi - \Phi_m)\|_{L^\i_{\ub}L^\i(S_{0,\ub})}) \\
\ls &\: 2^{-m} + \sum_{0\leq k\leq K} \f{1}{\sqrt{A}} (\|e^{-A\ub} \mathring{\nab}{}^{k}(\Omg^{-2}\f{\rd(\Phi-\Phi_m)}{\rd \ub})\|_{L^2_{\ub}L^\i(S_{0,\ub})})\\
&\: + (1+ \|\Phi_m^{-1}\|_{L^\i_{\ub} W^{K,\i}(S_{0,\ub})}) \sum_{0\leq k\leq K} \f{1}{\sqrt{A}}\|e^{-A\ub}\mathring{\nab}{}^{k}(\Phi_m - \Phi)\|_{L^{\i}_{\ub}L^{\i}(S_{0,\ub})})
\end{split}
\end{equation}
\textbf{uniformly for all subintervals $[0,\Ub]\subseteq [0,\ub_*]$ and for all $m\in \mathbb N$}.

Using \eqref{eq:main.Phi.m.est.8}, we first claim that for $m\in \mathbb N$ sufficiently large, $\|\Phi_m^{-1}\|_{L^\i_{\ub} W^{K,\i}(S_{0,\ub})} \leq 2 \|\Phi^{-1}\|_{L^\i_{\ub} W^{K,\i}(S_{0,\ub})}$. If not, then by continuity there exists $\widetilde{\Ub}$ such that 
\begin{equation}\label{eq:main.Phi.m.contradiction}
\|\Phi_m^{-1}\|_{W^{K,\i}(S_{0,\widetilde{\Ub}})} = 2\|\Phi^{-1}\|_{L^\i_{\ub} W^{K,\i}(S_{0,\ub})},\quad \sup_{\ub\in [0,\widetilde{\Ub}]}\|\Phi_m^{-1}\|_{W^{K,\i}(S_{0,\ub})} \leq 2\|\Phi^{-1}\|_{L^\i_{\ub} W^{K,\i}(S_{0,\ub})}.
\end{equation}
However, for $A$ sufficiently large (independently of $\widetilde{U}$ or $m$), \eqref{eq:main.Phi.m.est.8} and \eqref{eq:main.Phi.m.contradiction} imply that on the interval $[0,\widetilde{\Ub}]$,
\begin{equation}\label{eq:main.Phi.m.almost.the.end}
\sum_{0\leq k\leq K}(\|e^{-A\ub} \mathring{\nab}{}^{k} \f{\rd(\Phi-\Phi_m)}{\rd \ub}\|_{L^2_{\ub}L^\i(S_{0,\ub})} + \|e^{-A\ub}\mathring{\nab}{}^{k}(\Phi - \Phi_m)\|_{L^\i_{\ub}L^\i(S_{0,\ub})}) \ls 2^{-m}.
\end{equation}
For $m\in \mathbb N$ sufficiently large, \eqref{eq:main.Phi.m.almost.the.end} contradicts \eqref{eq:main.Phi.m.contradiction}. Hence 
\begin{equation}\label{eq:main.Phi.m.almost.the.end.2}
\|\Phi_m^{-1}\|_{L^\i_{\ub} W^{K,\i}(S_{0,\ub})} \leq 2 \|\Phi^{-1}\|_{L^\i_{\ub} W^{K,\i}(S_{0,\ub})}
\end{equation}

Now using \eqref{eq:main.Phi.m.almost.the.end.2}, we can repeat the above argument to derive \eqref{eq:main.Phi.m.almost.the.end} from \eqref{eq:main.Phi.m.est.8} as long as $A$ and $m$ are chosen to be sufficiently large. Finally, we fix $A$ in \eqref{eq:main.Phi.m.almost.the.end} and absorb it into the implicit constants. Since $\Ub$ is arbitrary, we have thus obtained \eqref{eq:phi.diff.est}. 

\pfstep{Step~3: Proof of \eqref{eq:duphi.uniform}} Finally, the proof of \eqref{eq:duphi.uniform} is much more straightforward since we are only have to deal with the \emph{smooth} ODE \eqref{Phim.ODE}. 

Using the uniform estimates given by \eqref{eq:fm.uniform}, \eqref{eq:f.approx.Phi.bkg.est}, \eqref{eq:f.approx.gamma.bkg.est} and \eqref{eq:phi.diff.est}, we know that 
\begin{equation}\label{eq:uniform.est.for.Rm}
\sup_m \|\f 18 |\f{\rd\hat{\gamma}_m}{\rd\ub}|^2_{\hat{\gamma}_m}\Phi_m + \f 12 f_m \Phi_m^{-1}\|_{L^1_{\ub} W^{K,\infty}(S_{0,\ub})}<+\infty.
\end{equation}

We now let $\Psi_m = \f{\rd\Phi_m}{\rd\ub}$, $H_m = 2 \f{\rd \log\Om}{\rd\ub}$ and $R_m = \f 18 |\f{\rd\hat{\gamma}_m}{\rd\ub}|^2_{\hat{\gamma}_m}\Phi_m + \f 12 f_m \Phi_m^{-1}$. Then by \eqref{Phim.ODE}, \eqref{eq:f.approx.Omg.bkg.est} and \eqref{eq:uniform.est.for.Rm}, it suffices to show that for $\Psi_m$ solving the ODE
\begin{equation}\label{Phim.ODE.revisit}
\f{\rd \Psi_m}{\rd \ub} + H_m \Psi_m + R_m  =0,
\end{equation}
with the estimates
\begin{equation}
\sup_m (\| \Psi_m(0,\vartheta)\|_{W^{K,\infty}(S_{0,0})}  + \|H_m\|_{L^2_{\ub} W^{K,\infty}(S_{0,\ub})} + \|R_m \|_{L^1_{\ub} W^{K,\infty}(S_{0,\ub})}) \leq D,
\end{equation}
we can obtain uniform (in $m$) estimates for the $L^\i_{\ub}W^{K,\infty}(S_{0,\ub})$ norm of $\Psi_m$.

Assume as a bootstrap assumption that $\| \Psi_m\|_{W^{K,\infty}(S_{0,\ub})}\leq \sqrt{B} D e^{B\ub}$ (which is satisfied initially as long as $B\geq 1$). Differentiating \eqref{Phim.ODE.revisit} by $\mathring{\nab}{}^k$ (for $0\leq k\leq K$), integrating in $\ub$, and then using the bootstrap assumption and the Cauchy--Schwarz inequality, we obtain
\begin{equation*}
\begin{split}
\|\mathring{\nab}{}^k \Psi_m\|_{L^\i(S_{0,\ub})} \ls &\: D + \sum_{k_1+k_2 = k} \|\mathring{\nab}{}^{k_1} H_m\|_{L^2_{\ub} L^\i(S_{0,\ub})} \|\mathring{\nab}{} ^{k_2}\Psi_m\|_{L^2_{\ub} L^\i(S_{0,\ub})} + \| \mathring{\nab}{}^k R_m\|_{L^1_{\ub} L^\i(S_{0,\ub})} \\
\ls &\:  D + D (\sqrt{B}D) (\int_0^{\ub} e^{2 B\ub}\,\ud \ub')^{\f 12} \ls D +  D^2 e^{B\ub} \ls D(D+1) e^{B\ub}.
\end{split}
\end{equation*}
In particular, we obtain $\| \Psi_m\|_{W^{K,\infty}(S_{0,\ub})}\leq D(D+1) e^{B\ub}$. If we choose $B$ such that $(D+1)\ll \sqrt{B}$, we have improved the bootstrap assumption. Finally, after fixing $B$ and allowing the implicit constant to depend on $B$, we obtain that $\|\Psi_m\|_{L^\i_{\ub}W^{K,\infty}(S_{0,\ub})} \ls D(D+1)$. As argued above, this implies \eqref{eq:duphi.uniform}. \qedhere
\end{proof}

\subsection{Putting everything together: Approximating measure-valued null dust data by vacuum data}\label{sec:final.approx}

\begin{proposition}\label{prop:final.approx}
Let $K\in \mathbb N$. Suppose we are given $\ud\nu_{\mathrm{init}}$ and $\mathring{\gamma}$ as in the assumptions of Proposition~\ref{prop:f.approx}, and $\Phi$, $\log\Omg$, $\hat{\gamma}$ as in Proposition~\ref{prop:f.approx.Phi.conv}. Assume moreover that $\ud\nu_{\mathrm{init}}$ is supported on $[0,\ub_*]\times U^c$ for some $U\subset  \mathbb S^2$.

Then there exist smooth $\{(\hat{\gamma}_m^{(vac)},\Phi^{(vac)}_m)\}_{m=m_0}^{+\infty}$ on $[0,\ub_*]\times \mathbb S^2$ such that the following holds:
\begin{enumerate}
\item (Basic properties) $\Phi^{(vac)}_m$ is strictly positive. $\hat{\gamma}_m^{(vac)}$ is a positive definite metric on each $S_{0,\ub}$. Moreover,
\begin{equation}\label{eq:final.approx.det}
\f{\det\hat{\gamma}_m^{(vac)}}{\det \mathring{\gamma}} = 1
\end{equation}
\item (Vacuum constraint) For every $m\in \mathbb N$ with $m\geq m_0$, 
$$\begin{cases}
\f{\rd^2 \Phi^{(vac)}_m}{\rd\ub^2}=2\f{\rd \log\Omega}{\rd\ub}\f{\rd\Phi^{(vac)}_m}{\rd\ub}-\f {1}8 |\f{\rd\hat{\gamma}_m^{(vac)}}{\rd\ub}|_{\hat{\gamma}_m^{(vac)}}^2 \Phi^{(vac)}_m,\\
\Phi^{(vac)}_m(\ub=0) = \Phi(\ub = 0),\quad \f{\rd \Phi_m^{(vac)}}{\rd\ub}(\ub = 0) = \f{\rd \Phi}{\rd\ub}(\ub = 0).
\end{cases}$$
\item (Uniform estimates)
\begin{equation}\label{eq:final.approx.ue.1}
\|\hat{\gamma}_m^{(vac)} \|_{L^\i_{\ub}W^{K,\infty}(S_{0,\ub},\mathring{\gamma})} + \| \f{\rd\hat{\gamma}_m^{(vac)}}{\rd \ub} \|_{L^\i_{\ub}W^{K,\infty}(S_{0,\ub},\mathring{\gamma})}\ls 1,
\end{equation}
\begin{equation}\label{eq:final.approx.ue.2}
\|\Phi_m^{(vac)} \|_{L^\i_{\ub}W^{K,\infty}} + \|\f{\rd \Phi_m^{(vac)}}{\rd \ub}\|_{L^\i_{\ub}W^{K,\infty}(S_{0,\ub},\mathring{\gamma})}\ls 1,
\end{equation}
\begin{equation}\label{eq:final.approx.ue.3}
\Phi_m^{(vac)} \gtrsim 1.
\end{equation}
\item (Convergence) The following convergences hold:
\begin{equation}\label{eq:final.approx.convergence.easy}
\hat{\gamma}_m^{(vac)}\to \hat{\gamma},\quad \Phi_m^{(vac)}\to \Phi \quad \mbox{uniformly},
\end{equation}
and for every $\varphi \in C^0_c([0,\ub_*]\times \mathbb S^2)$,
\begin{equation}\label{eq:final.approx.convergence.hard}
\begin{split}
&\: \int_{(0,\ub_*)\times \mathbb S^2} \varphi \,\ud\nu_{\mathrm{init}}  \\
= &\: \f 14 \lim_{m\to +\infty} (\int_{[0,\ub_*]\times \mathbb S^2} \varphi \Omg^{-2}\, |\f{\rd\hat{\gamma}_m^{(vac)}}{\rd\ub}|_{\hat{\gamma}_m^{(vac)}}^2 \,\mathrm{dA}_{\gamma_{m}^{(vac)}}\, \ud \ub) - \f 14 \int_{\{u\}\times [0,\ub_*]\times \mathbb S^2} \varphi \Omg^{-2} \, |\f{\rd\hat{\gamma}}{\rd\ub}|_{\hat{\gamma}}^2 \,\mathrm{dA}_{\gamma}\,\ud\ub.
\end{split}
\end{equation}
\end{enumerate}
\end{proposition}
\begin{proof}
\pfstep{Step~1: Approximate with a sequence of smooth null dust data} Let $\{f_m\}_{m=1}^{+\infty}$ be as in the conclusion of Proposition~\ref{prop:f.approx}. Since $\ud\nu_{\mathrm{init}}$ is supported on $[0,\ub_*]\times U^c$, by Proposition~\ref{prop:f.approx}, we choose $f_m$ so that $\mathrm{supp}(f_m)\subset [0,\ub_*]\times V^c$ for some fixed non-empty open $V\subseteq U$ with $\overline{V} \subset U$.

Define a \emph{smooth} sequence $\{(\hat{\gamma}_m^{(dust)},\, \overline{\Phi}_m^{(dust)}, \, \overline{\Psi}_m^{(dust)})\}_{m=1}^{+\infty}$ so that the estimates \eqref{eq:f.approx.gamma.est} and \eqref{eq:f.approx.data.est} hold (with $(\hat{\gamma}_m,\, \overline{\Phi}_m, \, \overline{\Psi}_m) = (\hat{\gamma}_m^{(dust)},\, \overline{\Phi}_m^{(dust)}, \, \overline{\Psi}_m^{(dust)})$). We in addition choose 
\begin{equation}\label{eq:final.approx.extra.small}
\|\hat{\gamma}_m^{(dust)} - \hat{\gamma}\|_{L^\i_{\ub} W^{K,\infty}(S_{0,\ub},\mathring{\gamma})} \leq \min\{\ep, \, 2^{-m} \},
\end{equation}
where $\ep>0$ is as in Proposition~\ref{prop.data.approx}.4 (see in particular \eqref{eq:data.approx.uniform.assumption}) with $\hat{\gamma}_0 = \hat{\gamma}$.

We then apply Proposition~\ref{prop:f.approx.Phi.conv} so that by \eqref{eq:phi.diff.est} and \eqref{eq:duphi.uniform} we have
\begin{equation}\label{eq:final.approx.Phi.dust}
\|\Phi_m^{(dust)} - \Phi \|_{L^\i_{\ub} W^{K,\i}(S_{0,\ub},\mathring{\gamma})}\ls 2^{-m},\quad \sup_m \|\f{\rd\Phi_m^{(dust)}}{\rd\ub}\|_{L^{\i}_{\ub} W^{K,\i}(S_{0,\ub},\mathring{\gamma})} \ls 1.
\end{equation}

\pfstep{Step~2: Approximate with a sequence of vacuum data} From Step~1, we have obtain a sequence of smooth data $\{( f_m, \, \hat{\gamma}_m^{(dust)}, \, \Phi_m^{(dust)})\}_{m=1}^{+\infty}$ to the Einstein--null dust system. 
Now for each $m\in \mathbb N$, we apply Proposition~\ref{prop.data.approx}. We choose $n(m)$ sufficiently large depending on $m$ so that 
\begin{equation}\label{eq:n.m.compare}
n^{-1} \leq 2^{-m}
\end{equation}
and that all the $O(\f 1n)$ error terms in Proposition~\ref{prop.data.approx} are made $\ls 2^{-m}$.
Since we have already required \eqref{eq:final.approx.extra.small}, we thus obtain a sequence $\{ (\hat{\gamma}_m^{(vac)},\Phi_m^{(vac)})\}_{m=1}^{+\infty}$ so that 
\begin{equation}\label{eq:final.approx.gamma.vac}
\|\hat{\gamma}_m^{(vac)} - \hat{\gamma}^{(dust)}_m\|_{L^\i_{\ub}W^{K,\infty}(S_{0,\ub},\mathring{\gamma})} \ls 2^{-m},
\end{equation}
\begin{equation}\label{eq:final.approx.gamma.weak.vac}
\||\f{\rd\hat{\gamma}_m^{(vac)}}{\rd \ub}|_{\hat{\gamma}_m^{(vac)}}^2 (\Phi_m^{(dust)})^2 - |\f{\rd\hat{\gamma}_m^{(dust)}}{\rd \ub}|_{\hat{\gamma}_m^{(dust)}}^2 (\Phi_m^{(dust)})^2 - 4 f_m  - \f 1{n(m)} \f{\rd F_m}{\rd\ub}\|_{L^\i_{\ub}W^{K,\infty}(S_{0,\ub},\mathring{\gamma})} \ls 2^{-m},
\end{equation}
\begin{equation}\label{eq:final.approx.Phi.vac}
\|\Phi_m^{(vac)} - \Phi_m^{(dust)}\|_{L^\i_{\ub}W^{K,\infty}} \ls 2^{-m}, \quad \|\f{\rd (\Phi_m^{(vac)} - \Phi_m^{(dust)})}{\rd \ub}\|_{L^\i_{\ub}W^{K,\infty}(S_{0,\ub},\mathring{\gamma})}\ls 2^{-m},
\end{equation}
and
\begin{equation}\label{eq:final.approx.uniform.dubgamma}
\|\f{\rd \hat{\gamma}_m^{(vac)}}{\rd\ub} \|_{L^2_{\ub}W^{K,\infty}(S_{0,\ub},\mathring{\gamma})} \ls 1.
\end{equation}
Here, in \eqref{eq:final.approx.gamma.weak.vac}, $F_m$ are smooth functions which according to \eqref{eq:data.approx.Fn.est} satisfy 
\begin{equation}\label{eq:final.approx.Fm}
\|F_m\|_{L^\i_{\ub} W^{K,\infty}(S_{0,\ub},\mathring{\gamma})} \ls 1.
\end{equation}
Note also that in deriving \eqref{eq:final.approx.uniform.dubgamma}, we have also used the estimate \eqref{eq:fm.uniform} for $f_m$.

\pfstep{Step~3: Putting everything together} In this last step, we check that for sufficiently large $m_0$, the sequence $\{ (\hat{\gamma}_m^{(vac)},\Phi_m^{(vac)})\}_{m=m_0}^{+\infty}$ constructed in Step~2 indeed satisfies the necessary requirements.

First, the constructions in Propositions~\ref{prop.data.approx} guarantees \eqref{eq:final.approx.det}. After choosing $m_0$ to be sufficiently large, the positivity follows from \eqref{eq:final.approx.extra.small}, \eqref{eq:final.approx.Phi.dust}, \eqref{eq:final.approx.gamma.vac}, \eqref{eq:final.approx.Phi.vac} and the triangle inequality.

Second, since $\{ (\hat{\gamma}_m^{(vac)},\Phi_m^{(vac)})\}_{m=1}^{+\infty}$ are constructed by Proposition~\ref{prop.data.approx}, the vacuum constraints are satisfied by definition.

Third, the uniform estimates \eqref{eq:final.approx.ue.1} and \eqref{eq:final.approx.ue.2} follow from \eqref{eq:final.approx.extra.small}, \eqref{eq:final.approx.Phi.dust}, \eqref{eq:final.approx.gamma.vac}, \eqref{eq:final.approx.Phi.vac}, \eqref{eq:final.approx.uniform.dubgamma} and the triangle inequality. The lower bound \eqref{eq:final.approx.ue.3} follows from the lower bound of $\Phi$, the estimates \eqref{eq:final.approx.Phi.dust}, \eqref{eq:final.approx.Phi.vac} and the triangle inequality.

Fourth, the convergence statements in \eqref{eq:final.approx.convergence.easy} follow from \eqref{eq:final.approx.extra.small}, \eqref{eq:final.approx.Phi.dust}, \eqref{eq:final.approx.gamma.vac}, \eqref{eq:final.approx.Phi.vac} and the triangle inequality.

Finally, we prove the convergence statements in \eqref{eq:final.approx.convergence.hard}. By density we assume that $\varphi$ is $C^1$. 

Note that $\mathrm{dA}_\gamma = \Phi^2 \,\mathrm{dA}_{\mathring{\gamma}}$. Also, by \eqref{eq:final.approx.det}, $\mathrm{dA}_{\gamma_m^{(vac)}} = (\Phi_m^{(vac)})^2 \,\mathrm{dA}_{\mathring{\gamma}}$. Then, using \eqref{eq:f.approx.Phi.bkg.est}, \eqref{eq:f.approx.gamma.bkg.est}, \eqref{eq:final.approx.Phi.dust}, \eqref{eq:final.approx.gamma.weak.vac}, \eqref{eq:final.approx.Phi.vac} and \eqref{eq:final.approx.uniform.dubgamma}, we obtain that  for each fixed $m\in \mathbb N$ with $m\geq m_0$,
\begin{equation}\label{eq:final.approx.final.1}
\begin{split}
&\: \f 14 \int_{[0,\ub_*]\times \mathbb S^2} \varphi \Omg^{-2}\, |\f{\rd\hat{\gamma}_m^{(vac)}}{\rd\ub}|_{\hat{\gamma}_m^{(vac)}}^2 \,\mathrm{dA}_{\gamma_{m}^{(vac)}}\, \ud \ub - \f 14 \int_{\{u\}\times [0,\ub_*]\times \mathbb S^2} \varphi \Omg^{-2} \, |\f{\rd\hat{\gamma}}{\rd\ub}|_{\hat{\gamma}}^2 \,\mathrm{dA}_{\gamma}\,\ud\ub \\
 = &\: \f 14 \int_{[0,\ub_*]\times \mathbb S^2} \varphi \Omg^{-2}\, (|\f{\rd\hat{\gamma}_m^{(vac)}}{\rd\ub}|_{\hat{\gamma}_m^{(vac)}}^2 (\Phi_m^{(vac)})^2 - |\f{\rd\hat{\gamma}}{\rd\ub}|_{\hat{\gamma}}^2 \Phi^2) \,\mathrm{dA}_{\mathring{\gamma}}\, \ud \ub \\
= &\:  \int_{[0,\ub_*]\times \mathbb S^2} \varphi \Omg^{-2}\, (f_m + \f 1{4n(m)}\f{\rd F_m}{\rd\ub} ) \,\mathrm{dA}_{\mathring{\gamma}}\, \ud \ub + O(2^{-m}).
\end{split}
\end{equation}
We then integrate by parts and use \eqref{eq:n.m.compare} and \eqref{eq:final.approx.Fm} to obtain  
\begin{equation}\label{eq:final.approx.final.2}
\begin{split}
 \left| \int_{[0,\ub_*]\times \mathbb S^2} \varphi \Omg^{-2}\,  \f 1{4n(m)}\f{\rd F_m}{\rd\ub}  \,\mathrm{dA}_{\mathring{\gamma}}\, \ud \ub\right| 
 \ls 2^{-m}.
\end{split}
\end{equation}
Plugging \eqref{eq:final.approx.final.2} into \eqref{eq:final.approx.final.1}, taking the $m\to +\infty$ limit, and using the first conclusion in Proposition~\ref{prop:f.approx}, we obtain 
\begin{equation*}
\begin{split}
\mbox{RHS of \eqref{eq:final.approx.convergence.hard}} =  \lim_{m\to +\infty} \int_{[0,\ub_*]\times \mathbb S^2} \varphi \Omg^{-2}\, f_m \,\mathrm{dA}_{\mathring{\gamma}}\, \ud \ub 
=&\: \mbox{LHS of \eqref{eq:final.approx.convergence.hard}}.
\end{split}
\end{equation*}
\qedhere
\end{proof}

The same construction as Proposition~\ref{prop:final.approx} can be carried out on the $\ub=0$ hypersurface in a completely analogous manner. We state this as a proposition:
\begin{proposition}\label{prop:final.approx.1}
Proposition~\ref{prop:final.approx} holds on $\Hb_0$ after replacing $\ub\mapsto u$, $\ud\nu_{\mathrm{init}} \mapsto \ud\nub_{\mathrm{init}}$.
\end{proposition}

\subsection{Proofs of Theorem~\ref{thm:main.local.dust} and Theorem~\ref{thm:reverse.Burnett}}\label{sec:approx.final}

\begin{proof}[Proof of Theorem~\ref{thm:main.local.dust}]
We first consider the case that the strongly angularly regular reduced characteristic initial data set satisfies the additional assumption that $\mathrm{supp}(\ud\nu_{\mathrm{init}}) \subset [0,\underline{I} ]\times U^c$  and  $\mathrm{supp}(\ud\nub_{\mathrm{init}}) \subset [0,I]\times U^c$ for some non-empty open $U\subset  \mathbb S^2$.

We first apply Propositions~\ref{prop:final.approx} and \ref{prop:final.approx.1} in the previous subsection so as to obtain a sequence of reduced characteristic initial data sets (see Section~\ref{sec:reduced.data}) for the Einstein vacuum equations.

By Lemma~\ref{lem:suff.cond.on.data}, the sequence of vacuum characteristic initial data sets corresponding to this sequence of vacuum reduced characteristic initial data set satisfies the assumption of Theorem~\ref{main.thm}. Therefore, Theorem~\ref{main.thm} shows that there exists $\ep>0$ so that the sequence of vacuum solutions arising from this sequence of data admits a subsequence that converges to an angularly regular weak solution to the Einstein--null dust system in $[0,u_*]\times [0,\ub_*]\times \mathbb S^2$ whenever $u_*\in (0,I]$ and $\ub_* \in (0,\ep]$. Because of the convergence statements \eqref{eq:final.approx.convergence.easy} and \eqref{eq:final.approx.convergence.hard}, it follows moreover that this solution indeed achieve the prescribed initial data. We have thus proven the existence part of the theorem.

Uniqueness then follows from Theorem~\ref{thm:uniqueness}.

Finally, in the general case where $\ud\nu_{\mathrm{init}}$ or $\ud \nub_{\mathrm{init}}$ is not supported away from some angular direction, we can cut off and use the finite speed of propagation to reduce to the previous case. \qedhere
\end{proof}

\begin{proof}[Proof of Theorem~\ref{thm:reverse.Burnett}]
This is similar to the proof of Theorem~\ref{thm:main.local.dust} so we will be brief. Given an angularly regular weak solution to the Einstein--null dust system as in Theorem~\ref{thm:main.local.dust}, we first approximate the data using  Propositions~\ref{prop:final.approx} and \ref{prop:final.approx.1} and then use Lemma~\ref{lem:suff.cond.on.data} and Theorem~\ref{main.thm} to obtain a limiting angularly regular weak solution $(\widetilde{\mathcal M},g_\infty)$ (for $u_{**}$ sufficiently small) to the Einstein--null dust system.

By definition, $(\widetilde{M},\,g_\infty)$ is a limit of a sequence of smooth solutions $(\widetilde{\mathcal M},\,g_n)$ to the Einstein vacuum equations in the sense of Theorem~\ref{main.thm}. On the other hand, by Theorem~\ref{thm:uniqueness}, $(\widetilde{\mathcal M},\,g_\infty) = (\widetilde{\mathcal M},\,g\restriction_{\widetilde{\mathcal M}})$. Combining these two facts gives the conclusion of the theorem. \qedhere
\end{proof}

\section{Relation with the formation of trapped surfaces}\label{sec:trapped.surfaces}

In this final section, we discuss a connection of high-frequency limits and null dust shells with Christodoulou's work \cite{Chr} on the formation of trapped surfaces in vacuum. As is well-known, Penrose proved that a vacuum spacetime (or more generally a spacetime obeying the null energy condition) with a non-compact Cauchy hypersurface and a \emph{trapped surface} must be future causally geodesically incomplete. In a monumental work \cite{Chr}, Christodoulou showed moreover that trapped surfaces could form dynamically in vacuum spacetimes from characteristic initial data which are arbitrarily dispersed. In particular, he found open sets of initial data such that there are no trapped surfaces initially, but a trapped surface is formed in the dynamical evolution. We will recall the results of \cite{Chr} in \textbf{Section~\ref{sec:Chr.trapped.result}}; see also \cite{An.AH, AnAthanasiou, AL, Jaffe, KLR, KlRo.scarred, KlRo.trapped, Le, Li.Schwarzschild, LiLiu, LiMei, LiYu.glue, LR2, Yu.Maxwell, Yu.CMP} and references therin for various extensions.

Christodoulou's construction is based on what he called the \emph{short pulse method}, for which the large initial data is concentrated on a short length scale of size $\de$. It is precisely the short length scale that allowed Christodoulou to propagation a hierarchy of $\de$-dependent estimates for the geometric quantities so that the estimates can be closed despite being a large data problem

We will show below (see \textbf{Section~\ref{sec:connection}}) that if we take the $\de \to 0^+$ limit in Christodoulou's short pulse ansatz, then one obtains a limiting spacetime which is not vacuum, but solves the Einstein--null dust system with a null dust shell (in a similar manner as the main results of this paper). The limiting solution in fact coincides with the Synge--Gibbons--Penrose construction (see \textbf{Section~\ref{sec:Gibbons.Penrose}}) of collapsing null dust shell solutions in which trapped surfaces form dynamically. In particular, one could think of Christodoulou's solutions as ``approximating'' the spacetimes with collapsing null shells.

\subsection{Christodoulou's trapped surface formation result}\label{sec:Chr.trapped.result}

In  \cite{Chr}, Christodoulou proved the dynamical formation of trapped surfaces by considering a characteristic initial value problem, where the initial data are prescribed on two intersecting null hypersurfaces.

Before describing Christodoulou's data, first define $\mathcal M^- :=  \{ (u,\ub, \vartheta):  0 < u \leq \ub+1 <1 ,\,\vartheta\in \mathbb S^2\}$ to be the past light cone of a point in Minkowski spacetime\footnote{In polar coordinates, the Minkowski metric is given by $m= -dt^2 + dr^2 + r^2 \mathring{\gamma}_{\mathbb S^2(1)}$. Here, the null coordinates correspond to $u = \f 12 (t-r) +1$, $\ub = \f 12 (t+r)$. Thus in the $(t,r,\vartheta)$ coordinate system, we have $\mathcal M^- = \{(t,r,\vartheta): -2< t+r < 0,\,-2<t-r<0\}$, which is a truncated subset of the past light cone of the origin in Minkowski spacetime.}, i.e.~consider the metric $g$ on $\mathcal M^-$ taking the form
$$g = -2\Omega^2(\ud u\otimes \ud\ub+\ud\ub\otimes \ud u)+\gamma_{AB}(\ud\th^A-b^A\ud u)\otimes (\ud\th^B-b^B\ud u)$$
with $\Omg^2\restriction_{\mathcal M^-} = 1$, $b\restriction_{\mathcal M^-} = 0$ and $\gamma\restriction_{\mathcal M^-} = (\ub - u + 1)^2 \mathring{\gamma}_{\mathbb S^2(1)}$, where $\mathring{\gamma}_{\mathbb S^2(1)}$ is the round metric on the $2$-sphere with radius $1$.

For the characteristic initial value problem that Christodoulou considered, one initial characteristic hypersurface is $\Hb_0:= \{ (u,\ub, \vartheta) \in \mathcal M^-: \ub = 0\}$ with the induced geometry. The other initial characteristic surface is given by $H_{0} = \{u = 0\} \times [0,\de]\times \mathbb S^2$, and the data consist of a ``short pulse'' mentioned above, where $\de>0$ is a small parameter. The following is a version\footnote{The original theorem in fact applies when the data are posed on past null infinity to obtain a semi-global spacetime; see details in \cite{Chr}.} of the main theorem in \cite{Chr}, which gives a condition on the initial data on $H_{-1}$ which guarantees the dynamical formation of trapped surfaces.

\begin{theorem}[Christodoulou \cite{Chr}]\label{thm:Chr.FOTS}
For every $B>0$ and $u_* < 1$, there exists $\de = \de(B,u_*) > 0$ sufficiently small such that if the initial $\chih$ (denoted $\chih_0$), prescribed
on $H_{0}:=\{(u,\ub, \vartheta): u = 0,\,\ub \in [0,\de],\,\vartheta \in \mathbb S^2 \}$ satisfy
\begin{equation}\label{eq:Chr.upper.bd}
\sum_{i\leq 5, \, j \leq 3} \de^{\f 12 + j} \|\nab^i \nab_4^j \chih_0\|_{L^\i_{\ub} L^2(S_{u,\ub})} \leq B,
\end{equation}
then there is a unique solution to the Einstein vacuum equation in double null coordinates in $\{(u,\ub): u \in[0, u_*],\, \ub\in [0,\de]\} \times \mathbb S^2$
with the prescribed data.

Moreover, if the initial data also verify the lower bound
\begin{equation}\label{eq:Chr.lower.bd}
\inf_{\vartheta \in \mathbb S^2} \int_0^\de |\chih_0|^2_\gamma (\ub', \vartheta) \,\ud \ub'\geq M_* > 2 (1-u_*),
\end{equation}
then, after choosing $\de$ to be smaller (depending on $B$, $u_*$ and $M_*$) if necessary,
the sphere $S_{-1 + u_*,\de}$ is a trapped surface.
\end{theorem}

We remark that by definition, the sphere $S_{-1 + u_*,\de}$ is a trapped surface exactly when the inequalities $\trch>0$ and $\trchb<0$ hold pointwise on $S_{-1 + u_*,\de}$.

As already noted in \cite{LR2}, the scaling of the initial data in Theorem~\ref{thm:Chr.FOTS} is such that the $L^2(H_0)$ norm of the initial data for $\chih$ and its angular derivatives are uniformly bounded as $\delta\to 0$. In particular, suppose we extend the data in a ``regular'' way to $\ub \in [0,\underline{I}]$ but still requiring the data to concentrate in $\ub\in [0,\de]$, we can study the $\de \to 0^+$ limit in view of Theorem~\ref{ext.thm}. Before we discuss that, we consider in the next subsection the formation of trapped surface in the presence of a null dust shell. (This will turn out to be connected to the $\de \to 0^+$ limit of Christodoulou's spacetimes; see Section~\ref{sec:connection}.)

\subsection{The Synge--Gibbons--Penrose construction}\label{sec:Gibbons.Penrose}

While to show that trapped surfaces form dynamically in \emph{vacuum} is a very difficult problem, to show that trapped surfaces can arise from the collapsing of a null dust shell is significantly easier. In fact, the null dust shell example of Synge \cite{Synge} already demonstrates the dynamical formation of trapped surfaces. Other examples were later considered in the works of Gibbons \cite{Gibbons.shell} and Penrose \cite{Penrose.shell}.

The setup of \cite{Gibbons.shell, Penrose.shell} is as follows. (The example in \cite{Synge} is a special case where the function $m: \mathbb S^2\to \mathbb R_{>0}$ below is a constant function.) Consider a spacetime $\mathcal M$ with a null dust shell supported on $\mathcal N$, which is an ingoing null hypersurface. The spacetime is given as $\mathcal M = \mathcal M^+ \cup \mathcal M^- \cup \mathcal N$, where $\mathcal M^-$ is the spacetime to one side of $\mathcal N$ and $\mathcal M^+$ is the spacetime to the other side of $\mathcal N$. Assume moreover that $\mathcal M^-$ is isometric to a region in Minkowski spacetime in exactly the same way as in Section~\ref{sec:Chr.trapped.result}.

Introduce now a double null coordinate system so that the metric is in the form
\begin{equation}\label{eq:double.null.trapped.surface}
g = -2\Omega^2(\ud u\otimes \ud\ub+\ud\ub\otimes \ud u)+\gamma_{AB}(\ud\th^A-b^A\ud u)\otimes (\ud\th^B-b^B\ud u).
\end{equation}

Define, in exactly the same manner as in Section~\ref{sec:Chr.trapped.result}, $\mathcal M^- :=  \{ (u,\ub, \vartheta):  0 < u \leq \ub+1 < 1 ,\,\vartheta\in \mathbb S^2\}$ and impose that $\Omg^2\restriction_{\mathcal M^-} = 1$, $b\restriction_{\mathcal M^-} = 0$ and $\gamma\restriction_{\mathcal M^-} = (\ub - u+1)^2 \mathring{\gamma}_{\mathbb S^2(1)}$, where $\mathring{\gamma}_{\mathbb S^2(1)}$ is the round metric on the $2$-sphere with radius $1$. 

As a result, it follows that in $\mathcal M^-$, $\trch = \f 2{\ub - u+1}$ and $\trchb = -\f 2{\ub - u+1}$ and all the other Ricci coefficients vanish.

Define $\mathcal M^+ := \{(u,\ub,\vartheta): 0<u<1,\, 0\leq \ub \leq f(u)\}$ for some decreasing function $f:[0,1]\to \mathbb R$ with $\lim_{u\to 1^-} f(u) = 0$. The choice of the metric in $\mathcal M^+$ does not matter so much, let us just assume that it is a vacuum metric in a double null coordinate system \eqref{eq:double.null.trapped.surface}. Assume also that the metric coefficients $\Omg$, $\gamma$ and $b$ are continuous up to and across $\mathcal N$ and that all the Ricci coefficients \emph{except} for $\trch$, $\chih$ and $\om$ are also continuous up to and across $\mathcal N$.

Impose that the Ricci coefficient\footnote{Notice that $\chih$ and $\om$ could have a jump discontinuity across $\{\ub = 0\}$ in this construction. (A jump discontinuity of $\chih$ corresponds to an impulsive gravitational wave.) Nevertheless, whether $\chih$ and $\om$ have a jump discontinuity does not affect conclusions regarding trapped surface formation. } $\trch$ has a jump discontinuity across $\mathcal N$ so that\footnote{We note that the propagation equation \eqref{eq:nu} essentially forces the jump $\trch^+ - \trch^-$ to be of this form.} $\trch^+ - \trch^- = \f{ m(\vartheta) }{(1-u)^2}$, for some smooth function $m:\mathbb S^2 \to \mathbb R_{\geq 0}$. By \eqref{eq:trch} (and taking appropriate limit), this means that the spacetime has a null dust shell given by
$$\ud\nu_u = m(\vartheta) \de(\ub),$$
where $\de(\ub)$ denotes the delta measure at $\ub = 0$ and $m(\vartheta)$ is as before. Note also that with this definition of $\ud\nu_u$, it is easy to check that the propagation equation \eqref{eq:nu} holds.

Suppose now that there are $M_*>0$ and $u_* \in (0,1)$ such that
\begin{equation}\label{eq:null.shell.m.lower.bd}
\inf_{\vartheta \in \mathbb S^2} m(\vartheta) \geq M_* > 2(1-u_*) >0.
\end{equation}
We claim that in fact the sphere $S_{u_*,\epsilon}$ is \emph{trapped} for $\epsilon>0$ sufficiently small. In order to prove that, it suffices to show that $ \trchb(u = u_*, \,\ub = 0,\, \vartheta)<0$ and that\footnote{Recall again that $\trch$ is not continuous across the hypersurface $\{\ub = 0\}$!} $\trch^+(u=u_*,\, \ub = \epsilon, \,\vartheta) < 0$ for all $\vartheta \in \mathbb S^2$. These are very easy to check: since $\trchb$ is continuous across the null dust shell, it takes the Minkowskian value $\trchb(u=u_*,\,\ub = 0,\,\vartheta) = -\f 2{1 - u_*}$; on the other hand, since $\trch^-(u = u_*,\,\ub = 0,\,\vartheta) = \f 2{1- u_*}$ (taking Minkowskian value), the jump condition above gives
$$\trch^+(u = u_*,\,\ub = 0,\,\vartheta) = \trch^-(u = u_*,\,\ub = 0,\,\vartheta) - \f{ m(\vartheta) }{(1-u_*)^2} =  \f 2{1- u_*} - \f{ m(\vartheta) }{(1-u_*)^2}.$$
In particular, this means that whenever $m(\theta)$ obeys the lower bound \eqref{eq:null.shell.m.lower.bd}, we have $\trch^+<0$ everywhere on $S_{u_*, 0}$.  Now since the metric is smooth in $\mathcal M^+ \cap \{ \ub >0\}$, it follows that $S_{u_*,\ub}$ is trapped for some $\ub>0$ sufficiently small.

This simple construction shows that a trapped surface is formed dynamically from the collapse of a null dust shell.

\subsection{Connection between the Synge--Gibbons--Penrose construction and trapped surface formation in vacuum}\label{sec:connection}

We now observe that by taking the $\de\to 0^+$ limit in Christodoulou's construction in Theorem~\ref{thm:Chr.FOTS}, one obtains a solution to the Einstein--null dust system with a null dust shell as in Section~\ref{sec:Gibbons.Penrose}. 

To make this precise, we need slightly more assumptions than that in Theorem~\ref{thm:Chr.FOTS}. In order to streamline the exposition, let us first state a propagation of regularity result, before turning to the precise setup relating Christodoulou's result to Section~\ref{sec:Gibbons.Penrose}. The following propagation of regularity result is a small modification of \cite[Proposition~52]{LR2}, and can be proven in exactly the same manner.

\begin{lemma}\label{lem:prop.of.singularities}
Consider a characteristic initial value problem with initial data satisfying the assumptions of Theorem~\ref{ext.thm}. 
\begin{enumerate}
\item (Propagation of angular regularity) If $ \forall i \in \mathbb N \cup\{0\}$, $\exists C_i>0$ such that the initial data satisfy (in addition to the assumptions of Theorem~\ref{ext.thm})
\begin{equation}\label{eq:trapped.surface.hr.assumption.1}
\begin{split}
 \sum_{\psi \in \{\eta,\etab,\trch,\trchb,K\}}\| \nab^i \psi \|_{L^\i_{\ub} L^\i(S_{0,\ub})} +  \sum_{\psi \in \{\eta,\etab,\trch,\trchb,K\}}\| \nab^i \psi \|_{L^\i_{u} L^\i(S_{u,0})} & \\
+ \sum_{\psi_H \in \{\chih,\om\}}\|\nab^i \psi_H\|_{L^2_{\ub} L^\i(S_{0,\ub})} + \sum_{\psi_{\Hb} \in \{\chibh,\omb\}}\|\nab^i \psi_{\Hb}\|_{L^2_{u} L^\i(S_{u,0})}  &\: \leq C_i.
\end{split}
 \end{equation}
Then $\forall i'\in \mathbb N\cup \{0\}$, $\exists C'_{i'} > 0$ (where for each $i'$, $C'_{i'}$ depends only on the constants in Theorem~\ref{ext.thm} and finitely many $C_i$'s) such that the following bounds hold in $[0,u_*]\times [0,\ub_*]\times \mathbb S^2$:
\begin{equation}\label{eq:trapped.surface.hr.conclusion.1}
\begin{split}
 \sum_{\psi \in \{\eta,\etab,\trch,\trchb,K\}}\| \nab^{i'} \psi \|_{L^\i_u L^\i_{\ub} L^\i(S_{u,\ub})} & \\
+ \sum_{\psi_H \in \{\chih,\om\}}\|\nab^i \psi_H\|_{L^2_{\ub} L^\i_u L^\i(S_{u,\ub})} + \sum_{\psi_{\Hb} \in \{\chibh,\omb\}}\|\nab^i \psi_{\Hb}\|_{L^2_{u} L^\i_{\ub} L^\i(S_{u,\ub})}  &\:  \leq C'_{i'}.
\end{split}
 \end{equation}
\item (Propagation of higher regularity in the ``regular region'') Suppose the assumptions of part (1) hold, and
\begin{itemize}
\item $\exists u_1,\,u_2,\,\ub_1,\,\ub_2$ with $0 \leq  u_1 < u_2 \leq u_*$ and $0 \leq \ub_1 < \ub_2 \leq \ub_*$, 
\item $\exists J,\, L\in \mathbb N \cup \{0\}$, and
\item $\forall i\in \mathbb N\cup \{0\}$, $\exists \widetilde{C}_{i}^{(J,L)}>0$
\end{itemize}
such that the initial data satisfy
\begin{equation}\label{eq:trapped.surface.hr.assumption.2}
\begin{split}
 &\: \sum_{\psi \in \{\eta,\etab,\trch,\trchb,K,\chih,\om\}} \sum_{\ell \leq L} \| \nab^i \nab_4^\ell \psi \|_{L^\i_{\ub}([\ub_1,\ub_2 ]; L^\i(S_{0,\ub}))} \\
 &\: +  \sum_{\psi \in \{\eta,\etab,\trch,\trchb,K,\chibh,\omb\}} \sum_{j\leq J} \| \nab^i \nab_3^j\psi \|_{L^\i_{u}([u_1,u_2]; L^\i(S_{u,0}))}  
 \leq \widetilde{C}_{i}^{(J,L)}.
\end{split}
 \end{equation}
Then $\forall i'\in \mathbb N\cup \{0\}$, $\exists \widetilde{C'}_{i}^{(J,L)} > 0$ (where for each $i'$, $ \widetilde{C'}_{i}^{(J,L)}$ depends only on the constants in Theorem~\ref{ext.thm} and finitely many $C_i$'s and $\widetilde{C}_{i}^{(J,L)}$'s) such that the following estimates hold:
\begin{equation}\label{eq:trapped.surface.hr.conclusion.2}
\begin{split}
 \sup_{(u,\ub) \in [u_1,u_2]\times [\ub_1,\ub_2]} \sum_{\psi \in \{\eta,\etab,\trch,\trchb,K,\chih,\om,\chibh,\omb\}} \sum_{\substack{ j\leq J \\ \ell\leq L}} \| \nab^{i'} \nab_3^j \nab_4^\ell \psi \|_{L^\i_u L^\i_{\ub} L^\i(S_{u,\ub})}  \leq \widetilde{C'}_{i}^{(J,L)}.
\end{split}
 \end{equation}
\end{enumerate}

\end{lemma}
We remark that an important point of part (2) Lemma~\ref{lem:prop.of.singularities} is that the higher regularity estimates hold even when the data are singular for $u<u_1$ or $\ub < \ub_1$.

We now return to the discussion of the relation between Theorem~\ref{thm:Chr.FOTS} and null dust shells. 

Consider a sequence of characteristic initial data on $\Hb_0$ is the backward Minkowskian light cone as in Theorem~\ref{thm:Chr.FOTS} and $H_0 = \{0\}\times [0, \underline{I}]\times \mathbb S^2$. Fix a decreasing sequence $\de_n \to 0$. On $H_0$, require the (smooth) characteristic initial data to obey the following:
\begin{enumerate}
\item (Christodoulou's conditions) Fix $M_*,\,B>0$ and $u_*\in (0,1)$. When restricted to $\ub\in [0, \de_n]$, \eqref{eq:Chr.upper.bd} and \eqref{eq:Chr.lower.bd} both hold (with $\de$ replaced by $\de_n$).
\item (Additional angular regularity) Assume pointwise estimates for \emph{all} higher angular derivatives, i.e.~assume \eqref{eq:trapped.surface.hr.assumption.1} holds.
\item (Additional regularity beyond short pulse) Impose that when restricted to $\ub \in (\de_m, \underline{I}]$, the metric components $(\gamma_n, \log\Omg_n, b_n)$ are uniformly bounded in the $C^k$ norm (with respect to derivatives tangential to $H_0$) for every $n\geq m$ and for every $k\in \mathbb N \cup \{0\}$.
\end{enumerate}

By Theorem~\ref{main.thm}, there exists an $\ep>0$ such that if $\ub_* \leq \ep$, then for every $n\in \mathbb N$ there is a solution to the Einstein vacuum equations in the region $[0,u_*]\times [0,\ub_*]\times \mathbb S^2$ with the prescribed initial data. Moreover, there exists a limiting solution to the Einstein--null dust system in $[0,u_*]\times [0,\ub_*]\times \mathbb S^2$.

Given the conditions (1)--(3) that we imposed above for the sequence of vacuum initial data, we can in fact conclude that the limiting solution to the Einstein--null dust system has the following features:
\begin{enumerate}
\item (Propagation of angular regularity) By part (1) of Lemma~\ref{lem:prop.of.singularities}, the improved angular regularity estimates \eqref{eq:trapped.surface.hr.conclusion.1} hold for the full sequence of vacuum solutions, and hence also hold for the limiting solution.
\item (Regularity in the $\nab_3$ direction) Since the data on $\Hb_0$ are smooth, by part (2) of Lemma~\ref{lem:prop.of.singularities}, the estimates \eqref{eq:trapped.surface.hr.conclusion.2} hold for $L = 0$ and for all $J \geq 0$ for the full sequence of vacuum solutions. As in (1), this implies that the same estimates hold for the limiting solution.
\item (Improved regularity away from $\{\ub = 0\}$) For every $\widetilde{\de}>0$, (since $\de_n \to 0^+$), there exists $N\in \mathbb N$ such that for every $n\geq N$, the assumption \eqref{eq:trapped.surface.hr.assumption.2} holds for $(u_1, u_2, \ub_1,\ub_2) := (0, u_*, \widetilde{\de}, \ub_*)$ and any $J,\,L\in \mathbb N\cup\{0\}$. It follows that for the limiting spacetime, \eqref{eq:trapped.surface.hr.assumption.1} holds away from $\{\ub = 0\}$. In particular, the limiting spacetime metric is smooth away from $\{\ub = 0\}$.
\item (Continuity of $\eta$, $\etab$, $\trchb$, $\chibh$ and $\omb$) By Theorem~\ref{main.thm} and the definition of angular regularity (Definition~\ref{double.null.def.2}), the limiting spacetime has continuous $\eta$, $\etab$. Using angular regularity and the improved $\nab_3$ regularity established above, it follows that $\trchb$, $\chibh$ and $\omb$ are also continuous. 
\item (Presence of a null shell) The Chrisotodoulou conditions \eqref{eq:Chr.upper.bd} and \eqref{eq:Chr.lower.bd} exactly imply that in the limit there is a null dust shell --- supported exactly on the $\{\ub = 0\}$-hypersurface --- given by the measure $\ud\nu_u = m(\vartheta) \de(\ub)$, where $m:\mathbb S^2\to \mathbb R_{>0}$ which is both bounded above and bounded below away from $0$. (Correspondingly, this gives a jump of $\trch$ across the $\{\ub = 0\}$ hypersurface.)
\end{enumerate}

With this it is easy to conclude that the limiting spacetime is exactly one as in Section~\ref{sec:Gibbons.Penrose}, i.e.~there is a propagating null dust shell which drives the dynamical formation of trapped surfaces. This demonstrates a connection between Christodoulou's construction in \cite{Chr} and collapsing null dust shells.

\appendix
\section{Derivation of the estimates in Theorem~\ref{thm:ext.est}}\label{app:est}

In this appendix, we derive the estimates in Theorem~\ref{thm:ext.est}. The most difficult bounds are already in \cite{LR2}. Here we indicate how to obtain the remaining estimates.

\begin{proof}[Proof of Theorem~\ref{thm:ext.est}]
In this proof, all constants $C>0$ and implicit constants in $\ls$ depend only on the constants in the assumptions of Theorems~\ref{ext.thm} and \ref{thm:ext.est}.

\pfstep{Step~1: Proof of \eqref{eq:bdd.psi}--\eqref{eq:bdd.psiHb}: Ricci coefficient estimates from \cite{LR2}} These estimates follow directly from \cite[Theorem~4]{LR2}. (Note that the norms $\mathcal O$, $\mathcal O_{3,2}$ and $\mathcal R$ defined in \cite{LR2} controls all the norms in \eqref{eq:bdd.psi}--\eqref{eq:bdd.psiHb}.)

\pfstep{Step~2: Proof of \eqref{eq:bdd.metric} I: Estimates for the metric components} In this step we prove the $C^0_u C^0_{\ub} W^{3,2}(S_{u,\ub},\gamma)$ estimates for $\slashed g$ in \eqref{eq:bdd.metric}.

\pfstep{Step~2(a): Preliminaries} We first argue as in Proposition~\ref{prop:norms.compare} to obtain that for any $p\in [1,+\infty]$,  
\begin{equation}\label{eq:appendix.norm.compare}
C^{-1} \|\xi \|_{L^p(S_{u,\ub},\gamma_{0,0})}\leq  \|\xi\|_{L^p(S_{u,\ub},\gamma)} \leq C\|\xi \|_{L^p(S_{u,\ub},\gamma_{0,0})}.
\end{equation}
(For this we first use the transport equation \eqref{eq:comparability.2} in the $u$ direction on $\Hb_0$, where $b\restriction_{\Hb_0} = 0$. We then use the transport equation \eqref{eq:comparability.1} in the $\ub$ direction for the rest of the spacetime.)

Next we derive some commutation formulas. By \eqref{metric.derivative.invar},
\begin{equation}\label{eq:appendix.comm.ub}
[\slashed{\mathcal L}_{\f{\rd}{\rd \ub}}, \nab_B] \psi_{A_1\cdots A_r}= 2\sum_{i=1}^r (\gamma^{-1})^{CD} \{ \nab_D (\Omg \chi)_{A_i B}  - 2 \nab_{(A_i} (\Omg \chi)_{B)D} \} \psi_{A_1\cdots \hat{A}_i C\cdots A_r},
\end{equation}
where $\hat{A}_i$ means that $A_i$ is removed.

Similarly, since $b\restriction_{\Hb_0} = 0$, we have by \eqref{metric.derivative.invar}
\begin{equation}\label{eq:appendix.comm.u}
[\slashed{\mathcal L}_{\f{\rd}{\rd u}}, \nab_B] \psi_{A_1\cdots A_r} \restriction_{\Hb_0}= 2\sum_{i=1}^r (\gamma^{-1})^{CD} \{ \nab_D (\Omg \chib)_{A_i B}  - 2 \nab_{(A_i} (\Omg \chib)_{B)D} \} \psi_{A_1\cdots \hat{A}_i C\cdots A_r} \restriction_{\Hb_0}.
\end{equation}

\pfstep{Step~2(b): Estimates for $\Omg$} The zeroth estimates for $\Omg$ in \eqref{eq:bdd.metric} can be obtained from \cite[Propositions~1, 3]{LR2}. The higher order derivatives for $\Omg$ follows from the third equation in \eqref{Ricci.relation} together with the estimates for $\eta$ and $\etab$ in \eqref{eq:bdd.psi}.

\pfstep{Step~2(c): Estimates\footnote{Note that the estimates we need for $\gamma$ here are slightly different from those in \cite{LR2}. In \cite{LR2}, we gave the bounds for $\gamma$ in local coordinates.} for $\gamma$} Using \eqref{metric.derivative.invar} and $\slashed{\mathcal L}_{\f{\rd}{\rd \ub}} \gamma_{0,0} = 0$, we have
\begin{equation}\label{eq:appendix.gamma.norm.transport}
\f{\rd}{\rd\ub} |\gamma - \gamma_{0,0}|^2_{\gamma_{0,0}} = 4 \Omg \langle \chi, \gamma - \gamma_{0,0}\rangle_{\gamma_{0,0}}, \quad \f{\rd}{\rd u} |\gamma - \gamma_{0,0}|^2_{\gamma_{0,0}} \restriction_{\Hb_0} = 4 \Omg \langle \chib, \gamma - \gamma_{0,0}\rangle_{\gamma_{0,0}} \restriction_{\Hb_0}.
\end{equation}
Integrating \eqref{eq:appendix.gamma.norm.transport} (first in $u$ along $\Hb_0$, then in $\ub$), and using \eqref{eq:appendix.norm.compare} together with the estimates for $\chi$, $\chib$, $\Omg$ established above, we obtain the $C^0_u C^0_{\ub} L^2(S_{u,\ub}, \gamma)$ bounds for $\gamma - \gamma_{0,0}$ in \eqref{eq:bdd.metric}.

To obtain the (first to third) derivative estimates for $\gamma - \gamma_{0,0}$, we first derive the transport equations for $\nab^i (\gamma - \gamma_{0,0})$ using \eqref{eq:appendix.comm.ub} and \eqref{eq:appendix.comm.u}, and then argue similarly as above.

\pfstep{Step~2(d): Estimates for $b$} The $C^0_u C^0_{\ub} W^{3,2}(S_{u,\ub}, \gamma)$ estimates for $b$ can be proven in a similar way as those for $\gamma - \gamma_{0,0}$, except that we instead use as transport equation the third equation in \eqref{metric.derivative.invar} and also the fact that $b\restriction_{\Hb_0} = 0$.


\pfstep{Step~3: Proof of \eqref{eq:bdd.isoperimetric} and \eqref{eq:bdd.density}: area density, isoperimetric constant and area  estimates} 
For the area density estimate in \eqref{eq:bdd.density}, note that by \eqref{metric.derivative.invar} and $b\restriction_{\Hb_0} = 0$,
$$\f{\rd}{\rd \ub}(\log \f{\det\gamma}{\det\gamma_{0,0}}) = 2\Omg \trch,\quad \f{\rd}{\rd u}(\log \f{\det\gamma}{\det\gamma_{0,0}}) \restriction_{\Hb_0} = 2\Omg \trch \restriction_{\Hb_0}.$$
Integrating (first in $u$ along $\Hb_0$ and then in $\ub$) and using the already established estimates with \eqref{eq:appendix.norm.compare}, we obtain \eqref{eq:bdd.density}.

Next, consider the isoperimetric constant bound in \eqref{eq:bdd.isoperimetric}. Take $(u,\ub) \in [0,u_*]\times [0,\ub_*]$. To bound ${\bf I}(S_{u,\ub},\gamma)$, we first bound ${\bf I}(S_{u,0},\gamma)$ in terms of ${\bf I}(S_{0,0},\gamma)$, and then bound ${\bf I}(S_{u,\ub},\gamma)$ in terms of ${\bf I}(S_{u,0},\gamma)$.

Let $U\subset S_{u,\ub}$ be a domain as in the definition in \eqref{eq:def.IPC}. We first map $U$ to $S_{u,0}$ via the flow of $\f{\rd}{\rd\ub}$ and then map it to $S_{0,0}$ via the flow of $\f{\rd}{\rd u} \restriction_{\Hb_0}$. According to \eqref{eq:def.IPC}, it then suffices to control the change of $\mathrm{Area}(U)$, $\mathrm{Area}(U^c)$ and $\mathrm{Perimeter}(\rd U)$ under these maps. The changes of $\mathrm{Area}(U)$ and $\mathrm{Area}(U^c)$ have already been bounded in \eqref{eq:bdd.density}. The change in $\mathrm{Perimeter}(\rd U)$ can be controlled in a similar manner as \cite[Lemmas~5.3, 5.4]{Chr}. This gives a uniform bound on ${\bf I}(S_{u,\ub}, \gamma)$.

Finally, the area estimate in \eqref{eq:bdd.isoperimetric} is an immediate consequence of \eqref{eq:bdd.density}.

\pfstep{Step~4: Proof of \eqref{eq:bdd.psi.trans}--\eqref{eq:bdd.psiH.psiHb.trans}: Using the equations for the Ricci coefficients} For these estimates involving $\slashed{\mathcal L}_{\f{\rd}{\rd\ub}}$ and $\slashed{\mathcal L}_{\f{\rd}{\rd u}}$, we use the null structure equations.

\pfstep{Step~4(a): \eqref{eq:bdd.psi.trans} for $\eta$ and $\etab$} For $\psi \in \{\eta,\etab\}$, by Proposition~\ref{diff.formula},
\begin{equation}\label{eq:L.as.nab}
\slashed{\mathcal L}_{\f{\rd}{\rd\ub}} \psi = \Om \nab_4 \psi + \Om \chi \cdot \psi,\quad \slashed{\mathcal L}_{\f{\rd}{\rd u}} \psi = \Om \nab_3 \psi + \Om \chib\cdot  \psi - \nab_b \psi - \nab_{\psi^\sharp} b.
\end{equation}

Let us consider the case $\psi = \eta$ ($\etab$ is similar). First, using \eqref{eq:L.as.nab} and \eqref{Ric4A},
\begin{equation}\label{eq:L.ub.eta}
\slashed{\mathcal L}_{\f{\rd}{\rd\ub}} \eta = \Om \chi \cdot \eta + \Omg\{ -\f 34 \trch (\eta-\etab) +  \div\chih -\frac 12 \nab \trch - \f 12(\eta - \etab)\cdot  \chih\}.
\end{equation}
It suffices\footnote{Note that while \eqref{eq:bdd.psi.trans} only requires the slightly weaker $C^0_u L^2_{\ub} W^{2,2}(S_{u,\ub},\gamma)$ estimates, the stronger estimate holds.} to bound each term on the RHS of \eqref{eq:L.ub.eta} in $L^2_{\ub} C^0_u W^{2,2}(S_{u,\ub},\gamma)$. Since we have obtained a uniform bound on the isoperimetric constant in Step~3, we can apply Sobolev embedding in Proposition~\ref{prop:Sobolev} (together with H\"older's inequality) so that given any two tensor fields $\phi^{(1)}$ and $\phi^{(2)}$
\begin{equation}\label{eq:Sobolev.product}
\|\phi^{(1)} \otimes \phi^{(2)} \|_{W^{2,2}(S_{u,\ub},\gamma)} \ls \|\phi^{(1)}\|_{W^{2,2}(S_{u,\ub},\gamma)}  \| \phi^{(2)} \|_{W^{2,2}(S_{u,\ub},\gamma)}.
\end{equation}
With the product estimate \eqref{eq:Sobolev.product}, it is not hard to control terms on the RHS of \eqref{eq:L.ub.eta}. To consider a couple of representative terms, we have
\begin{equation}\label{eq:L.ub.eta.est.1}
\|\Omg \div\chih\|_{L^2_{\ub} C^0_u W^{2,2}(S_{u,\ub},\gamma)} \ls \|\Omg\|_{C^0_{\ub} C^0_u W^{2,2}(S_{u,\ub},\gamma)}\|\chih\|_{L^2_{\ub} C^0_u W^{3,2}(S_{u,\ub},\gamma)} \ls 1,
\end{equation}
\begin{equation}\label{eq:L.ub.eta.est.2}
\|\Omg\eta\cdot \chih\|_{L^2_{\ub} C^0_u W^{2,2}(S_{u,\ub},\gamma)} \ls \|\Omg\|_{C^0_{\ub} C^0_u W^{2,2}(S_{u,\ub},\gamma)} \|\eta\|_{C^0_{\ub} C^0_u W^{2,2}(S_{u,\ub},\gamma)} \|\chih\|_{L^2_{\ub} C^0_u W^{2,2}(S_{u,\ub},\gamma)} \ls 1.
\end{equation}

For $\slashed{\mathcal L}_{\f{\rd}{\rd u}} \eta$, note that even though our list of null structure equations do not explicitly contain a $\nab_3\eta$ equation, combining \eqref{Ric3A} with the second and third equations in \eqref{Ricci.relation},
\begin{equation}\label{eq:L.u.eta}
\begin{split}
\slashed{\mathcal L}_{\f{\rd}{\rd u}} \eta =&\:  \Om \chib\cdot  \eta - \nab_b \eta - \nab_{\eta^\sharp} b -\Omg \{-\f 34 \trchb (\etab-\eta) + \div\chibh - \frac 12 \nab \trchb - \f 12(\etab-\eta) \cdot \chibh\} \\
&\: - 4\nab (\Omg\omb) - 2\nab (\nab_b \log\Omg).
\end{split}
\end{equation}
Using equation \eqref{eq:L.u.eta}, we can then bound the terms on its RHS in $C^0_{\ub} L^2_u W^{2,2}(S_{u,\ub},\gamma)$ in a similar way as \eqref{eq:L.ub.eta.est.1} and \eqref{eq:L.ub.eta.est.2}. Notice that because of the vector $b$, we need the third order angular derivative estimates for $b$ and $\eta$, which are both established above. (On the other hand, it is precisely because of the third order angular derivatives for $\eta$ that we only have $C^0_{\ub} L^2_u W^{2,2}(S_{u,\ub},\gamma)$, instead of $L^2_u C^0_{\ub} W^{2,2}(S_{u,\ub},\gamma)$, estimates.)

\pfstep{Step~4(b): \eqref{eq:bdd.psiH.psiHb.trans} for $\chih$, $\om$, $\chibh$ and $\omb$} As in \eqref{eq:L.as.nab} we rewrite $\slashed{\mathcal L}_{\f{\rd}{\rd\ub}}$ and $\slashed{\mathcal L}_{\f{\rd}{\rd u}}$ as $\Omg \nab_4$ and $\Omg \nab_3$ (plus lower order terms) and then use the relevant equations. 

Consider $\psi_{\Hb} \in \{ \chibh,\,\omb\}$, which obeys $\nab_4$ equations \eqref{RicAB} and \eqref{Ric34}. (Note that  $\psi_{\Hb} \in \{ \chibh,\,\omb\}$ does not obey $\nab_3$ equations. We thus only have estimates for $\slashed{\mathcal L}_{\f{\rd}{\rd \ub}} \psi_{\Hb}$.)

Rewriting the $\nab_4$ equations for $\psi_{\Hb} \in \{ \chibh,\,\omb\}$ in terms of $\slashed{\mathcal L}_{\f{\rd}{\rd \ub}} \chibh$ and $\slashed{\mathcal L}_{\f{\rd}{\rd \ub}} \omb$, it can be checked that all the terms on the RHS can be bounded in $L^2_u L^2_{\ub} W^{2,2}(S_{u,\ub},\gamma)$ using the estimates we have derived. Consider for instance the following term on the RHS of \eqref{RicAB}, which can be bounded using \eqref{eq:Sobolev.product} and estimates we have established:
$$\|\Omg \om \chibh\|_{L^2_u L^2_{\ub} W^{2,2}(S_{u,\ub},\gamma)} \ls \|\Omg\|_{C^0_{\ub} C^0_u W^{2,2}(S_{u,\ub},\gamma)} \|\om\|_{L^2_{\ub} C^0_u W^{2,2}(S_{u,\ub},\gamma)} \|\chibh\|_{C^0_{\ub} L^2_u W^{2,2}(S_{u,\ub},\gamma)} \ls 1.$$
This term in particular limits the regularity in $u$ and $\ub$ that can be proven for $\slashed{\mathcal L}_{\f{\rd}{\rd \ub}}\psi_{\Hb}$. Other terms are either similar or easier; we omit the details. 

The terms $\slashed{\mathcal L}_{\f{\rd}{\rd u}}\psi_H$ for $\psi_H \in \{\chih,\,\om\}$ are similar except for having to handle terms related to $b$.

\pfstep{Step~4(c): \eqref{eq:bdd.trch.trans} and \eqref{eq:bdd.trchb.trans} for $\trch$ and $\trchb$} Finally, we argue similarly for $\trch$ and $\trchb$. Note that both $\trch$ and $\trchb$ obey both $\nab_3$ and $\nab_4$ equations; see \eqref{Ric44}, \eqref{Ric33}, \eqref{trRicAB} and \eqref{trRicAB.1}. As above, we rewrite these equations in terms of $\slashed{\mathcal L}_{\f{\rd}{\rd u}}$ and $\slashed{\mathcal L}_{\f{\rd}{\rd \ub}}$ and then estimate the RHSs.

Let us just point out the terms that limit the regularity. For $\slashed{\mathcal L}_{\f{\rd}{\rd \ub}} \trch$, we have the term
$$\|\Omg |\chih|^2_\gamma\|_{C^0_u L^1_{\ub} W^{2,2}(S_{u,\ub},\gamma)} \ls \|\Omg\|_{C^0_{\ub} C^0_u W^{2,2}(S_{u,\ub},\gamma)} \|\chih\|_{C^0_u L^2_{\ub} W^{2,2,}(S_{u,\ub},\gamma)}^2 \ls 1;$$
and for $\slashed{\mathcal L}_{\f{\rd}{\rd u}} \trch$, we have the term
$$\|\Omg \omb \trch \|_{L^2_u C^0_{\ub} W^{2,2}(S_{u,\ub},\gamma)} \ls \|\Omg\|_{C^0_{\ub} C^0_u W^{2,2}(S_{u,\ub},\gamma)} \|\omb\|_{L^2_u C^0_{\ub}W^{2,2,}(S_{u,\ub},\gamma)} \|\trch\|_{C^0_{\ub} C^0_u W^{2,2}(S_{u,u},\gamma)} \ls 1.$$
These terms are responsible for the choice of function spaces for $\slashed{\mathcal L}_{\f{\rd}{\rd \ub}} \trch$ and $\slashed{\mathcal L}_{\f{\rd}{\rd u}} \trch$ in \eqref{eq:bdd.trch.trans}. The estimates \eqref{eq:bdd.trchb.trans} are similar.

\pfstep{Step~5: Proof of \eqref{eq:bdd.metric} II: Estimates for $u$, $\ub$ derivatives of the metric components} We first consider $\gamma$. The first equation in \eqref{metric.derivative.invar} together with estimates we have established above gives the necessary bound for $\slashed {\mathcal L}_{\f{\rd}{\rd\ub}}\gamma$. For $\slashed {\mathcal L}_{\f{\rd}{\rd u}}\gamma$, note that $\slashed {\mathcal L}_{\f{\rd}{\rd u}}\gamma - \slashed {\mathcal L}_{\f{\rd}{\rd u}+b^A \f{\rd}{\rd\th^A}}\gamma$ can be expressed as angular covariant derivatives of $b$. Thus, after using the estimates that we have established above, the desired estimate in \eqref{eq:bdd.metric} for $\slashed {\mathcal L}_{\f{\rd}{\rd u}}\gamma$ follows from (second equation in) \eqref{metric.derivative.invar}.

The estimates for $\slashed {\mathcal L}_{\f{\rd}{\rd\ub}}\log\Omg$ and $\slashed {\mathcal L}_{\f{\rd}{\rd u}}\log\Omg$ in \eqref{eq:bdd.metric} are similar (as those for $\gamma$) except that we use instead the first two equations of \eqref{Ricci.relation}.

Finally, we turn to $b$. The estimates for $\slashed{\mathcal L}_{\f{\rd}{\rd \ub}} b$ in \eqref{eq:bdd.metric} follow directly from \eqref{metric.derivative.invar}. For $\slashed{\mathcal L}_{\f{\rd}{\rd u}} b$, we derive from \eqref{metric.derivative.invar} that 
\begin{equation}\label{eq:transport.for.du.b}
\slashed {\mathcal L}_{\f{\rd}{\rd \ub}} (\slashed{\mathcal L}_{\f{\rd}{\rd u}} b)=-2\slashed{\mathcal L}_{\f{\rd}{\rd u}} \{\Omega^2(\eta^\sharp-\etab^\sharp)\}.
\end{equation} 
Using \eqref{metric.derivative.invar}, \eqref{Ricci.relation}, \eqref{eq:L.u.eta} (and analogous equation for $\mathcal L_{\f{\rd}{\rd u}} \etab$), and the established estimates, it follows that 
\begin{equation}\label{eq:transport.for.du.b.1}
\|\mbox{RHS of \eqref{eq:transport.for.du.b}}\|_{C^0_{\ub} L^2_u W^{2,2}(S_{u,\ub},\gamma)} \ls 1.
\end{equation}
Finally, integrating \eqref{eq:transport.for.du.b} using the bound \eqref{eq:transport.for.du.b.1} yields the estimate for $\slashed{\mathcal L}_{\f{\rd}{\rd u}} b$ in \eqref{eq:bdd.metric}. \qedhere

\end{proof}

\section{Proof of Lemma~\ref{lem:compensated.compactness}}\label{app:CC}



In this appendix we prove Lemma~\ref{lem:compensated.compactness}, which is a compensated compactness lemma. We will first give a high-level proof, assuming a few results that we will later prove in Propositions~\ref{fL}--\ref{fg.HH}.
\begin{proof}[Proof of Lemma~\ref{lem:compensated.compactness}] It will be convenient to introduce the following notations. Let $(x^1,\,x^2,\,x^3,\,x^4) = (u,\,\ub,\,y^1,\,y^2)$. Denote also $\ud x = \ud x^1\,\ud x^2\,\ud x^3\,\ud x^4$. Denote the corresponding Fourier variables by $(\xi_1, \,\xi_2,\,\xi_3,\,\xi_4)$ and $\ud \xi = \ud \xi_1\,\ud \xi_2\,\ud \xi_3\,\ud \xi_4$.

\pfstep{Step~1: A simple reduction} We first extend $f_n$ and $h_n$ so that they are compacted supported in a fixed ball $\mathcal U\subset \mathbb R^4$ and satisfy the estimates
\begin{equation}\label{fn.bd}
\sup_n \sum_{i\leq 1,\,j\leq 1,\,k\leq 1} \int_U \left((\f{\rd}{\rd x^1})^{i}(\f{\rd}{\rd x^3})^{j}(\f{\rd}{\rd x^4})^{k} f_n\right)^2\,\ud x \ls C_0
\end{equation}
and
\begin{equation}\label{gn.bd}
\sup_n \sum_{i\leq 1,\,j\leq 1,\,k\leq 1} \int_U \left((\f{\rd}{\rd x^2})^{i}(\f{\rd}{\rd x^3})^{j}(\f{\rd}{\rd x^4})^{k} h_n\right)^2 \,\ud x \ls C_0.
\end{equation}

By \eqref{fn.bd}, \eqref{gn.bd} and one-dimensional Sobolev embedding, we have 
$$\|f_n h_n\|_{L^2(U)}\lesssim \|f_n\|_{L^\infty_{x^1}L^2_{x^2}L^2_{x^3}L^2_{x^4}} \|h_n\|_{L^2_{x^1}L^\infty_{x^2}L^\infty_{x^3}L^\infty_{x^4}}\lesssim _{C_0} 1.$$
Therefore, by a simple density argument, it suffices to show that after passing to a subsequence,
$$\int_U (f_{n_k} h_{n_k}- f_{\infty} h_{\infty})\psi \, \ud x\to 0 $$
as $k\to \infty$ for all $\psi \in L^\infty\cap L^2$.

For the rest of the proof fix a function $\psi\in L^\infty\cap L^2$ and fix $\ep>0$. 

\pfstep{Step~2: Frequency decomposition} Denote by either $\hat{ }$ or $\mathcal F$ the Fourier transform on $\mathbb R^4$. First, notice that since $\hat{\psi}$ is in $L^2(\mathbb R^4)$, there exists some $C_1>1$ such that $\hat{\psi}$ satisfies
\begin{equation}\label{eq:small.high.frequency}
\int_{|\xi|\geq C_1} |\hat{\psi}|^2 \ud \xi \leq \ep.
\end{equation}
Let\footnote{We note that in the appendix, $\chi$, $\eta$ will be used as real-valued functions and are not to be confused with the notation for the Ricci coefficient in the main text!} $\chi:\mathbb R\to \mathbb R_{\geq 0}$ be a non-negative, smooth function which is compactly support in $[-2,2]$ and is identically $1$ in $[-1,1]$. For each $n$, we then define\footnote{We note that of course $f_{n,H}$, $f_{n,L}$, $g_{n,H}$, $g_{n,L}$ are no longer supported in $U$.}
$$f_n=f_{n,H,1}+f_{n,H,2}+f_{n,L},\quad h_n=h_{n,H,1}+h_{n,H,2}+h_{n,L}$$
where 
$$\mathcal F(f_{n,L}):=\hat{f}_n(\xi) \chi(\f{|\xi|}{2C_1}),\quad \mathcal F(f_{n,H,1}):=\hat{f}_n(\xi) \left(1-\chi(\f{|\xi|}{2C_1})\right)\chi(\f{100 C_1 |\xi_2|}{|\xi_1|}),$$
and
$$\mathcal F(h_{n,L}):=\hat{h}_n(\xi) \chi(\f{|\xi|}{2C_1}),\quad \mathcal F(h_{n,H,2}):=\hat{h}_n(\xi) \left(1-\chi(\f{|\xi|}{2C_1})\right)\chi(\f{100 C_1 |\xi_1|}{|\xi_2|}).$$
We note explicitly that the definitions for $f_{n,H,1}$ and $h_{n,H,1}$ are \emph{different}. Moreover, in the support of $\mathcal F(f_{n,H,1})$, $50 C_1|\xi_2|\leq |\xi_1|$ and $|\xi|\geq 4C_1$; while in the support of $\mathcal F(h_{n,H,2})$, $50 C_1|\xi_1|\leq |\xi_2|$ and $|\xi|\geq 4C_1$.
Similarly, we define
$$f_\infty=f_{\infty,H,1}+f_{\infty,H,2}+f_{\infty,L},\quad h_\infty= h_{\infty,H,1} + h_{\infty,H,2} + h_{\infty,L}$$
where 
$$\mathcal F(f_{\infty,L}):=\hat{f}_\infty(\xi) \chi(\f{|\xi|}{2C_1}),\quad \mathcal F(f_{\infty,H,1}):=\hat{f}_\infty(\xi) \left(1-\chi(\f{|\xi|}{2C_1})\right)\chi(\f{100 C_1 |\xi_2|}{|\xi_1|}),$$
and
$$\mathcal F(h_{\infty,L}):=\hat{h}_\infty(\xi) \chi(\f{|\xi|}{2C_1}),\quad \mathcal F(h_{\infty,H,2}):=\hat{h}_\infty(\xi) \left(1-\chi(\f{|\xi|}{2C_1})\right)\chi(\f{100 C_1 |\xi_1|}{|\xi_2|}).$$

Let us briefly explain this decomposition. For say $f_n$, we decompose into a piece $f_{n,H,1}$ where the $\xi_1$ frequency dominates and a piece $f_{n,H,2}$ where the $\xi_2$ frequency dominates. We add in an extra low frequency piece $f_{n,L}$ in the decomposition to ensure in particular that all the Fourier multipliers are smooth. We also make a similar decomposition for $g_n$. 

\pfstep{Step~3: Completion of the argument} It is clear that we have the weak convergences $f_{n,L} \rightharpoonup f_{\infty,L}$, $f_{n,H,1} \rightharpoonup f_{\infty,H,1}$, $f_{n,H,1} \rightharpoonup f_{\infty,H,1}$, $h_{n,L} \rightharpoonup h_{\infty,L}$, $h_{n,H,1} \rightharpoonup h_{\infty,H,1}$, $h_{n,H,1} \rightharpoonup h_{\infty,H,1}$ in $L^2$. More importantly, by Propositions~\ref{fL}, \ref{fH2} and \ref{gstrong} below, after passing to subsequences $f_{n_k,L}$, $f_{n_k,H,2}$, $h_{n_k,L}$, $h_{n_k,H,1}$ in fact converges in the $L^2$ norm. Therefore,
\begin{equation*}
\begin{split}
&\int_{\mathbb R^4} \left(f_{n_k} (h_{{n_k},H,1}+h_{{n_k},L}) + (f_{{n_k},H,2}+f_{{n_k},L}) h_{n_k} \right)\psi \,\ud x\\
\to &\int_{\mathbb R^4} \left(f_\infty (h_{\infty,H,1} + h_{\infty,L}) + (f_{\infty,H,2}+f_{\infty,L}) h_\infty \right)\psi \,\ud x.
\end{split}
\end{equation*}
In order to conclude the proof, it therefore suffices to show that
\begin{equation}\label{eq:cc.hf.goal}
\int_{\mathbb R^4} f_{n,H,1}h_{n,H,2}\psi \,\ud x \lesssim_{C_0} \ep,\quad\int_{\mathbb R^4} f_{\infty,H,1}h_{\infty,H,2}\psi \,\ud x \lesssim_{C_0} \ep.
\end{equation}
The point is that is that the cutoffs have been chosen such that $\mathrm{supp}(\mathcal F(f_{n,H,1} h_{n,H,2}))\subset \{|\xi|\geq C_1\}$ and $\mathrm{supp}(\mathcal F(f_{\infty,H,1}h_{\infty,H,2}))\subset \{|\xi|\geq C_1\}$ (to be proven in Proposition \ref{fg.HH}). Therefore, using \eqref{eq:small.high.frequency},\begin{equation*}
\begin{split}
\int_{\mathbb R^4} f_{n,H,1} h_{n,H,2}\psi \,dx
=&\:\int_{\mathbb R^4} \mathcal F(f_{n,H,1} h_{n,H,2})(\xi) \hat\psi(\xi) \,d\xi 
= \int_{|\xi|\geq C_1} \mathcal F(f_{n,H,1} h_{n,H,2})(\xi) \hat\psi(\xi) \,d\xi\\
\leq &\:\epsilon \|\mathcal F(f_{n,H,1} h_{n,H,2})\|_{L^2_\xi(\mathbb R^4)}
\leq  \epsilon \|f_{n,H,1} h_{n,H,2}\|_{L^2_x(\mathbb R^4)}\\
\leq &\:\epsilon \|f_{n,H,1}\|_{L^2_{x^1}L^\infty_{x^2}L^4_{x^3}L^4_{x^4}} \|h_{n,H,2}\|_{L^\infty_{x^1}L^2_{x^2}L^4_{x^3}L^4_{x^4}}\\
\leq &\:C\epsilon \|f_{n,H,1}\|_{L^2_{x^1}H^1_{x^2}H^1_{x^3}H^1_{x^4}} \|h_{n,H,2}\|_{H^1_{x^1}L^2_{x^2}H^1_{x^3}H^1_{x^4}}\\
\leq &\:C\epsilon \|f_{n}\|_{L^2_{x^1}H^1_{x^2}H^1_{x^3}H^1_{x^4}} \|h_{n}\|_{H^1_{x^1}L^2_{x^2}H^1_{x^3}H^1_{x^4}},
\end{split}
\end{equation*}
where in the second to last step, we have used one-dimensional Sobolev embedding theorem in three of the directions and in the last step, we have used the boundedness of the Fourier multipliers. By \eqref{fn.bd} and \eqref{gn.bd}, we have
$$\|f_{n}\|_{L^2_{x^1}H^1_{x^2}H^1_{x^3}H^1_{x^4}}+\|h_{n}\|_{H^1_{x^1}L^2_{x^2}H^1_{x^3}H^1_{x^4}}\lesssim_{C_0} 1.$$
Therefore, 
$$\int_{\mathbb R^4} f_{n,H,1}h_{n,H,2}\psi \,\ud x\lesssim_{C_0} \ep.$$
Similarly, 
$$\int_{\mathbb R^4} f_{\infty,H,1} h_{\infty,H,2}\psi \, \ud x\lesssim_{C_0} \ep$$
so that we have proven \eqref{eq:cc.hf.goal}. This concludes the proof. \qedhere
\end{proof}
It now remains to prove the compactness results and the support properties that we have used in the proof of Lemma~\ref{lem:compensated.compactness}. First, the low frequency part of $f_n$ converges in $L^2$ norm:
\begin{proposition}\label{fL}
Let $f_{n,L}$, $f_{\infty,L}$ be as in the proof of Lemma~\ref{lem:compensated.compactness}. Then, after passing to a subsequence,
$$\|f_{n_k,L}-f_{\infty,L}\|_{L^2(\mathbb R^4)}\to 0$$
as $k\to +\infty$.
\end{proposition}
\begin{proof}
By a standard extension of Rellich's theorem to weighted $H^1$ spaces in non-compact domain, it suffices to show that
$$\sup_n\sum_{|\alpha|\leq 1,\,|\beta|\leq 1}\|x^{\alpha}\rd_x^\beta f_{n,L}\|_{L^2(\mathbb R^4)}<\infty.$$
(Here, $\alpha=(\alpha_1,\alpha_2,\alpha_3,\alpha_4)$ and $\beta=(\beta_1,\beta_2,\beta_3,\beta_4)$ are multi-indices.)
This obviously holds since $f_n$ have uniform $L^2$ norm and $\chi(\f{|\xi|}{2C_1})$ is smooth and compactly supported.
\end{proof}
We also have norm convergence for $f_{n,H,2}$. The key point is to notice that the cutoff in frequency guarantees that the $|\xi_1|$ frequency is controlled by the $|\xi_2|$ frequency. This then allows us to use the assumption \eqref{fn.bd}.
\begin{proposition}\label{fH2}
Let $f_{n,H,2}$, $f_{\infty,H,2}$ be as in the proof of Lemma~\ref{lem:compensated.compactness}. Then, after passing to a subsequence,
$$\|f_{n_k,H,2}-f_{\infty,H,2}\|_{L^2(\mathbb R^4)}\to 0$$
as $k \to +\infty$.
\end{proposition}
\begin{proof}
As in the proof of Proposition \ref{fL}, it suffices to show that
$$\sup_n\sum_{|\alpha|\leq 1,\,|\beta|\leq 1}\|x^{\alpha}\rd_x^\beta f_{n,H,2}\|_{L^2(\mathbb R^4)}$$
is bounded.

By definition,
$$\mathcal F(f_{n,H,2}):=\hat{f}(\xi) \left(1-\chi(\f{|\xi|}{2C_1})\right)\left(1-\chi(\f{100 C_1 |\xi_2|}{|\xi_1|})\right).$$
To simplify notation, we denote
$$\tilde{\chi}(\xi):=\left(1-\chi(\f{|\xi|}{2C_1})\right)\left(1-\chi(\f{100 C_1 |\xi_2|}{|\xi_1|})\right).$$

By Plancherel's theorem, 
$$\| x^\alp \rd_x^\beta f_{n,H,2}(x)\|_{L^2_x} \ls \| \rd_\xi^\alp (\xi^\bt \tilde{\chi}(\xi) \hat{f}(\xi)) \|_{L^2_\xi}.$$
Using the product rule and the fact that $100 C_1 |\xi_2|\geq |\xi_1|$ in the support of $\tilde{\chi}$, we have
$$\| \rd_\xi^\alp (\xi^\bt \tilde{\chi}(\xi) \hat{f}(\xi)) \|_{L^2_\xi} \ls_{C_1} \|(1+|\xi_2| + |\xi_3| + |\xi_4|)  \hat{f}(\xi) \|_{L^2_\xi} + \|(1+|\xi_2| + |\xi_3| + |\xi_4|) (\rd_\xi^\alp \hat{f})(\xi) \|_{L^2_\xi}.$$
Therefore, combining the above estimates and using Plancherel's theorem again, we obtain
\begin{equation*}
\begin{split}
&\: \| x^\alp \rd_x^\beta f_{n,H,2}(x)\|_{L^2_x}\\
 \ls_{C_1} &\: \| f_n\|_{L^2_x} + \|\rd_{x^2} f_n\|_{L^2} + \|\rd_{x^3} f_n\|_{L^2} + \|\rd_{x^4} f_n\|_{L^2} \\
 &\: \quad + \|x^\alp f_n\|_{L^2_x} + \|\rd_{x^2} (x^\alp f_n)\|_{L^2} + \|\rd_{x^3} (x^\alp f_n)\|_{L^2} + \|\rd_{x^4} (x^\alp f_n) \|_{L^2} \\
 \ls_{C_1} &\: \| f_n\|_{L^2_x} + \|\rd_{x^2} f_n\|_{L^2} + \|\rd_{x^3} f_n\|_{L^2} + \|\rd_{x^4} f_n \|_{L^2} \ls_{C_1} C_0, 
\end{split}
\end{equation*}
where in the last line we use $\mathrm{supp}(f_n) \subset U$ and \eqref{fn.bd}. \qedhere
\end{proof}
Similar convergence statements can also be proved for $h_{n,L}$ and $h_{n,H,1}$. The proof is similar to that of Propositions \ref{fL} and \ref{fH2} and is omitted.
\begin{proposition}\label{gstrong}
Let $h_{n,L}$, $h_{\infty,L}$, $h_{n,H,1}$, $h_{\infty,H,1}$ be as in the proof of Lemma~\ref{lem:compensated.compactness}. Then, after passing to a subsequence,
$$\|h_{n_k,L}-h_{\infty,L}\|_{L^2(\mathbb R^4)}\to 0,\quad\|h_{n_k,H,1}-h_{\infty,H,1}\|_{L^2(\mathbb R^4)}\to 0$$
as $k \to +\infty$.
\end{proposition}
Finally, we conclude with the frequency support properties of the products $f_{n,H,1} h_{n,H,2}$ and $f_{\infty,H,1} h_{\infty,H,2}$:
\begin{proposition}\label{fg.HH}
$$\mathrm{supp} (\mathcal F(f_{n,H,1} h_{n,H,2}))\subset \{|\xi|\geq C_1\}\quad\mbox{for all } n$$
and
$$\mathrm{supp}(\mathcal F(f_{\infty,H,1} h_{\infty,H,2}))\subset \{|\xi|\geq C_1\}.$$
\end{proposition}
\begin{proof}
We claim that if $\xi^{(1)} = (\xi^{(1)}_1,\, \xi^{(1)}_2,\,\xi^{(1)}_3,\,\xi^{(1)}_4)$ and $\xi^{(2)} = (\xi^{(2)}_1,\, \xi^{(2)}_2,\,\xi^{(2)}_3,\,\xi^{(2)}_4)$ satisfy
\begin{equation}\label{xi1.assumption}
|\xi^{(1)}_1|\geq 50 C_1 |\xi^{(1)}_2|
\end{equation}
and 
\begin{equation}\label{xi2.assumption}
|\xi^{(2)}_2|\geq 50 C_1 |\xi^{(2)} _1|,\quad |\xi^{(2)}|\geq 4C_1,
\end{equation} 
then $|\xi^{(1)}-\xi^{(2)}|\geq C_1$. This then implies the conclusion of the proposition.

To prove the claim, assume for the sake of contradiction that it is not true, i.e.~assume $|\xi^{(1)}-\xi^{(2)}|< C_1$.
This immediately implies that
$$|\xi^{(2)}_1|\geq |\xi^{(1)}_1|-C_1,\quad |\xi^{(2)}_2|\leq |\xi^{(1)}_2|+C_1.$$
Combining this with \eqref{xi1.assumption} and \eqref{xi2.assumption}, we
get
$$|\xi^{(2)}_2|\leq \f{1}{50 C_1}|\xi^{(1)}_1|+C_1\leq \f{1}{2500 C_1^2}|\xi^{(2)}_2|+2C_1.$$
Rearranging, this gives
\begin{equation}\label{xi22}
|\xi^{(2)}_2|\leq \f{2C_1}{1-\f{1}{2500C_1^2}}.
\end{equation}
Returning \eqref{xi2.assumption} again, we obtain
\begin{equation}\label{xi21}
|\xi^{(2)}_1|\leq \f{1}{25(1-\f{1}{2500C_1})}.
\end{equation}
Combining \eqref{xi22} and \eqref{xi21}, and using the fact that $C_1>1$, we get 
$$|\xi^{(2)}| <4C_1,$$
which contradicts \eqref{xi2.assumption}. \qedhere
\end{proof}
\bibliographystyle{DLplain}
\bibliography{HFlimit}

\begin{thebibliography}{10}

\bibitem{An.AH}
X.~An.
\newblock Emergence of apparent horizon in gravitational collapse.
\newblock {\em arXiv:1703.00118, preprint}, 2017.

\bibitem{AnAthanasiou}
X.~An and N.~Athanasiou.
\newblock A scale-critical trapped surface formation criterion for the
  {E}instein--{M}axwell system.
\newblock {\em arXiv:2005.12699, preprint}, 2020.

\bibitem{AL}
X.~An and J.~Luk.
\newblock Trapped surfaces in vacuum arising dynamically from mild incoming
  radiation.
\newblock {\em Adv. Theor. Math. Phys.}, 21(1):1--120, 2017.

\bibitem{Barrabes.shell}
C.~Barrab\`es.
\newblock Prolate collapse of string loops and domain walls.
\newblock {\em Classical Quantum Gravity}, 8(10):L199--L204, 1991.

\bibitem{BH.book}
C.~Barrab\`es and P.~A. Hogan.
\newblock {\em Singular null hypersurfaces in general relativity}.
\newblock World Scientific Publishing Co., Inc., River Edge, NJ, 2003.
\newblock Light-like signals from violent astrophysical events.

\bibitem{cBwIpL91}
C.~Barrab\`es, W.~Israel, and P.~S. Letelier.
\newblock Analytic models of nonspherical collapse, cosmic censorship and the
  hoop conjecture.
\newblock {\em Phys. Lett. A}, 160(1):41--44, 1991.

\bibitem{cBwIeP90}
C.~Barrabes, W.~Israel, and E.~Poisson.
\newblock Collision of light-like shells and mass inflation in rotating black
  holes.
\newblock {\em Classical and Quantum Gravity}, 7(12):L273--L278, dec 1990.

\bibitem{Burnett}
G.~A. Burnett.
\newblock The high-frequency limit in general relativity.
\newblock {\em J. Math. Phys.}, 30(1):90--96, 1989.

\bibitem{CB.HF}
Y.~Choquet-Bruhat.
\newblock Construction de solutions radiatives approch\'{e}es des \'{e}quations
  d'{E}instein.
\newblock {\em Comm. Math. Phys.}, 12:16--35, 1969.

\bibitem{Chr}
D.~Christodoulou.
\newblock {\em The formation of black holes in general relativity}.
\newblock EMS Monographs in Mathematics. European Mathematical Society (EMS),
  Z\"urich, 2009.

\bibitem{CK}
D.~Christodoulou and S.~Klainerman.
\newblock {\em The global nonlinear stability of the {M}inkowski space},
  volume~41 of {\em Princeton Mathematical Series}.
\newblock Princeton University Press, Princeton, NJ, 1993.

\bibitem{DL}
M.~Dafermos and J.~Luk.
\newblock The interior of dynamical vacuum black holes {I}: The
  ${C}^0$-stability of the {K}err {C}auchy horizon.
\newblock {\em arXiv:1710.01722, preprint}, 2017.

\bibitem{tDgtH85}
T.~Dray and G.~'t~Hooft.
\newblock The effect of spherical shells of matter on the {S}chwarzschild black
  hole.
\newblock {\em Comm. Math. Phys.}, 99(4):613--625, 1985.

\bibitem{tDgtH86}
T.~Dray and G.~'t~Hooft.
\newblock The gravitational effect of colliding planar shells of matter.
\newblock {\em Classical Quantum Gravity}, 3(5):825--840, 1986.

\bibitem{Evans}
L.~C. Evans and R.~F. Gariepy.
\newblock {\em Measure theory and fine properties of functions}.
\newblock Textbooks in Mathematics. CRC Press, Boca Raton, FL, revised edition,
  2015.

\bibitem{Gibbons.shell}
G.~W. Gibbons.
\newblock Collapsing shells and the isoperimetric inequality for black holes.
\newblock {\em Classical Quantum Gravity}, 14(10):2905--2915, 1997.

\bibitem{GW1}
S.~R. Green and R.~M. Wald.
\newblock New framework for analyzing the effects of small scale
  inhomogeneities in cosmology.
\newblock {\em Phys. Rev. D}, 83:084020, Apr 2011.

\bibitem{GW2}
S.~R. Green and R.~M. Wald.
\newblock Examples of backreaction of small-scale inhomogeneities in cosmology.
\newblock {\em Phys. Rev. D}, 87:124037, Jun 2013.

\bibitem{GW.FLRW}
S.~R. Green and R.~M. Wald.
\newblock How well is our universe described by an {FLRW} model?
\newblock {\em Classical Quantum Gravity}, 31(23):234003, 16, 2014.

\bibitem{GW.simple}
S.~R. Green and R.~M. Wald.
\newblock A simple, heuristic derivation of our `no backreaction' results.
\newblock {\em Classical Quantum Gravity}, 33(12):125027, 11, 2016.

\bibitem{Hawking.shell}
S.~Hawking.
\newblock Gravitational radiation from collapsing cosmic string loops.
\newblock {\em Physics Letters B}, 246(1):36 -- 38, 1990.

\bibitem{pHtF93}
P.~A. Hogan and T.~Futamase.
\newblock Some high-frequency spherical gravity waves.
\newblock {\em J. Math. Phys.}, 34(1):154--169, 1993.

\bibitem{HL.HF}
C.~Huneau and J.~Luk.
\newblock High-frequency backreaction for the {E}instein equations under
  polarized {$\Bbb{U}(1)$}-symmetry.
\newblock {\em Duke Math. J.}, 167(18):3315--3402, 2018.

\bibitem{HL.Burnett}
C.~Huneau and J.~Luk.
\newblock Trilinear compensated compactness and {B}urnett's conjecture in
  general relativity.
\newblock {\em arXiv:1907.10743, preprint}, 2019.

\bibitem{I1}
R.~A. Isaacson.
\newblock Gravitational radiation in the limit of high frequency. i. the linear
  approximation and geometrical optics.
\newblock {\em Phys. Rev.}, 166:1263--1271, Feb 1968.

\bibitem{I2}
R.~A. Isaacson.
\newblock Gravitational radiation in the limit of high frequency. ii. nonlinear
  terms and the effective stress tensor.
\newblock {\em Phys. Rev.}, 166:1272--1280, Feb 1968.

\bibitem{Jaffe}
E.~Y. Jaffe.
\newblock Asymptotic description of the formation of black holes from
  short-pulse data.
\newblock {\em arXiv:2003.05985, preprint}, 2020.

\bibitem{KLR}
S.~Klainerman, J.~Luk, and I.~Rodnianski.
\newblock A fully anisotropic mechanism for formation of trapped surfaces in
  vacuum.
\newblock {\em Invent. Math.}, 198(1):1--26, 2014.

\bibitem{KlRo.scarred}
S.~Klainerman and I.~Rodnianski.
\newblock On emerging scarred surfaces for the {E}instein vacuum equations.
\newblock {\em Discrete Contin. Dyn. Syst.}, 28(3):1007--1031, 2010.

\bibitem{KlRo.trapped}
S.~Klainerman and I.~Rodnianski.
\newblock On the formation of trapped surfaces.
\newblock {\em Acta Math.}, 208(2):211--333, 2012.

\bibitem{L21}
S.~Klainerman, I.~Rodnianski, and J.~Szeftel.
\newblock The bounded {$L^2$} curvature conjecture.
\newblock {\em Invent. Math.}, 202(1):91--216, 2015.

\bibitem{Le}
P.~Le.
\newblock The intersection of a hyperplane with a lightcone in the {M}inkowski
  spacetime.
\newblock {\em J. Differential Geom.}, 109(3):497--507, 2018.

\bibitem{Li.Schwarzschild}
J.~Li.
\newblock On the focusing effect of gravitational waves.
\newblock {\em Ann. Henri Poincar\'{e}}, 17(7):1909--1936, 2016.

\bibitem{LiLiu}
J.~Li and J.~Liu.
\newblock Instability of spherical naked singularities of a scalar field under
  gravitational perturbations.
\newblock {\em arXiv:1710.02422, preprint}, 2017.

\bibitem{LiMei}
J.~Li and H.~Mei.
\newblock A {C}onstruction of {C}ollapsing {S}pacetimes in {V}acuum.
\newblock {\em Comm. Math. Phys.}, 378(2):1343--1389, 2020.

\bibitem{LiYu.glue}
J.~Li and P.~Yu.
\newblock Construction of {C}auchy data of vacuum {E}instein field equations
  evolving to black holes.
\newblock {\em Ann. of Math. (2)}, 181(2):699--768, 2015.

\bibitem{Lott1}
J.~Lott.
\newblock Backreaction in the future behavior of an expanding vacuum spacetime.
\newblock {\em Classical Quantum Gravity}, 35(3):035010, 10, 2018.

\bibitem{Lott3}
J.~Lott.
\newblock Collapsing in the {E}instein flow.
\newblock {\em Ann. Henri Poincar\'{e}}, 19(8):2245--2296, 2018.

\bibitem{Lott2}
J.~Lott.
\newblock Corrigendum: {B}ackreaction in the future behavior of an expanding
  vacuum spacetime (2018 {\it {c}lass. {q}uantum {g}rav.} 35 035010) [
  {MR}3755966].
\newblock {\em Classical Quantum Gravity}, 35(8):089501, 1, 2018.

\bibitem{LukWeakNull}
J.~Luk.
\newblock Weak null singularities in general relativity.
\newblock {\em J. Amer. Math. Soc.}, 31(1):1--63, 2018.

\bibitem{LR}
J.~Luk and I.~Rodnianski.
\newblock Local propagation of impulsive gravitational waves.
\newblock {\em Comm. Pure Appl. Math.}, 68(4):511--624, 2015.

\bibitem{LR2}
J.~Luk and I.~Rodnianski.
\newblock Nonlinear interaction of impulsive gravitational waves for the vacuum
  {E}instein equations.
\newblock {\em Camb. J. Math.}, 5(4):435--570, 2017.

\bibitem{LVdM1}
J.~Luk and M.~Van~de Moortel.
\newblock Nonlinear interaction of three impulsive gravitational waves {I}:
  Main result and the geometric estimates.
\newblock {\em preprint}, 2020.

\bibitem{LVdM2}
J.~Luk and M.~{Van de Moortel}.
\newblock Nonlinear interaction of three impulsive gravitational waves {II}:
  The wave estimates.
\newblock {\em preprint}, 2020.

\bibitem{MacCallumTaub}
M.~A.~H. MacCallum and A.~H. Taub.
\newblock The averaged {L}agrangian and high-frequency gravitational waves.
\newblock {\em Comm. Math. Phys.}, 30:153--169, 1973.

\bibitem{Hagen}
H.~M\"{u}ller~zum Hagen.
\newblock Characteristic initial value problem for hyperbolic systems of second
  order differential equations.
\newblock {\em Ann. Inst. H. Poincar\'{e} Phys. Th\'{e}or.}, 53(2):159--216,
  1990.

\bibitem{Penrose.massless}
R.~Penrose.
\newblock General-relativistic energy flux and elementary optics.
\newblock In {\em Perspectives in {G}eometry ({E}ssays in {H}onor of {V}.
  {H}lavat\'{y})}, pages 259--274. Indiana Univ. Press, Bloomington, Ind.,
  1966.

\bibitem{Penrose.shell}
R.~Penrose.
\newblock Naked singularities.
\newblock {\em Annals of the New York Academy of Sciences}, 224(1):125--134,
  1973.

\bibitem{Penrose2018}
R.~Penrose.
\newblock The big bang and its dark-matter content: Whence, whither, and
  wherefore.
\newblock {\em Foundations of Physics}, 48(10):1177--1190, Oct 2018.

\bibitem{iR85}
I.~H. Redmount.
\newblock Blue-sheet instability of {S}chwarzschild wormholes.
\newblock {\em Progr. Theoret. Phys.}, 73(6):1401--1426, 1985.

\bibitem{Synge}
J.~L. Synge.
\newblock A model in general relativity for the instantaneous transformation of
  a massive particle into radiation.
\newblock {\em Proc. Roy. Irish Acad. Sect. A}, 59:1--13, 1957.

\bibitem{SC.standing}
S.~J. Szybka and A.~Cie{\'s}lik.
\newblock Standing waves in general relativity.
\newblock {\em Phys. Rev. D}, 100:064025, Sep 2019.

\bibitem{SGWK}
S.~J. Szybka, K.~G{\l}{\'o}d, M.~J. Wyr{\c e}bowski, and A.~Konieczny.
\newblock Inhomogeneity effect in wainwright-marshman space-times.
\newblock {\em Phys. Rev. D}, 89:044033, Feb 2014.

\bibitem{SW}
S.~J. Szybka and M.~J. Wyr{\c e}bowski.
\newblock Backreaction for {E}instein-{R}osen waves coupled to a massless
  scalar field.
\newblock {\em Phys. Rev. D}, 94(2):024059, 12, 2016.

\bibitem{Tod.shell}
K.~P. Tod.
\newblock The hoop conjecture and the {G}ibbons-{P}enrose construction of
  trapped surfaces.
\newblock {\em Classical Quantum Gravity}, 9(6):1581--1591, 1992.

\bibitem{Yu.Maxwell}
P.~Yu.
\newblock Dynamical formation of black holes due to the condensation of matter
  field.
\newblock {\em arXiv:1105.5898, preprint}, 2011.

\bibitem{Yu.CMP}
P.~Yu.
\newblock Energy estimates and gravitational collapse.
\newblock {\em Comm. Math. Phys.}, 317(2):273--316, 2013.

\end{thebibliography}

\end{document}